\documentclass[11pt,a4paper,reqno]{amsart}

\usepackage{tikz}
\usepackage{enumerate}
\usepackage{aliascnt}
\usepackage{hyperref}
\usepackage{amsfonts}
\usepackage{mathrsfs}
\usepackage{amsmath,multirow}
\usepackage{amsmath,graphics}
\usepackage{amssymb}
\usepackage{indentfirst,latexsym,bm,amsmath,pstricks,amssymb,amsthm,graphicx,fancyhdr,float,color}
\usepackage[all]{xy}
\usepackage{amsthm}
\usepackage{amscd}
\usepackage[all]{hypcap}

\newtheorem{theorem}{Theorem}[section]

\newaliascnt{lemma}{theorem}
\newtheorem{lemma}[lemma]{Lemma}
\aliascntresetthe{lemma}

\newaliascnt{conjecture}{theorem}

\aliascntresetthe{conjecture}

\newaliascnt{proposition}{theorem}
\newtheorem{proposition}[proposition]{Proposition}
\aliascntresetthe{proposition}

\newaliascnt{corollary}{theorem}
\newtheorem{corollary}[corollary]{Corollary}
\aliascntresetthe{corollary}

\newaliascnt{problem}{theorem}

\aliascntresetthe{problem}

\newaliascnt{question}{theorem}
\newtheorem{question}[question]{Question}
\aliascntresetthe{question}

\newaliascnt{claim}{theorem}
\newtheorem{claim}[claim]{Claim}
\aliascntresetthe{claim}

\newtheorem*{PoincareProblem}{Poincar\'e Problem}

\newtheorem*{PainleveProblem}{Painlev\'e Problem}

\theoremstyle{definition}

\newaliascnt{definition}{theorem}
\newtheorem{definition}[definition]{Definition}
\aliascntresetthe{definition}

\newaliascnt{example}{theorem}
\newtheorem{example}[example]{Example}
\aliascntresetthe{example}

\newaliascnt{assumption}{theorem}

\aliascntresetthe{assumption}

\theoremstyle{remark}

\newaliascnt{remark}{theorem}
\newtheorem{remark}[remark]{Remark}
\aliascntresetthe{remark}

\newaliascnt{remarks}{theorem}

\aliascntresetthe{remarks}

\numberwithin{equation}{section}

\numberwithin{figure}{section}

\def\wt{\widetilde}
\def\ol{\overline}

\def\lra{\longrightarrow}

\def\div{\text{\rm{div\,}}}

\def\({$($}
\def\){$)$}
\def\chit{\chi_{\rm top}}
\def\bbp{\mathbb P}

\def\cald{\mathcal D}

\def\call{\mathcal L}

\def\calo{\mathcal O}

\def\Pic{\text{{\rm Pic\,}}}

\DeclareMathOperator*{\vol}{\ensuremath{vol}}

\newcommand{\df}{\mathrm{d}f}
\newcommand{\tang}{\mathrm{tang}}
\newcommand{\sF}{\mathcal{F}}
\newcommand{\sG}{\mathcal{G}}
\newcommand{\cs}{\mathrm{CS}}

\def\kod{\mathrm{Kod}}

\begin{document}

	\title{The Poincar\'e Problem for a foliated surface}	
	
	\author{Xin L\"u}
	
	\address{School of Mathematical Sciences,  Key Laboratory of MEA(Ministry of Education) \& Shanghai Key Laboratory of PMMP, East China Normal University, Shanghai 200241, China}
	
	\email{xlv@math.ecnu.edu.cn}

		\author{Sheng-Li Tan}
	
	\address{School of Mathematical Sciences,  Key Laboratory of MEA(Ministry of Education) \& Shanghai Key Laboratory of PMMP,  East China Normal University, Shanghai 200241, China}
	
	\email{sltan@math.ecnu.edu.cn}
	
	\thanks{This work is supported by Shanghai Pilot Program for Basic Research,
		National Natural Science Foundation of China,  Fundamental Research Funds for central Universities,
		and Science and Technology Commission of Shanghai Municipality (No. 22DZ2229014)}
	
	
	
	
	
	

	\begin{abstract}
		Let $\sF$ be a foliation on a smooth projective surface $S$ over the complex number $\mathbb{C}$.
		We introduce three birational non-negative invariants $c_1^2(\mathcal F)$,
		$c_2(\mathcal F)$ and $\chi(\sF)$, called the Chern numbers.
		If the foliation $\sF$ is not of general type,
		the first Chern number $c_1^2(\mathcal F)=0$, and $c_2(\sF)=\chi(\sF)=0$ except when $\sF$ is induced by a non-isotrivial fibration of genus $g=1$.
		If $\sF$ is of general type,
		we obtain a slope inequality when $\sF$ is algebraically integral.
		As a corollary, $\sF$ is always transcendental if the slope is less than $2$.
		On the other hand, we also prove three sharp Noether type inequalities
		if $\sF$ is of general type.
		As applications, we obtain a criterion for foliations to be transcendental
		using Noether type inequalities,
		and we also give a partial positive answer to the question on the lower bound on the volume of a foliation of general type.
%
	\end{abstract}
	
		\maketitle
	
	\setcounter{tocdepth}{2}
	\tableofcontents
	
	\section{Introduction}\label{sec-intro}
	The main purpose of this paper tries to solve the following problem of Poincar\'e
	on the algebraic integrability of a foliation $\sF$ on a smooth projective surface $S$ over the complex number $\mathbb{C}$.
	 \begin{PoincareProblem}
	 	Is it possible to decide whether a foliation $\sF$ on a smooth projective surface $S$ is algebraically integral or not?
	 \end{PoincareProblem}
	 
	 Historically, Darboux \cite{dar-78,dar-78-2}, Poincar\'e \cite{poi-85,poi-91-diff} and Painlev\'e \cite{pain-92} etc. studied the foliation given by a differential equation
	 of first order on the affine plane $\mathbb{A}^2$,
	 \begin{equation*}
	 \dfrac{dy}{dx}=\dfrac{p(x,y)}{q(x,y)}, \qquad \text{or equivalently,}\quad \omega:=p(x,y)dx-q(x,y)dy=0.
	 \end{equation*}
	 where $p$ and $q$ are relatively coprime polynomials.  Hilbert
	 listed this kind of study as {\it Hilbert's 16th Problem} in
	 \cite{hil-02}. One of the central problems is about the above Poincar\'e problem on the algebraic integrability
	 of the above differential equation, namely to find criterions so that the above differential
	 equation has a rational first integral.
	 If $f(x,y)$ is a rational first integral of the above differential equation, then the integral curves defined by
	 $$f(x,y)= \text{constant},$$
	 give a pencil of algebraic curves.
	 Because of this, the differential equation is {\it algebraically integral} if it admits a rational first integral; otherwise we call it {\it non-algebraic} or {\it transcendental}.
	 If $\sF$ is algebraically integral, then it gives a pencil of algebraic curves of the genus $g$, in which case we will call $g$ the genus of $\sF$ for convenience.
	 \begin{PainleveProblem}
	 Is it possible to recognize the genus $g$ of an algebraically integral foliation $\sF$?
      \end{PainleveProblem}
	 
	 Poincar\'e suggested to study the above two problems on the projective plane $\mathbb{P}^2$, or more generally on a smooth projective surface $S$.
	 By a {\it foliated surface $(S,\sF)$} we mean a foliation $\sF$ on a  smooth projective surface $S$.
	 It can be defined by a collection of one-forms
	 $\omega_i\in \Omega^1_S(U_i)$ with at most isolated zeros
	 and there exists $f_{ij} \in \mathcal{O}^*_S(U_i\cap U_j)$ whenever $U_i\cap U_j \neq \emptyset$ such that
	 \begin{equation*}
	 	\omega_i|_{U_i\cap U_j}=f_{ij}\omega_j|_{U_i\cap U_j}.
	 \end{equation*}
	 Equivalently, we can say that $\sF$ is defined by a twisted one form
	 $\omega \in H^0\big(S, \Omega_S^1(\mathcal{N}_{\sF})\big)$,
	 where $\mathcal{N}_{\sF}$ is the normal bundle of $\sF$, see \autoref{sec-pre} for details.
	 
	 Let $(S,\sF)$ be a foliated surface, a theorem of Darboux-Jouanolou-Ghys \cite{dar-78,jou-78,ghy-00} says roughly that if $\sF$ admits sufficiently many invariant irreducible algebraic curves, then the foliation $\sF$ is algebraically integral.
	 This gives a solution to the Poincar\'e problem in some sense.
	 As already noticed by Poincar\'e \cite{poi-91-diff}, it is difficult to obtain the algorithm of that invariant curves, even for foliations on $\bbp^2$.
	 In \cite{per-02}, Pereira obtained a nice solution to Poincar\'e's and Painlev\'e's problems by showing that the degree of a first integral for a plane foliation $(\bbp^2,\sF)$ is bounded from above by the degree and pluri-genera of $\sF$.
	 On the other hand, Lins Neto \cite{lin-02} constructed families of foliations on $\bbp^2$, where the analytic type of each singularity is constant but one cannot decide the algebraic integrability.
	 Furthermore, the genera of algebraically integral foliations in the families are either not
	 bounded.
	 This shows in particular that there might be some obstruction to Poincar\'e's and Painlev\'e's problems, and hence one has to find more invariants
	 (should be global) in order to solve these two problems.
	 
	 In the last several decades, the foliation theory has attracted more and more attention in algebraic geometry, especially in birational geometry.
	 One can even consider foliations on higher dimensional variety, not only on surfaces.
	 For a foliation $\sF$ on a normal variety $X$, it is possible to define its canonical divisor $K_{\sF}$.
	 The canonical divisor $K_{\sF}$ of a foliation plays a similar role in many aspects as the canonical divisor $K_X$ of a variety.
	 For instance, the criterion on the uniruledness of foliated varieties with non-pseudo-effective canonical divisors by Miyaoka and Campana-P\u{a}un, \cite{miy-87,cp-19},
	 the minimal model theory on foliated surfaces due to Brunella and McQuillan etc. \cite{bru-99,mcq-08},
	 and the deformation invariance of the pluri-genera by Cascini-Floris \cite{cf-18} and Liu \cite{liu-19}.
	 There have been also many works on the adjoint canonical divisors recently;
	 see for example \cite{ps-19,ss-23,chlmssx}.
	 In particular, Pereira-Svaldi \cite{ps-19} proved a cohomological characterization of foliations with non-pseudo-effective divisors.
	 Recently, the minimal model program for foliated varieties was introduced in the beautiful paper \cite{cas-21}.
	 It stimulates a new project in birational geometry to generalize everything to foliated varieties.
	 
	 The algebraic integrability of a foliated surface is clearly a birational property.
	 Hence it is reasonable to study the Poincar\'e problem
	 using the birational invariants of foliated surfaces.
	 In fact, the interaction of the birational geometry with surface foliations have also made great success in 
	 the Poincar\'e problem.
	 The following criterion of Miyaoka \cite{miy-87} (see also \cite{bm-16,cp-19})
	 solves the Poincar\'e problem in the case when $K_{\sF}$ is not pseudo-effective.
	 \begin{theorem}[Miyaoka]\label{thm-miyaoka}
	 	Let $(S,\sF)$ be a foliated surface with canonical divisor $K_{\sF}$.
	 	Suppose that $K_{\sF}$ is not pseudo-effective. Then $\sF$ is induced by a family of rational curves. In particular, it is always algebraically integral.
	 \end{theorem}
	 
	 Motivated by the formulas for the modular invariants of families of
	 algebraic curves (cf. \cite{Tan1996, Tan2010}),
	 we will first introduce three new birational invariants $c_1^2(\sF), c_2(\sF)$ and $\chi(\sF)$ for a foliated surface $(S,\sF)$.
	 The three invariants are all non-negative rational numbers satisfying the following properties.
	 \begin{theorem}[{=\autoref{def-chern-numbers}+\autoref{thm-4-2}+\autoref{cor-4-1}+\autoref{prop-invariant-c_2}+\autoref{thm813b}}]\label{thm-chern-1}
	 \mbox{~}	
	 Let $(S,\sF)$ be a foliated surface.
	 	We can define its Chern numbers $c_1^2(\sF), c_2(\sF)$ and $\chi(\sF)$,
	 	which are non-negative rational numbers.
	 	\begin{enumerate}[$($i$)$]
	 	    \item The three Chern numbers are birational invariants.
	 	    \item The three Chern numbers satisfy the Noether equality:
	 	    \begin{equation}\label{eqn-chern-1-1}
	 	    	12\chi(\sF) = c_1^2(\sF)+c_2(\sF).
	 	    \end{equation}
	 	    \item The first Chern number $c_1^2(\sF)=\vol(\sF)$, where $\vol(\sF)$ is the volume of the foliation $\sF$, whose definition is given in \autoref{def-birational}.
	 	    \item Suppose that the foliation $\sF$ is algebraically integral, i.e., $\sF$ is induced by a family $f:\,S \to B$ of curves. Then
	 	    \begin{equation}\label{eqn-chern-1-2}
	 	    	c_1^2(\sF)=\kappa(f),\qquad c_2(\sF)=\delta(f),\qquad \chi(\sF)=\lambda(f),
	 	    \end{equation}
	 	    where $\kappa(f),\,\delta(f)$ and $\lambda(f)$ are the modular invariants of $f$,
	 	    whose definitions will be recalled in \autoref{sec-alg}.
	 	\end{enumerate}	 	
	 \end{theorem}
 
Foliations of non-general type are classified, cf. \cite{bru-03,bru-04,mcq-08}.
 We determine the Chern numbers of foliations of non-general type as follows.
 \begin{theorem}\label{thm-chern-2}
 	Let $(S,\sF)$ be a foliated surface.
 	Suppose that $\sF$ is not of general type, i.e., its Kodaira dimension $\kod(\sF)\in \{-\infty,0,1\}$.
 \begin{enumerate}[$(i)$]
 	\item The first Chern number $c_1^2(\sF)=0$.
 	\item The other two Chern numbers $c_2(\sF)=\chi(\sF)=0$ except when $\sF$ is induced by a non-isotrivial fibration of genus $g=1$.
 \end{enumerate} 
 \end{theorem}

As a direct consequence, we obtain the following criterion on the algebraic integrability of a foliated surface $(S, \sF)$ of non-general type.
\begin{corollary}\label{cor-chern-1}
	Let $(S,\sF)$ be a foliated surface.
	Suppose that $\sF$ is not of general type, and that $c_2(\sF)>0$ or $\chi(\sF)>0$.
	Then $\sF$ is algebraically integral.
\end{corollary}

Lins Neto's \cite{lin-02} famous counterexamples with $d=2,3$ or $4$
are foliations of non-general type with $c_2(\sF)=\chi(\sF)=0$ by \cite{llt-23}.
In particular, we can not determine the algebraic integrability of foliated surfaces with $c_2(\sF)=\chi(\sF)=0$ using the Chern numbers.
Combined with the above corollary, the Poincar\'e problem is completely solved for foliated surface $(S, \sF)$ of non-general type.

We next turn to the case when $\sF$ is of general type.
In this case, $c_1^2(\sF)=\vol(\sF)>0$, and hence $\chi(\sF)>0$ due to the Noether equality \eqref{eqn-chern-1-1} and the non-negativity of $c_2(\sF)$.
We may thus define the {\it slope} of $\sF$ as
$$\lambda(\sF):=\frac{c_1^2(\sF)}{\chi(\sF)}=\frac{\vol(\sF)}{\chi(\sF)}.$$
According to the non-negativity of the Chern numbers and the Noether equality \eqref{eqn-chern-1-1},
it holds that $0<\lambda(\sF) \leq 12$.
We are interested in a positive lower bound on the slope $\lambda(\sF)$.
Cornalba-Harris-Xiao's slope inequality \cite{ch-88,xia-87} on fibred surface together with \eqref{eqn-chern-1-2} gives a positive lower bound on the slope $\lambda(\sF)$ when $\sF$ is algebraically integral.
\begin{theorem}\label{thm-chern-3}
	Let $(S,\sF)$ be a foliated surface.
	Suppose that $\sF$ is of general type and algebraically integral,  i.e., $\sF$ is induced by a non-isotrivial family $f:\,S \to B$ of curves of genus $g\geq 2$. Then
	\begin{equation}\label{eqn-chern-3}
		\lambda(\sF) \geq \frac{4(g-1)}{g},\qquad\text{or equivalently,}\qquad
		c_1^2(\sF)=\vol(\sF)\geq \frac{4(g-1)}{g} \chi(\sF).
	\end{equation}
\end{theorem}
As an application, we obtain a partial solution to both Poincar\'e's and Painlev\'e's problems for foliations of general type with lower slope.	 
\begin{corollary}\label{cor-chern-2}
	Let $(S,\sF)$ be a foliated surface, and suppose that $\sF$ is of general type.
	\begin{enumerate}[$($i$)$]
		\item Suppose that $\lambda(\sF)<4$ and that $\sF$ is algebraically integral,  i.e., $\sF$ is induced by a non-isotrivial family $f:\,S \to B$ of curves of genus $g\geq 2$. Then
		\begin{equation}\label{eqn-chern-4}
			g \leq \frac{4}{4-\lambda(\sF)}.
		\end{equation}
		\item If $\lambda(\sF)<2$, then $\sF$ is transcendental.
	\end{enumerate}
\end{corollary}

	It is well-known that the bound \eqref{eqn-chern-4} is sharp, and the equality can be achieved by semi-stable fibrations for every genus $g\geq 2$.
	Moreover, it is not difficult to construct algebraically integral foliations with $\lambda(\sF)=4$ but the genus $g$ can be arbitrarily large, cf. \cite{xia-87}.
	On the other hand, we will construct a transcendental foliation of general type (cf. \autoref{exam-slope<2}) with $\lambda(\sF)=\frac{12}{7}<2$.
	However, the sharp lower bound on the slope $\lambda(\sF)$ is still mysterious to us for a transcendental foliation $\sF$ of general type.
	We tend to guess $\lambda(\sF)\geq 1$, i.e., $c_1^2(\sF) \geq \chi(\sF)$.
	
	\autoref{cor-chern-2} shows in particular that the Poincar\'e problem is solved when the slope $\lambda(\sF)<2$ for foliations of general type.
	Moreover, if $\lambda(\sF)<4$ and $\sF$ is algebraically integral, one obtains an upper bound on the genus $g$, i.e., Painlev\'e's problem can be also solved.
	In fact, by \cite{lz-18}, any algebraically integral foliation with $\lambda(\sF)<4$ must be induced by a hyperelliptic fibration if $g\geq 16$.
	Hence such a foliation has a very special structure: it is a double cover of a $\bbp^1$-fibration.
	It is reasonable to wonder whether the Poincar\'e problem can be solved for foliations with $\lambda(\sF)<4$?
	Namely, besides the slope, can one construct more invariants, so that the algebraic integrability of $\sF$ can be completely determined if $\lambda(\sF)<4$?
	For instance, if $\lambda(\sF)<\frac83$, then $\sF$ is either transcendental or algebraically integral of genus $g=2$.
	The rich geometry of genus-$2$ fibrations (cf. \cite{xiao-85}) might help us to characterize the algebraic integrability of foliations with $\lambda(\sF)<\frac83$.
	On the other hand, Lins Neto's famous counterexamples \cite{lin-02} show that there indeed exist obstructions to determine the algebraic integrability in general.
	In fact, it has been shown in \cite{llt-23} that the foliations of general type constructed by Lins Neto have slope $\lambda(\sF)\geq 6$.
	It is natural to ask whether there exist counterexamples to Poincar\'e's problem (as well as Painlev\'e's problem) with $\lambda(\sF)<6$.
	More precisely, we would like to propose the following question.
	\begin{question}
		What is the critical slope $\lambda_0$, such that the algebraic integrability can be determined if $\lambda(\sF)<\lambda_0$, and that there exist counterexamples if $\lambda(\sF)\geq \lambda_0$?
	\end{question}
	
	\vspace{2mm}
	Another approach to Poincar\'e's problem for foliations of general type is the Noether type inequalities.
	The classical Noether inequality \cite{noether} asserts that
	\begin{equation}\label{eqn-1-1}
		\vol(S) \geq 2p_g(S)-4,
	\end{equation}
	for every complex smooth projective surface $S$ of general type.
	Here we recall that the volume $\vol(S)$ and the geometric genus $p_g(S)$ of $S$ are defined as follows.
	Let $K_S$ be the canonical divisor of $S$.
	Then
	$$\begin{aligned}
		p_g(S)&\,=\dim H^0(S, K_S);\\
		\vol(S)&\,=\limsup_{n\to +\infty} \frac{\dim H^0(S, nK_S)}{n^2/2}.
	\end{aligned}$$
	These are two important birational invariants of $S$.
	If $S$ is minimal, then the volume $\vol(S)$ is equal to the intersection number $K_S^2$,
	which is also the first Chern number of $S$.
	The Noether inequality \eqref{eqn-1-1} is one of the fundamental inequality in the surface theory.
	Minimal surfaces of general type with the equality are usually said on the Noether line, which has been systematically studied by Horikawa \cite{hor-76}.
	
	In a series of nice works \cite{kob-92,mchen-04,ccz-06,mchen-07,cc-15,ch-17,ccj-201,ccj-202},
	the following sharp Noether inequality has been established for every minimal $3$-fold $X$ of general type
	with $p_g(X)\leq 4$ or $p_g(X) \geq 11$:
	\begin{equation*}
		\vol(X) \geq \frac{4}{3}\,p_g(X)-\frac{10}{3}.
	\end{equation*}
	It is proved in \cite{cj-17} that the Noether type inequality holds also in higher dimension:
	there exist positive numbers $a_d$ and $b_d$, depending only on the dimension $d$ of the variety $X$,
	such that
	$$\vol(X) \geq a_d\,p_g(X)-b_d.$$
    Inspired by the Noether type inequalities for algebraic varieties,
    it is natural to ask the following question.
    \begin{question}\label{question-1}
    	Do there exist Noether type inequalities for a foliation of general type?
    \end{question} 
    
    We will settle the above question in dimension two.
    \begin{theorem}\label{thm-main}
    	Let $(S,\sF)$ be a foliated surface, and suppose that $\sF$ is of general type.
    	Let $p_g(\sF)$ be its geometric genus, and $\varphi$ be the canonical map as in \autoref{sec-2-3}.
    	\begin{enumerate}[$(i)$]
    	\item The following Noether inequality holds
    		\begin{equation}\label{eqn-noe1}
    			c_1^2(\sF)=\vol(\sF) \geq p_g(\sF)-2.
    		\end{equation}
    		Moreover, if the equality holds in \eqref{eqn-noe1}, then the canonical map $\varphi$ defines a birational map whose image is a surface of minimal degree $($equal to $p_g(\sF)-2)$ in $\bbp^{p_g(\sF)-1}$.
    		
    	\item Suppose that $c_1^2(\sF)=\vol(\sF) \neq p_g(\sF)-2$ with $p_g(\sF)>0$. Then
    	\begin{equation}\label{eqn-noe2}
    		c_1^2(\sF)=\vol(\sF) \geq p_g(\sF)-2+\frac{1}{p_g(\sF)}.
    	\end{equation}
    	Moreover, if the equality holds in \eqref{eqn-noe2}, then the image $\Sigma=\varphi(S)$ is of one-dimensional.
    	
    	\item Suppose moreover that the foliation $\sF$ is algebraically integral. Then
    	\begin{equation}\label{eqn-noe3}
    	c_1^2(\sF)=\vol(\sF) \geq p_g(\sF)-\frac32+\frac{3}{2\big(2p_g(\sF)+1\big)}.
    	\end{equation}
    	Moreover, if the equality holds in \eqref{eqn-noe3}, then the image $\Sigma=\varphi(S)$ is of one-dimensional.
    	\end{enumerate}
    \end{theorem}

Applying the third Noether type inequality \eqref{eqn-noe3} to the Poincar\'e problem, we get
    \begin{corollary}\label{cor-noether}
    	Let $(S,\sF)$ be a foliated surface, and suppose that $\sF$ is of general type.
    	If $$c_1^2(\sF)=\vol(\sF)< p_g(\sF)-\frac32+\frac{3}{2\big(2p_g(\sF)+1\big)},$$ then $\sF$ is transcendental.
    \end{corollary}

Based on the first two Noether type inequalities \eqref{eqn-noe1} and \eqref{eqn-noe2},
we can also give a partial answer to the question
on a positive lower bound on the volume of a foliation of general type,
cf. \cite[\S\,3]{cas-21} and \cite[Question\,4]{hl-21}.
\begin{corollary}\label{cor-volume}
	Let $(S,\sF)$ be a foliated surface.
	Suppose that $\sF$ is of general type with $p_g(\sF)\geq 2$. Then
	$$c_1^2(\sF)=\vol(\sF)\geq \frac12.$$
\end{corollary}
    
    \begin{remark}
	(i)\, We will construct examples in \autoref{sec-example}
	reaching the equalities in \eqref{eqn-noe1}, \eqref{eqn-noe2} and \eqref{eqn-noe3},
	which show that all the three Noether type inequalities in \autoref{thm-main} are sharp.
	In particular, taking $n=2$ in \autoref{exam-6-2}, one obtains
	an example reaching the equality in \autoref{cor-volume}.
%
	
	(ii)\, The inequality \eqref{eqn-noe2} will be referred as the second Noether inequality for foliations of general type.
	This is slightly different from the second Noether inequality for
	algebraic varieties $X$ of general type.
	When $X$ is of dimension two, the volume is an integer, and hence
	$$\vol(X)\geq 2p_g(X)-3=(2p_g(X)-4)+1,\qquad\text{if $\vol(X) \neq 2p_g(X)-4$.}$$
	When $X$ is a threefold, Hu-Zhang \cite{hz-22} proved that if $p_g(X) \not \in [5,10]$ and $\vol(X) \neq \frac43p_g(X)-\frac{10}{3}$, then
	$$\vol(X)\geq \frac{4}{3}p_g(X)-\frac{19}{6}
	=\Big(\frac43p_g(X)-\frac{10}{3}\Big)+\frac16.$$	
	Such a phenomenon is different for foliations of general type.
	Indeed, we will construct in \autoref{exam-6-2} a sequence of foliated surfaces $(S_n,\sF_n)$ reaching the equality in \eqref{eqn-noe2}; namely 
	$\vol(\sF_n) > p_g(\sF_n)-2$, but the difference $\vol(\sF_n) - \big( p_g(\sF_n)-2 \big)$ can be arbitrarily close to zero.
\end{remark}

Both the slope inequalities and Noether type inequalities can give lower bounds on the volume $\vol(\sF)$ in terms of the Chern number $\chi(\sF)$ and the geometric genus $p_g(\sF)$ respectively.
It is natural to ask the relation between the two birational invariants $\chi(\sF)$ and $p_g(\sF)$.
Combining the Noether inequality \eqref{eqn-noe1} together with  \eqref{eqn-chern-1-1} and the non-negativity of $c_2(\sF)$, one obtains immediately the following.
\begin{corollary}
	Let $(S,\sF)$ be a foliated surface, and suppose that $\sF$ is of general type.
	Then
	\begin{equation}\label{eqn-1-4}
		p_g(\sF) \leq 12\chi(\sF)+2.
	\end{equation}
\end{corollary}

\noindent Conversely, the Chern number $\chi(\sF)$ (as well as the other two Chern numbers $c_1^2(\sF)$ and $c_2(\sF)$) can not be bounded from above by the geometric genus $p_g(\sF)$ in general; see \autoref{exam-6-4}.

	\vspace{2mm}
	The paper is organized as follows.
	In \autoref{sec-pre}, we review some basic facts about foliated surfaces.
	In \autoref{sec-zariski}, we recall the Zariski decomposition of the canonical divisor $K_{\sF}$ for a relatively minimal foliation, and prove in \autoref{prop-3-1} an upper bound on the coefficients $a_C$'s of the negative part $N=\sum a_CC$.
	In \autoref{sec-alg-foliation}, we restrict ourselves to the algebraically integral foliations, namely foliations with rational first integrals.
	In particular, we will present the resolution of non-reduced singularities on an algebraically integral foliation, which will be helpful in obtaining the Noether inequalities for an algebraically integral foliation.
	In \autoref{sec-chern-number}, we will introduce the Chern numbers of foliations with pseudo-effective canonical divisors, and prove some basic properties of these Chern numbers.
	In particular, \autoref{thm-chern-1} will be proved in this section.
	In \autoref{sec-non-general-type}, we compute the Chern numbers of foliations which are not of general type, and prove \autoref{thm-chern-2} and \autoref{cor-chern-1}.
	In \autoref{sec-slope}, we consider the slope of foliations of general type.
	In particular, we will prove \autoref{thm-chern-3} and \autoref{cor-chern-2}.
	We also construct an example of transcendental foliation of general type with slope $\lambda(\sF)=\frac{12}{7}<2$ in this section.
	In \autoref{sec-noether}, we prove the three Noether type inequalities in \autoref{thm-main}; while both \autoref{cor-noether} and \autoref{cor-volume}
	follow directly from the Noether type inequalities.
	Finally, we will construct in \autoref{sec-example} several examples reaching the equalities in the three Noether type inequalities in \autoref{thm-main},
	which illustrate the sharpness of those inequalities.

	\section{Preliminaries}\label{sec-pre}
	In this section, we recall some basic facts about the foliations on smooth surfaces. For more details we refer to \cite{bru-04,mcq-08}.
	
	
	\subsection{Basic definitions}
	We work over the complex number $\mathbb{C}$.
	By a {\it foliated surface} $(S,\sF)$ we mean a foliation $\sF$ on the smooth projective surface $S$ over the complex number $\mathbb{C}$.
		\begin{definition}\label{def-foliation}
			A foliation $\mathcal{F}$ on a smooth projective surface $S$ over the complex number $\mathbb{C}$ is defined by a saturated invertible subsheaf $T_{\mathcal{F}} \subseteq T_S$ in the tangent sheaf $T_{S}$,
			i.e.,  $T_S/T_{\mathcal{F}}$ is torsion free.
		\end{definition}
		
	Equivalently, a foliation $\sF$ on $S$ can be given by an exact sequence
	$$0\lra T_{\sF} \lra T_S  \lra I_{\Delta}(N_{\sF}) \lra 0,$$
	where $T_{\sF}$ and $N_{\sF}$ are respectively called the {\it tangent bundle} and {\it normal bundle} of $\sF$,
	and $I_{\Delta}$ is an ideal sheaf.
	In other words,
	a foliation on $S$ is given by the data $\{(U_i, v_i)\}_{i\in I}$,
	where $\{U_i\}_{i\in I}$ is an open covering of $S$,
	$v_i$ is a holomorphic vector field on $U_i$ with at most isolated zeros,
	and there exists $g_{ij} \in \mathcal{O}^*_S(U_i\cap U_j)$ whenever $U_i\cap U_j \neq \emptyset$ such that
	\begin{equation}\label{eqn-2-8}
		v_i|_{U_i\cap U_j}=g_{ij}v_j|_{U_i\cap U_j}.
	\end{equation}
	The cocycle $\{g_{ij}\}$ defines a line bundle which is nothing but the dual of the tangent bundle $T_{\sF}^*$.
    The bundle $T_{\sF}^*$ will be called the {\it canonical divisor} of $\mathcal{F}$, and we denote it by $K_{\mathcal{F}}:=T_{\sF}^*$.
	
	Alternatively, one can also define $\sF$ using one-forms instead of vector fields, as studied by Poincar\'e and Painlev\'e etc. over one hundred years ago.
	A foliation on $S$ is given by a collection of one-forms
	$\omega_i\in \Omega^1_S(U_i)$ with at most isolated zeros
	and there exists $f_{ij} \in \mathcal{O}^*_S(U_i\cap U_j)$ whenever $U_i\cap U_j \neq \emptyset$ such that
	\begin{equation}\label{eqn-2-9}
		\omega_i|_{U_i\cap U_j}=f_{ij}\omega_j|_{U_i\cap U_j}.
	\end{equation}
	The cocycle $\{f_{ij}\}$ defines a line bundle which is the conormal bundle $N_{\sF}^*$.
	One can also translate this into an exact sequence:
	$$0\lra N_{\sF}^* \lra \Omega_S^1  \lra I_{\Delta}(K_{\sF}) \lra 0.$$
	We remark that a foliation is classically defined by a
	nonzero global meromorphic 1-form $\alpha$ on $S$.
	 The one-dimensional zeros and poles of $\alpha$ define a divisor
	$$\textrm{div}(\alpha)=(\alpha)_0-(\alpha)_\infty\,.$$
	Let $s$ be the canonical nonzero meromorphic section of $\mathcal
	O_S(\textrm{div}(\alpha))$ such that
	$\textrm{div}(s)=\textrm{div}(\alpha)$, and let $\alpha=s\cdot
	\omega$. Then $\omega$ is a twisted holomorphic 1-form $\omega$ in
	$\Omega^1_S(-\textrm{div}(\alpha))$ with isolated zeros,
	which coincides with the above definitions satisfying $\mathcal{N}_{\sF}^*=\mathcal{O}_S\big(\div(\alpha)\big)$.
	For any surjective morphism $\Pi:\,\wt{S} \to S$ from another smooth projective surface $\wt{S}$,
	there is an induced foliation $\wt{\sF}=\Pi^*(\sF)$ defined by
	$\Pi^*(\alpha)$ if $\alpha$ is a nonzero global meromorphic $1$-form defining $\sF$.
	
	\vspace{2mm}
	An irreducible curve $C\subseteq S$ is said to be {\it $\sF$-invariant} if the inclusion
	$T_{\sF}|_{C} \hookrightarrow T_X|_{C}$ factors through $T_C$.
	By a curve $C\subseteq S$ we mean a reduced and compact algebraic curve.
	So it might be singular and reducible.
	Suppose that $C$ is not $\sF$-invariant, or more precisely every irreducible component of $C$ is not $\sF$-invariant.
	Then one defines the tangency of $\sF$ to $C$ as follows.
	Let $p\in C$ be any point.
	Around $p$, let $\{f=0\}$ be a local equation of $C$,
	and $v$ be a local holomorphic vector field defining $\sF$ around $p$.
	Then the tangency of $\sF$ to $C$ at $p$ is defined to be
	$$\tang(\sF,C,p)=\dim_{\mathbb{C}} \frac{\mathcal{O}_{S,p}}{\langle f, v(f) \rangle}.$$
	As $C$ is not $\sF$-invariant, $\tang(\sF,C,p)<+\infty$ and  $\tang(\sF,C,p)=0$ except for finitely many points.
	Hence one defines the tangency of $\sF$ to $C$.
	$$\tang(\sF,C) = \sum_{p\in C} \tang(\sF,C,p).$$
	\begin{proposition}[{\cite[Proposition\,2.2]{bru-04}}]\label{prop-2-1}
		Let $C$ be a curve on $S$ which is not $\sF$-invariant. Then
		$$\tang(\sF,C)=K_{\sF}C+C^2.$$
		In particular,
		\begin{equation*} 
		K_{\sF}C+C^2\geq 0.
		\end{equation*}
	\end{proposition}
	
	We now suppose that $C$ is $\sF$-invariant, or more precisely every irreducible component of $C$ is $\sF$-invariant.
	Given any point $p\in C$, let $\{f=0\}$ be a local equation of $C$,
	and $\omega$ be a local holomorphic one-form defining $\sF$ around $p$.
	Because $C$ is $\sF$-invariant, we may write
	$$g\omega=hdf+f\eta,$$
	for some holomorphic one-form $\eta$ and holomorphic functions $g,h$ around $p$,
	such that $h$ and $f$ are coprime.
	We define
	$$\begin{aligned}
	Z(\sF,C,p)&\,=\text{vanishing order of $~\frac{h}{g}\Big|_C$ at $p$},\\
	\cs(\sF,C,p)&\,=\text{residue of $~-\frac{\eta}{h}\Big|_C$ at $p$}.
	\end{aligned}$$
	By definition, both $Z(\sF,C,p)$ and $\cs(\sF,C,p)$ are zero if $p$ is not a singular point of $\sF$.
	If $\sF$ is reduced, then $Z(\sF,C,p)\geq 0$ for any $p\in C$ \cite{bru-97}.
	Let
	$$\begin{aligned}
	Z(\sF,C) &\,= \sum_{p\in C} Z(\sF,C,p)=\sum_{p\in C\,\cap\, \text{Sing}(\sF)} Z(\sF,C,p),\\
	\cs(\sF,C) &\,= \sum_{p\in C} \cs(\sF,C,p)=\sum_{p\in C\,\cap\, \text{Sing}(\sF)} \cs(\sF,C,p).
	\end{aligned}$$
	\begin{proposition}[{\cite[Proposition\,2.2 and Theorem\,3.2]{bru-04}}]\label{prop-2-2}
		Let $C$ be a curve on $S$ which is $\sF$-invariant. Then
		\begin{eqnarray}
		Z(\sF,C)&=&K_{\sF}C+\chi(C), \quad \text{where~}\chi(C)=-K_SC-C^2; \nonumber\\
		\cs(\sF,C) &=&C^2. \label{eqn-2-7}
		\end{eqnarray}
		The second equality above is called the Camacho-Sad formula.
		The first equality implies in particular that if $\sF$ is reduced and $C$ is an $\sF$-invariant curve, then
		\begin{equation*} 
		0\leq K_{\sF}C+\chi(C)=K_{\sF}C-K_SC-C^2=N_{\sF}C-C^2.
		\end{equation*}
	\end{proposition}
	
	\subsection{The singularities of a foliation}\label{sec-singularities}
	Let $(S,\sF)$ be a foliated surface.
	The singular locus of $\mathcal{F}$ is the set of points $p\in S$
	where the quotient sheaf $T_S/T_{\mathcal{F}}$
	fails to be locally free at $p$.
	The torsion-freeness of $T_S/T_{\mathcal{F}}$ implies that the codimension of the singular locus of $\sF$ in $S$ is at least two,
	i.e., it consists of finitely many points.
	Let $p$ be a singular point of $\sF$,
	and $v=A\frac{\partial}{\partial x}+B\frac{\partial}{\partial y}$ is a local vector field defining $\sF$.
	One defines the {\it multiplicity} of $\sF$ at $p$ as
	$$m_p(\sF)=\dim_{\mathbb{C}}\frac{\mathcal{O}_p}{\langle A,B \rangle}.$$
	The total number of singular points of $\sF$ is (cf. \cite[Proposition\,2.1]{bru-04}):
	\begin{equation}\label{eqn-2-4}
	m(\sF):=\sum_{p}m_p(\sF)=c_2(\Omega_S(N_\sF))=c_2(S)+N_\sF K_\sF.
	\end{equation}
	The two eigenvalues $\lambda_1,\lambda_2$ of the linear part $(Dv)(p)$
	at a singular point of $\sF$
	are well-defined up to multiplication by a non-zero constant.
	\begin{definition}\label{def-2-1}
		Let $p\in S$ be a singularity of $\sF$ and $v$ be a local vector field defining $\sF$.
		\begin{enumerate}[(i)]
			\item The singularity $p$ is called non-degenerate if both of the eigenvalues $\{\lambda_1,\lambda_2\}$ of $(Dv)(p)$ are non-zero;
			it is called a saddle-node if exactly one the two eigenvalues
			$\{\lambda_1,\lambda_2\}$ is non-zero.
			\item The singularity $p$ is called a {\it reduced singularity} if at least one of the two eigenvalues (say, $\lambda_2$) is not zero and the quotient $\lambda=\frac{\lambda_1}{\lambda_2}$ is not a positive rational number. In particular, a reduced singularity is either non-degenerate or a saddle-node.
			\item The foliation $\sF$ is said to be reduced if any singularity of $\sF$ is reduced.
		\end{enumerate}
	\end{definition}
	If $\lambda_2\neq 0$, then the quotient $\lambda=\frac{\lambda_1}{\lambda_2}$ is unchanged by multiplication of $v$ by a non-vanishing holomorphic function.
	Of course, if $\lambda_1\neq 0$, we could also consider the quotient
	$\lambda^{-1}=\frac{\lambda_2}{\lambda_1}$ instead of $\lambda$,
	but then $\lambda \not\in \mathbb{Q}^{+}$ iff $\lambda^{-1} \not\in \mathbb{Q}^{+}$.
	The complex number $\lambda=\frac{\lambda_1}{\lambda_2}$,
	with an inessential abuse due to the exchange $\lambda \leftrightarrow \lambda^{-1}$,
	is called the {\it eigenvalue} of $\sF$ at $p$ following \cite{bru-04}.
	For convenience, we always assume the eigenvalue of $\sF$ at a saddle-node is $0$.
	Given any foliation, one can obtain a reduced one by a sequence of blowing-ups, cf. \cite{sei-68} or \cite[Theorem\,1.1]{bru-04}.
	\begin{theorem}[Seidenberg]\label{thm-seidenberg}
		Given any foliated surface $(S,\sF)$, there exists a sequence of blowing-ups $\pi:\,\wt{S} \to S$,
		such that the induced foliation $\wt\sF$ on $\wt{S}$ is reduced.
	\end{theorem}

We want to introduce two new invariants for a singularity of $\sF$.
First, for a complex number $u\in \mathbb C$, we define
\begin{equation}\label{eqn-def-beta}
\left\{	
\begin{aligned}
\beta(u)&\,=\begin{cases}\frac{\gcd(a,b)^2}{ab}, &\quad\textrm{~if~}
u=\dfrac{a}{b}\in\mathbb Q\setminus\{0\},
\\
0, &\quad\textrm{~otherwise}.
\end{cases} \\[2pt]
\chi(u)&\,=\dfrac{1}{12}
\left(u+\dfrac1{u}+\beta(u)-3\right), \hskip0.5cm
\text{~if~}u\neq0.
\end{aligned}\right.\end{equation}
We can check easily that
$$\left\{\begin{aligned}
&\beta(1)=1,\qquad \beta(-1)=-1;\\
&\beta(u^{-1})=\beta(u),\quad\forall\,u\neq 0.
\end{aligned}\right.\qquad\qquad
\left\{\begin{aligned}
&\chi(1)=0,\qquad \chi(-1)=-1/2;\\
&\chi(u^{-1})=\chi(u),\quad\forall\,u\neq 0.
\end{aligned}\right.$$
Moreover, if $u\neq 0,1$, then
\begin{equation}\label{eqn-2-11}
	\beta(u)=\beta(u+1)+\beta(u^{-1}+1),\qquad
	\chi(u)=\chi(u+1)+\chi(u^{-1}+1).
\end{equation}
\begin{definition}\label{def-beta-chi}
	Let $p\in S$ be either a non-degenerate singularity, or a saddle node of $\sF$.
	Let $\lambda_p$ be the eigenvalue of $\sF$ at $p$.
	Define
	\begin{equation}\label{eqn-def-beta-2}
		\beta_p(\sF)=\beta(-\lambda_p),\qquad\qquad
		\chi_p(\sF)=-\frac1{12}\big(BB_p(\sF)+m_p(\sF)-\beta(-\lambda_p)\big),
	\end{equation}
	where $BB_p(\sF)$ is the Baum-Bott index of $\sF$ at $p$,
	and $m_p(\sF)$ is the multiplicity of $\sF$ at $p$.
\end{definition}

With the help of the equalities above,
one checks easily that these two invariants are well-defined
(invariant under the exchange $\lambda \leftrightarrow \lambda^{-1}$).
Moreover, if $p$ is a non-degenerate singularity with eigenvalue $\lambda_p$,
then $m_p(\sF)=1$ and
$BB_p(\sF)=\lambda_p+\lambda_p^{-1}+2$, cf. \cite[\S\,3]{bru-04}. Thus we have
$$\chi_p(\sF)=\chi(-\lambda_p),\qquad\text{if $p$ is a non-degenerate singularity}.$$

\begin{lemma}\label{lem-2-1}
	Let $\sigma:\,S' \to S$ be a blowing-up centered at some point $p\in S$ with exceptional curve $E$,
	and $\sF'=\sigma^*(\sF)$ the induced foliation on $S'$.
	\begin{enumerate}[$(i)$]
		\item If $p$ is a regular point, then $E$ is $\sF'$-invariant and there is a unique singularity $q_1$ of $\sF'$ on $E$ with eigenvalue $\lambda_{q_1}=-1$.
		In particular,
		$$\beta_{q_1}(\sF')=1,\qquad \chi_{q_1}(\sF')=0.$$
		
		\item If $p$ is a non-degenerate singularity with eigenvalue $\lambda_p\neq 1$,
		then $E$ is $\sF'$-invariant, and there are two non-degenerate singularities $\{q_1,q_2\}$ of $\sF'$ on $E$.
		Moreover the eigenvalues of these two singularities are $\{\lambda_{q_1},\lambda_{q_2}\}=\{\lambda_p-1,\lambda_p^{-1}-1\}$.
		In particular,
		\begin{equation}\label{eqn-2-12}
			\beta_p(\sF)=\beta_{q_1}(\sF')+\beta_{q_2}(\sF'),
			\qquad\chi_p(\sF)=\chi_{q_1}(\sF')+\chi_{q_2}(\sF').
		\end{equation}
		
		\item If $p$ is a non-degenerate singularity with eigenvalue $\lambda_p=1$ (in particular, $p$ is a non-reduced singularity),
		then $E$ can be $\sF'$-invariant or not:
		\begin{enumerate}
			\item[$(iii$-$1)$] if $E$ is $\sF'$-invariant, then there is a unique singularity $q_1$ of $\sF'$ on $E$ which is a saddle-node of multiplicity $m_{q_1}=2$ and
			\begin{equation}\label{eqn-2-15}
			\beta_p(\sF)=\beta_{q_1}(\sF')-1,\qquad \chi_p(\sF)=\chi_{q_1}(\sF')-\frac{1}{12};
		\end{equation}
			
			\item[$(iii$-$2)$] if $E$ is not $\sF'$-invariant, then there is no singularity of $\sF'$ on $E$.
		\end{enumerate}		
		
		\item If $p$ is a saddle-node of multiplicity $m_p(\sF)$, then $E$ is $\sF'$-invariant, and there are two singularities $\{q_1,q_2\}$ of $\sF'$ on $E$.
		Moreover, one of these two singularities, say $q_1$, is a saddle-node of multiplicity $m_{q_1}(\sF')=m_p(\sF)$,
		and the other one is a non-degenerate singularity with eigenvalue $\lambda_{q_2}=-1$. In particular,
		\begin{equation}\label{eqn-2-14}
		\beta_p(\sF)=\beta_{q_1}(\sF')+\beta_{q_2}(\sF')-1,
		\qquad\chi_p(\sF)=\chi_{q_1}(\sF')+\chi_{q_2}(\sF')-\frac{1}{12}.
	\end{equation}		
	\end{enumerate}
\end{lemma}
\begin{proof}
	The first statement is clear.
	We consider next $p$ is a non-degenerate singularity with eigenvalue $\lambda_p\neq 1$.
	Then the local one-form $\omega$ defining $\sF$ has the form 
	$$\omega=x\big(1+o(1)\big)dy-\lambda_py\big(1+o(1)\big)dx.$$
	Using local equations defining the blowing-up $\sigma$:
	$$(x,t) \mapsto (x,y)=(x,xt), \qquad\text{and}\qquad
	(s,y) \mapsto (x,y)=(sy,y),$$
	one shows directly that the exceptional curve $E$ is $\sF'$-invariant and
	there are two non-degenerate singularities $\{q_1,q_2\}$ of $\sF'$ on $E$
	with eigenvalues $\{\lambda_{q_1},\lambda_{q_2}\}=\{\lambda_p-1,\lambda_p^{-1}-1\}$.
	The equalities in \eqref{eqn-2-12} follows from \eqref{eqn-2-11}.
	
	If $p$ is a non-degenerate singularity with eigenvalue $\lambda_p= 1$,
	then by Poincar\'e-Dulac Normal Form Theorem (cf. the end of Chapter 1 in \cite{bru-04}), $\sF$ is locally defined by a one-form
	$$\omega=(\epsilon x+y)dx-xdy,\qquad \text{where~}\epsilon\in\{0,1\}.$$
	If $\epsilon=0$ (resp. $\epsilon=1$), then again by a direct computation one shows that
	$E$ is not $\sF'$-invariant and that there is no singularity of $\sF'$ on $E$
	(resp. that $E$ is $\sF'$-invariant and that there is a unique singularity $q_1$ of $\sF'$ on $E$ which is a saddle-node of multiplicity $m_{q_1}=2$).
	Moreover, $\beta_p(\sF)=-1$, $\beta_{q_1}(\sF')=0$,
	$BB_{p}(\sF)=4$, and $BB_{q_1}(\sF')=BB_{p}(\sF)-1=3$ according to the Baum-Bott formula (cf. \cite[Theorem\,3.1]{bru-04}).
	Hence
	$$\chi_p(\sF)=-\frac{1}{2},\qquad \chi_{q_1}(\sF')=-\frac{5}{12}.$$
	This proves \eqref{eqn-2-15}.
	
	Finally, we consider the case when
	$p$ is a saddle-node of multiplicity $m_p(\sF)$.
	In this case, the blowing-up process is exhibited in \cite[\S\,1.2]{bru-04},
	and the conclusions follow from a direct computation as before.
	We just remark
	the formula $BB_p(\sF)=BB_{q_1}(\sF')+BB_{q_2}(\sF')+1$,
	which can deduced from Baum-Bott formula (cf. \cite[Theorem\,3.1]{bru-04}).
	This completes the proof.
\end{proof}

In the proof of the above lemma, we have seen that the blowing-up of a non-degenerate singularity produces non-degenerate singularities or saddle-nodes.
By a sequence of blowing-ups, one can resolve a non-degenerate non-reduced singularity $p$.
Such a process can be found at the end of Chapter 1 in \cite{bru-04}.
It depends on the eigenvalue $\lambda_p$, which is a positive rational number by assumption.
For convenience, we briefly recall it in the following.\vspace{2mm}

(i).  If $\lambda_p\not\in \mathbb{N}^+\cup \frac{1}{\mathbb{N}^+}$,
	then $\sF$ can be generated by a local vector field
	$v=x\frac{\partial}{\partial x}+\lambda_p y\frac{\partial}{\partial y}$,
	where $\lambda_p=\frac{m_1}{m_2}\in \mathbb{Q}^{+}$ with $\gcd(m_1,m_2)=1$.
	Equivalently, $\sF$ is given by the levels of the meromorphic function $\frac{x^{m_1}}{y^{m_2}}$.
	In particular, the local separatrices passing through $p$ are defined by
	$\frac{x^{m_1}}{y^{m_2}}=c$ or $x^{m_1}-cy^{m_2}=0$, where $c\in \mathbb{C}\cup \{\infty\}$.
	There are exactly two separatrices $\{C_1,C_1'\}$ smooth at $p$,
	defined respectively by $x=0$ and $y=0$.
	Let $\sigma_1:\,S_1 \to S$ be the blowing-up centered at $p$ with exceptional curve $E_1$.
	Then $E_1$ is $\sF_1=\sigma_1^*(\sF)$-invariant, and there are exactly two singularities on $E_1$, which are $q_1=\ol C_1\cap E_1$ and $q_1'=\ol C_1' \cap E_1$, where $\ol C_1$ (resp. $\ol C_1'$) is the strict transform of $C_1$ (resp. $C_1$).
	The eigenvalue of $\{\lambda_{q_1},\lambda_{q_1'}\}=\{\lambda_p-1,\lambda_p^{-1}-1\}$.
	It follows that one singularity, say $q_1$, is reduced; and the other singularity $q_1'$ is still non-reduced with positive rational eigenvalue.
	Let $\sigma_2:\,S_2 \to S_1$ be the blowing-up centered at $q_1'$ and continue the above process. Finally, one obtains a birational map
	$\tilde \pi:\, \wt S \to S$, such that there is only reduced singularities
	of $\wt\sF=\tilde\pi^*(\sF)$ on $\tilde\pi^{-1}(p)$. Moreover, all the exceptional curves except the last one $E_k$ are $\wt \sF$-invariant,
	there is no singularity of $\wt \sF$ on the last exceptional curve $E_n$, and 
	\begin{eqnarray}
		K_{\wt \sF}&=&\tilde\pi^*(K_{\sF})-E_n,\label{eqn-2-16}\\
		\sum_{q\in \tilde\pi^{-1}(p)}\beta_{q}(\wt \sF)&=&\beta_p(\sF_0)+1.\label{eqn-7-13}
	\end{eqnarray}
	The second equality	which follows from \autoref{lem-2-1}.
	Indeed, by \eqref{eqn-2-12}, one proves inductively that
	$$\begin{aligned}
	\beta_{p}(\sF_0)&\,=\beta_{q_1}(\sF_1)+\beta_{q_1'}(\sF_1)\\
	&\,=\beta_{q_1}(\sF_2)+\beta_{q_2}(\sF_2)+\beta_{q_2'}(\sF_2)=\cdots\\
	&\,=\sum_{i=1}^{n-1}\beta_{q_i}(\sF_{n-1})+\beta_{q_{n-1}'}(\sF_{n-1})
	=\sum_{i=1}^{n-1}\beta_{q_i}(\sF_{n})-1,
	\end{aligned}$$
	where we view $q_i$ also as a singularity of $\sF_j$ for $j>i$ with the same $\beta$-invariants.
	These points $\{q_1,\cdots,q_{n-1}\}$ are all the singularities of $\sF_n=\wt \sF$ lying on $\tilde \pi^{-1}(p)$.
			
(ii). If $\lambda_p\in \mathbb{N}^+\cup \frac{1}{\mathbb{N}^+}$,
	then $\sF$ can be generated by a local vector field
	$v=x\frac{\partial}{\partial x}+(my+\epsilon x^m)\frac{\partial}{\partial y}$,
	where $m=\lambda_p$ or $\lambda_p^{-1}$ belongs to $\mathbb{N}^+$, and $\epsilon\in\{0,1\}$.
	If $\epsilon=0$, then the situation is almost the same as the previous case
	(the only difference is that there might be infinitely many separatrices smooth at $p$, and we donot know exactly the position of the two singularities on $E_1$), and we are done by a similar argument.
	All the exceptional curves except the last one $E_k$ are $\wt \sF$-invariant,
	and \eqref{eqn-2-16} still holds in this case.
	If $\epsilon=1$, then as exhibited at the end of Chapter 1 in \cite{bru-04},
	the resolution $\tilde\pi:\,\wt S\to S$ of the non-reduced singularity $p$ would create a saddle-node of multiplicity two at the last step.
	In this case, all the exceptional curves are $\wt \sF$-invariant, and it holds
	\begin{equation}\label{eqn-2-17}
	K_{\wt \sF}=\tilde\pi^*(K_{\sF}).
	\end{equation}

	\begin{lemma}\label{lem-2-2}
		Let $(S,\sF)$ be a foliated surface, $C_1,C_2$ be two $\sF$-invariant curves, and $p\in C_1\cap C_2$ be a non-degenerate singularity of $\sF$
		with eigenvalue $\lambda_p$.
		Suppose that both $C_1$ and $C_2$ are smooth at $p$ and intersect transversely at $p$.
		Then module the exchange $\lambda_p \leftrightarrow \lambda_p^{-1}$,
		it holds that
		\begin{equation}\label{eqn-2-6}
			\cs(\sF,C_1,p)=\frac{1}{\lambda_p},\qquad \cs(\sF,C_2, p)=\lambda_p.
		\end{equation}
	\end{lemma}
\begin{proof}
	If $p$ is reduced singularity, it is actually proved in \cite[\S\,3.2]{bru-04}.
	We briefly recall the proof.
	The foliation $\sF$ can be locally defined by a one-form
	$$\omega=x(1+o(1))dy-\lambda_p y(1+o(1))dx,\qquad\text{with~}\lambda_p\not\in\mathbb{Q}_{\geq 0}.$$
	The invariant $\lambda_p$ is nothing but the eigenvalue of $\sF$ at $p$.
	There are exactly two separatrices through $p$, defined respectively by $\{x=0\}$ and $\{y=0\}$, which are nothing but the two $\sF$-invariant curves
	$C_1$ and $C_2$ by our assumption.
	Recall that a separatrix of $\sF$ at $p$ is a local holomorphic irreducible curve $C$ (possibly singular) on a neighborhood of $p$ which is $\sF$-invariant and passes through $p$.
	By a direct computation, one proves \eqref{eqn-2-6}.
	
	Suppose now that $p$ is a non-reduced singularity, i.e., the eigenvalue $\lambda_p\in \mathbb{Q}^+$.
	If $\lambda_p\not\in \mathbb{N}^+\cup \frac{1}{\mathbb{N}^+}$,
	then as recalled above, $\sF$ can be generated by a local vector field
	$v=x\frac{\partial}{\partial x}+\lambda_p y\frac{\partial}{\partial y}$,
	where $\lambda_p=\frac{m_1}{m_2}\in \mathbb{Q}^{+}$ with $\gcd(m_1,m_2)=1$.
	It follows that the local separatrices passing through $p$ are defined by
	$\frac{x^{m_1}}{y^{m_2}}=c$ or $x^{m_1}-cy^{m_2}=0$, where $c\in \mathbb{C}\cup \{\infty\}$.
	There are exactly two separatrices smooth at $p$,
	defined respectively by $x=0$ and $y=0$,
	which are nothing but the two $\sF$-invariant curves
	$C_1$ and $C_2$ by our assumption.
	Again \eqref{eqn-2-6} follows by a direct computation.
	If $\lambda_p\in \mathbb{N}^+\cup \frac{1}{\mathbb{N}^+}$,
	then $\sF$ can be generated by a local vector field
	$v=x\frac{\partial}{\partial x}+(my+\epsilon x^m)\frac{\partial}{\partial y}$,
	where $m=\lambda_p$ or $\lambda_p^{-1}$, and $\epsilon\in\{0,1\}$.
	If $\epsilon=0$, then the situation is almost the same as the previous case
	(the only difference is that there might be infinitely many separatrices smooth at $p$, and hence we can not decide the local defining equations of $C_1$ and $C_2$, but this would not infect the CS-indices),
	and we are done by a similar argument.
	If $\epsilon=1$, then except the strict transform of $\{x=0\}$, all the other
	separatrices through a point on $\pi^{-1}(p)$ is actually contained in $\{x=0\}$.
	It follows that there is only one separatrix of $\sF$ through $p$,
	since the strict transform of any separatrix of $\sF$ through $p$ gives a separatrix through a point on $\pi^{-1}(p)$ and not contained in $\{x=0\}$.
	This is a contradiction, since $C_1$ and $C_2$ are already two different separatrices of $\sF$ through $p$.
\end{proof}

\begin{remark}
	If $p$ is a saddle node, one can also compute the $\cs$-indices on the local separatrices through $p$; see \cite[\S\,3.2]{bru-04}.
\end{remark}

	\subsection{Birational invariants of a foliated surface}\label{sec-2-3}
	In this subsection, we recall the birational invariants for a foliated surface $(S,\sF)$; see \cite{bru-99,bru-04,men-00} for more details.
	
	Let $(S,\sF)$ be a foliated surface, not necessarily reduced.
	Let $\sigma:\,S'\to S$ be a blowing-up centered at some point $p\in S$, and $\sF'=\sigma^*(\sF)$ the induced foliation on $S'$.
	Let $q_1,\cdots,q_r$ be the singularities of $\sF'$ on the exceptional curve $E$.
	Let $a_p$ be the {\it vanishing order} of $\omega$ at $p$, where $\omega$ is a local one-form defining $\sF$.
	The number $a_p$ is independent on the choice of $\omega$, and it is also the vanishing order of a local vector field $v$ defining $\sF$.
	\begin{definition}\label{def-vanishing-order}
		The above number $a_p$ is called the {\it vanishing order} of $\sF$ at $p$,
		and will be denoted by $a_p(\sF)$, or simply by $a_p$ if there is no confusion.
	\end{definition}
    For a singularity $p$, the vanishing order $a_p(\sF)$ is different from the multiplicity $m_p(\sF)$ in \autoref{sec-singularities}.
    However, one checks easily that $m_p(\sF)\geq a_p(\sF)^2$.
	According to \cite[\S\,2.3, Example\,(1)]{bru-04}, it holds that
	\begin{equation}\label{eqn-2-13}
	K_{\sF'}=\sigma^*K_{\sF}-(\ell_p-1)E,\quad \text{~where~}\ell_p=\left\{\begin{aligned}
	&a_p, &~&\text{if $E$ is $\sF$-invariant};\\
	&a_p+1, &~& \text{if $E$ is not $\sF$-invariant}.
	\end{aligned}\right.
	\end{equation}
	Conversely, suppose $\sF'$ is any foliation on $S'$. Then the birational morphism $\sigma$ induces a foliation
	$\sF:=\sigma_*(\sF')$ on $S$ by extending $\sF'|_{S'\setminus E}$ to $S$, where we view $S'\setminus E \cong S\setminus\{p\} \subseteq S$.
	Moreover, $K_{\sF}=\sigma_*(K_{\sF'})$.
	
	In particular, if $\sF$ is reduced, then $\ell_p=0$ or $1$, and the exceptional curve $E$ is always $\sF$-invariant.
	Hence for any birational morphism $\pi:\,S'\to S$,
	$$K_{\sF'}=\pi^*K_{\sF}+\mathcal{E},$$
	where $\mathcal{E}\geq 0$ is supported on the exceptional curves of $\pi$.
	In particular,
	$$\dim H^0(S,nK_{\sF})=\dim H^0(S',nK_{\sF'}).$$
	This stimulates the following definitions.
	\begin{definition}\label{def-birational}
		Let $(S,\sF)$ be a foliated surface, and let $(S',\sF')$ be any reduced model, i.e., $(S',\sF')$ is birational to $(S,\sF)$ and $\sF'$ is reduced.
		\begin{enumerate}[(i)]
			\item For any integer $n\geq 1$, the $n$-th pluri-genus $p_n(\sF)$ of $(S,\sF)$ is defined to be
			$$p_n(\sF):=\dim H^0(S',nK_{\sF'}).$$
			When $n=1$, it is the geometric genus, and also denoted by $p_g(\sF)=p_1(\sF)$. \vspace{2mm}
			
			\item The volume $\vol(\sF)$ and the Kodaira dimension $\kod(\sF)$ are defined as follows.
			$$\left\{\begin{aligned}
			\vol({\sF})&\,=\limsup_{n\to +\infty} \frac{\,p_n(\sF)}{n^2/2}=\limsup_{n\to +\infty} \frac{\,\dim H^0(S', nK_{\sF'})}{n^2/2};\\[3pt]
			\kod(\sF)&\,=\left\{\begin{aligned}
			&-\infty, &\quad&\text{if~$p_n(\sF)=0$ for any $n\geq 1$},\\
			&\limsup_{n\to +\infty} \frac{\log p_n(\sF)}{\log n}, &&\text{otherwise}.
			\end{aligned}\right.
			\end{aligned}\right.$$
		\end{enumerate}
	\end{definition}

By Seidenberg's \autoref{thm-seidenberg},
every foliated surface admits a reduced model by a sequence of blowing-ups.
The above definitions are independent on the choice of the reduced model,
and they are all birational invariants for a foliated surface $(S,\sF)$.
Indeed, suppose that there are two reduced models $(S',\sF')$ and $(S'',\sF'')$.
Since $(S',\sF')$ is birational to $(S'',\sF'')$, there is a foliated surface $(\wt S,\wt \sF)$ with two birational morphisms $\rho':\,\wt S \to S'$ 
and $\rho'':\,\wt S \to S''$ such that
\begin{enumerate}[(i)]
	\item the foliated surface $(\wt S, \wt \sF)$ is reduced;
	\item we have the isomorphism $(\rho')^*(\sF')\cong \wt \sF  \cong (\rho'')^*(\sF'')$.
\end{enumerate}
$$\xymatrix{  &\wt S \ar[dl]_-{\rho'} \ar[dr]^-{\rho''}&\\
S' && S''}$$
In particular,
$$\dim H^0(S',nK_{\sF'})=\dim H^0(\wt S,nK_{\wt \sF})=\dim H^0(S'',nK_{\sF''}).$$

\begin{remark}
	(i) Historically, one defined the pluri-genera (as well as the volume and Kodaira dimension) using the divisor $K_{\sF}$ on $S$ directly in some literatures, i.e., defined them to be $\dim H^0(S,nK_{\sF})$.
	However, it turns out that such definitions behave not well in birational geometry;
	they might change under the birational morphisms between foliated surfaces.
	This is one of the reasons for us to define them over a reduced model. 
	
	(ii) In the study of the minimal model program,
	people prefer to consider foliated varieties with canonical singularities,
	especially for foliations over high dimensional varieties, cf.	 \cite{mcq-08,cas-21}.
	For foliated surfaces, it makes no essential difference between those with reduced singularities and with canonical singularities
	because of Seidenberg's \autoref{thm-seidenberg}.
	Moreover, it would be more convenient to work with a reduced model to
	introduce the new birational invariants: the Chern numbers; see \autoref{sec-chern-number} for more details.
\end{remark}
	
Suppose that the canonical divisor $K_{\sF'}$ is pseudo-effective for some reduced model $(S',\sF')$.
Let $K_{\sF'}=P'+N'$ be the Zariski decomposition, where $P'$ and $N'$ are respectively the nef and negative parts. By the Riemann-Roch theorem,
	\begin{equation}\label{eqn-2-5}
		\vol(\sF)=(P')^2,
	\end{equation}
	and $\vol(\sF)>0$ if and only if $\sF$ is of general type, i.e., the Kodaira dimension $\kod(\sF)=2$.
	In this case, we will also call the foliated surface $(S,\sF)$ is of general type,
	in which case, the surface $S$ might be of non-general type.	
	Foliated surfaces of non-general type are classified up to birational equivalence, cf. \cite{bru-03,bru-04,mcq-08}.\vspace{2mm}	
	
	\begin{center}
		\begin{tabular}{|c|l|l|}
			\hline
			$\kod(\sF)$ &  Algebraic foliations (Fibrations of genus $g$) & Transcendental foliations \\\hline
			$-\infty$  & $g=0:$ \quad $K_{\sF}$ is not pseudo-effective & Hilbert modular foliations \\\hline
			\multirow{2}{*}{$0$} &\multirow{2}{*}{Isotrivial fibrations of genus $g=1$} & Quotients of Foliations generated by\\
			&&global vector fields with isolated zeros \\\hline
			\multirow{2}{*}{$1$} & Non-isotrivial fibrations of genus $g=1$  & Riccati foliations or  \\
			& or isotrivial fibrations of genus $g\geq 2$ & Turbulent foliations  \\\hline
		\end{tabular}\vspace{2mm}
	\end{center}
We remark that in the above tabular,
the transcendental foliations of Kodaira dimension one
must be Riccati foliations or Turbulent foliations.
But the converse is not true.
Namely, there exist Riccati foliations or Turbulent foliations of Kodaira dimension zero.
     	
	\section{The Zariski decomposition of the canonical divisor $K_{\sF}$}\label{sec-zariski}
	In this section, we are concerned about the Zariski decomposition of the canonical divisor $K_{\sF}$.
	We will first briefly recall the Zariski decomposition of the canonical divisor $K_{\sF}$ for a relatively minimal foliation $\sF$
	with $K_{\sF}$ being pseudo-effective  and refer to \cite{mcq-08} and \cite[Chapter\,8]{bru-04} for more details.
	Then we give a careful analysis on the negative part, and prove a technical result about the coefficients appearing in the negative part, which will be helpful in proving the Noether inequalities.
	
	\begin{definition}
		Let $\sF$ be reduced foliation on a smooth projective $S$.
		An irreducible curve $C\subseteq S$ is $\sF$-exceptional if
		\begin{enumerate}[(i).]
			\item $C$ is an exceptional curve of first kind on $S$, i.e., it is a smooth rational curve with $C^2=-1$;
			\item the contraction of $C$ to a point produces a new foliation $(S_0,\sF_0)$ which is still reduced.
		\end{enumerate}
	\end{definition}
	
	\begin{definition}\label{def-min}
		A foliated surface $(S,\sF)$ is called relatively minimal if
		\begin{enumerate}[(i).]
			\item the foliation $\sF$ is reduced;
			\item there is no $\sF$-exceptional curve on $S$.
		\end{enumerate}
	\end{definition}
	
	It is proved that any foliated surface $(S,\sF)$ has a relatively minimal model, cf. \cite[Proposition\,5.1]{bru-04}.
	The  relatively minimal model is not necessarily
	unique, the foliations with at least two  relatively minimal models
	have been classified \cite[Theorem 5.1]{bru-04}: {\it
		fibrations by rational curves, Riccati foliations} and {\it very
		special foliations.}
	
	We assume in the following that $\sF$ is a relatively minimal foliation on a smooth projective surface $S$
	such that $K_{\sF}$ is pseudo-effective.
	In fact, the canonical divisor $K_{\sF}$ is pseudo-effective if and only if $\sF$ is not induced by a $\bbp^1$-fibration, cf. \cite{miy-87}.
	Denote the Zariski decomposition of $K_{\sF}$ by
	\begin{equation}\label{eqn-zariski}
		K_{\sF}=P+N,
	\end{equation}
	where $P$ is the nef part and $N$ is the negative one.
	In the case when $\sF$ is relatively minimal with pseudo-effective canonical divisor,
	McQuillan proved that the support of the negative part $N$ is a disjoint union of maximal $\sF$-chains.
	\begin{definition}\label{def-max-chain}
		Let $\sF$ be a relatively minimal foliation on a smooth projective surface $S$.
		We say a curve $C\subseteq S$ is an $\sF$-chain if
		\begin{enumerate}[(i).]
			\item the curve $C$ is a Hirzebruch-Jung string, i.e., $C=\cup_{j=1}^{r} C_j$, each $C_j$ is a smooth rational curve with $C_j^2\leq -2$, $C_j\cdot C_i=1$ if $|i-j|=1$ and $0$ if $|i-j|\geq 2$;
			\item each irreducible component $C_j$ is $\sF$-invariant;
			\item $\mathrm{Sing}(\sF)\cap C$ are all reduced and non-degenerate;
			\item $Z(\sF,C_1)=1$, and $Z(\sF,C_j)=2$ for any $2\leq j\leq r$.
		\end{enumerate}
	\end{definition}
	Since each irreducible component $C_j$ is a smooth rational curve, by \autoref{prop-2-2} the last condition (iv) is also equivalent to
	\begin{enumerate}[(iv)$'$.]
		\item $K_{\sF}C_1=-1$, and $K_{\sF}C_j=0$ for any $2\leq j\leq r$.
	\end{enumerate}
	
	\begin{theorem}[{\cite[Theorem\,8.1]{bru-04}}]\label{thm-3-4}
		Let $\sF$ be a relatively minimal foliation on a smooth projective surface $S$.
		Suppose that $K_{\sF}$ is pseudo-effective with the Zariski decomposition as in \eqref{eqn-zariski}.
		Then the support $\mathrm{Supp}(N)$ is a disjoint union of maximal $\sF$-chains,
		and $\lfloor N \rfloor=0$.
	\end{theorem}
	The above theorem shows that all the coefficients in $N$ are less than $1$.
	In fact, since the support $\mathrm{Supp}(N)$ is a disjoint union of maximal $\sF$-chains, which can be contracted to singularities of Hirzebruch-Jung type,
	these coefficients can be  explicitly computed out using continued fractions \cite[\S\,III.5]{bhpv}.
	By contracting the support $\mathrm{Supp}(N)$, one obtains a surface $S_0$ with finitely many singularities.
	Then the negative part $N$ can be decomposed into
	$$N=\sum_{Q} N_Q,$$
	where the sum runs over all singularities $Q$'s on $S_0$, and $\mathrm{Supp}(N_Q)$ is supported on the inverse image of $Q$ in $S$.
	\begin{lemma}\label{lem-coefficient-N}
		Suppose that $Q$ is a singularity of type $A_{n,q}$,
		and let $N_Q=\sum\limits_{j=1}^{r} b_jC_j$ with $C=\bigcup\limits_{j=1}^{r} C_j$ being a maximal $\sF$-chain as above.
		Then \begin{equation}\label{eqn-3-3}
		b_j=\frac{\xi_j}{n}, \qquad \forall\,1\leq j\leq r,
		\end{equation}
		where $\xi_{r+1}=0$, $\xi_r=1$, and the rest $\xi_j$'s are given by the following recursion formula:
		\begin{equation}\label{eqn-2-2}
		\xi_{j-1}-e_j\xi_j+\xi_{j+1}=0, \qquad\text{~$e_j=-C_j^2\geq 2$.}
		\end{equation}
		In particular,
		\begin{equation}\label{eqn-N_Q^2}
			N_Q^2=-\frac{q}{n}.
		\end{equation}
	\end{lemma}
    \begin{proof}
    	Define $\xi_{r+1}=0$, $\xi_r=1$, and the rest $\xi_j$'s by the following recursion formula \eqref{eqn-2-2} above.
    	Then it has been shown in \cite[\S\,III.5]{bhpv} that $\xi_1=q$,
    	$\xi_{0}=n$, and that
    	$$\large \frac{n}{q}=e_1-\frac{1}{e_2-\frac{1}{\cdots \,-\, \frac{1}{e_r}}}.$$
    	Moreover, according to the property of the Zariski decomposition
    	$$N_QC_1=K_{\sF}C_1=-1,\qquad N_QC_j=K_{\sF}C_j=0,\quad\forall\,2\leq j\leq r.$$
    	On the other hand, define $N'=\sum\limits_{j=1}^{r} \frac{\xi_j}{n}C_j$.
    	Based on \eqref{eqn-2-2}, one checks directly that
    	$$N'C_1=-1=N_QC_1,\qquad N'C_j=0=N_QC_j,\quad\forall\,2\leq j\leq r.$$
    	Since the intersection matrix $\big(C_iC_j\big)$ is negatively definite,
    	it follows that $N_Q=N'$, i.e., the equality \eqref{eqn-3-3} holds.
    	Moreover,
    	$$N_Q^2=\sum_{j=1}^{r}b_jN_Q\cdot C_j=-b_1=-\frac{q}{n}.$$
    	This proves \eqref{eqn-N_Q^2}.    	
    \end{proof}

	\begin{proposition}\label{prop-3-1}
		Let $N_Q=\sum\limits_{j=1}^{r} b_jC_j$ with $C=\bigcup\limits_{j=1}^{r} C_j$ being a maximal $\sF$-chain as above.
		Then
		\begin{equation}\label{eqn-2-1}
			b_j < \left\{\begin{aligned}
				&\,\frac{1}{e_1-1}, &~&\text{~if~}j=1;\\[2mm]
				&\,\frac{1}{2e_j-3},&&\text{~if~} j\geq 2,
			\end{aligned}\right. 
		\end{equation}
		where $e_j=-C_j^2 \geq 2$.
	\end{proposition}
	\begin{proof}
		Let $\xi_{r+1}=0$, and $b_j=\frac{\xi_j}{n}$ for $1\leq j \leq r$.
		By the above arguments, $\xi_r=1$, and the rest $\xi_j$'s can be computed by \eqref{eqn-2-2}.
		As proved in \cite[\S\,III.5]{bhpv}, $\xi_1=q$ and $\xi_{0}=n$.
		Note that each $e_j\geq 2$.
		With the help of \eqref{eqn-2-2}, one proves inductively that
		\begin{equation}\label{eqn-3-4}
			\xi_{j-1} > \xi_{j}, \qquad \forall\, 1\leq j \leq r+1.
		\end{equation}
		Hence
		\[\xi_{j-1}=e_j\xi_j-\xi_{j+1}>e_j\xi_j-\xi_{j}=(e_j-1)\xi_j.\]
		In particular, $n=\xi_0>(e_1-1)\xi_1$.
		Equivalently, $b_1=\frac{\xi_1}{n}<\frac{1}{e_1-1}$ as required.
		
		Suppose now that $j\geq 2$. Then
		$$n=\xi_0\geq \xi_{j-2}=e_{j-1}\xi_{j-1}-\xi_{j}\geq 2\xi_{j-1}-\xi_j> \big(2(e_j-1)-1\big)\xi_j=(2e_j-3)\xi_j.$$
		It follows that $b_j=\frac{\xi_j}{n}<\frac{1}{2e_j-3}$ if $j\geq 2$.
		This completes the proof of \eqref{eqn-2-1}.
	\end{proof}
	
	\begin{remark}\label{rem-3-1}
		The above bounds on the coefficients of the negative part $N_Q$ will be key to prove the Noether type inequalities for the canonical divisor $K_{\sF}$.
		One can similarly define the volume $\vol(L)$ for any big divisor $L$ on a smooth projective surface $S$ by
		$$\vol(L)=\limsup_{n\to +\infty} \frac{\dim H^0(X, nL)}{n^2/2}.$$
		The naive Noether type inequalities do NOT hold for $L$.
		For instance, let $e>0$ and $S=\mathbb{P}_{\mathbb{P}^1}\big(\mathcal{O}_{\mathbb{P}^1} \oplus \mathcal{O}_{\mathbb{P}^1}(e)\big)$ be the Hirzebruch surface admitting a unique section $C_0$ with $C_0^2=-e<0$.
		Let $f:\,S \to \mathbb{P}^1$ be the geometrical ruling on $S$, $F$ be a general fiber of $f$, and $L=mF+C_0$.
		Suppose that $0<m<e$. Then one checks easily that the Zariski decomposition of $L$ is $$L=\Big(mF+\frac{m}{e}C_0\Big)+\frac{e-m}{e}C_0,$$
		and hence
		$$\vol(L)=\frac{m^2}{e},\qquad h^0(L)=m+1.$$
		Fixing $m \gg 0$ and letting $e \to +\infty$, one deduces that there can not exist positive constants $a,b$ such that
		$\vol(L)\geq ah^0(L)-b$.
		\,Of course, in this concrete example, one shows that
		$L=K_{\sF}$ for some foliation $\sF$ on $S$ can happen only when $m=e-1$ based on \eqref{eqn-2-2} and \eqref{eqn-3-3}.
		If it is indeed the case, then
		$$\vol(\sF)=\frac{m^2}{e}=\frac{(e-1)^2}{e}=e-2+\frac{1}{e}>e-2=p_g(\sF)-2.$$
		We refer to \autoref{exam-6-2} for a construction of such a foliation.
	\end{remark}

\section{Algebraically integral foliations}\label{sec-alg-foliation}
In this section, we will first recall in \autoref{sec-alg} the algebraic integrability of a foliation as well as the modular invariants of a fibration.
\autoref{sec-non-reduced} is devoted to the resolution of non-reduced singularities of an algebraically integral foliation, which will be key to prove the Noether inequalities for algebraically integral foliations in \autoref{sec-noether}.

\subsection{The algebraic integrability of a foliation and modular invariants}\label{sec-alg}
A foliation on a smooth projective surface
is called {\it algebraically integral} if it admits a rational first integral;
otherwise it is called {\it non-algebraic} or {\it transcendent}.
Equivalently, a foliation is algebraically integral if its integral curves form a pencil of algebraic curves.
It is clear that the algebraic integrability is a birational property:
if $(S,\sF)$ and $(S',\sF')$ are two foliated surface birational to each other,
then $(S,\sF)$ is algebraic integral if and only if $(S',\sF')$ is algebraic integral.

Let $\sF_0$ be an algebraically integral foliation on a smooth projective surface $S_0$,
and $\Lambda_{\sF_0}$ be the corresponding pencil.
By \cite[Proposition\,1.1]{bru-04}, one proves easily that for any $p\in S_0$,
\begin{equation}\label{eqn-2-10}
\begin{aligned}
&\qquad \text{the point $p$ is a dicritical (i.e., $\exists$ infinitely many separatrices of $\sF$ through $p$),}\\
&\Longleftrightarrow~\text{the point $p$ is a base point of the pencil $\Lambda_{\sF}$.}
\end{aligned}
\end{equation}
By a sequence of possible blowing-ups $\sigma:\,S \to S_0$,
one obtains a surface fibration $f:\,S \to B$,
i.e., $f$ is proper surjective with connected fibers.
It defines a foliation $\sF$ on $S$ by taking the saturation of $\ker(df:\,T_{S} \to f^*T_{B})$ in $T_{S}$.
By construction, $\sF=\sigma^*(\sF_0)$ and it admits no dicritical points.
The canonical divisor $K_{\sF}$ is simple:
\begin{equation}\label{eqn-2-3}
K_{\sF}=K_{S/B} \otimes \mathcal{O}_{S}\Big(\sum(1-a_i)C_i\Big),
\end{equation}
where $K_{S/B}=K_{S}-f^*(K_B)$ is the relative canonical divisor,
the sum is taken over all components $C_i$'s in fibers of $f$,
and $a_i$ is the multiplicity of $C_i$ in fibers of $f$.

The foliation $\sF$ is reduced if and only if every possible singular fiber of $f$ is normal crossing.
For any given foliation $\sF$ corresponding to a fibration $f:\,S \to B$ above,
one can obtain a reduced one by a sequence of possible blowing-ups.
Such a foliation is relatively minimal if there is no redundant exceptional curves,
i.e., any possible exceptional curve in fibers of $f$ intersects other components at at least three points.
In particular, if the fibration is semi-stable,
i.e., any possible singular fiber of $f$ is a reduced node curve,
and any possible smooth rational component in such a singular fiber intersects other components at least two points,
then $\sF$ is relatively minimal and $K_{\sF}=K_{S/B}$ by \eqref{eqn-2-3}.

There are several well-studied invariants for a surface fibration.
Let $f:\,S \to B$ be a surface fibration (simply fibration).
Suppose that $f$ is relatively minimal, i.e., there is no exceptional curve contained in fibers of $f$. Remark that this is different from the relative minimality of the foliation $\sF$ associated to $f$, as mentioned above.
Consider the following invariants.
$$\left\{\begin{aligned}
K_f^2:&=K_{S/B}^2=K_S^2-8(g-1)(g(B)-1),\\
e_f:&=\sum_{F}\big(\chit(F)+(2g-2)\big)=\chit(S)-4(g-1)(g(B)-1),\\
\chi_f:&=\deg f_*\mathcal{O}_S(K_{S/B})= \chi(\mathcal{O}_S)-(g-)(g(B)-1),
\end{aligned}\right.$$
where $g$ and $g(B)$ are respectively the genus of a general fiber of $f$ and the base $B$,
and $\chit(F)$ is the Euler topological characteristic of the reduced curve $F_{\rm red}$.
These are three non-negative integers and satisfy the Noether equality:
$$K_f^2+e_f=12\chi_f.$$
By the stable reduction theorem (cf. \cite{delignemumford}),
there exists a base change $\phi:\wt B \to B$ of finite degree, possibly ramified,
such that the pull-back fibration $\tilde f: \wt S \to \wt B$ is semi-stable.
Here the pull-back fibration $\tilde f:\, \wt S \to \wt B$ is constructed as follows.
Let $S_1$ be the resolution of singularities of the fiber-product $S\times_{B}\wt B$.
Then $\tilde f:\, \wt S \to \wt B$ is just the relatively minimal model of $S_1$.
\begin{center}\mbox{}
	\xymatrix{
		\wt S \ar@{<-}[rr]^-{\theta} \ar[d]_-{\tilde f} && S_1 \ar[rr]^-{\Phi_1} \ar[d]_-{f_1}
		&& S\times_{B}\wt B \ar[rr]^-{\Phi_0} \ar[d] && S \ar[d]^-{f}\\
		\wt B \ar@{=}[rr]  && \wt B \ar@{=}[rr] && \wt B \ar[rr]^-{\phi} && B}
\end{center}
The modular invariants of $f$ is then defined to be
$$\kappa(f):=\frac{K_{\tilde f}^2}{\deg\phi},\qquad
\delta(f):=\frac{e_{\tilde f}}{\deg\phi},\qquad
\chi(f):=\frac{\chi_{\tilde f}}{\deg\phi}.$$
The three invariants defined above are independent on the choices of the semi-stable reductions.
Moreover, these are three non-negative rational numbers satisfying the Noether equality
\begin{equation}\label{eqn-noe-modular}
\kappa(f)+\delta(f)=12\chi(f).
\end{equation}

\begin{lemma}\label{lem-2-3}
	Let $\sF$ be a (not necessarily reduced) foliation induced by a fibration $f:\, S \to B$.
	Suppose that $E$ is an exceptional curve and $\sF$-invariant.
	Let $\sigma:\,S \to S_0$ be the blowing-down by contracting $E$.
	Then there is an induced fibration $f_0:\,S_0 \to B$ which induces the foliation $\sF_0=\sigma_*(\sF)$ on $S_0$ and satisfies $f=f_0\circ \sigma$.
\end{lemma}
\begin{proof}
	This is clear, since any $\sF$-invariant curve is by definition contained in fibers of $f$.
	By contracting a vertical exceptional curve, it still induces a fibration $f_0:\,S_0 \to B$ such that $f=f_0\circ \sigma$.
\end{proof}
\begin{lemma}\label{lem-7-10}
	Let $\sF$ be an algebraically integral reduced foliation of general type on $S$.
	Suppose that
	$$\pi=\sigma_1\circ\cdots\circ\sigma_n:\,S \to S_0=\mathbb{\bbp}_{\bbp^1}\big(\mathcal{O}_{\bbp^1}\oplus\mathcal{O}_{\bbp^1}(e)\big)$$
	is a birational morphism to a Hirzebruch surface $S_0$,
	where each $\sigma_i$ is a blowing-up.
	Then $n\geq 9$, i.e., $\pi$ consists of at least $9$ blowing-ups.
\end{lemma}
\begin{proof}
	Since $\sF$ is reduced of general type, it is induced by a fibration $f:\,S \to B$ of genus $g\geq 2$.
	Let $F$ be a general fiber of $f$ and $A=\pi_*(F)$.
	Let $C_0$ be a section of the Hirzebruch surface with $C_0^2=-e$ (unique if $e\geq 1$) and $\Gamma_0\subseteq S_0$ be a general fiber.
	Then $\Pic(S_0)$ is generated by $C_0$ and $\Gamma_0$.
	Let $A \sim aC_0+b\Gamma_0$.
	Then $b-ae\geq 0$ since $A\cdot C_0\geq 0$.
	Let
	$$F=\pi^*(A)-\sum_{i=1}^{n} a_i\mathcal{E}_i,$$
	where $\mathcal{E}_i$ is the total transform of the exceptional curve of $\sigma_i$.
	Then $a_i\leq a=\Gamma_0\cdot A$, and
	$$a(2b-ae)=A^2=\sum_{i=1}^{n}a_i^2.$$
	Hence 
	$$\begin{aligned}
	0<2g-2=K_{S}F&\,=\Big(\pi^*(K_{S_0})+\sum_{i=1}^{n} \mathcal{E}_i\Big)\Big(\pi^*(A)-\sum_{i=1}^{n} a_i\mathcal{E}_i\Big)\\
	&\,=K_{S_0}A-\sum_{i=1}^{n}a_i=\sum_{i=1}^{n}a_i-(2b-ae+2a).
	\end{aligned}$$
	In particular,
	$$
	\sum_{i=1}^{n}a_i>2b-ae+2a=\sum_{i=1}^{n}\frac{a_i^2}{a}+2a.
	$$
	Hence
	$$2a^2 <\sum_{i=1}^{n}(aa_i-a_i^2)=\sum_{i=1}^{n}\Big(\frac{a^2}{4}-\big(a/2-a_i\big)^2\Big) \leq \frac{na^2}{4}.$$
	It follows that $n\geq 9$ as required.
%
\end{proof}

\begin{corollary}\label{cor-7-3}
	Let $\sF$ be an algebraically integral reduced foliation of general type on $S$.
	Suppose that
	$$\pi=\sigma_1\circ\cdots\circ\sigma_n:\,S \to S_0=\bbp^2$$
	is a birational morphism,
	where each $\sigma_i$ is a blowing-up.
	Then $n\geq 10$, i.e., $\pi$ consists of at least $10$ blowing-ups.
\end{corollary}
\begin{proof}
	Let $\sigma_1:\,S_1 \to \bbp^2$ and $\pi_1=\sigma_2\circ\cdots\circ\sigma_n:\,S \to S_1$.
	Then $\pi=\pi_1\circ\sigma_1$ and $S_1\cong \mathbb{\bbp}_{\bbp^1}\big(\mathcal{O}_{\bbp^1}\oplus\mathcal{O}_{\bbp^1}(1)\big)$.
	According to \autoref{lem-7-10}, $\pi_1$ consists of at least $9$ blowing-ups.
	Hence $n\geq 10$ as required.
\end{proof}

\subsection{Resolution of the non-reduced singularities}\label{sec-non-reduced}
In this subsection, we study the resolution of non-reduced singularities of an algebraically integral foliation,
which will be crucial in proving the third Noether inequality \eqref{eqn-noe3} in \autoref{sec-noether}.
We will not insist on resolving it to be a reduced foliation;
instead, we weaken a little sometimes: the singularities on the resolution are all non-degenerate.
Remark that any reduced singularity of an algebraically integral foliation is non-degenerate,
since there is no saddle-node.
However, the converse is not true.

\begin{lemma}\label{lem-7-1}
	Let $(S_0,\sF_0)$ be a foliated surface, and $p\in S_0$ be a non-reduced singularity of $\sF_0$.
	Let
	$$\pi:\,S=S_n \overset{\sigma_n}{\lra} S_{n-1} \overset{\sigma_{n-1}}{\lra}\cdots
	\overset{\sigma_2}{\lra} S_1 \overset{\sigma_1}{\lra} S_0$$
	be a sequence of blowing-ups resolving the non-reduced singularity $p$,
	such that $S\setminus \pi^{-1}(p) \cong S_0\setminus \{p\}$,
	and that $\sF$ is relatively minimal over $\pi^{-1}(p)$,
	i.e., all the singularities	of $\sF=\pi^*(\sF_0)$ on $\pi^{-1}(p)$ are reduced and there is no $\sF$-exceptional curve contained in $\pi^{-1}(p)$.
	Let
	$$K_{\sF}=\pi^*K_{\sF_0}-\sum_{i=1}^{n} (\ell_i-1)\mathcal{E}_i.$$
	where $\mathcal{E}_i$ is the total inverse image of the exceptional curve of the blowing-up $\sigma_i$,
	and $\sF_i$ be the induced foliation on the surface $S_i$.
	\begin{enumerate}[$(i)$]
		\item It holds $\ell_i\geq 1$ for any $1\leq i \leq n$.
		In particular, $p_g(\sF_0)\geq p_g(\sF)$.
		\item If the foliation $\sF_0$ is algebraically integral, then $\ell_n\geq 2$.
	\end{enumerate}	
\end{lemma}
\begin{proof}
	(i). Let $\sigma_i:\,S_i \to S_{i-1}$ be a blowing-up centered at $q_{i-1}\in S_{i-1}$ for $1\leq i \leq n$ with $q_0=p$.
	Since $\sF_n=\sF$ is relatively minimal over $\pi^{-1}(p)$, it follows that each $q_{i-1}$ is a singularity of $\sF_{i-1}$;
	otherwise, $\sigma_i$ is a blowing-up centered at a smooth point of $\sF_{i-1}$,
	and hence the inverse image of $q_{i-1}$ in $S$ would contain an $\sF$-exceptional curve,
	which contradicts the relative minimality of $\sF$ over $\pi^{-1}(p)$.
	Hence $a_{q_{i-1}} \geq 1$,
	where $a_{q_{i-1}}$ is the vanishing order of $\sF_{i-1}$ around $q_{i-1}$.
	According to \eqref{eqn-2-13}, $\ell_i\geq a_{q_{i-1}} \geq 1$ as required.	
	
	(ii). Replacing $\sF_0$ by $\sF_{n-1}$, we may assume that $\pi=\sigma:\,S \to S_0$ is a blowing-up centered at $p$,
	such that the exceptional curve $E$ is not $\sF$-exceptional, where $\sF=\sigma^*(\sF_0)$.
	By blowing-ups centered at points away from $E$, we may assume that there is a fibration $f:\,S \to B$ such that $\sF$ is induced by the fibration $f$.
	If the exceptional curve $E$ is not $\sF$-invariant, then $\ell_n=a_p+1\geq 2$ by \eqref{eqn-2-13}.
	Thus we may assume that $E$ is $\sF$-invariant, i.e., it is contained in fibers of $f$.
	Since the exceptional curve $E$ of $\sigma$ is not $\sF$-exceptional,
	it follows that $E$ would intersect other components of $F$ at at least three points, where $F$ is the fiber of $f$ containing $E$.
	Hence
	$$3\leq Z(\sF,E)=K_{\sF}\cdot E+2=\big(\sigma^*K_{\sF_0}-(\ell_n-1)E\big)\cdot E+2=\ell_n+1.$$
	Thus $\ell_n\geq 2$ as required.		
\end{proof}

\begin{lemma}\label{lem-7-8}
	Let $(S_0,\sF_0)$ be a foliated surface,
	and $p\in S_0$ be a singularity of $\sF_0$.
	Let $C$ be an $\sF_0$-invariant curve passing through $p$, and suppose that $C$ is smooth at $p$.
	Let $\sigma:\,S\to S_0$ be the blowing-up centered at $p$ with exceptional curve $E$, 
	$\sF=\sigma^*(\sF_0)$, and $\ol C$ be the strict transform of $C$ with $q=E\cap \ol C$.
	Let $K_{\sF}=\sigma^*(K_{\sF_0})-(\ell_p-1)E$. Then
	\begin{equation}\label{eqn-7-22}
	Z(\sF,\ol C,q)=Z(\sF_0,C,p)-(\ell_p-1).
	\end{equation}
\end{lemma}
\begin{proof}
	This follows from \autoref{prop-2-2}.
	Indeed, let $p_1=p,p_2,\cdots, p_m$ be all the singularities of $\sF_0$ on $C$,
	and let $q_1=q$ and $q_i=\sigma^{-1}(p_i)$ for $i\geq 2$.
	It is clear that $\chi(\ol C)=\chi(C)$, and $Z(\sF,\ol C,q_i)=Z(\sF_0,C,p_i)$ for $i\geq 2$.
	Thus by \autoref{prop-2-2},
	$$Z(\sF,\ol C)-\big(K_{\sF_0}C-(\ell_p-1)\big)=Z(\sF,\ol C)-K_{\sF}\ol C=\chi(\ol C)=\chi(C)=Z(\sF_0,C)-K_{\sF_0}C.$$
	It follows that
	$$\sum_{i=1}^{m}Z(\sF,\ol C,q_i)=\sum_{i=1}^{m}Z(\sF_0, C,p_i)-(\ell_p-1).$$
	Combining this with the equalities $Z(\sF,\ol C,q_i)=Z(\sF_0,C,p_i)$ for $i\geq 2$, we prove \eqref{eqn-7-22}.
\end{proof}

\begin{lemma}\label{lem-7-9}
	Let $\sF$ be an algebraically integral reduced foliation on $S$,
	and $\pi:\,S \to S_0$ be a birational morphism with $\sF_0=\pi_*(\sF)$.
	Let $$|K_{\sF}|=|M|+Z,$$
	where $M$ (resp. $Z$) is the moving (resp. fixed) part of $|K_{\sF}|$.
	Suppose that $M=\pi^*(M_0)$ for some divisor $M_0$ on $S_0$.
	\begin{enumerate}[(i)]
		\item All the possible non-reduced singularities of $\sF_0$ are on $Z_0=\pi_*(Z)$.
		In particular, if $Z_0=0$, then $\sF_0$ is also reduced;
		and $\pi$ is an isomorphism if moreover $\sF$ is relatively minimal.
		\item Let $\sigma_1:\,S_1 \to S_0$ be a blowing-up centered at a non-reduced singularities $p$ with exceptional curve $E_1$.
		Then $\pi$ factors through $\sigma_1$.
		Let $\sF_1=\sigma_1^*(\sF_0)$ and $K_{\sF_1}=\sigma_1^*(K_{\sF_0})-(\ell_p-1)E_1$.
		Then $m_p(Z_0) \geq \ell_p-1$, where $m_p(Z_0)$ is the multiplicity of $Z_0$ at $p$.
	\end{enumerate}
\end{lemma}
\begin{proof}
	(i).
	By assumption,
	$$K_{\sF_0}=\pi_*(K_{\sF})=\pi_*\big(\pi^*(M_0)+Z\big)=M_0+Z_0.$$
	On the other hand, by \autoref{lem-7-1},
	$$M+Z=K_{\sF}=\pi^*K_{\sF_0}-\sum_{i=1}^{n} (\ell_i-1)\mathcal{E}_i=M+\pi^*(Z_0)-\sum_{i=1}^{n} (\ell_i-1)\mathcal{E}_i,$$
	where $\mathcal{E}_i$ is the total inverse image of the exceptional curve of the $i$-th blowing-up $\sigma_i$ contained in $\pi$.
	Thus
	\begin{equation}\label{eqn-7-35}
	\pi^*(Z_0)=Z+\sum_{i=1}^{n} (\ell_i-1)\mathcal{E}_i.
	\end{equation}
%
	Suppose that there is a non-reduced singularity $p$ which is not on $Z_0$.
	By \autoref{lem-7-1}\,(ii), there is at least one exceptional curve $\mathcal{E}_i$, which is over $p$ with $\ell_i\geq 2$.
	This is a contradiction to \eqref{eqn-7-35},
	since the total inverse image $\pi^*(Z_0)$ is the strict transform of $Z_0$ plus some curves contracted to points on $Z_0$.
	This proves that all the possible non-reduced singularities of $\sF_0$ are on $Z_0$. The rest statements follow immediately.
	
	(ii).
	Since $p$ is non-reduced, $\pi$ is not an isomorphism around $p$.
	By the universal properties of blowing-ups,
	$\pi$ factors through $\sigma_1$.
	Let $\pi_1:\,S \to S_1$ be the induced birational morphism with $\pi=\sigma_1\circ\pi_1$. Then
	$$K_{\sF_1}=(\pi_1)_*(K_{\sF})=(\pi_1)_*(M)+(\pi_1)_*(Z)=\sigma_1^*(M_0)+Z_1, \qquad\text{where~}Z_1=(\pi_1)_*(Z).$$
	By assumption, 
	$$K_{\sF_1}=\sigma_1^*(K_{\sF_0})-(\ell_p-1)E_1=\sigma_1^*(M_0)+\sigma_1^*(Z_0)-(\ell_p-1)E_1=\sigma_1^*(M_0)+\ol Z_0+\big(m_p(Z_0)-(\ell_p-1)\big)E_1,$$
	where $\ol Z_0$ is the strict transform of $Z_0$ in $S_1$.
	It follows that
	$$Z_1=\ol Z_0+\big(m_p(Z_0)-(\ell_p-1)\big)E_1.$$
	Hence $m_p(Z_0)\geq \ell_p-1$ as required.
\end{proof}

\begin{lemma}\label{lem-7-2}
	Let $\sF_0$ be an algebraically integral foliation on $S_0$,
	and $p\in S_0$ be a non-reduced singularity of $\sF_0$.
	Let $C$ be an $\sF_0$-invariant curve passing through $p$.
	Suppose that $C$ is smooth at $p$.
	If $Z(\sF_0,C,p)=1$, then locally around $p$,
	there exists a coordinate $(z,w)$ such that $C=\{z=0\}$
	and that $\sF_0$ is generated by the vector field
	$$v=z\frac{\partial}{\partial z}+\lambda w\frac{\partial}{\partial w},
	\qquad \text{where $\lambda=\frac{m_1}{m_2}\in \mathbb{Q}^{+}$}.$$
	If moreover there exists another $\sF_0$-invariant curve $C'$ passing through $p$ and smooth at $p$, then the above local coordinate can be chosen such that $C'=\{w=0\}$.
\end{lemma}
\begin{proof}
	Let $v$ be a local vector field defining $\sF_0$ around $p$.
	We first show that the two eigenvalues $\lambda_1,\lambda_2$ of the linear part $(Dv)(p)$ are both non-zero with $\lambda:=\frac{\lambda_1}{\lambda_2}\in \mathbb{Q}^{+}$.
	To this aim, we choose a local coordinate $(x,y)$ with $C=\{x=0\}$.
	Let $v=A\frac{\partial}{\partial x}+B\frac{\partial}{\partial y}$.
	Since $C$ is $\sF_0$-invariant, $A=xA_1$ for some holomorphic function $A_1$.
	Hence
	\begin{equation}\label{eqn-7-6}
	(Dv)(p)=\left(\begin{aligned}
	A_1(p) ~&~ \quad 0\\
	B_x(p) ~&~ B_y(p)
	\end{aligned}\right),\qquad\text{where $B_x=\frac{\partial B}{\partial x}$ and $B_y=\frac{\partial B}{\partial y}$}.
	\end{equation}
	On the other hand, $\omega:=xA_1dy-Bdx$ would be a local one-form defining $\sF_0$.
	Hence by definition, $Z(\sF_0,C,p)$ equals to the vanishing order of the restriction $(-B)|_C$ at $p$.
	Since $Z(\sF_0,C,p)=1$ by assumption, it follows that
	$$B(x,y)=xB_1(x,y)+yB_2(y),$$
	where $B_1(x,y)$ and $B_2(y)$ are both holomorphic with $B_2(0)=\lambda_2\neq 0$.
	In particular, $B_y(p)=\lambda_2\neq 0$.
	According to \eqref{eqn-7-6}, the two eigenvalues of the linear part $(Dv)(p)$ are equal to $A_1(p)$ and $\lambda_2$.
	It follows that the eigenvalue of $\sF_0$ at $p$ is equal to $\lambda=\frac{A_1(p)}{\lambda_2}$.
	Since $p$ is assumed to be a non-reduced singularity of $\sF_0$,
	it follows that $A_1(p)=\lambda_1\neq 0$ (otherwise it is a saddle-node),
	and that $\lambda=\frac{\lambda_1}{\lambda_2}\in \mathbb{Q}^{+}$ as required.
	
	Now we have proved that, locally around $p$ the foliation $\sF_0$ can be a generated by a vector field whose linear part has eigenvalues $\lambda_1,\lambda_2$,
	both non-zero with $\lambda=\frac{\lambda_1}{\lambda_2}\in \mathbb{Q}^{+}$.
	Such a singularity can be resolved by a standard way as recalled above \autoref{lem-2-2}.
	
	If $\lambda\not\in \mathbb{N}^+\cup \frac{1}{\mathbb{N}^+}$,
	then $\sF_0$ can be generated by a local vector field
	$v=z\frac{\partial}{\partial z}+\lambda w\frac{\partial}{\partial w}$,
	where $\lambda=\frac{m_1}{m_2}\in \mathbb{Q}^{+}$ with $\gcd(m_1,m_2)=1$.
	Equivalently, $\sF_0$ is given by the levels of the meromorphic function $\frac{z^{m_1}}{w^{m_2}}$.
	In particular, the local separatrices passing through $p$ are defined by
	$\frac{z^{m_1}}{w^{m_2}}=c$ or $z^{m_1}-cw^{m_2}=0$, where $c\in \mathbb{C}\cup \{\infty\}$.
	Since $C$ is a separatrix passing through $p$ and $C$ is smooth at $p$ by assumption,
	it follows that $C$ is defined by $z=0$ or $w=0$.
	By exchange $z$ and $w$ (this would change $\lambda$ to $\lambda^{-1}$), we may assume that $C=\{z=0\}$ as required.
	
	If $\lambda\in \mathbb{N}^+\cup \frac{1}{\mathbb{N}^+}$,
	then $\sF_0$ can be generated by a local vector field
	$v=z\frac{\partial}{\partial z}+(mw+\epsilon z^m)\frac{\partial}{\partial w}$,
	where $m=\lambda$ or $\lambda^{-1}$, and $\epsilon\in\{0,1\}$.
	If $\epsilon=0$, then the situation is almost the same as the previous case, and we are done by a similar argument.
	If $\epsilon=1$, then the resolution of the non-reduced singularity $p$ would create a saddle-node of multiplicity two.
	This is impossible since $\sF_0$ is assumed to be algebraically integral.
	
	\vspace{2mm}
	Finally, it remains to prove that $C'=\{w=0\}$ if there exists another $\sF_0$-invariant curve $C'$ passing through $p$ and smooth at $p$.
	If $\lambda=\frac{m_1}{m_2}\neq 1$, then by the above argument, the local separatrices passing through $p$ are defined by $z^{m_1}-cw^{m_2}=0$.
	Hence $C'=\{w=0\}$ as required.
	If $\lambda=1$, then the local separatrices passing through $p$ are defined by $z-cw=0$. Hence $C'=\{z-c_0w=0\}$ for some $c_0\neq 0$.
	Replacing $(z,w)$ by $(z,w')=(z,c_0w-z)$, one sees that the local vector field defining $\sF_0$ is
	$$v=z\frac{\partial}{\partial z}+\lambda w\frac{\partial}{\partial w}
	=z\frac{\partial}{\partial z}+\lambda w'\frac{\partial}{\partial w'},$$
	and $C=\{z=0\}$ and $C'=\{w'=0\}$ as required.
\end{proof}

\begin{corollary}\label{cor-7-2}
	Let $\sF_0$ be an algebraically integral foliation on $S_0$,
	and $C,C'$ be two $\sF_0$-invariant curves which intersect transversely at $p$.
	Suppose that $Z(\sF_0,C,p)=k\geq 2$.
	\begin{enumerate}[(i)]
		\item It holds that $k':=Z(\sF_0,C',p)\geq 2$.
		\item If moreover the vanishing order of $\sF_0$ at $p$ is $a_p(\sF_0)\leq 2$,
		then $a_p=\min\{k,k'\}=2$.
	\end{enumerate}	
\end{corollary}
\begin{proof}
	(i). Since $Z(\sF_0,C,p)=k\geq 2$, $p$ is a singularity of $\sF_0$, and hence
	$Z(\sF_0,C',p)=k'\geq 1$.
	Suppose that $Z(\sF_0,C',p)=1$. Then by \autoref{lem-7-2}, it follows that $Z(\sF_0,C,p)=1$ too, which is a contradiction.
	Hence $Z(\sF_0,C',p)=k'\geq 2$ as required.
	
	(ii).
	Let $\sigma_1:\,S_1 \to S_0$ be the blowing-up centered at $p$ with exceptional curve $E_1$.
	If $E_1$ is not $\sF_1=\sigma_1^*(\sF_0)$-invariant,
	then $K_{\sF_1}=\sigma_1^*(K_{\sF_0})-a_pE_1$ by \eqref{eqn-2-13}.
	Hence $\tang(\sF_1,E_1)=a_p-1\leq 1$ by \autoref{prop-2-1}.
	It follows that there is at most one singularity of $\sF_1$ on $E_1$.
	We may thus assume without loss of generality that $q=\ol C \cap E_1$ is not a singularity of $\sF_1$ (the arguments would be similar if $q'=\ol C' \cap E_1$ is not a singularity), where $\ol C$ (resp. $\ol C'$) is the strict transform of $C$ (resp. $C'$).
	By \eqref{eqn-7-22},
	$$0=Z(\sF_1,\ol C,q)=Z(\sF_0,C,p)-a_p,\quad \Longrightarrow\quad
	k=Z(\sF_0,C,p)=a_p\leq 2.$$
	It follows that $a_p=k=2$ as required, since $k$ is assumed to be $\geq 2$.
	Suppose next that $E_1$ is $\sF_1$-invariant.
	Then $K_{\sF_1}=\sigma_1^*(K_{\sF_0})-(a_p-1)E_1$ by \eqref{eqn-2-13}.
	Hence $Z(\sF_1,E_1)=a_p+1\leq 3$ by \autoref{prop-2-2}.
	Note that $Z(\sF_1,E_1,q)\geq 1$ and $Z(\sF_1,E_1,q')\geq 1$.
	We may thus assume without loss of generality that $Z(\sF_1,E_1,q)=1$;
	otherwise, $Z(\sF_1,E_1)\geq Z(\sF_1,E_1,q)+Z(\sF_1,E_1,q')\geq 4$, which is a contradiction.
	By \autoref{lem-7-2}, one sees that $Z(\sF_1,\ol C,q)=1$ too.
	According to \eqref{eqn-7-22},
	\[1=Z(\sF_1,\ol C,q)=Z(\sF_0,C,p)-(a_p-1)=k+1-a_p.\]
	As $k\geq 2$ and $a_p\leq 2$ by assumption,
	it follows also that $a_p=k=2$ as required.
\end{proof}

\begin{lemma}\label{lem-7-??}
	Let $\sF_0$ be an algebraically integral foliation on $S_0$,
	and $p\in S_0$ be a non-reduced singularity of $\sF_0$.
	Let $\sigma_1:\,S_1 \to S_0$ be the blowing-up centered at $p$ with exceptional curve $E_1$.
	Suppose that $E_1$ is $\sF_1=\sigma_1^*\sF_0$-invariant,
	and that there are exactly two singularities $q_1,q_2$ of $\sF_1$ on $E_1$ with
	\begin{equation}\label{eqn-7-28}
	Z(\sF_1,E_1,q_1)=Z(\sF_1,E_1,q_2)=1.
	\end{equation}
	Then $p$ is a non-degenerate singularity of $\sF_0$.
\end{lemma}
\begin{proof}
	Let $K_{\sF_1}=\sigma_1^*K_{\sF_0}-(\ell-1)E_1$.
	By \autoref{prop-2-2},
	$$2=Z(\sF_1,E_1)=K_{\sF_1}\cdot E_1+2=(\ell-1)+2.$$
	Hence $\ell=1$, i.e., $K_{\sF_1}=\sigma_1^*K_{\sF_0}$.
	
	Suppose first that both singularities $q_1,q_2$ are reduced.
	Then $E_1$ is an $\sF_1$-exceptional curve by the assumption \eqref{eqn-7-28}, see the text after \cite[Definition\,5.2]{bru-04}.
	Hence there is at least one point, say $q_1$, is non-reduced.
	By \autoref{lem-7-2}, locally around $q_1$, there exists a local coordinate $(z,w)$ such that
	$E_1=\{z=0\}$ and that $\sF_1$ is generated by the vector field
	$$v=z\frac{\partial}{\partial z}+\lambda w\frac{\partial}{\partial w},
	\qquad \text{where $\lambda=\frac{m}{n}\in \mathbb{Q}^{+}$}.$$
	Let $C_1=\{w=0\}$ be the local holomorphic curve which is $\sF_1$-invariant and $C_0=\sigma_1(C_1)$.
	Then $C_0$ is $\sF_0$-invariant and smooth around $p$ and by \autoref{lem-7-8},
	$$Z(\sF_0,C_0,p)=Z(\sF_1,C_1,q_1)=1.$$
	This implies that $p$ is non-degenerate by \autoref{lem-7-2}.
\end{proof}

\begin{lemma}\label{lem-7-7}
	Let $\sF_0$ be an algebraically integral foliation on $S_0$,
	and $p\in S_0$ be a singularity of $\sF_0$.
	Let $C$ be an $\sF_0$-invariant curve passing through $p$ with $Z(\sF_0,C,p)=k\geq 2$
	Suppose that $C$ is smooth at $p$.
	\begin{enumerate}[$(i)$]
		\item Locally around $p$,
		there exists a coordinate $(x,y)$ such that $C=\{x=0\}$
		and that $\sF_0$ is generated by the vector field
		$$v=xA_1(x,y)\frac{\partial}{\partial x}+\big(xB_1(x,y)+y^kB_2(y)\big)\frac{\partial}{\partial y},$$
		where $A_1(x,y)$, $B_1(x,y)$ and $B_2(y)$ are all holomorphic with
		$A_1(0,0)=0$ and $B_2(0)\neq 0$.
		In particular, $p$ is a non-reduced singularity.
		
		\item Suppose moreover that the vanishing order of $\sF_0$ at $p$ is equal to $1$. Then there is a sequence of blowing-ups:
		$$\pi:\,S_k \overset{\sigma_k}{\lra} S_{k-1} \overset{\sigma_{k-1}}{\lra} ~\cdots~ \overset{\sigma_{2}}{\lra} S_1
		\overset{\sigma_1}{\lra} S_0,$$
		where $\sigma_1:\,S_1\to S_0$ is the blowing-up centered at $p$ with exceptional curve $E_1$,
		and $\sigma_i:\,S_i \to S_{i-1}$ (for $2\leq i \leq k$) is the blowing up centered at the intersection $q_{i-1}=E_{i-1}\cap \ol C$ with exceptional curve $E_i\subseteq S_i$.
		Here we denote by $\ol C$ the strict transform in each $S_i$ ($i\geq1$) by abuse of notations, and similarly for $\ol E_j$ with $j<i$.	
		The inverse image of $C$ around $p$ is as follows:
		$${\setlength{\unitlength}{6mm}
			\begin{tikzpicture}
			\draw[thick] (14,1) -- (14,3);
			\filldraw[black] (14,2) circle (1.5pt);
			\filldraw[black] (14,2) node[anchor=west]{$p$};
			\filldraw[black] (14,1) node[anchor=west]{$C$};

			\draw[->] (12,2) -- (13.5,2);
			\filldraw[black] (12.5,2.3) node[anchor=west]{\Large  $\sigma_1$};
			
			\draw[thick] (11,1) -- (11,3);
			\draw[thick] (10.75,2.5) -- (11.5,1);
			\filldraw[black] (11,2) circle (1.5pt);
			\filldraw[black] (11,2) node[anchor=west]{$q_1$};
			\filldraw[black] (10.5,1) node[anchor=west]{$\ol{C}$};
			\filldraw[black] (11.5,1) node[anchor=west]{$E_1$};
			
			\draw[->] (9,2) -- (10.5,2);
			\filldraw[black] (9.5,2.3) node[anchor=west]{\Large  $\sigma_2$};
			
			\draw[thick] (8,1) -- (8,3);
			\draw[thick] (6.5,1) -- (6.5,3);
			\draw[thick] (6,2) -- (8.5,2);
			\filldraw[black] (6.5,2) circle (1.5pt);
			\filldraw[black] (8,2) circle (1.5pt);
			\filldraw[black] (8,1.7) node[anchor=west]{$q_2$};
			\filldraw[black] (6,1.7) node[anchor=west]{$\bar q_1$};
			\filldraw[black] (6.5,1) node[anchor=west]{$\ol{E}_1$};
			\filldraw[black] (8,1) node[anchor=west]{$\ol{C}$};
			\filldraw[black] (5.6,2.2) node[anchor=west]{$E_2$};

			\draw[->] (4,2) -- (5.5,2);
			\filldraw[black] (4.5,2.3) node[anchor=west]{\Large  $\sigma_3$};
			
			\draw[thick] (3,1) -- (3,3);
			\draw[thick] (0.5,1) -- (0.5,3);
			\draw[thick] (1.34,0.88) -- (3.5,2.5);
			\draw[thick] (2.16,0.88) -- (0,2.5);
			\filldraw[black] (0.5,2.12) circle (1.5pt);
			\filldraw[black] (1.75,1.2) circle (1.5pt);
			\filldraw[black] (3,2.12) circle (1.5pt);
			\filldraw[black] (3,2) node[anchor=west]{$q_3$};
			\filldraw[black] (1.45,1.5) node[anchor=west]{$\bar q_2$};
			\filldraw[black] (0.45,2.2) node[anchor=west]{$\bar q_1$};
			\filldraw[black] (-0.1,1) node[anchor=west]{$\ol{E}_1$};
			\filldraw[black] (3,1) node[anchor=west]{$\ol{C}$};
			\filldraw[black] (-0.65,2.6) node[anchor=west]{$\ol E_2$};
			\filldraw[black] (3.3,2.7) node[anchor=west]{$E_3$};
			
			\draw[->] (2,-1) -- (2,0.5);
			\filldraw[black] (2,-0.2) node[anchor=west]{\Large  $\sigma_4$};
			\filldraw[black] (1.3,-1.5) circle (1.5pt);
			\filldraw[black] (1.6,-1.5) circle (1.5pt);
			\filldraw[black] (1.9,-1.5) circle (1.5pt);
			\filldraw[black] (2.2,-1.5) circle (1.5pt);
			\filldraw[black] (2.5,-1.5) circle (1.5pt);
			\filldraw[black] (2.8,-1.5) circle (1.5pt);
			
			\draw[->] (5,-1.5) -- (3.5,-1.5);
			\filldraw[black] (4,-1.3) node[anchor=west]{\Large  $\sigma_k$};

			\draw[thick] (6,-1) -- (6,-3);
			\draw[thick] (6.84,-3.12) -- (9,-1.5);
			\draw[thick] (7.66,-3.12) -- (5.5,-1.5);
			\draw[thick] (9.2,-2.4) -- (8,-1.5);
			\filldraw[black] (6,-1.88) circle (1.5pt);
			\filldraw[black] (7.25,-2.8) circle (1.5pt);
			\filldraw[black] (8.5,-1.88) circle (1.5pt);
			
			\filldraw[black] (9.9,-2.4) circle (1.5pt);
			\filldraw[black] (10.2,-2.4) circle (1.5pt);
			\filldraw[black] (10.5,-2.4) circle (1.5pt);
			\filldraw[black] (10.8,-2.4) circle (1.5pt);
			\filldraw[black] (11.1,-2.4) circle (1.5pt);
			\filldraw[black] (11.4,-2.4) circle (1.5pt);

			\draw[thick] (14,-1) -- (14,-3);
			\draw[thick] (12.34,-3.12) -- (14.5,-1.5);
			\draw[thick] (13.16,-3.12) -- (12,-2.25);
			\filldraw[black] (12.75,-2.8) circle (1.5pt);
			\filldraw[black] (13.35,-2.35) circle (1.5pt);
			\filldraw[black] (14,-1.88) circle (1.5pt);
			
			\filldraw[black] (5.3,-3) node[anchor=west]{$\ol E_1$};
			\filldraw[black] (7.7,-3.3) node[anchor=west]{$\ol E_2$};
			\filldraw[black] (9,-1.3) node[anchor=west]{$\ol E_3$};
			\filldraw[black] (9.1,-2.7) node[anchor=west]{$\ol E_4$};
			\filldraw[black] (11.5,-1.9) node[anchor=west]{$\ol E_{k-1}$};
			\filldraw[black] (14.4,-1.4) node[anchor=west]{$E_k$};
			\filldraw[black] (14,-3) node[anchor=west]{$\ol C$};
			
			\filldraw[black] (6,-1.8) node[anchor=west]{$\bar q_1$};
			\filldraw[black] (7,-2.5) node[anchor=west]{$\bar q_2$};
			\filldraw[black] (8.2,-1.6) node[anchor=west]{$\bar q_3$};
			\filldraw[black] (11.7,-2.8) node[anchor=west]{$\bar q_{k-1}$};
			\filldraw[black] (13.2,-2.5) node[anchor=west]{$q_{k}'$};
			\filldraw[black] (14,-2) node[anchor=west]{$q_k$};	
	\end{tikzpicture}}$$
	On $S_k$, all the singularities over $\pi^{-1}(p)$ are non-degenerate (might not be reduced).
	Moreover, let $\sF_i$ be the induced foliation on $S_i$. Then
	\begin{equation}\label{eqn-7-32}
	\left\{\begin{aligned}
	&K_{\sF_k}=\pi^*(K_{\sF_0})-\sum_{i=2}^{k}\mathcal{E}_i=\pi^*(K_{\sF_0})- \bigg(\sum_{i=2}^{k-1} (i-1)\ol E_i+(k-1)E_k\bigg),\\
	&\pi^*(C)=\ol C+\sum_{i=1}^{k}\mathcal{E}_i=\ol C+\sum\limits_{i=1}^{k-1} i\ol E_i+kE_k,
	\end{aligned}\right.
	\end{equation}
	where $\mathcal{E}_i=\sum\limits_{j=i}^{k-1}\ol E_j+E_k$ is the total transform of the exceptional curve of $\sigma_i$ for $1\leq i \leq k-1$.
	\end{enumerate}
	\end{lemma}
	
	\begin{proof}
		(i).
		Let $(x,y)$ be a local coordinate around $p$ such that $C=\{x=0\}$.
		Let $v=A\frac{\partial}{\partial x}+B\frac{\partial}{\partial y}$.
		Since $C$ is $\sF_0$-invariant, $A=xA_1$ for some holomorphic function $A_1$.
		Moreover, similar as in the proof of \autoref{lem-7-2},
		one shows that the linear part of $(Dv)(p)$ is of the form in \eqref{eqn-7-6}, and that
		$$Z(\sF_0,C,p)=\text{vanishing order of $(-B)|_C$ at $p$}.$$
		Thus
		$$B(x,y)=xB_1(x,y)+y^kB_2(y),$$
		where $B_1(x,y)$ and $B_2(y)$ are both holomorphic with $B_2(0)\neq 0$.
		In particular, $B_y(p)= 0$ since $k\geq 2$.
		According to \eqref{eqn-7-6}, the two eigenvalues of the linear part $(Dv)(p)$ are equal to $A_1(p)$ and $B_y(p)= 0$.
		Since $p$ is assumed to be a non-reduced singularity of $\sF_0$,
		it follows that $A_1(p)=A_1(0,0)=0$ as required; otherwise it is a saddle-node.
		
		(ii).
		By (i), the foliation $\sF_0$ is defined by the one-form
		$$\omega=\big(xB_1(x,y)+y^kB_2(y)\big)dx-xA_1(x,y)dy.$$
		Since the vanishing order of $\sF_0$ at $p$ is equal to $1$,
		it follows that $B_1(0,0)\neq 0$.
		Let $\sigma_1$ be given by $(x,y)=(x_1y_1,y_1)$.
		Then $\ol{C}=\{x_1=0\}$ and $E_1=\{y_1=0\}$.
		Moreover,
		$$\sigma_1^*(\omega)=y_1^2\big(x_1B_1(x_1y_1,y_1)+y_1^{k-1}B_2(y_1)\big)dx_1+x_1y_1\big(x_1B_1(x_1y_1,y_1)+y_1^{k-1}B_2(y_1)-A_1(x_1y_1,y_1)\big)dy_1.$$
		Since $A_1(0,0)=0$, $B_1(0,0)\neq 0$ and $B_2(0)\neq 0$,
		it follows that the vanishing order of $\sigma_1^*(\omega)$ along $E_1=\{y_1=0\}$ is exactly $1$.
		Hence the foliation $\sF_1$ is defined by the one-form
		$\omega_1=\frac{\sigma_1^*(\omega)}{y_1}$, namely,
		$$\omega_1=
		y_1\big(x_1B_1(x_1y_1,y_1)+y_1^{k-1}B_2(y_1)\big)dx_1
		+x_1\big(x_1B_1(x_1y_1,y_1)+y_1^{k-1}B_2(y_1)-A_1(x_1y_1,y_1)\big)dy_1.
		$$
		Hence the vanishing order of $\omega_1$ at $\ol{C}\cap E_1=q_1=(0,0)$ is $2$.
		Note that the exceptional curve $E_1$ is $\sF_1$-invariant since $y$ divides the two-form $\omega_1\wedge dy$.
		By \eqref{eqn-2-13} and \eqref{eqn-7-22},
		$$K_{\sF_1}=\sigma_1^*(K_{\sF_0}),\qquad Z(\sF_1,E_1)=2,\qquad
		Z(\sF_1,\ol C,q_1)=Z(\sF_1,C,p)=k.$$	 
		Combining with \autoref{cor-7-2}, one obtains that $Z(\sF_1,E_1,q_1)=Z(\sF_1,E_1)=2$, and hence
		$q_1$ is the unique singularity of $\sF_1$ on $E_1$.
		
		\begin{claim}\label{claim-7-6}
			For any $2\leq i \leq k-1$, let $\sigma_i:\,S_i \to S_{i-1}$ be the blowing-up centered at $q_{i-1}=E_{i-1}\cap \ol C$ with exceptional curve $E_i\subseteq S_i$,
			and let $\sF_i=\sigma_i^*(\sF_{i-1})$.
			\begin{enumerate}[(1).]
				\item Locally around around $q_i$, $\sF_i$ can be defined by the one-form
				\begin{equation}\label{eqn-7-33}
				\begin{aligned}
				\omega_i=\,&y_i\big(x_iB_1(x_iy_i^i,y_i)+y_i^{k-i}B_2(y_i)\big)dx_i\\
				&+x_i\Big(i\big(x_iB_1(x_iy_i^i,y_i)+y_i^{k-i}B_2(y_i)\big)-\frac{A_1(x_iy_i^i,y_i)}{y_i^{i-1}}\Big)dy_i.
				\end{aligned}
				\end{equation}
				Moreover, the function $\frac{A_1(x_iy_i^i,y_i)}{y_i^{i-1}}$ takes value equal to $0$ when $x_i=y_i=0$.
				In other words, one can express $A_1(x,y)$ as
				$A_1(x,y)=xA_2(x,y)+y^iA_3(y)$ for some holomorphic functions $A_2(x,y)$ and $A_3(y)$.
				In particular, the vanishing order of $\omega_i$ at $q_i$ is $2$,
				$E_i$ is $\sF_i$-invariant, and
				$K_{\sF_i}=\sigma_i^*(K_{\sF_{i-1}})-E_i$.
				\item The two points $\bar q_{i-1}=E_i\cap \ol E_{i-1}$ and $q_i=E_i\cap \ol C$ are the only two singularities of $\sF_i$ on $E_i$, where $\bar q_{i-1}$ is non-degenerate, and
				$$Z(\sF_i,\ol C,q_i)=k-(i-1)\geq 2,\qquad Z(\sF_i,E_i, q_i)=2.$$
			\end{enumerate}
		\end{claim}
		
		\begin{proof}[{Proof of \autoref{claim-7-6}}]
			We prove this by induction.
			Clearly $\omega_1$ has the form as in \eqref{eqn-7-33}.
			Suppose that $\omega_{i-1}$ has the form as in \eqref{eqn-7-33}.
			Locally, $\sigma_i$ is given by $(x_{i-1},y_{i-1})=(x_iy_i,y_i)$
			with $\ol C=\{x_i=0\}$ and $E_i=\{y_i=0\}$.
			Hence
			$$\begin{aligned}
			\sigma_i^*(\omega_{i-1})=\,&y_i^3\big(x_iB_1(x_iy_i^i,y_i)+y_i^{k-i}B_2(y_i)\big)dx_i\\
			&+x_iy_i^2\Big(i\big(x_iB_1(x_iy_i^i,y_i)+y_i^{k-i}B_2(y_i)\big)-\frac{A_1(x_iy_i^i,y_i)}{y_i^{i-1}}\Big)dy_i.
			\end{aligned}$$
			Since $B_1(0,0) \neq 0$, it follows that the multiplicity of $E_i=\{y_i=0\}$ in $\sigma_i^*(\omega_{i-1})$ is exactly $2$.
			Hence $\sF_i$ is defined by
			$$\begin{aligned}
			\omega_i=\frac{\sigma_i^*(\omega_{i-1})}{y_i^2}=\,&y_i\big(x_iB_1(x_iy_i^i,y_i)+y_i^{k-i}B_2(y_i)\big)dx_i\\
			&+x_i\Big(i\big(x_iB_1(x_iy_i^i,y_i)+y_i^{k-i}B_2(y_i)\big)-\frac{A_1(x_iy_i^i,y_i)}{y_i^{i-1}}\Big)dy_i.
			\end{aligned}$$
			Equivalently, $\sF_i$ can be defined locally by the vector field
			$$\begin{aligned}
			v_i=x_i\Big(i\big(x_iB_1(x_iy_i^i,y_i)+y_i^{k-i}B_2(y_i)\big)-\frac{A_1(x_iy_i^i,y_i)}{y_i^{i-1}}\Big)\frac{\partial}{\partial x_i}
			+y_i\big(x_iB_1(x_iy_i^i,y_i)+y_i^{k-i}B_2(y_i)\big)\frac{\partial}{\partial y_i}.
			\end{aligned}$$
			Hence the linear part of $(Dv_i)(q_i)$ is
			\begin{equation*}
			(Dv_i)(q_i)=\left(\begin{aligned}
			\frac{A_1(x_iy_i^i,y_i)}{y_i^{i-1}}\Big(0,0\Big) ~&~ \quad 0~\\
			0~\qquad\,\quad&~\quad 0~
			\end{aligned}\right).
			\end{equation*}
			Here, $\frac{A_1(x_iy_i^i,y_i)}{y_i^{i-1}}\Big(0,0\Big)$ stands for the value of the function $\frac{A_1(x_iy_i^i,y_i)}{y_i^{i-1}}$ when $x_i=y_i=0$.
			Since $\sF_i$ is algebraically integral, it follows that $\frac{A_1(x_iy_i^i,y_i)}{y_i^{i-1}}\Big(0,0\Big)=0$ as required.
			In particular, the vanishing order of $\omega_i$ at $q_i$ is $2$ since $B_1(0,0)\neq 0$.
			Moreover, from the expression of $\omega_i$, one sees easily that $E_i$ is $\sF_i$-invariant.
			It follows that $K_{\sF_i}=\sigma_i^*(K_{\sF_{i-1}})-E_i$ by \eqref{eqn-2-13},
			and hence by \autoref{lem-7-8} and \autoref{prop-2-2},	
			$$Z(\sF_i,\ol C,q_i)=Z(\sF_{i-1},\ol C,q_{i-1})-1=k+1-i,\qquad Z(\sF_i,E_i)=3.$$
			As $\ol E_{i-1},E_i,\ol C$ are all $\sF_i$-invariant,
			the two intersection points $\bar q_{i-1}=E_i\cap \ol E_{i-1}$ and $q_i=E_i\cap \ol C$ are singularities of $\sF_i$.
			Since $Z(\sF_i,\ol C,q_i)=k+1-i\geq 2$, it follows that $Z(\sF_i,E_i,q_i) \geq 2$ by \autoref{cor-7-2}.
			Thus 
			$$Z(\sF_i,E_i,\bar q_{i-1})=1,\qquad Z(\sF_i,E_i,q_i)=2.$$
			In particular, $\bar q_{i-1}$ and $q_i$ are the only two singularities of $\sF_i$ on $E_i$.
		\end{proof}
		Come back to the proof of \autoref{lem-7-7}.
		Let $\sigma_k:\,S_k \to S_{k-1}$ be the blowing-up centered at $q_{k-1}=E_{k-1}\cap \ol C$ with exceptional curve $E_k\subseteq S_k$.
		Locally, $\sigma_k$ is given by $(x_{k-1},y_{k-1})=(x_ky_k,y_k)$
		with $\ol C=\{x_k=0\}$ and $E_k=\{y_k=0\}$.
		Hence
		$$\begin{aligned}
		\sigma_k^*(\omega_{k-1})=\,&y_k^3\big(x_kB_1(x_ky_k^k,y_k)+B_2(y_k)\big)dx_k\\
		&+x_ky_k^2\Big(k\big(x_kB_1(x_ky_k^k,y_k)+B_2(y_k)\big)-\frac{A_1(x_ky_k^k,y_k)}{y_k^{k-1}}\Big)dy_k.
		\end{aligned}$$
		Since $B_1(0,0) \neq 0$ and $B_2(0)\neq 0$, it follows that the multiplicity of $E_k=\{y_k=0\}$ in $\sigma_k^*(\omega_{k-1})$ is exactly $2$.
		Hence $\sF_k$ is defined by
		$$\begin{aligned}
		\omega_k=\frac{\sigma_k^*(\omega_{k-1})}{y_k^2}=\,&y_k\big(x_kB_1(x_ky_k^k,y_k)+B_2(y_k)\big)dx_k\\
		&+x_k\Big(k\big(x_kB_1(x_ky_k^k,y_k)+B_2(y_k)\big)-\frac{A_1(x_ky_k^k,y_k)}{y_k^{k-1}}\Big)dy_k.
		\end{aligned}$$
		In particular, $E_k$ is $\sF_k$-invariant, and $K_{\sF_k}=\sigma_k^*(K_{\sF_{k-1}})-E_k$.
		By \autoref{lem-7-8},
		$$Z(\sF_k,\ol E_{k-1},\bar q_{k-1})=Z(\sF_{k-1}, E_{k-1},q_{k-1})-1=1,\quad
		Z(\sF_k,\ol C,q_{k})=Z(\sF_{k-1}, \ol C,q_{k-1})-1=1.$$
		Combining this with \autoref{cor-7-2},
		$$Z(\sF_k,E_k,\bar q_{k-1})=Z(\sF_k,E_k,q_k)=1.$$
		On the other hand, by \autoref{prop-2-2},
		$Z(\sF_k,E_k)=3$.
		Hence there must be a third singularity $q_k'$ with $Z(\sF_k,E_k,q_k')=1$, which is non-degenerate by \autoref{lem-7-2}.
		Finally, the first formula in \eqref{eqn-7-32} follows from the equalities $K_{\sF_1}=\sigma_1^*(K_{\sF_0})$ and $K_{\sF_i}=\sigma_i^*(K_{\sF_{i-1}})-E_i$ for $2\leq i \leq k$;
		and the second formula in \eqref{eqn-7-32} follows by a direct computation.
	\end{proof}
	
	\begin{lemma}\label{lem-7-3}
		Let $\sF_0$ be an algebraically integral foliation on $S_0$,
		and $p\in S_0$ be a singularity of $\sF_0$.
		Let $C$ be an $\sF_0$-invariant curve passing through $p$
		such that $Z(\sF_0,C,p)=2$ and that $C$ is smooth at $p$.
		Suppose that the vanishing order of $\sF_0$ at $p$ is $2$. Let $\sigma_1:\,S_1 \to S_0$ be the blowing-up centered at $p$ with exceptional curve $E_1$ and $\sF_1=\sigma_1^*(\sF_0)$. Then
		\begin{enumerate}[$(i)$]
			\item  either $E_1$ is not $\sF_1$-invariant, in which case $K_{\sF_1}=\sigma_1^*K_{\sF_0}-2E_1$ and $q_1=\ol C\cap E_1$
			is not a singularity of $\sF_1$;
			\item or $E_1$ is $\sF_1$-invariant, in which case $K_{\sF_1}=\sigma_1^*K_{\sF_0}-E_1$, $q_1=\ol C\cap E$
			is a non-degenerate singularity of $\sF_1$, and there are exactly another two singularities of $\sF_1$ on $E_1$
			with $Z(\sF_1,E_1,q_1')=Z(\sF_1,E_1,q_1'')=1$,
			where $\ol{C}$ is the strict transform of $C$;
			$${\setlength{\unitlength}{6mm}
				\begin{tikzpicture}
				\draw[thick] (14,1) -- (14,3);			
				\filldraw[black] (14,2) circle (1.5pt);
				\filldraw[black] (14,1.8) node[anchor=west]{$p$};
				\filldraw[black] (14,1) node[anchor=west]{$C$};

				\draw[->] (12,2) -- (13.5,2);
				\filldraw[black] (12.5,2.3) node[anchor=west]{\Large  $\sigma_1$};
				
				\draw[thick] (10.5,1) -- (10.5,3);
				\draw[thick] (8,2) -- (11,2);
				\filldraw[black] (10.5,2) circle (1.5pt);
				\filldraw[black] (9.5,2) circle (1.5pt);
				\filldraw[black] (8.5,2) circle (1.5pt);
				\filldraw[black] (10.5,1.8) node[anchor=west]{$q_1$};
				\filldraw[black] (9.3,1.7) node[anchor=west]{$q_1''$};
				\filldraw[black] (8.3,1.7) node[anchor=west]{$q_1'$};
				\filldraw[black] (10.5,1) node[anchor=west]{$\ol C$};
				\filldraw[black] (7.3,2) node[anchor=west]{$E_1$};
		\end{tikzpicture}}$$
		\item or $E_1$ is $\sF_1$-invariant, in which case
		$K_{\sF_1}=\sigma_1^*K_{\sF_0}-E_1$, $q_1=\ol C\cap E$
		is a non-degenerate singularity of $\sF_1$, and there is exactly another one singularity $q_1'$ of $\sF_1$ on $E_1$
		with $Z(\sF_1,E_1,q_1')=2$.
		$${\setlength{\unitlength}{6mm}
			\begin{tikzpicture}
			\draw[thick] (14,1) -- (14,3);			
			\filldraw[black] (14,2) circle (1.5pt);
			\filldraw[black] (14,1.8) node[anchor=west]{$p$};
			\filldraw[black] (14,1) node[anchor=west]{$C$};

			\draw[->] (12,2) -- (13.5,2);
			\filldraw[black] (12.5,2.3) node[anchor=west]{\Large  $\sigma_1$};
			
			\draw[thick] (10.5,1) -- (10.5,3);
			\draw[thick] (8,2) -- (11,2);
			\filldraw[black] (10.5,2) circle (1.5pt);
			\filldraw[black] (9,2) circle (1.5pt);
			\filldraw[black] (10.5,1.8) node[anchor=west]{$q_1$};
			\filldraw[black] (8.8,1.7) node[anchor=west]{$q_1'$};
			\filldraw[black] (10.5,1) node[anchor=west]{$\ol C$};
			\filldraw[black] (7.3,2) node[anchor=west]{$E_1$};
			\end{tikzpicture}}$$
		\end{enumerate}
		\end{lemma}
		\begin{proof}
		If $E_1$ is not $\sF_1$-invariant, then we are in case (i).
		In this case, $K_{\sF_1}=\sigma_1^*K_{\sF_0}-2E_1$ by \eqref{eqn-2-13},
		and hence by \autoref{lem-7-8},
		$$Z(\sF_1,\ol C,q_1)=Z(\sF_0,C,p)-2=0.$$
		Hence $q_1$ is not a singularity of $\sF_1$.
		
		If $E_1$ is $\sF_1$-invariant,
		we are in cases (ii) or (iii).
		By \eqref{eqn-2-13}, $K_{\sF_1}=\sigma_1^*K_{\sF_0}-E_1$ in this case,
		and hence $Z(\sF_1,E_1)=3$, and by \autoref{lem-7-8},
		$$Z(\sF_1,\ol C,q_1)=Z(\sF_0,C,p)-1=1.$$
		Hence $Z(\sF_1,E_1,q_1)=1$ by \autoref{cor-7-2}.
		It follows that either there are another two singularities $q_1'$ and $q_1''$ with
		$Z(\sF_1,E_1,q_1')=Z(\sF_1,E_1,q_1'')=1$,
		or there is another one singularity $q_1'$ with
		$Z(\sF_1,E_1,q_1')=2$. 
	\end{proof}
	
	By repeating the process for the singularity $q_1'$ with $Z(\sF_1, E_1, q_1')=2$ in case (iii) above,
	we can reach the following situation.
	\begin{corollary}\label{cor-7-1}
		Let $\sF_0$ be an algebraically integral foliation on $S_0$,
		and $p\in S_0$ be a non-reduced singularity of $\sF_0$.
		Let $C$ be an $\sF_0$-invariant curve passing through $p$.
		Suppose that $C$ is smooth at $p$, and that $Z(\sF_0,C,p)=2$.
		Then there exist a sequence blowing-ups
		$$\pi=\sigma_1\circ\cdots\circ\sigma_n:~S_n \overset{\sigma_n}{\lra} S_{n-1}
		\overset{\sigma_{n-1}}{\lra} \cdots \overset{\sigma_2}{\lra} S_{1}\overset{\sigma_1}{\lra} S_{0},$$
		such that the singularities of $\sF_{n}$ on $\pi^{-1}(p)$ belong to one of the following,
		where a solid curve (resp. dotted curve) means $\sF_n$-invariant (resp. non-$\sF_n$-invariant).
		\begin{enumerate}[$(i)$]
			\item A chain of $\sF_n$-invariant curves ($n\geq 1$) with all the singularities on $\pi^{-1}(p)$ being non-degenerate,
			$\ol C^2=C^2-1$,
			$\ol E_i$ is the strict transform of the exceptional curve of $\sigma_i$ and
			$\ol E_i^2=-2$ for $1\leq i\leq n-1$,
			and $E_n^2=-1$. Moreover,
			\begin{equation}\label{eqn-7-16}
			K_{\sF_n}= \pi^*K_{\sF_0}-\sum_{i=1}^{n}\mathcal{E}_i=\pi^*K_{\sF_0}-\sum_{i=1}^{n}i\ol E_i,
			\end{equation}
			where $\mathcal{E}_i=\sum\limits_{j=i}^{n}\ol E_j$ is total inverse image of the exceptional curve of $\sigma_i$, (here we understand $\mathcal{E}_n=E_n=\ol E_n$).
			$${\setlength{\unitlength}{6mm}
				\begin{tikzpicture}
				\draw[thick] (14.5,1) -- (13.5,3);
				\filldraw[black] (13.8,2.4) circle (1.5pt);
				\filldraw[black] (13.9,2.4) node[anchor=west]{$q_1$};
				\filldraw[black] (14.5,0.7) node[anchor=west]{$\ol C$};
				
				\draw[thick] (13.1,1) -- (14.1,3);
				\filldraw[black] (13.4,1.6) circle (1.5pt);
				\filldraw[black] (13.5,1.6) node[anchor=west]{$q_2$};
				\filldraw[black] (12.6,0.7) node[anchor=west]{$\ol E_1$};
				
				\draw[thick] (13.7,1) -- (12.7,3);
				\filldraw[black] (13,2.4) circle (1.5pt);
				\filldraw[black] (13.7,0.7) node[anchor=west]{$\ol E_2$};
				
				\draw[thick] (12.3,1) -- (13.3,3);
				\filldraw[black] (12.6,1.6) circle (1.5pt);
				
				\filldraw[black] (12.4,2) circle (1pt);
				\filldraw[black] (12.2,2) circle (1pt);
				\filldraw[black] (12,2) circle (1pt);
				\filldraw[black] (11.8,2) circle (1pt);
				\filldraw[black] (11.6,2) circle (1pt);
				\filldraw[black] (11.4,2) circle (1pt);

				\draw[thick] (11.5,1) -- (10.5,3);
				\filldraw[black] (10.8,2.4) circle (1.5pt);
				\filldraw[black] (11.2,1.6) circle (1.5pt);
				\filldraw[black] (10.9,2.4) node[anchor=west]{$q_n$};
				\filldraw[black] (11,0.7) node[anchor=west]{$\ol E_{n-1}$};
				
				\draw[thick] (9.6,1) -- (11.3,3);
				\filldraw[black] (10.35,1.9) circle (1.5pt);
				\filldraw[black] (10.3,1.7) node[anchor=west]{$q_n''$};
				\filldraw[black] (9.93,1.4) circle (1.5pt);
				\filldraw[black] (9.9,1.2) node[anchor=west]{$q_n'$};
				\filldraw[black] (9,0.7) node[anchor=west]{$E_n$};
		\end{tikzpicture}}$$
		
		\item A chain of $\sF_n$-invariant curves ($n\geq 2$) with all the singularities on $\pi^{-1}(p)$ being non-degenerate,
		$\ol C^2=C^2-1$ if $n\geq 3$ (resp. $\ol C^2=C^2-2$ if $n=2$),
		$\ol E_i$ is the strict transform of the exceptional curve of $\sigma_i$ and
		$\ol E_i^2=-2$ for $1\leq i\leq n-3$ or $i=n-1$,
		$\ol E_{n-2}^2=-3$,
		and $E_n^2=-1$. Moreover,
		\begin{equation}\label{eqn-7-17}
		K_{\sF_n}= \pi^*K_{\sF_0}-\sum_{i=1}^{n-2}\mathcal{E}_i-\mathcal{E}_n=\pi^*K_{\sF_0}-\sum_{i=1}^{n-2}i\ol E_i-(n-2)\ol E_{n-1}-(2n-3) E_n,
		\end{equation}
		where $\mathcal{E}_i=\sum\limits_{j=i}^{n-1}\ol E_j+2E_n$ is the total inverse image of the exceptional curve of $\sigma_i$ for $1\leq i \leq n-2$, and $\mathcal{E}_{n}=E_n$ the exceptional curve of $\sigma_n$.
		$${\setlength{\unitlength}{6mm}
			\begin{tikzpicture}
			\draw[thick] (14.5,1) -- (13.5,3);
			\filldraw[black] (13.8,2.4) circle (1.5pt);
			\filldraw[black] (13.9,2.4) node[anchor=west]{$q_1$};
			\filldraw[black] (14.5,0.7) node[anchor=west]{$\ol C$};
			
			\draw[thick] (13.1,1) -- (14.1,3);
			\filldraw[black] (13.4,1.6) circle (1.5pt);
			\filldraw[black] (13.5,1.6) node[anchor=west]{$q_2$};
			\filldraw[black] (12.6,0.7) node[anchor=west]{$\ol E_1$};
			
			\draw[thick] (13.7,1) -- (12.7,3);
			\filldraw[black] (13,2.4) circle (1.5pt);
			\filldraw[black] (13.7,0.7) node[anchor=west]{$\ol E_2$};
			
			\draw[thick] (12.3,1) -- (13.3,3);
			\filldraw[black] (12.6,1.6) circle (1.5pt);
			
			\filldraw[black] (12.4,2) circle (1pt);
			\filldraw[black] (12.2,2) circle (1pt);
			\filldraw[black] (12,2) circle (1pt);
			\filldraw[black] (11.8,2) circle (1pt);
			\filldraw[black] (11.6,2) circle (1pt);
			\filldraw[black] (11.4,2) circle (1pt);

			\draw[thick] (11.5,1) -- (10.5,3);
			\filldraw[black] (10.8,2.4) circle (1.5pt);
			\filldraw[black] (11.2,1.6) circle (1.5pt);
			\filldraw[black] (10.9,2.4) node[anchor=west]{$q_{n-1}$};
			\filldraw[black] (11,0.7) node[anchor=west]{$\ol E_{n-2}$};
			
			\draw[thick] (9.6,1) -- (11.3,3);
			\filldraw[black] (10.35,1.9) circle (1.5pt);
			\filldraw[black] (10.3,1.7) node[anchor=west]{$q_n'$};
			\filldraw[black] (9.93,1.4) circle (1.5pt);
			\filldraw[black] (9.3,1.4) node[anchor=west]{$q_n$};
			\filldraw[black] (9,0.7) node[anchor=west]{$E_n$};
			
			\draw[thick] (10.11,1) -- (9.11,3);
			\filldraw[black] (8.2,2.8) node[anchor=west]{$\ol E_{n-1}$};
	\end{tikzpicture}}$$
	
	\item A chain of rational curves ($n\geq 1$),
	such that $\ol E_i$ is $\sF_n$-invariant for $1\leq i \leq n-1$, $E_n$ is not $\sF_n$-invariant,
	all the singularities on $\bigcup\limits_{i=1}^{n-1} \ol E_i$ are non-degenerate,
	$\ol C^2=C^2-1$,
	$\ol E_i$ is the strict transform of the exceptional curve of $\sigma_i$ with
	$\ol E_i^2=-2$ for $1\leq i\leq n-1$, and $E_n^2=-1$.
	$${\setlength{\unitlength}{6mm}
		\begin{tikzpicture}
		\draw[thick] (14.5,1) -- (13.5,3);
		\filldraw[black] (13.8,2.4) circle (1.5pt);
		\filldraw[black] (13.9,2.4) node[anchor=west]{$q_1$};
		\filldraw[black] (14.5,0.7) node[anchor=west]{$\ol C$};
		
		\draw[thick] (13.1,1) -- (14.1,3);
		\filldraw[black] (13.4,1.6) circle (1.5pt);
		\filldraw[black] (13.5,1.6) node[anchor=west]{$q_2$};
		\filldraw[black] (12.6,0.7) node[anchor=west]{$\ol E_1$};
		
		\draw[thick] (13.7,1) -- (12.7,3);
		\filldraw[black] (13,2.4) circle (1.5pt);
		\filldraw[black] (13.7,0.7) node[anchor=west]{$\ol E_2$};
		
		\draw[thick] (12.3,1) -- (13.3,3);
		\filldraw[black] (12.6,1.6) circle (1.5pt);
		
		\filldraw[black] (12.4,2) circle (1pt);
		\filldraw[black] (12.2,2) circle (1pt);
		\filldraw[black] (12,2) circle (1pt);
		\filldraw[black] (11.8,2) circle (1pt);
		\filldraw[black] (11.6,2) circle (1pt);
		\filldraw[black] (11.4,2) circle (1pt);

		\draw[thick] (11.5,1) -- (10.5,3);
		\filldraw[black] (11.2,1.6) circle (1.5pt);
		\filldraw[black] (11.2,1.6) node[anchor=west]{$q_{n-1}$};
		\filldraw[black] (11,0.7) node[anchor=west]{$\ol E_{n-1}$};
		
		\draw[very thick, dashed] (9.6,1) -- (11.3,3);
		\filldraw[black] (9,0.7) node[anchor=west]{$E_n$};
\end{tikzpicture}}$$		
\end{enumerate}
\end{corollary}
\begin{proof}
If $Z(\sF_0,C,p)=2$, it follows that the vanishing order of $\sF_0$ at $p$ is at most $2$ by the definition of the $Z$-index, cf. \autoref{sec-pre}.
Almost all the statements follow from \autoref{lem-7-3} and \autoref{lem-7-7} (for $k=2$) by induction.
We remark that in case (iii), except these singularities $\{q_1,\cdots q_{n-1}\}$, which are all non-degenerate,
there might exist singularity of $\sF_n$ on $E_n$ (which is not $\sF_n$-invariant in this case).
Nevertheless, the intersection point $q_n=\ol E_{n-1} \cap E_n$ is not a singularity by \autoref{lem-7-3}\,(i).
\end{proof}

\begin{lemma}\label{lem-7-4}
	Let $\sF_0$ be an algebraically integral foliation on $S_0$,
	and $p\in S_0$ be a non-reduced singularity of $\sF_0$.
	Let $C$ be an algebraic curve passing through $p$, which not $\sF_0$-invariant.
	Suppose that $C$ is smooth at $p$, and that $\tang(\sF_0,C,p)=1$.
	Let $\sigma_1:\,S_1 \to S_0$ be the blowing-up centered at $p$ with exceptional curve $E_1$, and $\sF_1=\sigma_1^*\sF_0$.
	\begin{enumerate}[$(i)$]
		\item If $E_1$ is not $\sF_1$-invariant, then $\tang(\sF_1,E_1)=0$.
		In particular, there is no singularity of $\sF_1$ on $E_1$.
		\item If $E_1$ is $\sF_1$-invariant, then $Z(\sF_1,E_1)=2$.
		Moreover, the intersection $q=\ol{C}\cap E_1$ is not a singularity of $\sF_1$.
	\end{enumerate}
\end{lemma}
\begin{proof}
	Let $(x,y)$ be a local coordinate such that $C=\{x=0\}$.
	Let $v=A\frac{\partial}{\partial x}+B\frac{\partial}{\partial y}$ be a vector field generating $\sF_0$ around $p$.
	Then
	$$1=\tang(\sF_0,C,p)=I_p\big(x,v(x)\big)=I_p\big(x,A(x,y)\big).$$
	Hence $\frac{\partial A}{\partial y}(0,0) \neq 0$.
	In particular, the vanishing order of $v$ at $p$ is $a_p=1$.
	
	If $E_1$ is not $\sF_1$-invariant, then $K_{\sF_1}=\sigma_1^*K_{\sF_0}-E_1$ by \eqref{eqn-2-13}.
	Hence $\tang(\sF_1,E_1)=K_{\sF_1}E_1+E_1^2=0$ as required.
	
	If $E_1$ is $\sF_1$-invariant, then $K_{\sF_1}=\sigma_1^*K_{\sF_0}$ by \eqref{eqn-2-13}.
	Hence $Z(\sF_1,E_1)=K_{\sF_1}E_1+2=2$ by \autoref{prop-2-2}.
	Moreover,
	$$\tang(\sF_1,\ol{C})=K_{\sF_1}\ol{C}+\ol{C}^2=K_{\sF_0}C+(C^2-1)=\tang(\sF_0,C)-1.$$
	It follows that $\tang(\sF_1,\ol{C},q)=\tang(\sF_0,C,p)-1=0$,
	since $\sigma_1$ induces an isomorphism between $S_1\setminus E_1$
	and $S_0\setminus\{p\}$.
	In particular, the intersection $q=\ol{C}\cap E_1$ is not a singularity of $\sF_1$.
\end{proof}
	
	\section{The Chern numbers of foliations}\label{sec-chern-number}
	The main purpose of this section is to introduce three birational invariants for a foliated surface, which will be called the Chern numbers.
	Some basic properties of these three invariants will be also discussed.
	In particular, \autoref{thm-chern-1} will be proved in this section.
	More precisely, in \autoref{sec-definition} we will introduce the three Chern numbers and prove the Noether equality \eqref{eqn-chern-1-1}.
	In \autoref{sec-invariance} we show that the three Chern numbers are well-defined birational invariants. In particular, we will prove that the first Chern number $c_1^2(\sF)$ is nothing but the volume of $\sF$.
	Finally, in \autoref{sec-chern-algebraic} we consider the case when $\sF$ is algebraically integral, i.e., $\sF$ is induced by a family $f:\,S \to B$ of curves.
	In this case, we prove that the three Chern numbers are just the corresponding modular invariants of $f$.
	
	\subsection{Definitions and the Noether equality}\label{sec-definition}
	Let $(S,\sF)$ be a foliated surface, and $(S',\sF')$ be a relatively minimal model, i.e., $(S',\sF')$ is birational to $(S,\sF)$ and $(S',\sF')$ is relatively minimal.
	\begin{lemma}\label{lem-5-pseudo-eff}
		Let $(S',\sF')$ and $(S'',\sF'')$ be a two relatively minimal model of a foliated surface $(S,\sF)$.
		Then either
		\begin{enumerate}[(i)]
			\item both the canonical divisors $K_{\sF'}$ and $K_{\sF''}$ are non-pseudo-effective; or
			\item both the canonical divisors $K_{\sF'}$ and $K_{\sF''}$ are pseudo-effective.
		\end{enumerate}
	\end{lemma}
\begin{proof}
	Suppose for instance $K_{\sF'}$ is not pseudo-effective.
	Then $\sF'$ is induced by a family of rational curves by Miyaoka's \autoref{thm-miyaoka}, and so is $\sF''$.
	It implies that $K_{\sF''}$ is not pseudo-effective either.
	This completes the proof.
\end{proof}
	
	Suppose that $K_{\sF'}$ is pseudo-effective. Let
	$$K_{\sF'}=P'+N',$$
	be the Zariski decomposition of $K_{\sF'}$, where $P'$ and $N'$ are respectively the nef and negative parts.
	We will use the following notations.
	
	{\noindent\bf Notations.}
	For a singular point $p$ of $\sF'$, we denote by $p\in N'$ if $p$
	lies on the support of the negative part $N'$.
	In case $N'=0$, $p\in N'$ means $p$ is in an empty set.

	\begin{definition}\label{def-chern-numbers}
	Let $(S,\sF)$ be a foliated surface, and $(S',\sF')$ be a relatively minimal model as above.
	\begin{enumerate}[(i)]
		\item Suppose that the canonical divisor $K_{\sF'}$ is not pseudo-effective. We define all the three Chern numbers of $\sF$ to be zero:
		$$c_1^2(\sF)=c_2(\sF)=\chi(\sF)=0.$$
		\item Suppose that the canonical divisor $K_{\sF'}$ is pseudo-effective. The Chern numbers of $\sF$ is defined by
		the following formulas.
		\begin{equation}\label{eqn-4-1}
			\left\{\begin{aligned}
				c_1^2(\sF)&\,:=K_{\sF'}^2+\sum_{p\in N'
				}\beta_p(\sF'),\\
				c_2(\sF)&\,:=\sum_{p\not\in N'}\beta_p(\sF'),\\
				\chi(\sF)&\,:=\chi(\mathcal
				O_{S'})+\dfrac14K_{\sF'} N_{\sF'}+
				\sum_{p}\chi_p(\sF'),
			\end{aligned}\right.
		\end{equation}
		where $N_{\sF'}$ is the normal bundle of $\sF'$, and $\beta_p(\sF')$ and $\chi_p(\sF')$ are defined in \autoref{def-beta-chi}.
		
	\end{enumerate}
	\end{definition}

		We will prove in \autoref{sec-invariance} that the definition of Chern numbers is independent of the choices of the  relatively minimal models,
i.e., one can use any relatively minimal model to compute the Chern
numbers.
	
	\begin{theorem}\label{thm-4-2}
		 The following Noether equality holds for the three Chern numbers of a foliated surface $(S,\sF)$:
		\begin{equation*}
		c_1^2(\sF)+c_2(\sF)=12\chi(\sF).
		\end{equation*}
	\end{theorem}
	\begin{proof}
		Let $(S',\sF')$ be a relatively minimal model of $(S,\sF)$.
	    To prove the Noether equality, we may assume that $K_{\sF'}$ is pseudo-effective; otherwise all the three Chern numbers are zeros, and the Noether equality holds automatically.
		By the Noether equality
		$c_1^2(S')+c_2(S')=12\chi(\mathcal O_{S'})$ for the projective surface $S'$
		together with the Baum-Bott formula (cf. \cite[Theorem\,3.1]{bru-04}) and \eqref{eqn-2-4}, one obtains that
		\begin{align*}
		c_1^2(\sF)+c_2(\sF)&=K_{\sF'}^2+\sum_{p}\beta_p(\sF')
		=(K_{S'}+N_{\sF'})^2+\sum_{p}\beta_p(\sF')\\
		&=K_{S'}^2+c_2(S')+2K_{S'}N_{\sF'}+N_{\sF'}^2-c_2(S')+\sum_{p}\beta_p(\sF')\\
		&=12\chi(\mathcal O_{S'})+3K_{\sF'} N_{\sF'}-N_{\sF'}^2-m(\sF')+
		\sum_{p}\beta_p(\sF')\\
		&=12\chi(\mathcal O_{S})+3K_{\sF'} N_{\sF'}-
		\sum_{p}BB_p(\sF')-\sum_pm_p(\sF')+\sum_{p}\beta_p(\sF')\\
		&=12\chi(\mathcal
		O_{S'})+3K_{\sF'} N_{\sF'}+
		12\sum_{p}\chi_p(\sF') =12\chi(\sF).
		\end{align*}
		The last two equalities follow from the definitions of $\chi_p(\sF')$
		in \autoref{def-beta-chi} and $\chi(\sF)$ in \eqref{eqn-4-1}.
	\end{proof}
	
	\begin{remark}
		In the definition of
		$c_1^2(\sF)$ and $c_2(\sF)$ for foliated surfaces with pseudo-effective divisor $K_{\sF'}$, we can assume that
		$p$ runs over the singular points whose eigenvalues $\lambda_p\in \mathbb Q^-$
		are {\it negative rational numbers.}
		\begin{equation*}
		\left\{\begin{aligned}
			c_1^2(\sF)&\,=K_{\sF'}^2+\sum_{p\in N',\,\lambda_p\in \mathbb Q^-
			}\beta_p(\sF')
		=K_{\sF'}^2+\sum_{p\in N',\,\lambda_p\in \mathbb Q^-
			}\beta(-\lambda_p),\\
			c_2(\sF)&\,=\sum_{p\not\in N',\,\lambda_p\in \mathbb
				Q^-}\beta_p(\sF')
			=\sum_{p\not\in N',\,\lambda_p\in \mathbb
				Q^-}\beta(-\lambda_p).
		\end{aligned}\right.
		\end{equation*}
		Because the other singular
		points $q$ satisfy $\beta_q(\sF')=0$ by definition, which have no contributions
		to $c_1^2(\sF)$ and $c_2(\sF)$.
		In particular, saddle-nodes have no contributions to the Chern numbers.
		This is reasonable because
		the existence of a saddle-node implies that $\sF$ is non-algebraic,
		and hence Poincar\'e problem is solved automatically.
	\end{remark}
	
	\subsection{Birational invariance of the Chern numbers}\label{sec-invariance}
	In this subsection, we will prove that the Chern numbers are well-defined birational invariants.
	By their definitions in \autoref{def-chern-numbers}, if these Chern numbers are well-defined,
	then they are naturally birationally invariants.
Let $(S',\sF')$ and $(S'',\sF'')$ be two relatively minimal models of a given foliated surface $(S,\sF)$,
both $(S',\sF')$ and $(S'',\sF'')$ are relatively minimal, and birational to $(S,\sF)$.
According to \autoref{lem-5-pseudo-eff}, we may assume that
both the canonical divisors $K_{\sF'}$ and $K_{\sF''}$ are pseudo-effective.
By the Noether equality,
it suffices to prove that
\begin{equation*} 
\left\{\begin{aligned}
&K_{\sF'}^2+\sum_{p\in N'
}\beta_p(\sF')=K_{\sF''}^2+\sum_{p\in N''
}\beta_p(\sF'');\\
&\sum_{p\not\in N'}\beta_p(\sF')=\sum_{p\not\in N''}\beta_p(\sF''),
\end{aligned}\right.
\end{equation*}
where $N'$ and $N''$ are the nef parts in the Zariski decompositions of $K_{\sF'}$ and $K_{\sF''}$ respectively.
	
	\begin{proposition}\label{thm813c}
		Let $(S',\sF')$ be a relatively minimal foliated surface with pseudo-effective canonical divisor $K_{\sF'}$ as above.
		Then
		$$K_{\sF'}^2+\sum_{p\in N'
		}\beta_p(\sF')=\vol(\sF').$$
	\end{proposition}
	\begin{proof}
		Let
		$$K_{\sF'}=P'+N',$$
		be the Zariski decomposition, where $P'$ and $N'$ are respectively the nef and negative parts.
		Since $\sF'$ is relatively minimal, the support of $N'$ is a disjoint union of maximal $\sF'$-chains by \autoref{thm-3-4}.
		By contracting all such maximal $\sF'$-chains, one obtains a singular surface $S_0'$ with finitely many singularities of Hirzebruch-Jung type,
		and hence we can write
		$$N'=\sum_{Q} N_Q',$$
		where the sum runs over all singularities $Q$'s on $S_0'$, and $\mathrm{Supp}(N_Q')$ is supported on the inverse image of $Q$ in $S$.
		According to \eqref{eqn-2-5},
		$$\vol(\sF')=(P')^2=K_{\sF'}^2-(N')^2=K_{\sF'}^2-\sum_{Q}(N_Q')^2.$$
		In view of the definition of $c_1^2(\sF')$ in \eqref{eqn-4-1},
		it suffices to prove that for each singularity $Q\in S_0'$,
		\begin{equation}\label{eqn-4-21}
			\sum_{p\in
				N_Q'}\beta(-\lambda_p)=-(N_Q')^2.
		\end{equation}
		The sum above is taken over all singularities of $\sF'$ lying on the support of $N_Q'$ (namely the inverse image of $Q$).
		To this aim, we first claim that
		\begin{claim}\label{claim-4-1}
			Let $C=C_1+C_2+\cdots+C_r$ be a maximal $\sF'$-chain contracted to a singularity $Q\in S_0'$.
			Let $\mu_0=0$, $\mu_1=1$, and the rest $\mu_k$'s be given by the following recursion formula:
			\begin{equation}\label{eqn-4-22}
			\mu_{k-1}-e_k\mu_k+\mu_{k+1}=0, \hskip1cm k=1,2,\cdots,r,
			\end{equation}
			where $e_k=-C_k^2\geq 2$.
			Then around each singularity $p_k\in C_k\cap C_{k+1}$ ($1\leq k \leq r$),
			there exists local coordinate $(x_k,y_k)$, such that
			\begin{enumerate}[$(i)$]
				\item $C_k$ is locally defined by $x_k=0$, and $C_{k+1}$ is locally defined by $y_k=1$;
				\item the foliation $\sF'$ is locally defined by $\omega_k=\mu_ky_kdx_k+\mu_{k+1}x_kdy_k$. 
			\end{enumerate}
		Here $p_r$ is the singularity on $C_r$ different from the point $p_{r-1}=C_{r-1}\cap C_r$, and we understand that $C_{k+1}$ is the local separatrix of $\sF'$ passing through $p_r$ other than $C_r$. In particular, $C_{r+1}$ might not be algebraic, but only locally holomorphic.
		\end{claim}
	\begin{proof}[{Proof of \autoref{claim-4-1}}]
		This is in fact implicitly contained in \cite[\S\,8.2]{bru-04}, the context following Definition\,8.1 loc. cit.
		One can prove inductively the statements hold for each singularity.
		Consider first the case around the singularity $p_1=C_1\cap C_2$.
		Because	$C_1\setminus p_1$ is simply connected,
		the point $p_1$ is a reduced non-degenerate singularity of $\sF'$ with a separatrix inside $C_1$ without holonomy.
		Hence locally the foliation is of type $d(x_1y_1^{\ell_1})=0$, i.e.,
		the foliation $\sF'$ is defined around $p_1$ by
		$$\omega_1=y_1dx_1+\ell_1x_1dy_1,$$
		where $C_1=\{x_1=0\}$ and $C_2=\{y_1=0\}$ around $p_1$, and $\ell_1=e_1=-C_1^2$ by the Camacho-Sad formula \eqref{eqn-2-7} and \eqref{eqn-2-6}.
		By \eqref{eqn-4-22}, $\mu_2=\ell_1=e_1$.
		This proves the statement for the singularity $p_1$.
		
		Look now at the singularity $p_2$.
		The separatrix $C_2$ has $\mu_2$-periodic holonomy around $p_2$,
		and hence locally the foliation is of type $d(x_2^{\mu_2}y_2^{\ell_2})=0$, i.e.,
		the foliation $\sF'$ is defined around $p_2$ by
		$$\omega_2=\mu_2y_2dx_2+\ell_2x_2dy_2,$$
		where $C_2=\{x_2=0\}$ and $C_3=\{y_2=0\}$ around $p_2$. By the Camacho-Sad formula \eqref{eqn-2-7} and \eqref{eqn-2-6} one has
		$$e_2=-C_2^2=\cs(\sF',C_2)=\cs(\sF',C_2,p_1)+\cs(\sF',C_2,p_2)=\frac{1}{\mu_2}+\frac{\ell_2}{\mu_2}.$$
		Hence $\ell_2=\mu_3$ in view of \eqref{eqn-4-22}.
		This proves the statement for the singularity $p_2$.
		The rest cases can proved inductively and we omitted here.
	\end{proof}
		
		Come back to the proof of \autoref{thm813c}.
		Suppose that $C=C_1+C_2+\cdots+C_r$ is a maximal $\sF'$-chain contracted to a singularity $Q\in S_0'$ of type $A_{n,q}$.
		Let $\xi_i$'s and $\mu_i$'s are defined respectively in \autoref{lem-coefficient-N} and \autoref{claim-4-1}.
		Since $\mu_0=\xi_{r+1}=0$ and $\mu_1=\xi_{r}=1$, one deduces inductively by \eqref{eqn-2-2} and \eqref{eqn-4-22} that
		$$\mu_k=\xi_{r+1-k},\qquad \forall~0\leq k \leq r+1,$$
		and that (recall that $\xi_0=n$ and $\xi_1=q$ as proved in \cite[\S\,III.5]{bhpv})
		$$\xi_k\mu_{k+1}-\xi_{k+1}\mu_k=n, \qquad \forall~0\leq k \leq r.$$
		In particular, 
		$$
		\frac{\xi_k}{\mu_{k}}-
		\frac{\xi_{k+1}}{\mu_{k+1}}=\frac{n}{\mu_k\mu_{k+1}}, \qquad \forall~1\leq k \leq r.
		$$
	Taking the summation, we get
$$
\sum_{k=1}^r\dfrac{1}{\mu_k\mu_{k+1}}=\dfrac1n\sum_{k=1}^r
\left(\dfrac{\xi_k}{\mu_{k}}-
\dfrac{\xi_{k+1}}{\mu_{k+1}}\right)=
\dfrac1n\left(\dfrac{\xi_1}{\mu_1}-
\dfrac{\xi_{r+1}}{\mu_{r+1}}\right)=\dfrac{q}{n}.
$$
By \autoref{claim-4-1}, the foliation $\sF'$ is locally defined by $\omega_k=\mu_ky_kdx_k+\mu_{k+1}x_kdy_k$ around $p_k$.
It follows that the eigenvalue of $\sF'$ at $p_k$ is $\lambda_{p_k}=-\frac{\mu_{k+1}}{\mu_k}$.
Moreover, according to \eqref{eqn-4-22} one proves inductively that $\gcd(\mu_k,\mu_{k+1})=1$.
Hence
$\beta(-\lambda_{p_k})=\frac{1}{\mu_k\mu_{k+1}}$ by \eqref{eqn-def-beta}.
Therefore,
\begin{align*}
	\sum_{p\in
		N_Q'}\beta(-\lambda_p)&=\sum_{k=1}^r\beta(-\lambda_{p_k})=
	\sum_{k=1}^r\dfrac{1}{\mu_k\mu_{k+1}}
	=\dfrac{q}{n}=-(N_Q')^2.
\end{align*}
The last equality above follows from \eqref{eqn-N_Q^2}.
This proves \eqref{eqn-4-21} and hence completes the proof.
	\end{proof}
	
\begin{corollary}\label{cor-4-1}
	Let $(S,\sF)$ be a foliated surface.
	Then the first Chern number $c_1^2(\sF)$ defined in \autoref{def-chern-numbers} is a well-defined birational invariant satisfying $c_1^2(\sF)=\vol(\sF)$, which is a non-negative rational number.
\end{corollary}
\begin{proof}
	Let $(S',\sF')$ be any relatively minimal model of $(S,\sF)$.
	If the canonical divisor $K_{\sF'}$ is not pseudo-effective, then the canonical divisor $K_{\sF''}$ is not pseudo-effective for any relatively minimal model $(S'',\sF'')$ by \autoref{lem-5-pseudo-eff}. It follows that $c_1^2(\sF)=0=\vol(\sF)$ in this case.
	If $K_{\sF'}$ is pseudo-effective, then the corollary follows from \autoref{thm813c} and the fact that the volume $\vol(\sF)$ is
	a birational invariant, which is a non-negative rational number.
\end{proof}

\begin{corollary}\label{cor-4-2}
	Let $(S,\sF)$ be a reduced foliated surface (not necessarily relatively minimal)
	with pseudo-effective canonical divisor.
	Let
	$$K_{\sF}=P+N,$$
	be the Zariski decomposition, where $P$ and $N$ are respectively the nef and negative parts.
	Then
	$$c_1^2(\sF)=K_{\sF}^2+\sum_{p\in N
	}\beta_p(\sF).$$
	In other words, one can compute the first Chern number $c_1^2(\sF)$ over any reduced model (not necessarily relatively minimal).
\end{corollary}
\begin{proof}
	By contracting possible $\sF$-exceptional curves on $S$, one can obtain a relatively minimal model.
	Hence it is enough to prove that if $\sigma:\,S \to S'$ is a blowing-up centered at some point $p$, then
	\begin{equation}\label{eqn-4-23}
		K_{\sF}^2+\sum_{p\in N
		}\beta_p(\sF)=K_{\sF'}^2+\sum_{p\in N'
		}\beta_p(\sF'),
	\end{equation}
	where $(S,\sF)$ and $(S',\sF')=(S',\sigma_*\sF)$ are both reduced, and $N$ and $N'$ are respectively the negative parts of $K_{\sF}$ and $K_{\sF'}$.
	By construction, the exceptional curve $\mathcal{E}$ is always $\sF$-invariant (cf. \autoref{lem-2-1}) and
	\[K_{\sF}=\sigma^*(K_{\sF'})+(1-\ell_p)\mathcal{E},\]
	where $\ell_p=0$ (resp. $\ell_p=1$) if $p$ is a regular point (resp. singularity).
	Remark also that if $p$ is a saddle-node, then it does not lie on the negative part $N$.
	Hence \eqref{eqn-4-23} follows by \autoref{lem-2-1}. This completes the proof.
\end{proof}
\begin{remark}
	One can prove similarly that
	$$	K_\sF N_\sF=K_{\sF'} N_{\sF'},$$
	for any two reduced foliated surfaces $(S,\sF)$ and $(S',\sF')$ birational to each other.
	Suppose moreover that there are no saddle-nodes. Then 
	one shows also with the help of \autoref{lem-2-1} that
	$$\sum_{p\in Sing(\sF)}\chi_p(\sF)=\sum_{p'\in Sing(\sF')}\chi_{p'}(\sF'),$$
	for any two reduced foliated surfaces $(S,\sF)$ and $(S',\sF')$ birational to each other.
	In other words, one can compute $\chi(\sF)$ over any reduced model (not necessarily relatively minimal) if its reduced model has no saddle-nodes.
	In particular, $\chi(\sF)$ (as well as $c_2(\sF)$ by the Noether equality \eqref{eqn-chern-1-1}) is well-defined for any arbitrary reduced foliation $\sF$ without saddle-nodes.
	In general, we prove that $c_2(\sF)$ (and hence also $\chi(\sF)$) is well-defined
	 in the following.
\end{remark}

\begin{proposition}\label{prop-invariant-c_2}
	Let $(S',\sF')$ and $(S'',\sF'')$ be two relatively minimal foliated surfaces birational to each other.
	Then either
	\begin{enumerate}[(i)]
		\item both $K_{\sF'}$ and $K_{\sF''}$ are non-pseudo-effective, in which case $c_2(\sF)=0$ by definition;
		\item both $K_{\sF'}$ and $K_{\sF''}$ are pseudo-effective, in which case
		\begin{equation}\label{eqn-5-10-1}
			\sum_{p\not\in N'}\beta_p(\sF')=\sum_{p\not\in N'}\beta_p(\sF'').
		\end{equation}
	\end{enumerate} 
	In particular, the second Chern number $c_2(\sF)$ is well-defined for any arbitrary foliation $\sF$.	
\end{proposition}
\begin{proof}
	The two possibilities follow from \autoref{lem-5-pseudo-eff}.
	And in the first case, $c_2(\sF)=0$ by definition.
	It remains to prove \eqref{eqn-5-10-1}, if both $K_{\sF'}$ and $K_{\sF''}$ are pseudo-effective.
	According to \cite[Theorem\,5.1]{bru-04}, a foliated surface admitting two different relatively minimal model only if it is either the {\it very special foliation}
	 or {\it a Riccati foliation}.
	
	Suppose first that both $\sF'$ and $\sF''$ are the very
	special foliations,
	then the eigenvalues on a singularity $p$ away from the support of the negative parts are all non-rational, cf. \cite[\S\,4.2 and Example\,8.1]{bru-04}.
	It follows that
	$$\sum_{p\not\in N'}\beta_p(\sF')=0=\sum_{p\not\in N'}\beta_p(\sF'').$$
	
	Consider next the case when both $\sF'$ and $\sF''$ are Riccati foliations.
	Let $\varphi:\,S' \dashrightarrow S''$ be the birational map with $\sF'=\varphi^*\sF''$.
	Our aim is to prove that the birational map $\varphi$ between $S'$ and $S''$ is
	biholomorphic near any singular point whose eigenvalue
	is a negative rational number.
	Then according to the definitions \eqref{eqn-4-1} and \eqref{eqn-def-beta}, it follows that $$\sum_{p\not\in N'}\beta_p(\sF')=\sum_{p\not\in N'}\beta_p(\sF'')$$ as required.
	Suppose on the contrary that the birational map $\varphi$ is not
	biholomorphic near some singular point $p'\in S'$ with eigenvalue $\lambda_{p'}=-\frac{n'}{m'}\in \mathbb{Q}^-$.
	By exchanging $S'$ and $S''$ if possible, we may assume that $\varphi$ is not defined at $p'$.	
	Let $\sigma':X\to S'$ be the  minimal resolution of the
	indeterminacies of $\varphi: S'\dashrightarrow S''$. Then there is a
	sequence of blowing-ups $\sigma'':X\to S''$ such that
	$\sigma''=\varphi\circ\sigma'$, and the pullback foliations on $X$ are
	the same, $\sG:=(\sigma'')^*\sF''=(\sigma')^*\sF'$.
	$$
	\xymatrix{
		& X\ar[dl]_{\sigma'}  \ar[dr]^{\sigma''}             \\
		S' \ar@{-->}[rr]^{\varphi} & &     S''        }
	$$
	
	By assumption, the exceptional curve $C$ of the last blowing-up of
	$\sigma'$ above $p'$ is not contracted by $\sigma''$. Moreover, $\sigma''$ is not an
	isomorphism near $C$; otherwise, as $C$ is a $\sG$-exceptional curve,
	$\sigma''(C)$ would be an	$\sF''$-exceptional curve, which contradicts the relative minimality of $\sF''$.
	Hence $\sigma''$ can be factorized as $$(X,\sG)\overset{g}\lra
	(X',\sG')\overset{h}\lra (S'',\sF''),$$ where $g$ is the contraction of
	$(-1)$-curves disjoint from $C$.
	Then $g$ is a biholomorphism in a neighborhood of $C$,
	and hence $C_1=g(C)$ is a $\sG'$-exceptional curve.
	Moreover, since $h$ is not an isomorphism near $C_1$ by construction,
	$C_1$ intersects another $\sG'$-exceptional curve
	$C_2$ at one point $q$.
	According to the proof of \cite[Theorem\,5.1]{bru-04}, one sees that $\#(C_1\cap C_2)=1$,
	and that the image $h(C_1)$ must be a fiber $F_1$ of $f$,
	where $f:\,S'' \to B$ is the $\bbp^1$-fibration attached to the Riccati foliation $\sF''$ such that the general fiber of $f$ is transverse to $\sF''$.
	Since $p'$ is a reduced singularity with negative rational eigenvalue $\lambda_{p'}=-\frac{n'}{m'}$,
	the eigenvalues of any possible singularities on $(\sigma')^{-1}(p')$ are all negative rational numbers by \autoref{lem-2-1}\,(ii).
	It follows that the image $h(C_2)$ would be a singularity $p$ on $F_1$.
	Since $\sF''$ is reduced, the eigenvalue at $p$ is a negative rational eigenvalue $\lambda_{p}=-\frac{n}{m}$.
	Since $K_{\sF''}\cdot F_1=F_1^2=0$, by the two formulas in \autoref{prop-2-2} one sees that there is exactly one another singularity $\tilde p$ (other than $p$) on $F_1$ with eigenvalue $\lambda_{\tilde p}=-\lambda_p=\frac{n}{m}$.	
	This gives a contradiction since $\sF''$ is assumed to be reduced.
\end{proof}

To end this subsection, 
we prove a useful estimation on the lower bound of $c_1^2(\sF)=\vol(\sF)$
for a foliation with non-degenerate singularities.

\begin{lemma}\label{lem-7-5}
	Let $(S_0,\sF_0)$ be a foliated surface with pseudo-effective canonical divisor $K_{\sF_0}$, and
	$$K_{\sF_0}=P_0+N_0,$$ be the Zariski decomposition of $K_{\sF_0}$.
	Suppose that all the singularities of $\sF_0$ are non-degenerate,
	and denote by $T$ the set of singularities of $\sF_0$ which either lie on the support of $N_0$ or are non-reduced.
	Then
	\begin{equation}\label{eqn-7-12}
		\vol(\sF) =c_1^2(\sF) \geq K_{\sF_0}^2 + \sum_{p\in T}\beta_p(\sF_0).
	\end{equation}
\end{lemma}
\begin{proof}
	Let $\pi:\,S' \to S_0$ be a minimal resolution of non-reduced singularities of $\sF_0$,
	i.e., the induced foliation $\sF'=\pi^*(\sF_0)$ is reduced and there is no $\sF'$-exceptional curve contained in the inverse image of non-reduced singularities under $\pi$.
	Of course, $\sF'$ might not be relatively minimal. However it is reduced and birational to $\sF_0$.
	Hence one can compute the volume (or the first Chern number) using the reduced model $\sF'$ based on \autoref{thm813c} and \autoref{cor-4-2}.
	
	
	 Let
	\begin{equation}\label{eqn-7-14}
		 K_{\sF'}=P'+N',
	\end{equation}
	be the Zariski decompositions of $K_{\sF'}$.
	We claim
	\begin{claim}\label{claim-7-1}
		(i). Let $p$ be a reduced singularity of $\sF_0$ on $N_0$.
		Then the inverse image $\pi^{-1}(p)$ is a reduced singularity of $\sF'$ lying on the negative part $N'$.
		
		(ii). Let $p$ be a non-reduced degenerate singularity of $\sF_0$.
		Suppose that there is no saddle-node of $\sF'$ lying on $\pi^{-1}(p)$.
		Then all the singularities of $\sF'$ on $\pi^{-1}(p)$ lie on the negative part $N'$.
	\end{claim}
	
	We first prove \autoref{lem-7-5} based on \autoref{claim-7-1}, whose proof is postponed.
	Let $p_1,\cdots,p_k$ be the non-reduced singularities of $\sF_0$,
	such that there is no saddle-node of $\sF'$ lying on $\pi^{-1}(p_i)$.
	According to \eqref{eqn-2-16},
	\begin{equation}\label{eqn-7-15}
		K_{\sF'}=\pi^*K_{\sF_0}-\sum_{i=1}^{k} E_{p_i},
	\end{equation}
	where $E_{p_i}$ is the last exceptional curve lying on $\pi^{-1}(p_i)$.
	Note that $\beta_p(\sF_0)<0$ for any non-reduced non-degenerate singularity $p$ of $\sF_0$.
	Hence it suffices to prove
	\begin{equation}\label{eqn-4-2}
		\vol(\sF)\geq K_{\sF_0}^2+ \sum_{\text{$p$~is~reduced~}\in N_0}\beta_{p}(\sF_0)+\sum_{i=1}^{k}\beta_{p_i}(\sF_0).
	\end{equation}
	By definition, $\beta_p(\sF')\geq 0$ for any reduced singularity $p$ of $\sF'$.
	Hence by \autoref{cor-4-2} and \autoref{claim-7-1} together with \eqref{eqn-7-13},
	$$\begin{aligned}
		\vol(\sF)=c_1^2(\sF)=c_1^2(\sF')&\,=K_{\sF'}^2 + \sum_{p'\in N'} \beta_{p'}(\sF')\\
		&\,\geq K_{\sF'}^2 + \sum_{\text{$p$~is~reduced~}\in N_0}\beta_{\pi^{-1}(p)}(\sF')+\sum_{i=1}^{k}\sum_{p'\in \pi^{-1}(p_i)} \beta_{p'}(\sF')\\
		&\,=\big(K_{\sF_0}^2-k\big)+ \sum_{\text{$p$~is~reduced~}\in N_0}\beta_{p}(\sF_0)+\sum_{i=1}^{k}\big(\beta_{p_i}(\sF_0)+1\big)\\
		&\,=K_{\sF_0}^2+ \sum_{\text{$p$~is~reduced~}\in N_0}\beta_{p}(\sF_0)+\sum_{i=1}^{k}\beta_{p_i}(\sF_0).
	\end{aligned}$$
	In the above, if $p\in N_0$ is a reduced singularity,
	then $\pi$ is an isomorphism around $p$, and hence $\beta_{\pi^{-1}(p)}(\sF')=\beta_{p}(\sF_0)$.
	This proves \eqref{eqn-4-2}, and hence also \eqref{eqn-7-12}.
	It remains to prove \autoref{claim-7-1}.
\end{proof}

\begin{proof}[{Proof of \autoref{claim-7-1}}]
	(i). Consider the Zariski decompositions of $K_{\sF_0}$ and $K_{\sF'}$ in \eqref{eqn-7-14}.
	By \eqref{eqn-7-15},
	$$\pi^*P_0+\pi^*N_0=\pi^*K_{\sF_0}=K_{\sF'}+\sum_{i=1}^{k} E_k=P'+N'+\sum_{i=1}^{k} E_{p_i}.$$
	Note that the first equality is just the Zariski decomposition of $\pi^*K_{\sF_0}$.
	It follows that $P'\leq \pi^*P_0$, or equivalently,
	$$\pi^*N_0\leq N'+\sum_{i=1}^{k} E_{p_i}.$$
	In particular, the inverse image of any reduced singularity of $\sF_0$ on $N_0$
	is still reduced (since $\pi$ is an isomorphism around reduced singularities) lying on $\text{Supp}\big(N'\big)$.
	
	(ii). This is in fact follows from the resolution process of a non-reduced degenerate singularity presented above \autoref{lem-2-2}.
	Indeed, fixing $p=p_i$,
	the inverse image $\pi^{-1}(p)$ consists of a chain of rational curves
	$C_1+\cdots+C_n$ with $C_j=E_{p_i}$ being the last exceptional curve.
	Then there is no singularity of $\sF'$ lying on $C_j=E_{p_i}$, and
	$$C_{j-1}\cdot K_{\sF'}=C_{j-1}\cdot\Big(\pi^*K_{\sF_0}-\sum_{i=1}^{k} E_{p_i}\Big)=-C_{j-1}C_j=-1.$$
	Hence $C_{j-1}\subseteq \text{Supp}(N')$.
	Suppose $N'=b_{j-1}C_{j-1}+N_{j-1}'$
	Thus
	$$0\leq C_{j-2}\cdot P'=C_{j-2}\cdot (K_{\sF'}-N')=-C_{j-2}N'<-C_{j-2}\cdot N_{j-1}'.
	$$
	Hence $C_{j-2}\subseteq \text{Supp}(N_{j-1}') \subseteq \text{Supp}(N')$.
	One proves then inductively that $C_{s}\subseteq \text{Supp}(N')$ for any $1\leq s\leq j-1$.
	Similarly one can prove inductively that $C_{s}\subseteq \text{Supp}(N')$ for any $j+1\leq s\leq n$.
	This completes the proof.
\end{proof}

\begin{remark}
	Let
	$$T'=\big\{p\in\text{Sing}(\sF_0)~\big|~p\in N_0,\text{~or $p$ is non-reduced and there is no saddle-node in $\pi^{-1}(p)$}\big\},$$
	where $\sigma:\,S' \to S_0$ is a minimal resolution of the non-reduced singularity $p$. Then $T'\subseteq T$, and $T'=T$ if $\sF_0$ is algebraically integral.
	We have in fact proved in the above that
	$$\vol(\sF) =c_1^2(\sF) \geq K_{\sF_0}^2 + \sum_{p\in T'}\beta_p(\sF_0).$$
\end{remark}
	
	\subsection{The Chern numbers of an algebraically integral foliation}\label{sec-chern-algebraic}
	In this subsection, we compute the Chern numbers of an algebraically integral foliation $\sF$.    
	The main purpose
	is to prove that the Chern numbers are exactly the modular
	invariants of the corresponding fibration.
	Since the Chern numbers are birational invariants,
	we may thus always assume that $\sF$ is relatively minimal in this subsection.
	A relatively algebraically integral minimal foliation $\sF$ is defined by a normal-crossing
	fibration $f:S\to B$ of genus $g\geq 0$.
	If $g=0$, then by definition the Chern numbers are all zero, and all the modular invariants of $f$ are also zero.
	Thus we may assume $g\geq 1$ in the following.

	Let $f:\,S \to B$ be a surface fibration defining a relatively minimal foliation $\sF$
	by taking the saturation of $\ker(df:\,T_{S} \to f^*T_{B})$ in $T_{S}$.
	Let $F$ be a singular fiber of $f$. Then $F$ is a
	normal-crossing divisor containing no redundant $(-1)$-curves. Let
	$p$ be a singular point of $F_{\rm red}$. Then $(F,p)$ is defined
	locally by $x^ay^b=0$ for some positive integers $a,b$. Let
	$$\beta_p:=\frac{\gcd(a,b)^2}{ab}, \hskip1cm
	\beta_F:=\sum_{p\in F}\beta_p. $$
	Locally around $p$, the corresponding foliation $\sF$ is defined by $d(x^ay^b)=0$,
	or equivalently by $\omega=aydx+bxdy=0$.
	As a singular point of $\sF$, the eigenvalue
	$\lambda_p=-\frac{a}{b}\in\mathbb Q^-$. Thus
	$$\beta_p=\beta(-\lambda_p)=\beta_p(\sF).$$
	We can also talk about the maximal $\sF$-chains in the fibers. We
	use $p\in N$ to denote a singular point on some maximal $\sF$-chain.
	We define
	\begin{equation}\label{eqn-4-25}
	c_{-1}(F)=\sum_{p\in F\cap N}\beta(-\lambda_p).
	\end{equation}
	Let
	$$\mu_F=\sum\limits_{p\in F}\mu_p(F_{\rm red}),\qquad \alpha_F=\sum_{p\in F}(m_p(F_{\rm red})-2)^2,$$
	where $\mu_p(F_{\rm red})$ is the Milnor number of the reduced part $F_{\rm red}$ at $p$,
	$m_p(F_{\rm	red})$ is the multiplicity of $F_{\rm red}$ at $p$, and the sums run
	over all singular points of $F_{\rm red}$.
	The number $\mu_F$ is called the total Milnor
	number of the singular points of $F_{\rm red}$.
	Because $F$ is normal-crossing, it follows that
	\begin{equation}\label{eqn-4-24}
		\mu_F=\sum_{p\in F}1=\#\{\text{singularities of $F_{\rm red}$}\},\qquad \alpha_F=0.
	\end{equation}
	
	In \cite[Definition\,II]{Tan1996}, the local Chern
	numbers	$c_1^2(F)$, $c_2(F)$ and $\chi_F$ are introduced for a fiber $F$.
	They satisfy Noether's equality
	$$\chi_F=\frac1{12}(c_1^2(F)+c_2(F)).$$
	Moreover, we have the following formulas to compute them (cf. \cite[Theorem\,3.1]{Tan1996})
	\begin{equation*}
	\begin{cases}
	c_1^2(F)=4(g-p_a(F_{\rm red}))+F_{\rm red}^2+\alpha_F-c_{-1}(F), &\\
	c_2(F)=2(g-p_a(F_{\rm red}))+\mu_F-\beta_F+c_{-1}(F), &
	\end{cases}
	\end{equation*}
	where $p_a(F_{\rm red})$ is the arithmetic genus of $F_{\rm red}$.
	These local Chern numbers are closely related to the modular invariants recalled in \autoref{sec-alg},
	cf. the corollary after \cite[Theorem\,A]{Tan1996}.
	\begin{equation*}
		\left\{
		\begin{aligned}
			\kappa(f)&\,=K_S^2-8(g-1)(g(B)-1)-\sum_F c_1^2(F)=K_{f}^2-\sum_F c_1^2(F), \\
			\delta(f)&\,=c_2(S)-4(g-1)(g(B)-1)-\sum_{F}c_2(F)=e_f-\sum_{F}c_2(F),\\
			\chi(f)&\,=\chi(\mathcal O_S)-(g-1)(g(B)-1)-\sum_F\chi_F=\chi_f-\sum_F\chi_F.
		\end{aligned}\right.
	\end{equation*}
	We refer to \cite{LuTan} for more properties of these local Chern numbers.
	
	\begin{theorem}\label{thm813b}
		Let $f:\,S \to B$ be a surface fibration of genus $g\geq 1$ defining a relatively minimal foliation $\sF$
		by taking the saturation of $\ker(df:\,T_{S} \to f^*T_{B})$ in $T_{S}$.
		Then the Chern numbers of $\sF$ are exactly	the modular Chern numbers of the surface fibration $f$,
		$$c_1^2(\sF)=\kappa(f), \hskip0.5cm c_2(\sF)=\delta(f), \hskip0.5cm \chi(\sF)=
		\lambda(f).$$
	\end{theorem}
	\begin{proof}   Let
		$\Delta(f)=\sum\limits_{F}(F-F_{\rm red})$ be the discriminant divisor of
		$f$, where the sum runs over all singular fibers.
		By \eqref{eqn-2-3},
		$$ K_{\sF}=K_{S/B}-\Delta(f),\qquad
		N_\sF^*=K_S-K_{\sF}=f^*K_B+\Delta(f).
		$$
		Note that $\Delta(f)^2=\sum\limits_F F_{\rm red}^2$,
		and that the support of the negative part $N$ of $K_{\sF}$ is contained in singular fibers of $f$.
		By \eqref{eqn-4-25} and \eqref{eqn-4-24},
		\begin{align*}
		\sum_F c_1^2(F)&=\sum_F4(g-p_a(F_{\rm red}))+\sum_F F_{\rm
			red}^2-\sum_F\sum_{p\in F\cap N}
		\beta(-\lambda_p)\\
		&=\sum_F2K_S(F-F_{\rm red})-\sum_FF_{\rm red}^2-
		\sum_{p\in N}\beta(-\lambda_p)\\
		&=2K_{S/B}\cdot \Delta(f)-\Delta(f)^2-\sum_{p\in N}\beta(-\lambda_p).
		\end{align*}
		Hence
		\begin{align*}
		\kappa(f)&=K_{f}^2-\sum_{F}c_1^2(F)
		=K_{S/B}^2-2K_{S/B}\cdot \Delta(f)+\Delta(f)^2+\sum_{p\in N}\beta(-\lambda_p)\\
		&=(K_{S/B}-\Delta(f))^2+\sum_{p\in
			N}\beta(-\lambda_p)=K_\sF^2+\sum_{p\in
			N}\beta(-\lambda_p)=c_1^2(\sF).
		\end{align*}
		Similarly,
		\begin{align*}
		c_2(F)&=2(g-p_a(F_{\rm red}))+\mu_F-\sum_{p\in F,\, p\not\in
			N}\beta(-\lambda_p)=e_F-\sum_{p\in F,\, p\not\in
			N}\beta(-\lambda_p),
		\end{align*}
		where we use the following formula (cf. \cite[Formula (7) on p.229]{Tan1994}):
		$$e_F:=\chit(F_{\rm red})+2g-2=2(g-p_a(F_{\rm red}))+\mu_F.$$
		Therefore,
		\begin{align*}
		\delta(f)&=e_f-\sum_{F}c_2(F)=\sum_{F}(e_F-c_2(F))\\
		&=\sum_F\sum_{p\in F,\, p\not\in
			N}\beta(-\lambda_p)=\sum_{p\not\in N}\beta(-\lambda_p)=c_2(\sF).
		\end{align*}
		Finally, by Noether's equalities \eqref{eqn-chern-1-1} and \eqref{eqn-noe-modular},
		we get $\lambda(f)=\chi(\sF)$.
	\end{proof}

	

	\section{Foliations of non-general type}\label{sec-non-general-type}
	In this section, we compute the Chern numbers of foliations of non-general type.
	By the birational invariance of the Chern numbers, we may consider only the relatively minimal foliated surfaces.
	It depends on the classification of the foliated surface of non-general type; see the tabular in \autoref{sec-2-3}.
	\autoref{thm-chern-2} will be proved at the end of this section;
	while \autoref{cor-chern-1} is a direct consequence whose proof will be thus omitted.
	\begin{lemma}\label{lem-5-2}
		Let $\sF$ be an algebraically integral relatively minimal foliation of non-general type. Then
		\begin{equation}\label{eqn-5-6}
			\left\{\begin{aligned}
			&c_2(\sF)=12\chi(\sF)=0, &&\text{if $\sF$ is induced by an isotrivial fibration of genus $g\geq 0$};\\
			&c_2(\sF)=12\chi(\sF)>0, &&\text{if $\sF$ is induced by a non-isotrivial fibration of genus $g= 1$}.
			\end{aligned}\right.
		\end{equation}
	\end{lemma}
\begin{proof}
	By \eqref{eqn-chern-1-2}, the Chern numbers of $\sF$ are equal to the corresponding modular invariants of the fibration $f$.
	When the fibration $f$ is isotrivial, then all these modular invariants are zero by definition.
	When $f$ is a non-isotrivial fibration of genus $g=1$,
	it is well-known that $\delta(f)=12\lambda(f) \neq 0$.
	This completes the proof.
\end{proof}

\begin{lemma}\label{lem-5-3}
	Let $\sF$ be a Hilbert modular foliation. Then
	$c_2(\sF)=0$.
\end{lemma}
\begin{proof}
	We first briefly recall the construction of Hilbert modular foliations,
	and refer to \cite[\S\,9.5]{bru-04} and \cite[\S\,4]{bru-03} for more details.
	Let $\cald$ be the unit disc,
	and $\Gamma\subseteq \text{Aut}(\cald\times \cald)=PSL(2,\mathbb{R}) \times PSL(2,\mathbb{R})$
	be a subgroup such that it does not contain any finite index subgroup of type
    $\Gamma_1\times \Gamma_2$ with $\Gamma_i\subseteq PSL(2,\mathbb{R})$.
    The quotient surface $\cald\times \cald/\Gamma$ admits at most quotient singularities, resolved by Hirzebruch-Jung chains of rational curves (denoted by $D_i$'s).
    This quotient may not be compact, but it can be compactified by adding one or more cycles of rational curves (denoted by $E_j$'s).
    By a minimal resolution of these possible singularities and compactification,
    one obtains a smooth projective surface $S$, which is called a Hilbert modular surface.
    There are two natural foliations, denoted by $\sF$ and $\sG$ respectively,
    arising from the vertical and the horizontal ones on $\cald\times\cald$.
    Both foliations will be referred as Hilbert modular foliations.
    The Hirzebruch-Jung chains $D_i$'s and
    the cycles of rational curves $E_j$'s are the only
    compact leaves of both $\sF$ and $\sG$.
    Moreover, both foliations are reduced whose singularities are all non-degenerate and equal to the nodes of $D_i$'s and $E_j$'s.
    These Hirzebruch–Jung chains $D_i$'s are maximal $\sF$-chains and $\sG$-chains in sense of \autoref{def-max-chain}.
    Outside $D_i$'s and $E_j$'s, $\sF$ and $\sG$ are transverse to each other.
    Hence one has
    $$K_{\sF}=N_{\sG}^*\otimes \mathcal{O}_S(H),$$
    where $H$ is a positive integral divisor whose support is equal to the union of
    these compact leaves $D_i$'s and $E_j$'s.
    
   By the above construction, the singularities of $\sF$ are exactly the nodes of $D_i$'s and $E_j$'s, and the Hirzebruch-Jung chains $D_i$'s are maximal $\sF$-chains.
   It follows that the singularities of $\sF$ which are not on the support of $N$ are exactly the nodes of $E_j$'s, where $N$ is the negative part of the Zariski decomposition of $K_{\sF}=P+N$.
   By the definition of the second Chern number $c_2(\sF)$ in \autoref{def-chern-numbers},
   it is enough to prove that for any node $p$ of some cycle $E_j$ of rational curves,
   \begin{equation}\label{eqn-5-3}
   	\lambda_p \not\in \mathbb{Q}\setminus\{0\},
   \end{equation}
   where $\lambda_p$ is the eigenvalue of $\sF$ at $p$.
	
   Let $E_j=C_1+C_2+\cdots+C_r$ be a cycle of rational curve as follows.
   $${\setlength{\unitlength}{6mm}
   	\begin{tikzpicture}
   	\draw[thick] (0.5,1) -- (0.5,3);
   	\draw[thick] (0,1.5) -- (3,0.8);
   	\draw[thick] (0,2.5) -- (3,3.2);
   	\draw[thick] (2,3.2) -- (5,2.5);
   	
   	\filldraw[black] (4,1.9) circle (1pt);
   	\filldraw[black] (3.8,1.7) circle (1pt);
   	\filldraw[black] (3.6,1.5) circle (1pt);
   	\filldraw[black] (3.4,1.3) circle (1pt);
   	
   	\filldraw[black] (0.5,1.38) circle (1.5pt);
   	\filldraw[black] (0.5,2.62) circle (1.5pt);
   	\filldraw[black] (2.5,3.1) circle (1.5pt);
   	\filldraw[black] (4.5,2.6) circle (1.5pt);
   	\filldraw[black] (2.5,0.92) circle (1.5pt);
   	
   	\filldraw[black] (-0.1,2) node[anchor=west]{$C_1$};
   	\filldraw[black] (1,3.1) node[anchor=west]{$C_2$};
   	\filldraw[black] (3.3,3) node[anchor=west]{$C_3$};
   	\filldraw[black] (1,0.9) node[anchor=west]{$C_r$};
   	
   	\filldraw[black] (0.45,1.55) node[anchor=west]{$p_1$};
   	\filldraw[black] (0.45,2.45) node[anchor=west]{$p_2$};
   	\filldraw[black] (2.25,2.85) node[anchor=west]{$p_3$};
   	\filldraw[black] (4.2,2.35) node[anchor=west]{$p_4$};
   	\filldraw[black] (2.25,1.15) node[anchor=west]{$p_{r}$};
\end{tikzpicture}}$$
By \autoref{lem-2-2}, we may assume that
$$\cs(\sF,C_i,p_{i})=\lambda_i,\qquad \cs(\sF,C_{i-1},p_i)=\frac{1}{\lambda_i},\qquad \forall\,1\leq i\leq r,$$
where $\lambda_i$ is the eigenvalue of $\sF$ at $p_i$,
and we under stand the subscript $0$ as $r$ (resp. $r+1$ as $1$):
for instance, we understand $C_{0}=C_r$.
Note that $\sF$ is reduced. Hence $\lambda_i$ can not be a positive rational number.
We will prove by contradiction that $\lambda_i$ can not be negative rational number either.

Suppose that there exists some $1\leq j \leq r$, such that the eigenvalue $\lambda_{i_0}=-\frac{m_{i_0}}{n_{i_0}}$ is a negative rational number with $m_{i_0}\geq 1, n_{i_0}\geq 1$ and $\gcd(m_{i_0},n_{i_0})=1$.
By the formula \eqref{eqn-2-7},
\begin{equation}\label{eqn-5-5}
	\lambda_i+\frac{1}{\lambda_{i+1}}=\cs(\sF,C_i,p_{i})+\cs(\sF,C_i,p_{i+1})=C_i^2\leq -2,\qquad \forall\,1\leq i \leq r.
\end{equation}
If $\lambda_{i_0}\geq -1$, then one show by induction that $\lambda_i$ is a negative rational number with $\lambda_{i}\geq -1$ for $j\leq i \leq r+1$.
In particular, $\lambda_1=\lambda_{r+1}\geq -1$.
Then by repeatedly using \eqref{eqn-5-5}, one shows that every eigenvalue
$\lambda_i$ is a negative rational number with $\lambda_{i}\geq -1$.
If $\lambda_{i_0}\leq -1$, then one show by induction that $\lambda_i$ is a negative rational number with $\lambda_{i}\leq -1$ for $0\leq i \leq j$.
In particular, $\lambda_r=\lambda_0\leq -1$.
Similarly, one shows that every eigenvalue
$\lambda_i$ is a negative rational number with $\lambda_{i}\leq -1$.
In any case, one shows that every eigenvalue $\lambda_i=-\frac{m_i}{n_i}$ is a negative rational number with $m_i\geq 1, n_i \geq 1$ and $\gcd(m_i,n_i)=1$.
Let $e_i=-C_i^2\geq 2$. Then by \eqref{eqn-5-5}, for every $1\leq i \leq r$, it holds
$$\frac{m_i}{n_i}+\frac{n_{i+1}}{m_{i+1}}=e_i,\quad \Longrightarrow\quad
m_im_{i+1}+n_in_{i+1}=n_im_{i+1}e_i.$$
Since $\gcd(m_i,n_i)=1$ and $n_i$ divides both $n_in_{i+1}$ and $n_im_{i+1}e_i$,
it follows that $n_i$ divides $m_{i+1}$.
Similarly, one shows that $m_{i+1}$ divides $n_i$.
Hence $m_{i+1}=n_i$, and the above equality becomes
$$n_{i-1}+n_{i+1}=n_ie_i,\qquad \forall\,1\leq i \leq r.$$
Let $D=\sum\limits_{i=1}^{r}n_iC_i$ be an effective divisor supported on $E_j=C_1+C_2+\cdots+C_r$.
The above relation shows that $D\cdot C_i=0$ for any $1\leq i \leq r$.
On the other hand, it is known that the intersection matrix $\big(C_i\cdot C_j\big)$ is negatively definite, cf. \cite[p.146]{bru-03} or \cite[p.124]{bru-04}.
This gives a contradiction,
and the proof is complete.
\end{proof}
	
\begin{lemma}\label{lem-5-4} 
	Let $\sF$ be a transcendental Riccati foliation. Then
		$c_2(\sF)=0$.
\end{lemma}	
\begin{proof}
	First, recall that a (transcendental) Riccati foliation $(S, \sF)$ is a foliation
	associated with a rational fibration $\pi:S\to B$ whose general
	fiber $F$ is transverse to $\sF$. Every Riccati foliation has a
	standard form (\cite[Proposition\,4.2]{bru-04}) where $\pi$
	is a $\mathbb P^1$-bundle. The singularities of $\sF$ are on one of the
	invariant fibers $F$ of the following forms.	
	\begin{enumerate}
		\item[\textrm{I}$_\lambda$)] {\it Non-degenerate fiber:} $F$
		admits two non-degenerate singularities with eigenvalues $\pm \lambda\neq0$.
		
		\item[\textrm{II})] {\it Semi-degenerate fiber:}
		$F$ admits  two saddle-nodes of the same multiplicity, or one
		saddle-node of multiplicity two.
		
		\item[\textrm{III})] {\it Nilpotent fiber:}  $F$ admits only one nilpotent singularity.
	\end{enumerate}
	Based on \autoref{lem-2-1}\,(ii), by choosing proper flippings, one can further assume that
	the non-degenerate fibers {I}$_\lambda$ satisfy
	$0\leq\mathrm{Re}\,\lambda\leq\frac{1}{2}$. Then the minimal
	resolution $\sigma:\,\wt S \to S$ of this standard model is a relatively minimal
	model $\wt \sF$ (relatively minimal model for a Riccati foliation may not be unique!), and the canonical divisor $K_{\wt \sF}$ is described explicitly
	below \cite[Proposition\,4.2]{bru-04}. Based on this relatively
	minimal model, one sees that any singular point $p$ with
	$\lambda_p\in \mathbb Q^-$ is contained in $\wt N$, the negative part of
	the Zariski decomposition $K_{\wt \sF}=\wt P+ \wt N$.
	Indeed, any fiber admitting a non-reduced singularity is either of type I$_{\lambda}$) of type III).
	Let $F_1,\cdots, F_s$ be all such fibers, and let $\wt F_i$ be its corresponding fiber in its relatively minimal model as above.
	Then there is exists a unique component $C_i \subseteq \wt F_i$ for each $1\leq i \leq s$ (cf. the text at the end of \cite[\S\,4.1]{bru-04}), such that
	$$K_{\wt \sF}= \sigma^*(K_{\sF})-\sum_{i=1}^{s} C_i
	=\tilde \pi^*(\call)-\sum_{i=1}^{s} C_i=\wt P+ \wt N,$$
	where
	$$\wt P
	=\tilde \pi^*(\call)-\sum_{i=1}^{s}\frac{\wt F_i}{m_i},\qquad \wt N=\sum_{i=1}^{s}\Big(\frac{\wt F_i}{m_i}-C_i\Big),$$
	where $m_i$ is the multiplicity of $C_i$ in $\wt F_i$.
	This shows that all the singularities on $\wt F_i$ lie on the negative part $\wt N$ as required.
	By \autoref{prop-invariant-c_2}, one can compute the second Chern number using this relatively minimal model.
 Thus $c_2(\sF)=c_2(\wt \sF)=0$ by definition.
\end{proof}

\begin{lemma}\label{lem-5-5} 
	Let $\sF$ be a transcendental turbulent foliation. Then
	$c_2(\sF)=0$.
\end{lemma}	
\begin{proof}
	Since $\sF$ is a turbulent foliation,
	there is an isotrivial elliptic fibration $\pi:\,S\to B$ whose general
	fiber $F$ is transverse to $\sF$, equivalently, $K_\sF\cdot F=0$.
		
	We can choose a model such that $\pi$ is  relatively minimal
		normal-crossing fibration, i.e., the singular fibers
		$F_1,\cdots,F_s$ of $\pi$ are normal-crossing divisors without
		redundant $(-1)$-curves, see the list in \cite[\S\,4.3,
		p.57]{bru-04}. Then the singularities of $\sF$ are reduced and lie on the nodes of the singular fibers.
		Contracting the $\sF$-exceptional curves in the fiber, we get a
		relatively minimal model, still denoted by $\sF$ by abuse of notations.
		We claim that all the singularities of the
		relatively minimal model $\sF$ are in the negative part $N$ of the Zariski
		decomposition $K_{\sF}=P+N$.
		Indeed, let $F_{i,red}$ be the reduced part of the singular fiber $F_i$.
		Then
		$$K_{\sF}=\pi^*(\call)+\sum_{i=1}^{s} F_{i,red}=P+N,$$
		where
		$$P=\pi^*(\call)+\sum_{i=1}^{s}\frac{F_i}{m_i},\qquad N=\sum_{i=1}^{s}\Big(F_{i,red}-\frac{F_i}{m_i}\Big),
		\qquad\text{with~$m_i=\max\limits_{j}\{m_{ij}\}$ if $F_i=\sum\limits_{j}m_{ij}C_{ij}$.}$$
		Note that each $F_i$ can not be of type $I_b$ or $I_b^*$ with $b\geq 1$;
		see the argument on \cite[\S\,4.3,
		p.57-58]{bru-04}.
		According to the list in \cite[\S\,4.3,
		p.57]{bru-04},
		for each singular fiber $F_i$, there is at most one irreducible component $C_{ij}\subseteq F_i$ with $C_{ij} \not\subseteq \text{Supp}(N)$.
		In particular, all the singularities of $\sF$ on $F_i$ are contained in $N$ as required.
		Thus $c_2(\sF)=0$ by definition.
	\end{proof}

\begin{lemma}\label{lem-5-6} 
	Let $\sF$ be a transcendental foliation generated by a global vector field. Then
	$c_2(\sF)=0$.
\end{lemma}
\begin{proof}
	Since $\sF$ is generated by a global vector field $v$,
	it follows that the attached surface $S$ is of non-general type, cf. \cite[\S\,6.3]{bru-04}.
	Moreover, if $S$ has non-negative Kodaira dimension, the vector field $v$
	has no zeros; equivalently, $\sF$ has no singularities and hence $c_2(\sF)=0$ automatically by definition.
	We assume next that $S$ has negative Kodaira dimension, i.e., $S$ is a ruled surface.
	Note that $\sF$ is transcendental by assumption.
	In particular, $\sF$ is not induced by a rational fibration.
	According to \cite[Proposition\,6.6]{bru-04},
	$\sF$ is transverse to the ruling on $S$,
	i.e., $\sF$ is a Riccati foliation.
	Hence $c_2(\sF)=0$ by \autoref{lem-5-4}.
\end{proof}

We can now prove \autoref{thm-chern-2}.
The proof makes use of the classification of foliations of non-general type;
see the tabular in \autoref{sec-2-3}.
\begin{proof}[{Proof of \autoref{thm-chern-2}}]
	First, $c_1^2(\sF)=0$ since $\sF$ is of non-general type and $c_1^2(\sF)=\vol(\sF)$ by \autoref{thm-chern-1}\,(iii).
	This proves (i).
	If $\sF$ is algebraically integral, (ii) follows from \autoref{lem-5-2}.
	Thus, it remains to prove $c_2(\sF)=0$ if $\sF$ is transcendental in view of the Noether equality \eqref{eqn-chern-1-1}.

	If $\kod(\sF)=-\infty$, then $\sF$ is a Hilbert modular foliation, and $c_2(\sF)=0$ by \autoref{lem-5-3}.
	
	If $\kod(\sF)=0$, by \cite[Theorem\,8.2]{bru-04} and its poof, there exists a ramified cover $\pi':\,S' \to S$ together with a minimal resolution of singularities $\sigma:\,\wt S \to S'$,
	such that the induced foliation $\wt{\sF}=\pi^*(\sF)$ is birational to a foliation generated by a global section with isolated zeros,
	where $\pi=\pi'\circ\sigma$.
	Moreover, the cover $\pi'$ is branched  over $\text{Supp}(N)$,
	where $N$ is the negative part of the Zariski decomposition $K_{\sF}=P+N$.
	According to \autoref{lem-5-6}, $c_2(\wt\sF)=0$.
	In other words, any non-degenerated singularity $\tilde p$ of $\wt \sF$ with eigenvalue $\lambda_{\tilde p}\in \mathbb{Q}_{<0}$ lies on the negative part $\wt N$
	of the Zariski decomposition $K_{\wt \sF}=\wt P+\wt N$.
	Note also that $K_{\wt\sF}=\pi^*(K_{\sF})$ according to the proof of \cite[Theorem\,8.2]{bru-04}.
	It follows that any non-degenerated singularity $p$ of $\sF$ with eigenvalue $\lambda_{p}\in \mathbb{Q}_{<0}$ lies on $N$;
	otherwise, such singularity $p$ would produces $d=\deg(\pi)$ singularities with the same
	eigenvalues and do not lie on $\wt N$.
	Hence $c_2(\sF)=0$ by definition.
	
	Finally, if $\kod(\sF)=1$, then $\sF$ must be either a Riccati or a Turbulent foliation according to the classification of foliations of non-general type.
	Thus $c_2(\sF)=0$ by \autoref{lem-5-4} and \autoref{lem-5-5}.
\end{proof}
	
	\section{The slope inequalities for a foliation of general type}\label{sec-slope}
	In this section, we consider the slope inequalities for a foliation $\sF$ of general type, and prove \autoref{thm-chern-3} and \autoref{cor-chern-2}.
	The slope of $\sF$ is defined as
	$$\lambda(\sF):=\frac{c_1^2(\sF)}{\chi(\sF)}.$$
	We remark that if $\sF$ is of general type, then $c_1^2(\sF)>0$ and also $\chi(\sF)>0$ by the Noether equality \eqref{eqn-chern-1-1} with the non-negativity of $c_2(\sF)$.
	Hence the slope $\lambda(\sF)$ is well-defined and $0<\lambda(\sF)\leq 12$.
	\begin{proof}[{Proof of \autoref{thm-chern-3}}]
		This follows from \autoref{thm-chern-1}\,(iv) and the slope inequality for a semi-stable family of curves by Cornalba-Harris-Xiao (cf. \cite{ch-88,xia-87}).
		Indeed, 
		by the stable reduction theorem (cf. \cite{delignemumford}),
		there exists a base change $\phi:\wt B \to B$ of finite degree, possibly ramified,
		such that the pull-back fibration $\tilde f: \wt S \to \wt B$ is semi-stable.
		We refer to \autoref{sec-alg} for the explanation of the base change process.
		Moreover, by \autoref{thm-chern-1}\,(iv), it holds
		$$K_{\tilde f}^2=\deg(\phi)\cdot c_1^2(\sF),\qquad \chi_{\tilde f}=\deg(\pi)\cdot \chi(\sF).$$
		Thus by the slope inequality for the semi-stable fibration $\tilde f$ (cf. \cite{ch-88,xia-87}),
		it follows that
		\[\lambda(\sF)=\frac{c_1^2(\sF)}{\chi(\sF)}=\frac{K_{\tilde f}^2}{\chi_{\tilde f}} \geq \frac{4(g-1)}{g}.\qedhere\]
	\end{proof}

\begin{proof}[{Proof of \autoref{cor-chern-2}}]
	Both statements follow from the slope inequality \eqref{eqn-chern-3}.
	Indeed, if $\sF$ is algebraically integral with $\lambda(\sF)<4$, then $\lambda(\sF)\geq \frac{4(g-1)}{g}=4-\frac{4}{g}$; equivalently the inequality \eqref{eqn-chern-4} holds.
	Moreover, since $g\geq 2$, it follows that $\lambda(\sF)\geq \frac{4(g-1)}{g}\geq 2$ if $\sF$ is algebraically integral. This proves (ii).
\end{proof}

To end this section, we construct a transcendental foliation of general type whose slope $\lambda(\sF)<2$.
\begin{example}\label{exam-slope<2}
	In this example, we construct a relatively minimal foliation $\sF$ of general type whose slope is $$\lambda(\sF)=\frac{12}{7}<2.$$

	Let $X_0=\bbp^1\times \bbp^1$, and $\sG_0$ be the foliation on $S_0$ defined by
	the global vector field $v_0=x^2\frac{\partial}{\partial x}+y\frac{\partial}{\partial y}$,
	where $x$ and $y$ are respectively the affine coordinates of the factors $\bbp^1$ in $X_0$.
	One checks easily that $K_{\sG_0}=0$, and there are exactly two singularities
	$p_0=(0,0)$ and $p_\infty=(0,\infty)$, both of which are saddle-nodes of multiplicity two.
	In particular, $\sG_0$ is relatively minimal.
	
	Let $D=\{y+x^2(1+y^2)=0\}$, $C_0=\{y=0\}$ and $C_{\infty}=\{y=\infty\}$.
	Then $R_0=D+C_0+C_{\infty}$ is clearly an even divisor and decides a double cover
	$\pi_0:\,S_0 \to X_0$,
	whose branched divisor is exactly $R_0$.
	By a canonical resolution of the singularities of $R_0$, one obtains a smooth surface $S$.
	$$\xymatrix{S \ar[d]_-{\pi} \ar[rr]^-{\rho} && S_0 \ar[d]^-{\pi_0} \\
	X \ar[rr]^-{\sigma} && X_0}$$
	The two points $\{p_0,p_\infty\}$ are the two singularities of $R_0$, and the resolution of these two singularities are exhibited as follows.
   $${\setlength{\unitlength}{6mm}
	\begin{tikzpicture}
	\draw[very thick, dotted] (10,1) -- (10,4);
	\draw[thick] (9,1.5) -- (12,1.5);
	\draw[thick] (9,3.5) -- (12,3.5);
	\draw (11.7,2.98) arc (30:110:2 and 1);
	\draw (9.3,1.58) arc (250:330:2 and 1);
	\filldraw[black] (10,1.5) circle (1.5pt);
	\filldraw[black] (10,3.5) circle (1.5pt);
	\filldraw[black] (12,1.3) node[anchor=west]{$C_0$};
	\filldraw[black] (12,3.4) node[anchor=west]{$C_\infty$};
	\filldraw[black] (12.5,2.5) node[anchor=west]{$D$};
	\draw[->] (12.5,2.6) -- (11.9,2.9);
	\draw[->] (12.5,2.4) -- (11.9,2.1);
	\filldraw[black] (9.9,1.7) node[anchor=west]{$p_0$};
	\filldraw[black] (9.9,3.3) node[anchor=west]{$p_\infty$};
	\filldraw[black] (9.7,0.7) node[anchor=west]{$F_0$};
	
	\filldraw[black] (11.5,3.12) circle (1.5pt);
	\filldraw[black] (11.5,1.88) circle (1.5pt);
	\filldraw[black] (11,3.32) circle (1.5pt);
	\filldraw[black] (11,1.68) circle (1.5pt);
	
	\filldraw[black] (11.2,2.1) node[anchor=west]{$p_1$};
	\filldraw[black] (11.2,2.9) node[anchor=west]{$p_3$};
	\filldraw[black] (10.6,1.9) node[anchor=west]{$p_2$};
	\filldraw[black] (10.6,3.1) node[anchor=west]{$p_4$};

	\draw[->] (6.7,2.5) -- (8.7,2.5);
	\filldraw[black] (7.4,2.7) node[anchor=west]{$\sigma$};
	
	\draw[thick] (3,1) -- (6,1);
	\draw[thick] (3,4) -- (6,4);
	\draw[very thick, dotted] (4.1,4.3) -- (1.7,3.1);
	\draw[very thick, dotted] (4.1,0.7) -- (1.7,1.9);
	\draw[very thick, dotted] (2.8,1.6) -- (0.5,1.6);
	\draw[very thick, dotted] (2.8,3.4) -- (0.5,3.4);
	\draw[very thick, dotted] (0.8,1.2) -- (0.8,3.8);
	\draw[thick] (2.6,3.9) -- (4.5,3);
	\draw[thick] (2.6,1.1) -- (4.5,2);
	
	\filldraw[black] (3.5,4) circle (1.5pt);
	\filldraw[black] (2.3,3.4) circle (1.5pt);
	\filldraw[black] (0.8,3.4) circle (1.5pt);
	
	\filldraw[black] (3.5,1) circle (1.5pt);
	\filldraw[black] (2.3,1.6) circle (1.5pt);
	\filldraw[black] (0.8,1.6) circle (1.5pt);
	
	\filldraw[black] (6,0.9) node[anchor=west]{$\ol C_0$};
	\filldraw[black] (6,3.9) node[anchor=west]{$\ol C_\infty$};
	\filldraw[black] (5.2,2.5) node[anchor=west]{$\ol D$};
	\draw[->] (5.2,2.6) -- (4.6,2.9);
	\draw[->] (5.2,2.4) -- (4.6,2.1);
	
	\filldraw[black] (0.5,0.8) node[anchor=west]{$\ol F_0$};
	\filldraw[black] (-0.2,1.6) node[anchor=west]{$\ol E_1$};
	\filldraw[black] (4,0.4) node[anchor=west]{$E_2$};
	\filldraw[black] (-0.2,3.4) node[anchor=west]{$\ol E_3$};
	\filldraw[black] (4,4.5) node[anchor=west]{$E_4$};

	\filldraw[black] (3.4,1.2) node[anchor=west]{$q_0$};
    \filldraw[black] (3.4,3.8) node[anchor=west]{$q_\infty$};
    
    \filldraw[black] (2.2,1.8) node[anchor=west]{$q_5$};
    \filldraw[black] (2.2,3.2) node[anchor=west]{$q_7$};
    
    \filldraw[black] (0.7,1.8) node[anchor=west]{$q_6$};
    \filldraw[black] (0.7,3.2) node[anchor=west]{$q_8$};
    
    \filldraw[black] (4.3,3.1) circle (1.5pt);
    \filldraw[black] (4.3,1.9) circle (1.5pt);
    \filldraw[black] (3.7,3.37) circle (1.5pt);
    \filldraw[black] (3.7,1.63) circle (1.5pt);
    
    \filldraw[black] (4,2.12) node[anchor=west]{$q_1$};
    \filldraw[black] (4,2.88) node[anchor=west]{$q_3$};
    \filldraw[black] (3.3,1.85) node[anchor=west]{$q_2$};
    \filldraw[black] (3.3,3.15) node[anchor=west]{$q_4$};
\end{tikzpicture}}$$
In the above, $F_0=\{x=0\}$, and
$\ol F_0, \ol D, \ol C_0, \ol C_{\infty}$ are respectively the strict transforms of $F_0, D,C_0,C_{\infty}$.
A solid curve is branched; while a dotted line is not branched.
One of the dotted line is $\ol F_0$, and the rest four dotted ones are the exceptional curves of $\sigma$.
The points $\{p_1,p_2,p_3,p_4\}$ are the four points on $D$ with
$$\tang(\sG_0,D,p_1)=\tang(\sG_0,D,p_2)=\tang(\sG_0,D,p_3)=\tang(\sG_0,D,p_4)=1,$$
and $\{q_i=\sigma^{-1}(p_i),~1\leq i \leq 4\}$ are the inverse images on $\ol D$.

Let $\sG=\sigma^*(\sG_0)$ and $\sF=\pi^*(\sG)$ be the induced foliations.
Then $K_{\sG}=\sigma^*(K_{\sG_0})=0$.
Note that the two branched curves $C_0,C_{\infty}$ are $\sG_0$-invariant,
and $D$ is not $\sG_0$-invariant.
Hence
\begin{equation}\label{eqn-6-5}
	K_{\sF}=\pi^*\big(K_{\sG}+\ol D/2\big)=\pi^*(\ol D)/2.
\end{equation}
It follows immediately that $K_{\sF}$ is nef with $K_{\sF}^2=\frac{\ol D^2}{2}=2$.

Let's next compute the singularities of $\sF$.
First, $\sG_0$ admits exactly two singularities $\{p_0,p_\infty\}$, both of which are saddle-nodes as mentioned above.
The birational map $\sigma$ is composed of four blowing-ups.
By \autoref{lem-2-1}\,(iv) and its proof, one checks immediately that
$\sG$ admits exactly six singularities $\{q_0,q_{\infty},q_5,q_6,q_7,q_8\}$.
The two singularities $\{q_0,q_{\infty}\}$ are saddle-nodes, and the rest four singularities $\{q_5,q_6,q_7,q_8\}$ are non-degenerate with eigenvalues all equal to $-1$.
Besides the two saddle-nodes $\{q_0,q_{\infty}\}$ with
$$\tang(\sG,\ol D,q_0)=\tang(\sG,\ol D,q_\infty)=2,$$
there are four tangent points of $\sG$ on the branch divisor $\ol R=\ol D+\ol C_0+\ol C_{\infty}$, which are $q_i=\sigma^{-1}(p_i)$ ($1\leq i \leq 4$) with
$$\tang(\sG,\ol D,q_i)=1,\qquad \forall~1\leq i \leq 4.$$
Thus $\sF$ admits $14$ singularities, which are the inverse images of $\{q_i\}$'s with $0\leq i \leq 8$ and $i=\infty$.
It is clear that $\pi^{-1}(q_i)$ is still a saddle-node for $i=0$ or $\infty$,
and that $\pi^{-1}(q_i)$ consists two singularities with the same eigenvalues both equal to $-1$ for $5\leq i \leq 8$.
Moreover, a local computation (cf. \cite[Example\,3.4]{bru-04}) shows that
$\pi^{-1}(q_i)$ is a reduced singularity with eigenvalue equal to $-1$ for $1\leq i\leq 4$.
In conclusion, $\sF$ admits $14$ singularities, two of which are saddle-nodes,
and twelve of which are reduced with eigenvalues all equal to $-1$.

We claim that $\sF$ is relatively minimal.
Note first that there exist exactly three $\sG_0$-invariant curves, which are $C_0,C_\infty,F_0$.
Indeed, for any possible $\sG_0$-invariant curve $C$, other than $\{C_0,C_\infty,F_0\}$,
the intersection $C\cap F_0$ would be singularities of $\sG_0$.
Hence $C\cap F_0 \subseteq \{p_0,p_\infty\}$.
However, there are at most two separatrices through a saddle-node.
This shows that such a $\sG_0$-invariant curve $C$ does not exist.
It follows that all the $\sG$-invariant curves are
$$\{\ol F_0,\ol C_0, \ol C_\infty, \ol E_1, E_2, \ol E_3, E_4\}.$$
Hence all the $\sF$-invariant curves are the inverse images of the above seven curves.
On the other hand, let $C$ be any possible $\sF$-exceptional curve.
Then $C$ is an $\sF$-invariant smooth rational curve, and
either there is one singularity $p$ on $C$ with $Z(\sF,C,p)=1$,
or there is two singularities $\{p,p'\}$ on $C$ with $Z(\sF,C,p)=Z(\sF,C,p')=1$.
By a direct computation, one checks that
inverse images of the above seven $\sG$-invariant curves are not $\sF$-exceptional.
This proves that $\sF$ is relatively minimal.

Finally, let's compute the Chern numbers of $\sF$.
Since $\sF$ is relatively minimal and $K_{\sF}$ is nef, it follows that
$$c_1^2(\sF)=K_{\sF}^2=2.$$
Moreover, the nefness of $K_{\sF}$ implies that the negative part in the Zariski decomposition of $K_{\sF}$ is $N=0$.
Hence by \autoref{def-chern-numbers},
$$c_2(\sF)=\sum_{q \not\in N}\beta_q(\sF)=\sum_{q}\beta_q(\sF)=12.$$
By \eqref{eqn-chern-1-1}, one obtains $\chi(\sF)=\frac76$, and hence
$$\lambda(\sF)=\frac{c_1^2(\sF)}{\chi(\sF)}=\frac{12}{7}.$$
\end{example}

\begin{remark}
	We refer to \cite{lltx-24} for more examples of transcendental foliations with slope $\lambda(\sF)<2$.
\end{remark}
	
	\section{The Noether inequalities for a foliation of general type}\label{sec-noether}
In this section, we study the Noether type inequalities for a foliation of general type, and prove \autoref{thm-main};
while both \autoref{cor-noether} and \autoref{cor-volume} are direct consequence of \autoref{thm-main}, and the proofs of these two corollaries will be omitted.

As the invariants involved in \autoref{thm-main} are birational invariants,
we may assume that the foliated surface $(S,\sF)$ is reduced (or even relatively minimal if necessary).
Since the volume $\vol(\sF)>0$,
in order to prove the three Noether type inequalities in \autoref{thm-main},
one may assume that $p_g(\sF)\geq 2$.
As in the case proving the classical Noether inequality \eqref{eqn-1-1},
the starting point is to analyze the canonical map $\varphi_{|K_{\sF}|}$ defined by the complete linear system $|K_{\sF}|$:
	\begin{equation}\label{eqn-5-12}
		\varphi=\varphi_{|K_{\sF}|}:\,S \dashrightarrow \Sigma \subseteq \mathbb{P}^{p_g(\sF)-1}.
	\end{equation}	
	\begin{proposition}\label{prop-5-2}
	Let $(S,\sF)$ be a foliated surface of general type, i.e., $\kod(\sF)=2$.
	Suppose that the image $\Sigma=\varphi(S)$ is of dimension two.
	\begin{enumerate}[$(i).$]
		\item 	The following inequality holds.
		\begin{equation}\label{eqn-5-8}
		c_1^2(\sF)=\vol(\sF) \geq p_g(\sF)-2.
		\end{equation}
		Moreover, if the equality in \eqref{eqn-5-8} holds,
		then $\deg(\varphi)=1$ and the image $\Sigma$ is a surface of minimal degree $($equal to $p_g(\sF)-2)$ in $\bbp^{p_g(\sF)-1}$.
		\item If the equality in \eqref{eqn-5-8} does not hold,
		then \begin{equation}\label{eqn-5-8-1}
		c_1^2(\sF)=\vol(\sF) \geq p_g(\sF)-\frac{3}{2}.
		\end{equation}
		\item If moreover $\sF$ is algebraically integral,
		then \begin{equation}\label{eqn-5-8-2}
		c_1^2(\sF)=\vol(\sF) \geq p_g(\sF)-\frac54.
		\end{equation}
	\end{enumerate}
\end{proposition}

	\begin{proposition}\label{prop-5-1}
	Let $(S,\sF)$ be a foliated surface of general type, i.e., $\kod(\sF)=2$.
	Suppose that the image $\Sigma=\varphi(S)$ is of dimension one.
	\begin{enumerate}[$(i).$]
		\item The following inequality holds.
		\begin{equation}\label{eqn-5-4}
		\vol(\sF) \geq  p_g(\sF)-2+\frac{1}{p_g(\sF)}.
		\end{equation}
		\item If moreover $\sF$ is algebraically integral, then
		\begin{equation}\label{eqn-5-11}
		\vol(\sF) \geq p_g(\sF)-\frac{3}{2}+\frac{3}{2\big(2p_g(\sF)+1\big)}.
		\end{equation}
	\end{enumerate}
\end{proposition}
\begin{proof}[{Proof of \autoref{thm-main}}]
	It follows directly from the above two propositions.
	We just remark two points:
	firstly, we may assume that $p_g(\sF)\geq 2$ since $\vol(\sF)>0$;
	secondly, $p_g(\sF)\geq 3$ if $\Sigma=\varphi_{|K_{\sF}|}(S)$ is of dimension two.
\end{proof}
To complete the proof of the Noether type inequalities in \autoref{thm-main},
it remains to prove the above two propositions.
Before going to the detailed proofs, let's explain the main difficulties
compared to the case of $\varphi_{|K_S|}$:

(1). The canonical divisor $K_{\sF}$ might not be nef for a foliation $\sF$ of general type on a smooth surface $S$, even if the foliation $\sF$ is reduced or relatively minimal.
It makes some trouble in estimating the lower bound of $\vol(\sF)$.
For instance, suppose that
\begin{equation}\label{eqn-1-2}
|K_{\sF}|=|M|+Z,
\end{equation}
where $Z$ is the fixed part of $|K_{\sF}|$, and $M$ is the moving part.
Then it is no longer true that
$\vol(\sF) \geq K_{\sF} \cdot M$, cf. \autoref{rem-6-2}.
Of course, one can contract the support of the negative part of $K_{\sF}$ to a normal surface $S_0$, such that the induced foliation $\sF_0$ on $S_0$ has the advantage that $K_{\sF_0}$ is nef \cite{bru-99,mcq-08}.
However, the surface $S_0$ would be singular with klt singularities
and $K_{\sF_0}$ is no longer a line bundle (but a $\mathbb{Q}$-bundle).

(2). The second difficulty occurring in the case when
the image $\Sigma=\varphi(S)$ is of dimension one;
namely the canonical map $\varphi$ induces a fibration $f:\,S \to B$.
In the case proving the classical Noether inequality \eqref{eqn-1-1},
the general fiber $F$ of $f$ is of genus at least two,
and hence $K_S\cdot F=2g(F)-2\geq 2$.
However, in our case it is only known that $K_{\sF}\cdot F>0$ since $K_{\sF}$ is big.
It can happen that $K_{\sF}\cdot F=1$, even when the foliation $\sF$ is algebraically integral
(see \autoref{exam-6-2} and \autoref{rem-6-2} for such an example).
Suppose that
\begin{equation}\label{eqn-1-3}
K_{\sF}=P+N,
\end{equation}
is the Zariski decomposition of $K_{\sF}$,
where $P$ is the nef part and $N$ is the negative part of $K_{\sF}$.
It might happen that $N\cdot F>0$, and hence
$P\cdot F<K_{\sF}\cdot F =1$ (see \autoref{exam-6-2} for such an example).
This will cause trouble in estimating the lower bound of $\vol(\sF)=P^2$ along the usual way, cf. \autoref{rem-3-1}.

(3). The third difficulty happens when proving the Noether type inequalities
\eqref{eqn-5-8-2} and \eqref{eqn-5-11} for algebraically integral foliations.
This is completely new situation compared to the case of $\varphi_{|K_S|}$.
A priori, it is not clear what is the difference of the canonical map $\varphi_{|K_{\sF}|}$
between algebraically integral and transcendental foliations.
Indeed, any difference would give a positive solution to Poincar\'e's problem.

To overcome the above difficulties,
we need to control the negative part in the Zariski decomposition of the canonical divisor $K_{\sF}$,
as well as the structure of $K_{\sF}$ when the canonical map $\varphi_{|K_{\sF}|}$ induces a fibration or when the foliation is algebraically integral.

As pointed out at the beginning of this section,
we may assume that $\sF$ is reduced.
Let $\varphi=\varphi_{|K_{\sF}|}$ be the rational map defined by $|K_{\sF}|$ as in \eqref{eqn-5-12}.

When the image $\Sigma$ is of dimension two,
\eqref{eqn-5-8} follows by an easy argument.
To prove the rest two inequalities,
first by some elementary reduction,
we may assume that $\varphi$ is birational such that the image $\Sigma$ is a surface of minimal degree in $\bbp^{p_g(\sF)-1}$,
and that
$$K_{\sF_0}=M_0+Z_0,$$
where $\sF_0=\phi_*(\sF)$ is the induced foliation on the minimal desingularization $Y$ of $\Sigma$,
$M_0$ is the pulling-back of a hyperplane section on $\Sigma$,
and $Z_0=\phi_*(Z)$ with $\phi:\,S \to Y$ the induced birational map.
A key observation is that the divisor $Z_0$ is non-zero either if the equality does not hold in \eqref{eqn-5-8}
or if $\sF$ is algebraically integral.
With the help of \autoref{prop-3-1} and \autoref{lem-7-5},
we can estimate the contribution of $Z_0$ to the lower bound on the volume $\vol(\sF)$ and hence prove \eqref{eqn-5-8-1}.
The proof of \eqref{eqn-5-8-2} is the most technical part of \autoref{prop-5-2}.
Besides \autoref{prop-3-1} and \autoref{lem-7-5} mentioned above,
we have to more carefully analyze the possible non-reduced singularities lying on $Z_0$.
The difficult case is when $Z_0$ contains one unique irreducible rational curve with $K_{\sF_0}\cdot C_0=1$.
There are finitely many possibilities for the possible singularities on such a curve $C_0$.
The resolution of such singularities on $C_0$ presented in \autoref{sec-non-reduced} helps us to finish the proof of \eqref{eqn-5-8-2}
case-by-case.

When the image $\Sigma$ is of dimension one,
then the canonical map induces a fibration $f:\,S \to B$ by taking the normalization and Stein factorization:
$$\xymatrix{ &S \ar[rd]^-{\varphi} \ar[ld]_-{f} \\
	B \ar[rr] && \Sigma\,\ar@{^(->}[r] & \bbp^{p_g(\sF)-1}}$$
Consider the two decompositions in \eqref{eqn-1-2} and \eqref{eqn-1-3}.
Since the moving part $M$ is always nef,
it follows that $M\leq P$, or equivalently $N\leq Z$.
Let $F$ be a general fiber of $f$.
The moving part $M$ consists of several fibers, whose cardinality is at least $p_g(\sF)-1$ by the Riemann-Roch theorem.
Hence
\begin{equation*}
\vol(\sF)= P^2 \geq P\cdot M \geq \big(p_g(\sF)-1\big)P\cdot F = \big(p_g(\sF)-1\big)\big(K_{\sF}-N\big)\cdot F.
\end{equation*}
To obtain the Noether type inequality \eqref{eqn-5-4},
it suffices to get a lower bound on $\big(K_{\sF}-N\big)\cdot F$,
or equivalently, an upper bound on $N\cdot F$,
which will be done with the help of \autoref{prop-3-1}.
However, the proof of \eqref{eqn-5-11} for an algebraically integral foliation
is much more subtle.
First we may assume that $S$ is ruled on $B\cong \bbp^1$ based on \cite[Theorem\,1.7\,(ii)]{luxin-24} and its proof.
Let $\psi:\,S \to S_0$ be a sequence of blowing-down the exceptional curves in fibers
such that $f_0:\,S_0 \to \bbp^1$ is a $\bbp^1$-bundle
with $\sF_0=\psi_*(\sF)$ and $K_{\sF_0}=M_0+Z_0$,
where $M_0=\psi_*(M)$ and $Z_0=\psi_*(Z)$.
The key observation is, similar to the case when $\dim \Sigma=2$, that the divisor $Z_0$ contains at least one fiber of $f_0$.
By resolving the possible non-reduced singularities on such a fiber,
which is exhibited in \autoref{sec-non-reduced},
one proves that such a fiber would give a positive contribution
(almost $\frac12$) to the lower bound on the volume $\vol(\sF)$,
based on \autoref{lem-7-5}.
This helps us to complete the proof of \autoref{prop-5-1}.

\vspace{2mm}
The detailed proofs of \autoref{prop-5-2} and \autoref{prop-5-1} will be completed respectively in
\autoref{sec-noe-two} and \autoref{sec-noe-one}.

	\subsection{The canonical map is generically finite}\label{sec-noe-two}
	The aim of this subsection is to prove \autoref{prop-5-2}.
	Before going to the proof, let's do some general preparations.
	Given any foliated surface $(S,\sF)$,
	by Seidenberg's \autoref{thm-seidenberg} together with \cite[Proposition\,5.1]{bru-04},
	one may assume that $\sF$ is relatively minimal.
	Let $|K_{\sF}|=|M|+Z$ be the decomposition into moving and fixed parts as in \eqref{eqn-1-2}.
	Let $\rho:\,Y\to \Sigma$ be the minimal desingularization of $\Sigma$,
	where $\Sigma=\varphi_{|K_{\sF}|}(S)$ is the image of $S$ under the canonical map as in \eqref{eqn-5-12}.
	By a sequence of blowing-ups $\sigma:\,\wt{S} \to S$ centered on the base points of $|M|$, we obtain a well-defined morphism $\phi:\, \wt{S} \to Y$ with the following diagram.
$$\xymatrix{\wt{S} \ar[d]^-{\phi} \ar[rrr]^-{\sigma}  &&& S \ar@{-->}[d]^-{\varphi=\varphi_{|K_{\sF}|}}\\
		Y \ar[rrr]^-{\rho}_-{\text{desingularization}} &&&\Sigma\, \ar@{^(->}[r] & \bbp^{p_g(\sF)-1}}$$

	\begin{proof}[{Proof of \autoref{prop-5-2}}]
		(i). By construction, the map 
		$$\rho\circ\phi:\,\wt{S} \lra \Sigma \hookrightarrow \bbp^N,$$
		is defined by the complete linear system $|\wt{M}|$ with $\wt{M}=\sigma^*M-\sum a_j\mathcal{E}_j$.
		Let $M_0=\rho^*(H)$ be the pulling-back of a hyperplane section $H$ on $\Sigma$. Then $\wt{M}=\phi^*(M_0)$, and
		\begin{equation*}
			M^2=\wt{M}^2+\sum a_j^2 \geq \wt{M}^2=\deg(\phi)\cdot M_0^2= \deg(\varphi)\cdot \deg(\Sigma).
		\end{equation*}
		On the other hand, let $K_{\sF}=P+N$
		be the Zariski decomposition of $K_{\sF}$ as in \eqref{eqn-1-3}.
		Then
		$$\vol(\sF) =P^2 \geq M^2 \geq \deg(\varphi)\cdot \deg(\Sigma) \geq p_g(\sF)-2.$$
		The last inequality follows from \cite[Lemme\,1.4]{bea-79}.
		Moreover, if the equality holds, then $\deg(\varphi)=1$ and
		$\deg(\Sigma)=p_g(\sF)-2$ in $\bbp^{p_g(\sF)-1}$.
		
		(ii).
		According to the above argument, we may assume that
		\begin{equation}\label{eqn-7-30}
			\left\{\begin{aligned}
			& \text{$M^2=\wt M^2$, ~i.e.,~ $S=\wt S$};\\
			& \text{$\deg\varphi=1$, ~i.e.,~ $\varphi$ is birational};\\
			& \text{$\deg\Sigma=p_{g}(\sF)-2$}.
			\end{aligned}\right.
		\end{equation}
		Indeed, if any of the above three statements does not hold, then $\vol(\sF)\geq p_{g}(\sF)-1$ according to the argument above.
		Since $\Sigma$ is of minimal degree in $\bbp^{p_g(\sF)-1}$,
		its minimal desingularization $Y$ is either isomorphic to $\bbp^2$ or a Hirzebruch surface (cf. \autoref{prop-7-1}).
		Let $\sF_0$ be the induced foliation on $Y$.
		Then $\phi:\,S=\wt S \to Y$ is composed of a sequence of blowing-ups.
		By \eqref{eqn-1-2},
		$$K_{\sF_0}=\phi_*(M)+\phi_*(Z)=M_0+\phi_*(Z).$$
		\begin{claim}\label{claim-7-3}
			Denote by $Z_0=\phi_*(Z)$.
			Suppose that $\vol(\sF)\neq p_g(\sF)-2$.
			Then $Z_0\neq 0$, and there exists an irreducible curve $C_0\subseteq Z_0$ with $M_0\cdot C_0>0$.
		\end{claim}
	\begin{proof}[{Proof of \autoref{claim-7-3}}]
		We prove first that $Z_0\neq 0$ by contradiction. Suppose that $Z_0=0$, i.e., $K_{\sF_0}=M_0$.
		Then by \autoref{lem-7-1},
		$$Z+\phi^*(M_0)=Z+M=K_{\sF}=\phi^*K_{\sF_0}-\sum(\ell_i-1)\mathcal{E}_i=\phi^*(M_0)-\sum(\ell_i-1)\mathcal{E}_i,$$
		which implies that $Z=0$ and $K_{\sF}=\phi^*(M_0)$.
		This implies that $\vol(\sF)=M_0^2=p_g(\sF)-2$, which contradicts the assumption.
		
		We prove next that there exists an irreducible curve $D_0\subseteq Z_0$ with $M_0\cdot D_0>0$.
		Again, the proof is by contradiction.
		First note that if $M_0\cdot D_0=0$, then $Y$ is the Hirzebruch surface $\mathbb{P}_{\bbp^1}\big(\mathcal{O}_{\bbp^1}\oplus \mathcal{O}_{\bbp^1}(-e)\big)$ and $D_0$ is the rational curve with $D_0^2=-e\leq -2$.
		Suppose that there is no irreducible curve $D\subseteq Z_0$ with $M_0\cdot D>0$.
		Then $Z_0=aD_0$ for some $a>0$, and hence $K_{\sF_0}=M_0+aD_0$.
		If $D_0$ is not $\sF_0$-invariant, then $\tang(\sF_0,D_0)=K_{\sF_0}\cdot D_0+D_0^2<0$, which is a contradiction.
		Hence $D_0$ is $\sF_0$-invariant, and 	
		$$0\leq Z(\sF_0,D_0)=2+K_{\sF_0}\cdot D_0=2+aD_0^2=2-ae.$$
		It follows that $a=1$, $e=2$, and $Z(\sF_0,D_0)=0$.
		Since $K_{\sF_0}\cdot D_0<0$, it follows that $D_0$ is contained in the negative part $N_0$ of the Zariski decomposition of $K_{\sF_0}=P_0+N_0$.
		On the other hand, $Z(\sF_0,D_0)=0$ implies that there is no singularity of $\sF_0$ on $D_0$. This gives a contradiction to \autoref{thm-3-4}.
	\end{proof}
 Come back to the proof of \autoref{prop-5-2}.
		Let $D_0\subseteq Z_0$ be any irreducible curve with $M_0\cdot D_0=d\geq 1$, and $D\subseteq S$ be its strict transform. Then $D$ is clearly contained in the support of $Z$.
		Note that $M=\phi^*(M_0)$.
		
		If $D$ is not contained in the support of the negative part $N$,
		then $D$ is contained in the support of $P$, and hence
		\begin{equation}\label{eqn-7-1}
			\vol(\sF)=P^2\geq P\cdot M \geq (M+D)\cdot M =M^2+M\cdot D=M^2+M_0\cdot D_0\geq p_g(\sF)-1.
		\end{equation}
		Therefore, we may assume that $D$ is contained in the support of $N$.
		Let
		$$K_{\sF}=M+Z=\phi^*(M_0)+aD+Z',$$
		where $D$ is not contained in $Z'$.
		If $a\geq 2$, then
		$\phi^*(M_0)+D\subseteq P$, since $\lfloor N \rfloor=0$ by \autoref{thm-3-4}.
		Hence
		$$\vol(\sF)=P^2\geq \phi^*(M_0)\cdot P \geq \phi^*(M_0)\cdot \Big(\phi^*(M_0)+D\Big)=M_0^2+M_0\cdot D_0\geq p_g(\sF)-1.$$
		Hence we may assume that $a=1$.
		Suppose that the irreducible curve $D$ is contained in the maximal $\sF$-chain $C_1+\cdots +C_r$. Let $b$ be the coefficient of $D$ in the negative part $N$.
		\begin{enumerate}[(1)]
			\item If $r=1$ and $D=C_1$, then $b=\frac1n \leq \frac12$ by \eqref{eqn-3-3}. Hence $\phi^*(M_0)+\frac12D\subseteq P$, from which it follows that
			$$\begin{aligned}
			\vol(\sF)=P^2\geq \phi^*(M_0)\cdot P &\,\geq \phi^*(M_0)\cdot \Big(\phi^*(M_0)+\frac12D\Big)=M_0^2+\frac12M_0\cdot D_0=M_0^2+\frac{d}{2}\\
			&\,\geq \big(p_g(\sF)-2\big)+\frac{d}{2}\geq p_g(\sF)-\frac32.
			\end{aligned}$$
			\item If $D=C_r$ with $r\geq 2$, then $b=\frac1n \leq \frac13$ by \eqref{eqn-3-3}. Hence $\phi^*(M_0)+\frac23D\subseteq P$.
			Similar as above, it follows that
			$$\vol(\sF)=P^2\geq M_0^2+\frac23M_0\cdot D_0\geq \big(p_g(\sF)-2\big)+\frac{2d}{3}\geq p_g(\sF)-\frac43.$$
			\item If $D=C_j$ with $1<j<r$, then 
			$$0=K_{\sF}\cdot D\geq D^2+\phi^*(M_0)\cdot D+(C_{j-1}+C_{j+1})\cdot D \geq D^2+3.$$
			It follows that $D^2\leq -3$, and hence $b<\frac{1}{3}$ by \eqref{eqn-2-1}. Hence $\phi^*(M_0)+\frac23D\subseteq P$.
			Again, one obtains that
			$$\vol(\sF)=P^2\geq M_0^2+\frac23M_0\cdot D_0\geq \big(p_g(\sF)-2\big)+\frac{2d}{3}\geq p_g(\sF)-\frac43.$$
			\item If $r\geq 2$ and $D=C_1$, then
			$$-1=K_{\sF}\cdot D\geq D^2+\phi^*(M_0)\cdot D+D_2\cdot D \geq D^2+2.$$
			It follows that $D^2\leq -3$, and hence $b<\frac12$ by \eqref{eqn-2-1}. Hence $\phi^*(M_0)+\frac12D\subseteq P$.
			Therefore,
			$$\vol(\sF)=P^2\geq M_0^2+\frac12M_0\cdot D_0\geq \big(p_g(\sF)-2\big)+\frac{d}{2}\geq p_g(\sF)-\frac32.$$
		\end{enumerate}
	This completes the proof of (ii).
	
	\vspace{2mm}
	(iii).
	We first claim the following.
	\begin{claim}\label{claim-7-4}
		If $\sF$ is algebraically integral, then $\vol(\sF)\neq p_g(\sF)-2$.
	\end{claim}
\begin{proof}[{Proof of \autoref{claim-7-4}}]
	We prove by contradiction. Suppose that $\vol(\sF) = p_g(\sF)-2$.
	According to the proof of (i), the conditions in \eqref{eqn-7-30} are all satisfied, and $P^2=M^2$.
	Note that $M\leq P$, and the negative part $N$ satisfies $\lfloor N \rfloor=0$.
	It follows that $N\leq Z$, and that $\text{Supp}(N) \subseteq \text{Supp}(Z-N)=\text{Supp}(Z)$.
	Since both $P$ and $M$ are nef, one checks easily that
	$$P^2=P\cdot (M+Z-N)\geq P\cdot M+P\cdot (Z-N) \geq M^2+(P+M)\cdot (Z-N).$$
	Combining this with the condition that $P^2=M^2$, one obtains that
	$$P\cdot (Z-N)=M\cdot (Z-N)=0.$$
	Hence
	$$M^2=P^2=\big(M+(Z-N)\big)^2=M^2+(Z-N)^2,\quad\Longrightarrow\quad (Z-N)^2=0.$$
	By the Hodge index theorem, $Z-N=0$.
	Hence $N=Z=0$, since $\text{Supp}(Z-N)=\text{Supp}(Z)$.
	Thus $K_{\sF}=M=\phi^*(M_0)$,
	and $K_{\sF_0}=\phi_*(M)=M_0$ with $Z_0=\phi_*(Z)=0$.
	On the other hand, by \autoref{lem-7-1} one has
	$$\phi^*(M_0)=M=K_{\sF}=\phi^*(K_{\sF_0})-\sum(\ell_i-1)\mathcal{E}_i=\phi^*(M_0)-\sum(\ell_i-1)\mathcal{E}_i.$$
	It follows that $\sum(\ell_i-1)\mathcal{E}_i=0$, i.e., each $\ell_i=1$.
	Since $\sF$ is assumed to be algebraically integral, it follows that all singularity on $Y$ is reduced according to \autoref{lem-7-1}\,(ii).
	It implies that $S=Y$ and $\sF_0=\sF$, since $\sF$ is relatively minimal.
	As a relatively minimal algebraically integral foliation, $\sF_0=\sF$ is induced by a fibration of genus $g\geq 2$ with normal crossing fibers on the surface $S=Y$.
	However, the surface $Y$, as the minimal desingularization of the surface $\Sigma$ of minimal degree in $\bbp^{p_g(\sF)-1}$, is either $\bbp^2$ or a Hirzebruch surface, cf. \autoref{prop-7-1}.
	In particular, there is no fibration of genus $g\geq 2$ on $Y$.
	Therefore, $\vol(\sF)\neq p_g(\sF)-2$ as required.
\end{proof}
	
	According to \autoref{claim-7-4},
	we can assume that the conditions in \eqref{eqn-7-30} are all satisfied;
	otherwise, $\vol(\sF)\geq p_g(\sF)-1$ by the proof of (ii) above.

	\begin{claim}\label{claim-7-5}
		Let $Y$ be the minimal desingularization of $\Sigma$ with an induced morphism $\phi:\,S \to Y$ as in the proof of (ii).
		Suppose that $\vol(\sF)<p_g(\sF)-1$.
		\begin{enumerate}
			\item[(A-1).] Let $\sF_0=\phi_*(\sF)$. Then $K_{\sF_0}=M_0+Z_0$ with $Z_0\neq 0$.
			\item[(A-2).] There exists a unique irreducible curve $D_0\subseteq Z_0$ with $M_0\cdot D_0>0$. Moreover, such a curve $D_0$ is $\sF_0$-invariant whose strict transform $D\subseteq S$ is contained in the support of the negative part $N$ in the Zariski decomposition $K_{\sF}=P+N$,
			$d=M_0\cdot D_0=1$, and the coefficient of $D_0$ in $Z_0$ is $1$.
		\end{enumerate}
	\end{claim}
\begin{proof}[{Proof of \autoref{claim-7-5}}]
	This follows the proof of \autoref{prop-5-2}\,(ii) above.
	Indeed, by \autoref{claim-7-4} and \autoref{claim-7-3},
	$Z_0=\phi_*(Z)\neq 0$, and there exists at least one irreducible curve $D_0\subseteq Z_0$
	with $M_0\cdot D_0=d>0$.
	Moreover, any such a curve $D_0$ with $M_0\cdot D_0>0$ satisfies that
	\begin{enumerate}[$(1)\,$]
		\item $D_0$ is $\sF_0$-invariant, whose strict transform $D\subseteq S$ is contained in the support of the negative part $N$ in the Zariski decomposition $K_{\sF}=P+N$;
		\item the coefficient of $D_0$ in $Z_0$ is $1$;
		\item $d:=M_0\cdot D_0=1$.
	\end{enumerate}
Otherwise, if any of the above three conditions does not hold,
then it is proved that $\vol(\sF)\geq p_g(\sF)-1$ according the arguments in the proof of \autoref{prop-5-2}\,(ii) above.
Moreover, it is showed that any such a curve would contribute at least $1/2$ to the lower bound of the volume $\vol(\sF)$.
Hence such a curve $D_0$ with $M_0\cdot D_0>0$ is unique.
\end{proof}
	
Come back to the proof of \autoref{prop-5-2}\,(iii).
In order to prove \autoref{prop-5-2}\,(iii),
besides the conditions in \eqref{eqn-7-30},
we may also assume that the two conditions (A-1) and (A-2) in \autoref{claim-7-5} hold.
To go further, we recall the following description on a surface of minimal degree in the projective space (cf. \cite[Lemma\,1.2]{hor-76} or \cite[Chapter\,4.3]{gh-78}).
\begin{proposition}\label{prop-7-1}
	Let $\Sigma\subseteq \bbp^{m}$ be non-degenerate surface of minimal degree (equal to $m-1$), and $\rho:\,Y \to \Sigma$ its minimal desingularization with
	$M_0=\rho^*(H)$ be the pulling-back of a hyperplane section $H$ on $\Sigma$.
	Then $(Y,\Sigma,M_0)$ belongs to one of the following four cases:
	\begin{enumerate}[$(1).$]
		\item $Y=\Sigma=\mathbb{P}_{\mathbb{P}^1}\big(\mathcal{O}_{\mathbb{P}^1} \oplus \mathcal{O}_{\mathbb{P}^1}(e)\big)$ is a Hirzebruch surface,
		and $M_0=C_0+(e+1+k)F$, where $C_0$ is a section with $C_0^2=-e$ (unique if $e\geq 1$), $F$ is a general fiber of the ruling $Y\to \bbp^1$, and $k\geq 0$;
		\item $\Sigma$ is a cone over a rational curve of degree $m-1$ in $\bbp^{m-1}$ with $m\geq 3$,
		$Y=\mathbb{P}_{\mathbb{P}^1}\big(\mathcal{O}_{\mathbb{P}^1} \oplus \mathcal{O}_{\mathbb{P}^1}(m-1)\big)$, and $M_0=C_0+(m-1)F$,
		where $C_0$ is the unique section with $C_0^2=-(m-1)$, and $F$ is a general fiber of the ruling $Y\to \bbp^1$;
		\item $Y=\Sigma=\bbp^2$, $m=2$ and $M_0=L$, where $L$ is a general line in $\bbp^2$;
		\item $Y=\Sigma=\bbp^2$, $m=5$ and $M_0=2L$, where $L$ is a general line in $\bbp^2$.
	\end{enumerate}
\end{proposition}
Since $\Sigma$ is of minimal degree in $\bbp^{p_g(\sF)-1}$ by \eqref{eqn-7-30}.
Hence either $Y$ is a Hirzebruch surface or $Y\cong \bbp^2$ by \autoref{prop-7-1}.

{\noindent \bf Case (a).} 
The surface $Y$ is a Hirzebruch surface,
i.e., $Y=\Sigma=\mathbb{P}_{\mathbb{P}^1}\big(\mathcal{O}_{\mathbb{P}^1} \oplus \mathcal{O}_{\mathbb{P}^1}(e)\big)$.
Let $K_{\sF_0}=M_0+Z_0$.
By \eqref{eqn-7-30} together with two conditions (A-1) and (A-2) in \autoref{claim-7-5} and \autoref{prop-7-1},
there are two subcases:
\begin{enumerate}
	\item[\bf (a.1).] $Z_0=F_0$, and $M_0=C_0+(e+k)F$ for some $k\geq 0$,
	where $F_0$ is some fixed fiber of the ruling $Y\to \bbp^1$, such that its strict transform in $S$ is contained in the support of the negative part $N$ of the Zariski decomposition $K_{\sF}=P+N$;
	\item[\bf (a.2).] $Z_0=F_0+C_0$, $n=p_g(\sF)-2\geq 2$, and $M_0=C_0+eF=C_0+\big(p_g(\sF)-2\big)F$, where $C_0$
	is the unique section with $C_0^2=-e=2-p_g(\sF)$, and $F$ (resp. $F_0$) is a general fiber (some fixed fiber) of the ruling $Y\to \bbp^1$.
	Moreover, the strict transform of $F_0$ in $S$ is contained in the support of the negative part $N$ of the Zariski decomposition $K_{\sF}=P+N$
\end{enumerate}

{\bf Case (a.1).}
In this case,
$$K_{\sF_0}^2=(M_0+F_0)^2=M_0^2+2=\big(p_g(\sF)-2\big)+2=p_g(\sF).$$
By \autoref{lem-7-9}, all the possible non-reduced singularities of $\sF_0$ are on $F_0$.
Let $p_1,\cdots,p_k$ be all the singularities of $\sF_0$ on $F_0$.
By \autoref{prop-2-2},
\begin{equation}\label{eqn-7-37}
\sum_{i=1}^{k} Z(\sF_0,F_0,p_i)=Z(\sF_0,F_0)=K_{\sF_0}F_0+\chi(F_0)=3.
\end{equation}
In particular, $1\leq k \leq 3$.
We divide into three subcases to finish the proof according to the value of $k$.

{\bf Case (a.1.1).} Suppose that $k=1$.
In this case $Z(\sF_0,F_0,p_1)=3$ by \eqref{eqn-7-34}.
We will show that this subcase cannot happen.
First we claim that
\begin{claim}\label{claim-7-8}
	The vanishing order of $\sF_0$ at $p_1$ is $1$.
\end{claim}
\begin{proof}[{Proof of \autoref{claim-7-8}}]
	Suppose that the vanishing order of $\sF_0$ at $p_1$ is $a_{p_1}\geq 2$. We will derive a contradiction.
	To this aim, we will prove that
	\begin{equation}\label{eqn-7-38}
	\text{all the exceptional curves of the blowing-ups contained in $\phi:\,S \to Y$ are $\sF$-invariant.}
	\end{equation}
	We first derive a contradiction based on \eqref{eqn-7-38}.
	As a relatively minimal algebraically integral foliation of general type, $\sF$ is induced a fibration $f:\,S \to B$ of genus $g\geq 2$.
	By \autoref{lem-2-3} with \eqref{eqn-7-38}, such a fibration induces a fibration $f_0:\,Y \to B$ with $f=f_0\circ\phi$.
	However, $Y$ is a Hirzebruch surface, there is no fibration of genus $g\geq 2$ on $Y$. This gives a contradiction.
	Therefore, it suffices to prove \eqref{eqn-7-38}.
	
	Note that $p_1$ is a non-reduced singularity since $a_{p_1}\geq 2$.
	Let $\sigma_1:\,S_1 \to Y$ be the blowing-up centered at $p_1$ with exceptional curve $E_1$.
	Then $\phi:\,S \to Y$ factors through $\sigma_1$ as $\phi=\sigma_1\circ\phi_1$ by \autoref{lem-7-9}\,(ii).
	Let $\sF_1=\sigma_1^*(\sF_0)$ and $K_{\sF_1}=\sigma_1^*(K_{\sF_0})-(\ell_{p_1}-1)E_1$.
	Then by \autoref{lem-7-9}\,(ii),
	$$\ell_{p_1}\leq m_{p_1}(Z_0)+1=2.$$
	Combining this with \eqref{eqn-2-13}, $\ell_{p_1}=a_{p_1}=2$ and hence $E_1$ is $\sF_1$-invariant.
	It follows that
	$$K_{\sF_1}=\sigma_1^*(K_{\sF_0})-E_1=\sigma_1^*(M_0+F_0)-E_1=M_1+\ol F_0,$$
	where $M_1=\sigma_1^*(M_0)$, and $\ol F_0$ is the strict transforms of $F_0$ in $S_1$.
	Since $p_1$ is the unique non-reduced singularity of $\sF_0$,
	it follows that $q_1=E_1\cap \ol F_0$ is the unique possible non-reduced singularity
	of $\sF_1$, since all possible non-reduced singularities of $\sF_1$ would be on $Z_1:=\ol F_1$ by \autoref{lem-7-9}.
	Moreover, by \autoref{lem-7-8},
	$$Z(\sF_1,\ol F_0,q_1)=Z(\sF_0,F_0,p_1)-1=2.$$
	This implies in particular that the vanishing order of $\sF_1$ at $q_1$ is $a_{q_1} \leq 2$ by the definition of $Z$-index.
	By \autoref{cor-7-2}, $Z(\sF_1,E_1,q_1)\geq 2$ and $a_{q_1}=2$.
	Note also that
	$Z(\sF_1,E_1)=3$ by \autoref{prop-2-2}.
	Hence $Z(\sF_1,E_1,q_1)=2$ or $3$.
	
	Let $\sigma_2:\,S_2 \to S_1$ be the blowing-up centered at $q_1$ with exceptional curve $E_2$.
	Then $\phi_1$ factors through $\sigma_2$ as $\phi_1=\sigma_2\circ\phi_2$ by
	\autoref{lem-7-9}\,(ii).
	$K_{\sF_2}=\sigma_2^*(K_{\sF_1})-(\ell_{q_1}-1)E_2$.
	Then by \autoref{lem-7-9}\,(ii),
	$$\ell_{q_1}\leq m_{q_1}(Z_1)+1=2.$$
	Combining this with \eqref{eqn-2-13}, $\ell_{q_1}=a_{q_1}=2$ and hence $E_2$ is $\sF_2$-invariant.
	It follows that
	$$K_{\sF_2}=\sigma_2^*(K_{\sF_1})-E_2=\sigma_2^*(M_1+\ol F_0)-E_2=M_2+\ol F_0,$$
	where $M_2=\sigma_2^*(M_1)$, and we still denote by $\ol F_0$ the strict transforms of $F_0$ in $S_2$ by abuse of notations.
	Similar as above, one sees that the intersection $q_2=E_2\cap \ol F_0$ is the only possible non-reduced singularity of $\sF_2$ based on \autoref{lem-7-9}.
	Moreover, it cannot happen that $Z(\sF_1,E_1,q_1)=3$; indeed, if $Z(\sF_1,E_1,q_1)=3$, then similar as above one shows that $q_2'=\ol E_1\cap E_2$ is a singularity with $a_{q_2'}=2$ (and hence non-reduced),
	which is a contradiction because $q_2$ is only possible non-reduced singularity.
	Hence $Z(\sF_1,E_1,q_1)=2$, and hence there is another singularity $q_1'$ on $E_1$ with
	$$Z(\sF_1,E_1,q_1')=Z(\sF_1,E_1)-Z(\sF_1,E_1,q_1)=\big(K_{\sF_1}E_1+2\big)-2=1.$$
	Moreover, by \autoref{lem-7-8},
	$$Z(\sF_2,\ol F_0,q_2)=Z(\sF_1,\ol F_0,q_1)-1=1,\qquad 
	Z(\sF_2,\ol E_1,q_2')=Z(\sF_1,E_1,q_1)-1=1.$$
	Hence by \autoref{cor-7-2},
	$$Z(\sF_2,E_2,q_2)=Z(\sF_2,E_2,q_2')=1.$$
	Thus there is another singularity $q_2''$ on $E_2$ with
	$$Z(\sF_2,E_2,q_2'')=Z(\sF_2,E_2)-Z(\sF_2,E_2,q_2)-Z(\sF_2,E_2,q_2')=1.$$
	Hence all the singularities on $S_2$ are non-degenerate by \autoref{lem-7-2}; and except $q_2$, all the rest singularities are reduced.
	$${\setlength{\unitlength}{6mm}
		\begin{tikzpicture}
		\draw[thick] (14,1) -- (14,3);	
		\filldraw[black] (14,2.5) circle (1.5pt);
		\filldraw[black] (14,2.3) node[anchor=west]{$p_1$};
		\filldraw[black] (13.6,0.7) node[anchor=west]{$F_0$};

		\draw[->] (12,2) -- (13,2);
		\filldraw[black] (12.2,2.3) node[anchor=west]{\Large  $\sigma_1$};
		
		\draw[thick] (9,1) -- (9,3);	
		\draw[thick] (8.5,2.5) -- (11,2.5);	
		\filldraw[black] (9,2.5) circle (1.5pt);
		\filldraw[black] (10,2.5) circle (1.5pt);
		\filldraw[black] (9,2.3) node[anchor=west]{$q_1$};
		\filldraw[black] (9.8,2.2) node[anchor=west]{$q_1'$};
		\filldraw[black] (8.6,0.7) node[anchor=west]{$\ol F_0$};
		\filldraw[black] (11,2.5) node[anchor=west]{$E_1$};
		
		\draw[->] (7,2) -- (8,2);
		\filldraw[black] (7.2,2.3) node[anchor=west]{\Large  $\sigma_2$};
		
		\draw[thick] (4,1) -- (4,3);	
		\draw[thick] (3.5,2.7) -- (5.75,1.8);
		\draw[thick] (4.5,1.9) -- (6.5,2.7);
		\filldraw[black] (4,2.5) circle (1.5pt);
		\filldraw[black] (4.5,2.3) circle (1.5pt);
		\filldraw[black] (5,2.1) circle (1.5pt);
		\filldraw[black] (5.75,2.4) circle (1.5pt);
		\filldraw[black] (3.5,2.35) node[anchor=west]{$q_2$};
		\filldraw[black] (4.4,2.5) node[anchor=west]{$q_2''$};
		\filldraw[black] (4.7,1.8) node[anchor=west]{$q_2'$};
		\filldraw[black] (5.6,2.2) node[anchor=west]{$q_1'$};
		\filldraw[black] (3.6,0.7) node[anchor=west]{$\ol F_0$};
		\filldraw[black] (6.5,2.6) node[anchor=west]{$\ol E_1$};
		\filldraw[black] (5.6,1.6) node[anchor=west]{$E_2$};
\end{tikzpicture}}$$	
By \autoref{prop-2-2},
$$\cs(\sF_2,\ol F_0,q_2)=\ol F_0^2=-2.$$
It follows that $q_2$ is also a reduced singularity.
Since $\sF$ is relatively minimal, it follows that the induced birational morphism $\phi_2:\,S \to S_2$ is an isomorphism.
Hence the exceptional curves, which are $\ol E_1$ and $E_2$, are both $\sF$-invariant as required.
\end{proof}
Since the vanishing order of $\sF_0$ at $p_1$ is $1$.
The resolution of such a singularity $p_1$ into non-degenerate ones is given in \autoref{lem-7-7}.
The picture is as follows ($Z(\sF_0,F_0,p_1)=3$ in this case).
$${\setlength{\unitlength}{6mm}
	\begin{tikzpicture}
	\draw[thick] (6,1) -- (6,3);
	\filldraw[black] (6,2) circle (1.5pt);
	\filldraw[black] (6,2) node[anchor=west]{$p_1$};
	\filldraw[black] (6,1) node[anchor=west]{$F_0$};

	\draw[->] (4.3,2) -- (5.2,2);
	\filldraw[black] (4.5,2.3) node[anchor=west]{\Large  $\sigma$};

\draw[thick] (3,1) -- (3,3);
\draw[thick] (0.5,1) -- (0.5,3);
\draw[thick] (1.34,0.88) -- (3.5,2.5);
\draw[thick] (2.16,0.88) -- (0,2.5);
\filldraw[black] (0.5,2.12) circle (1.5pt);
\filldraw[black] (1.75,1.2) circle (1.5pt);
\filldraw[black] (2.4,1.68) circle (1.5pt);
\filldraw[black] (3,2.12) circle (1.5pt);
\filldraw[black] (3,2) node[anchor=west]{$q_3$};
\filldraw[black] (2.3,1.5) node[anchor=west]{$q_3'$};
\filldraw[black] (1.45,1.5) node[anchor=west]{$\bar q_2$};
\filldraw[black] (0.45,2.2) node[anchor=west]{$\bar q_1$};
\filldraw[black] (-0.1,1) node[anchor=west]{$\ol{E}_1$};
\filldraw[black] (3,1) node[anchor=west]{$\ol{F}_0$};
\filldraw[black] (-0.65,2.6) node[anchor=west]{$\ol E_2$};
\filldraw[black] (3.3,2.7) node[anchor=west]{$E_3$};
\end{tikzpicture}}$$
By \autoref{prop-2-2} and \autoref{lem-2-2}, one obtains that
\begin{equation*}
\left\{\begin{aligned}
&\cs(\sF_3,E_3,q_3)^{-1}=\cs(\sF_3,\ol F_0,q_3)=\ol F_0^2=-3;\\
&\cs(\sF_3,\ol E_2,\bar q_1)^{-1}=\cs(\sF_3,\ol E_1,\bar q_1)=\ol E_1^2=-2;\\
&\cs(\sF_3,E_3,\bar q_2)^{-1}=\cs(\sF_3,\ol E_2,\bar q_2)=\ol E_2^2-\cs(\sF_3,\ol E_2,\bar q_1)=-2+\frac12=-\frac32;\\
&\cs(\sF_3,E_3,q_3')=E_3^2-\cs(\sF_3,E_3,\bar q_2)-\cs(\sF_3,E_3,q_3)=-1+\frac23+\frac13=0.
\end{aligned}\right.
\end{equation*}
This gives a contradiction, since $\cs(\sF_3,E_3,q_3')$ is just the eigenvalue of $\sF_3$ at the non-degenerate singularity $q_3'$.
Hence this subcase cannot happen.

{\bf Case (a.1.2).} Suppose that $k=2$.
By \eqref{eqn-7-34}, we may assume without loss of generality that
$Z(\sF_0,F_0,p_1)=2$ and $Z(\sF_0, F_0, p_2)=1$.
Since $Z(\sF_0,F_0,p_1)=2$, the vanishing order of $\sF_0$ at $p_1$ is $a_{p_1}\leq 2$.

Suppose first that $a_{p_1}=2$. Let $\sigma_1:\,S_1 \to Y$ be the blowing-up centered at $p_1$ with exceptional curve $E_1$,
and let $K_{\sF_1}=\sigma_1^*(K_{\sF_0})-(\ell_{p_1}-1)E_1$.
By \autoref{lem-7-9}\,(ii), $\phi:\,S \to Y$ factors through $\sigma_1$ as $\phi=\sigma_1\circ\phi_1$, and $\ell_{p_1}\leq 2$.
Combining with \eqref{eqn-2-13},
one sees that $\ell_p=a_p=2$ and $E_1$ is $\sF_1$-invariant.
Moreover,
$$K_{\sF_1}=\sigma_1^*(K_{\sF_0})-(\ell_{p_1}-1)E_1=\sigma_1^*(M_0)+\ol F_0.$$
Hence the possible non-reduced singularities of $\sF_1$ are all on $\ol F_0$ by \autoref{lem-7-9}\,(i).
Let $q_1=\ol F_0\cap E_1$. By \autoref{lem-7-8}, $Z(\sF_1,\ol F_0,q_1)=1$,
which implies that $Z(\sF_1,E_1,q_1)=1$ as well by \autoref{cor-7-2}.
Note that $Z(\sF_1,E_1)=K_{\sF_1}\cdot E_1+\chi(E_1)=3$.
It follows that there are two another singularities $q_1',q_1''$ on $E_1$, which are both reduced.
$${\setlength{\unitlength}{6mm}
	\begin{tikzpicture}
	\draw[thick] (14,1) -- (14,3);	
	\filldraw[black] (14,2.5) circle (1.5pt);
	\filldraw[black] (14,2.3) node[anchor=west]{$p_1$};
	\filldraw[black] (14,1.5) circle (1.5pt);
	\filldraw[black] (14,1.3) node[anchor=west]{$p_2$};
	
	\filldraw[black] (13.6,0.7) node[anchor=west]{$F_0$};

	\draw[->] (12,2) -- (13,2);
	\filldraw[black] (12.2,2.3) node[anchor=west]{\Large  $\sigma_1$};
	
	\draw[thick] (8.5,1) -- (8.5,3);	
	\draw[thick] (8,2.5) -- (11,2.5);	
	\filldraw[black] (8.5,2.5) circle (1.5pt);
	\filldraw[black] (9.5,2.5) circle (1.5pt);
	\filldraw[black] (10.5,2.5) circle (1.5pt);
	\filldraw[black] (8.5,2.3) node[anchor=west]{$q_1$};
	\filldraw[black] (9.3,2.2) node[anchor=west]{$q_1'$};
	\filldraw[black] (10.3,2.2) node[anchor=west]{$q_1''$};
	\filldraw[black] (8.1,0.7) node[anchor=west]{$\ol F_0$};
	\filldraw[black] (8.5,1.5) circle (1.5pt);
	\filldraw[black] (8.5,1.3) node[anchor=west]{$p_2$};
	\filldraw[black] (11,2.5) node[anchor=west]{$E_1$};
\end{tikzpicture}}$$
Note that all the singularities of $\sF_1$ are non-degenerate, and $K_{\sF_1}$ is nef.
According to \autoref{lem-7-5},
\begin{equation}\label{eqn-7-41}
	\vol(\sF)\geq K_{\sF_1}^2+\sum_{p\in T}\beta_p(\sF_1)
	=\big(p_g(\sF)-1\big)+\min\big\{0,\beta_{q_1}(\sF_1)\big\}+\min\big\{0,\beta_{p_2}(\sF_1)\big\},
\end{equation}
where $T$ is the set of all non-reduced singularities of $\sF_1$, which contains at most $q_1$ and $p_2$ by the above argument.
By \autoref{prop-2-2},
$$\cs(\sF_1,\ol F_1,q_1)+\cs(\sF_1,\ol F_1,p_2)=-\ol F_0^2=-1.$$
Hence there is at most one of the singularities $\{q_1,p_2\}$ non-reduced.
If both of them are reduced, then $\vol(\sF)\geq p_g(\sF)-1$ by \eqref{eqn-7-41}.
If $q_1$ is non-reduced, then $p_2$ is reduced and $\beta_{q_1}(\sF_1)\geq -\frac{1}{4}$;
indeed, if $\beta_{q_1}(\sF_1)<-\frac14$, then the eigenvalue of $\sF_1$ at $q_1$ is
\begin{equation}\label{eqn-7-44}
	\lambda_{q_1} \not\in \big\{1,2,1/2,3,1/3\big\}.
\end{equation}
According to the resolution of a non-degenerate singularity above \autoref{lem-2-2},
the minimal resolution $\tilde \phi:\,\wt{S} \to S_1$ of the singularity $q_1$ consists of at most $3$ blowing-ups.
Moreover, the contraction $\pi_1$ factors through $\tilde\phi$ by \autoref{lem-7-9}\,(ii).
Since $\sF$ is relatively minimal, it follows that $S=\wt{S}$.
This contradicts \autoref{lem-7-10}.
Hence by \eqref{eqn-7-41},
$$\vol(\sF)\geq \big(p_g(\sF)-1\big)+\beta_{q_1}(\sF_1)
\geq \big(p_g(\sF)-1\big)-\frac14=p_g(\sF)-\frac54.$$
If $p_2$ is non-reduced, the argument would be completely the same and we omit it here.

Suppose next that $a_{p_1}=1$. By \autoref{lem-7-7}, one can resolve $p_1$ as follows.
$${\setlength{\unitlength}{6mm}
	\begin{tikzpicture}
	\draw[thick] (14,1) -- (14,3);	
	\filldraw[black] (14,1.5) circle (1.5pt);
	\filldraw[black] (14,1.3) node[anchor=west]{$p_2$};
	\filldraw[black] (14,2.5) circle (1.5pt);
	\filldraw[black] (14,2.3) node[anchor=west]{$p_1$};
	\filldraw[black] (13.6,0.7) node[anchor=west]{$F_0$};

	\draw[->] (12,2) -- (13,2);
	\filldraw[black] (12.2,2.3) node[anchor=west]{\Large  $\sigma_1$};
	
	\draw[thick] (9,1) -- (9,3);	
	\draw[thick] (8.5,2.5) -- (11,2.5);	
	\filldraw[black] (9,1.5) circle (1.5pt);
	\filldraw[black] (9,1.3) node[anchor=west]{$p_2$};
	\filldraw[black] (9,2.5) circle (1.5pt);
	\filldraw[black] (9,2.3) node[anchor=west]{$q_1$};
	\filldraw[black] (8.6,0.7) node[anchor=west]{$\ol F_0$};
	\filldraw[black] (11,2.5) node[anchor=west]{$E_1$};
	
	\draw[->] (7,2) -- (8,2);
	\filldraw[black] (7.2,2.3) node[anchor=west]{\Large  $\sigma_2$};
	
	\draw[thick] (4,1) -- (4,3);
	\draw[thick] (5.5,1) -- (5.5,3);	
	\draw[thick] (3.5,2.5) -- (6.3,2.5);	
	\filldraw[black] (4,1.5) circle (1.5pt);
	\filldraw[black] (4,1.3) node[anchor=west]{$p_2$};
	\filldraw[black] (4.75,2.5) circle (1.5pt);
	\filldraw[black] (4.55,2.25) node[anchor=west]{$q_2'$};
	\filldraw[black] (5.5,2.5) circle (1.5pt);
	\filldraw[black] (5.4,2.25) node[anchor=west]{$\bar q_1$};
	\filldraw[black] (4,2.5) circle (1.5pt);
	\filldraw[black] (3.5,2.3) node[anchor=west]{$q_2$};
	\filldraw[black] (3.6,0.7) node[anchor=west]{$\ol F_0$};
	\filldraw[black] (5.1,0.7) node[anchor=west]{$\ol E_1$};
	\filldraw[black] (6.25,2.5) node[anchor=west]{$E_2$};
\end{tikzpicture}}$$
Let $\sF_i$ be the induced foliation on $S_i$ as in \autoref{lem-7-7}.
By \autoref{prop-2-2} and \autoref{lem-2-2},
\begin{equation}\label{eqn-7-43}
	\left\{\begin{aligned}
	&\cs(\sF_2,E_2,\bar q_1)^{-1}=\cs(\sF_2,\ol E_1,\bar q_1)=\ol E_1^2=-2,\\
	&\cs(\sF_2,E_2,q_2)^{-1}+\cs(\sF_2,\ol F_0, p_2)=\cs(\sF_2,\ol F_0, q_2)+\cs(\sF_2,\ol F_0, p_2)=\ol F_0^2=-2,\\
	&\cs(\sF_2,E_2,q_2')+\cs(\sF_2,E_2,q_2)=E_2^2-\cs(\sF_2,E_2,\bar q_1)=-1+\frac12=-\frac12.
	\end{aligned}\right.
\end{equation}
According to \eqref{eqn-7-32},
$$K_{\sF_2}=\pi^*(K_{\sF_0})-E_2=\pi^*(M_0)+\ol F_0+E_2+\ol E_1.$$
Hence the Zariski decomposition of $K_{\sF_2}$ is as follows.
$$K_{\sF_2}=P_2+N_2, \qquad \text{where~}P_2=\pi^*(M_0)+\ol F_0+E_2+\frac{\ol E_1}{2},\,~\,N_2=\frac{\ol E_1}{2}.$$
By \autoref{lem-7-9}, $\phi:\,S \to Y$ factors through $\pi=\sigma_1\circ\sigma_2$ as $\phi=\pi\circ\phi_2$,
and that the possible non-reduced singularities are all on $\ol F_0+E_2+\ol E_1$.
Let $T_2$ be the set of non-reduced singularities of $\sF_2$.
According to \autoref{lem-7-5},
\begin{equation}\label{eqn-7-42}
\vol(\sF)\geq K_{\sF_2}^2+\beta_{\bar q_1}(\sF_1)+\sum_{q\in T_2}\beta_{q}(\sF_2).
\end{equation}
By \eqref{eqn-7-43}, $\bar q_1$ is reduced with $\beta_{\bar q_1}(\sF_1)=\frac12$.

If $q_2$ is non-reduced, i.e., $\cs(\sF_2,E_2,q_2)^{-1}=\cs(\sF_2,\ol F_0, q_2)$ is a positive rational number, then by \eqref{eqn-7-43}, both $\cs(\sF_2,E_2,q_2')$ and $\cs(\sF_2,\ol F_0, p_2)$ are negative rational,
which implies that both $q_2'$ and $p_2$ are reduced.
Moreover,  the eigenvalue $\lambda_{q_2}\neq 1$; otherwise,
$\phi:\,S \to Y$ consists of at most $3$ blowing-ups by a similar argument in proving \eqref{eqn-7-44}, a contradiction to \autoref{lem-7-10}.
Hence by \eqref{eqn-7-42},
$$\vol(\sF)\geq \big(p_g(\sF)-1\big)+\frac12+\sum_{q\in T_2}\beta_{q}(\sF_2)\geq \big(p_g(\sF)-1\big)+\frac12-\frac12=p_g(\sF)-1.$$

Suppose now that $q_2$ is reduced, from which it follows that $\cs(\sF_2,E_2,q_2)<0$.
If $T_2$ contains at most one singularity, i.e., $T_2=\{p_2\}$ or $T_2=\{q_2'\}$,
then similarly by \autoref{lem-7-10}, one shows that the eigenvalue $\lambda_q \neq 1$ for such a non-reduced singularity $q\in T_2$, and hence $\vol(\sF)\geq p_g(\sF)-1$.
Suppose next that $T_2=\{p_2,q_2'\}$, i.e., both $p_2$ and $q_2'$ are non-reduced.
In this case, the eigenvalue $\lambda_q\not \in \{1,2,1/2\}$ for any $q\in T_2$.
Indeed, if $\lambda_{p_2} \in \{1,2,1/2\}$, then $\lambda_{q_2'}<0$ by \eqref{eqn-7-43}; and similarly if $\lambda_{q_2'} \in \{1,2,1/2\}$, then $\lambda_{p_2}<0$. Hence by \eqref{eqn-7-42},
$$\vol(\sF)\geq \big(p_g(\sF)-1\big)+\frac12+\beta_{p_2}(\sF_2)+\beta_{q_2'}(\sF_2)\geq \big(p_g(\sF)-1\big)+\frac12-\frac13-\frac13=p_g(\sF)-\frac76.$$

{\bf Case (a.1.3).} Suppose that $k=3$.
By \eqref{eqn-7-34},
$$Z(\sF_0,F_0,p_1)=Z(\sF_0,F_0,p_2)=Z(\sF_0,F_0,p_3)=1.$$
In particular, all singularities of $\sF_0$ are non-degenerate.
According to \autoref{prop-2-2},
\begin{equation}\label{eqn-7-45}
	\cs(\sF_0,F_0,p_1)+\cs(\sF_0,F_0,p_2)+\cs(\sF_0,F_0,p_3)=F_0^2=0.
\end{equation}
It follows in particular that at most two of the above $\cs$-indices are positive rational.
We may assume without loss of generality that $\cs(\sF_0,F_0,p_3)$ is not positive rational.
By \autoref{lem-7-5},
\begin{equation}\label{eqn-7-46}
\vol(\sF)\geq K_{\sF_0}^2+\sum_{p\in T_0}\beta_{p}(\sF_0)=p_g(\sF)+\sum_{p\in T_0}\beta_{p}(\sF_0),
\end{equation}
where $T_0$ is the set of non-reduced singularities of $\sF_0$.
Suppose first that one of the $\cs$-indices equal to $1$, saying for instance $\cs(\sF_0,F_0,p_1)=1$. Then similar to prove \eqref{eqn-7-44}, by \autoref{lem-7-10} one shows that
$$\cs(\sF_0,F_0,p_2)\not \in \big\{1,2,3,1/2,1/3\}.$$
Hence by \eqref{eqn-7-46},
$$\vol(\sF)\geq p_g(\sF)-1-\frac14=p_g(\sF)-\frac54.$$
If $\cs(\sF_0,F_0,p_1)\neq 1$ and $\cs(\sF_0,F_0,p_2)\neq 1$, then by \eqref{eqn-7-46},
$$\vol(\sF)\geq p_g(\sF)-\frac12-\frac12=p_g(\sF)-1.$$

{\bf Case (a.2).} 
By \autoref{prop-2-1} one checks easily that $C_0$ is $\sF_0$-invariant.
As $F_0$ is also $\sF_0$-invariant, it follows that the intersection $p_0=F_0\cap C_0$ is a singularity of $\sF_0$.
Hence by \autoref{prop-2-2},
$$1\leq Z(\sF_0,C_0)=K_{\sF_0}\cdot C_0+\chi(C_0)=3-e.$$
Since $e\geq 2$ by assumption, it follows that $e=2$,
and hence $p_g(\sF)=2+e=4$ and
$$Z(\sF_0,C_0)=Z(\sF_0,C_0,p_0)=1.$$
In particular, $p_0$ is the unique singularity of $\sF_0$ on $C_0$,
and by \eqref{eqn-2-7},
$$\cs(\sF_0,C_0,p_0)=C_0^2=-2.$$
In particular, $p_0$ is a reduced singularity, and $Z(\sF_0,F_0,p_0)=1$ by \autoref{cor-7-2}.
In particular, all the possible non-reduced singularities of $\sF_0$ are on $F_0\setminus \{p_0\}$ by \autoref{lem-7-9}.
Let $p_0,p_1,\cdots,p_k$ be all the singularities of $\sF_0$ on $F_0$.
By \autoref{prop-2-2},
\begin{equation}\label{eqn-7-34}
	\sum_{i=1}^{k} Z(\sF_0,F_0,p_i)=Z(\sF_0,F_0)-Z(\sF_0,F_0,p_0)=3.
\end{equation}
In particular, $1\leq k \leq 3$. One can then apply a similar argument as in Case (a.1) to finish the proof according to the value of $k$.
We outline the main points in the following and omit the details.
Since $p_g(\sF)=4$ in this case, it suffices to show $\vol(\sF)\geq \frac{11}{4}$.

{\bf Case (a.2.1).} Suppose that $k=1$.
In this case $Z(\sF_0,F_0,p_1)=3$ by \eqref{eqn-7-34}.
First similar to \autoref{claim-7-8}, one shows that the vanishing order of $\sF_0$ at $p_1$ is $a_{p_1}=1$.
By \autoref{lem-7-7}, one can resolve $p_1$ into non-degenerate singularities as follows.
$${\setlength{\unitlength}{6mm}
	\begin{tikzpicture}
	\draw[thick] (8,1) -- (8,3);
	\draw[thick] (7.5,1.4) -- (9.4,1.4);
	\filldraw[black] (8,2.4) circle (1.5pt);
	\filldraw[black] (8,2.5) node[anchor=west]{$p_1$};
	\filldraw[black] (8,1.4) circle (1.5pt);
	\filldraw[black] (8,1.6) node[anchor=west]{$p_0$};
	\filldraw[black] (7.6,0.7) node[anchor=west]{$F_0$};
	\filldraw[black] (9.4,1.4) node[anchor=west]{$C_0$};

	\draw[->] (5.5,2) -- (6.7,2);
	\filldraw[black] (5.75,2.3) node[anchor=west]{\Large  $\sigma$};

	\draw[thick] (3,1) -- (3,3);
	\draw[thick] (0.5,1) -- (0.5,3);
	\draw[thick] (1.34,0.88) -- (3.5,2.5);
	\draw[thick] (2.16,0.88) -- (0,2.5);
	\filldraw[black] (0.5,2.12) circle (1.5pt);
	\filldraw[black] (1.75,1.2) circle (1.5pt);
	\filldraw[black] (2.4,1.68) circle (1.5pt);
	\filldraw[black] (3,2.12) circle (1.5pt);
	\draw[thick] (2.6,1.4) -- (4.4,1.4);
	\filldraw[black] (3,1.4) circle (1.5pt);
	\filldraw[black] (2.9,1.2) node[anchor=west]{$p_0$};
	\filldraw[black] (3,2) node[anchor=west]{$q_3$};
	\filldraw[black] (2,2) node[anchor=west]{$q_3'$};
	\filldraw[black] (1.45,1.5) node[anchor=west]{$\bar q_2$};
	\filldraw[black] (0.45,2.2) node[anchor=west]{$\bar q_1$};
	\filldraw[black] (-0.1,1) node[anchor=west]{$\ol{E}_1$};
	\filldraw[black] (2.6,0.7) node[anchor=west]{$\ol{F}_0$};
	\filldraw[black] (-0.65,2.6) node[anchor=west]{$\ol E_2$};
	\filldraw[black] (3.3,2.7) node[anchor=west]{$E_3$};
	\filldraw[black] (4.4,1.4) node[anchor=west]{$\ol C_0$};
\end{tikzpicture}}$$
Moreover, by \eqref{eqn-7-32},
$$\begin{aligned}
K_{\sF_3}&\,=\sigma^*(K_{\sF_0})-(\mathcal{E}_2+\mathcal{E}_3)=\sigma^*(K_{\sF_0})-(\ol E_2+2E_3)\\
&\,=\sigma^*(M_0)+\ol C_0+\ol F_0+\ol E_1+\ol E_2+E_3=P_3+N_3,
\end{aligned}$$
where $P_3=\sigma^*(M_0)+\frac15\ol C_0+\frac25\ol F_0+\frac13\ol E_1+\frac23\ol E_2+E_3$ and $N_3=\frac45\ol C_0+\frac35\ol F_0+\frac23\ol E_1+\frac13\ol E_2$
are respectively the nef and negative parts of the Zariski decomposition of $K_{\sF_3}$.
By \autoref{prop-2-2} and \autoref{lem-2-2}, one obtains that
\begin{equation}\label{eqn-7-39}
\left\{\begin{aligned}
&\cs(\sF_3,E_3,q_3)^{-1}=\cs(\sF_3,\ol F_0,q_3)=\ol F_0^2-\cs(\sF_3,\ol F_0,p_0)=-3+\frac12=-\frac{5}{2};\\
&\cs(\sF_3,\ol E_2,\bar q_1)^{-1}=\cs(\sF_3,\ol E_1,\bar q_1)=\ol E_1^2=-2;\\
&\cs(\sF_3,E_3,\bar q_2)^{-1}=\cs(\sF_3,\ol E_2,\bar q_2)=\ol E_2^2-\cs(\sF_3,\ol E_2,\bar q_1)=-2+\frac12=-\frac32;\\
&\cs(\sF_3,E_3,q_3')=E_3^2-\cs(\sF_3,E_3,\bar q_2)-\cs(\sF_3,E_3,q_3)=-1+\frac23+\frac25=\frac{1}{15}.
\end{aligned}\right.
\end{equation}
Hence by \autoref{lem-7-5},
$$\begin{aligned}
\vol(\sF)&\,\geq K_{\sF_3}^2+\beta_{p_0}(\sF_3)+\beta_{q_3}(\sF_3)+\beta_{\bar q_2}(\sF_3)+\beta_{\bar q_1}(\sF_3)+\beta_{q_3'}(\sF_3)\\
&\,=2+\frac12+\frac{1}{10}+\frac16+\frac12-\frac{1}{15}=\frac{16}{5}.
\end{aligned}$$

{\bf Case (a.2.2).} Suppose that $k=2$.
By \eqref{eqn-7-34}, we may assume without loss of generality that
$Z(\sF_0,F_0,p_1)=2$ and $Z(\sF_0, F_0, p_2)=1$.
Since $Z(\sF_0,F_0,p_1)=2$, the vanishing order of $\sF_0$ at $p_1$ is $a_{p_1}\leq 2$.

If the vanishing order $a_{p_1}=2$, then one can resolve $p_1$ into non-reduced singularities as follows.
$${\setlength{\unitlength}{6mm}
	\begin{tikzpicture}
	\draw[thick] (14,1) -- (14,3);	
	\draw[thick] (13.5,2.7) -- (15.5,2.7);		
	\filldraw[black] (14,2.7) circle (1.5pt);
	\filldraw[black] (14,2.5) node[anchor=west]{$p_0$};
	\filldraw[black] (14,2) circle (1.5pt);
	\filldraw[black] (14,1.8) node[anchor=west]{$p_1$};
	\filldraw[black] (14,1.3) circle (1.5pt);
	\filldraw[black] (14,1.1) node[anchor=west]{$p_2$};

	\filldraw[black] (13.6,0.7) node[anchor=west]{$F_0$};
	\filldraw[black] (15.5,2.7) node[anchor=west]{$C_0$};

	\draw[->] (12,2) -- (13,2);
	\filldraw[black] (12.2,2.3) node[anchor=west]{\Large  $\sigma_1$};
	
	\draw[thick] (8.5,1) -- (8.5,3);	
	\draw[thick] (8,2.7) -- (11,2.7);	
	\draw[thick] (8,2) -- (11,2);	
	\filldraw[black] (8.5,2.7) circle (1.5pt);
	\filldraw[black] (8.5,2.5) node[anchor=west]{$p_0$};
	\filldraw[black] (8.5,2) circle (1.5pt);
	\filldraw[black] (9.5,2) circle (1.5pt);
	\filldraw[black] (10.5,2) circle (1.5pt);
	\filldraw[black] (8.45,1.7) node[anchor=west]{$q_1$};
	\filldraw[black] (9.3,1.7) node[anchor=west]{$q_1'$};
	\filldraw[black] (10.3,1.7) node[anchor=west]{$q_1''$};
	\filldraw[black] (8.1,0.7) node[anchor=west]{$\ol F_0$};
	\filldraw[black] (8.5,1.3) circle (1.5pt);
	\filldraw[black] (8.5,1.1) node[anchor=west]{$p_2$};
	\filldraw[black] (11,2) node[anchor=west]{$E_1$};
	\filldraw[black] (11,2.7) node[anchor=west]{$\ol C_0$};
\end{tikzpicture}}$$
Moreover, it holds that $q_1'$ and $q_2''$ are reduced by \autoref{lem-7-9}, since
$$K_{\sF_1}=\sigma_1^*(K_{\sF_0})-E_1=\sigma_1^*(M_0)+\ol C_0+\ol F_0.$$
The Zariski decomposition is $K_{\sF_1}=\Big(\sigma_1^*(M_0)+\frac12\ol C_0+\ol F_0\Big)+\frac12\ol C_0$.
Moreover,
$$\begin{aligned}
&\cs(\sF_1,\ol F_0,p_0)^{-1}=\cs(\sF_1,\ol C_0,p_0)=-\ol C_0^2=-2,\\
&\cs(\sF_1,\ol F_0,q_1)+\cs(\sF_1,\ol F_0,p_2)=\ol F_0^2-\cs(\sF_1,\ol F_0,p_0)=-\frac12.
\end{aligned}$$
Hence at most one of the two singularities $\{q_1,p_2\}$ is non-reduced; say $q_1$.
Similar to the proof of \eqref{eqn-7-44}, one shows by \autoref{lem-7-10} that
$\lambda_{q_1}\not \in \{1,2,1/2,3,1/3\}.$
Hence
$$\begin{aligned}
\vol(\sF)&\,\geq K_{\sF_1}^2+\beta_{p_0}(\sF_1)+\beta_{q_1}(\sF_1)\geq 3+\frac12-\frac14=\frac{13}{4}.
\end{aligned}$$

If the vanishing order $a_{p_1}=1$, then by \autoref{lem-7-7}, one can resolve $p_1$ as follows.
$${\setlength{\unitlength}{6mm}
	\begin{tikzpicture}
	\draw[thick] (14,1) -- (14,3.5);	
	\draw[thick] (13.5,3) -- (15.5,3);		
	\filldraw[black] (14,1.5) circle (1.5pt);
	\filldraw[black] (14,1.3) node[anchor=west]{$p_2$};
	\filldraw[black] (14,3) circle (1.5pt);
	\filldraw[black] (14,2.8) node[anchor=west]{$p_0$};
	\filldraw[black] (14,2.25) circle (1.5pt);
	\filldraw[black] (14,2.05) node[anchor=west]{$p_1$};
	\filldraw[black] (13.6,0.7) node[anchor=west]{$F_0$};
	\filldraw[black] (15.5,3) node[anchor=west]{$C_0$};

	\draw[->] (12,2) -- (13,2);
	\filldraw[black] (12.2,2.3) node[anchor=west]{\Large  $\sigma_1$};
	
	\draw[thick] (9,1) -- (9,3.5);	
	\draw[thick] (8.5,3) -- (11,3);	
	\draw[thick] (8.5,2.25) -- (11,2.25);	
	\filldraw[black] (9,1.5) circle (1.5pt);
	\filldraw[black] (9,1.3) node[anchor=west]{$p_2$};
	\filldraw[black] (9,3) circle (1.5pt);
	\filldraw[black] (9,2.8) node[anchor=west]{$p_0$};
	\filldraw[black] (9,2.25) circle (1.5pt);
	\filldraw[black] (9,2.05) node[anchor=west]{$q_1$};
	\filldraw[black] (8.6,0.7) node[anchor=west]{$\ol F_0$};
	\filldraw[black] (11,3) node[anchor=west]{$\ol C_0$};
	\filldraw[black] (11,2.25) node[anchor=west]{$E_1$};
	
	\draw[->] (7,2) -- (8,2);
	\filldraw[black] (7.2,2.3) node[anchor=west]{\Large  $\sigma_2$};
	
	\draw[thick] (4,1) -- (4,3.5);
	\draw[thick] (5.5,1) -- (5.5,2.85);	
	\draw[thick] (3.5,2.25) -- (6.3,2.25);
	\draw[thick] (3.5,3.1) -- (6,3.1);	
	\filldraw[black] (4,3.1) circle (1.5pt);
	\filldraw[black] (4,2.9) node[anchor=west]{$p_0$};	
	\filldraw[black] (4,1.5) circle (1.5pt);
	\filldraw[black] (4,1.3) node[anchor=west]{$p_2$};
	\filldraw[black] (4.75,2.25) circle (1.5pt);
	\filldraw[black] (4.55,2) node[anchor=west]{$q_2'$};
	\filldraw[black] (5.5,2.25) circle (1.5pt);
	\filldraw[black] (5.4,2) node[anchor=west]{$\bar q_1$};
	\filldraw[black] (4,2.25) circle (1.5pt);
	\filldraw[black] (3.5,2.05) node[anchor=west]{$q_2$};
	\filldraw[black] (3.6,0.7) node[anchor=west]{$\ol F_0$};
	\filldraw[black] (6,3.1) node[anchor=west]{$\ol C_0$};
	\filldraw[black] (5.1,0.7) node[anchor=west]{$\ol E_1$};
	\filldraw[black] (6.25,2.25) node[anchor=west]{$E_2$};
\end{tikzpicture}}$$
Moreover,
\begin{equation*}
\left\{\begin{aligned}
&\cs(\sF_2,E_2,\bar q_1)^{-1}=\cs(\sF_2,\ol E_1,\bar q_1)=\ol E_1^2=-2,\\
&\cs(\sF_2,E_2,q_2)^{-1}+\cs(\sF_2,\ol F_0, p_2)=\cs(\sF_2,\ol F_0, q_2)+\cs(\sF_2,\ol F_0, p_2)\\
&\hspace{5.65cm}=\ol F_0^2-\cs(\sF_2,\ol F_0, p_0)=-\frac32,\\
&\cs(\sF_2,E_2,q_2')+\cs(\sF_2,E_2,q_2)=E_2^2-\cs(\sF_2,E_2,\bar q_1)=-1+\frac12=-\frac12.
\end{aligned}\right.
\end{equation*}
According to \eqref{eqn-7-32},
$$K_{\sF_2}=\pi^*(K_{\sF_0})-E_2=\pi^*(M_0)+\ol C_0+\ol F_0+E_2+\ol E_1.$$
Hence the Zariski decomposition of $K_{\sF_2}$ is as follows.
$$K_{\sF_2}=P_2+N_2, \qquad \text{where~}P_2=\pi^*(M_0)+\ol F_0+E_2+\frac{\ol C_0+\ol E_1}{2},\,~\,N_2=\frac{\ol C_0+\ol E_1}{2}.$$
Let $T_2$ be the set of non-reduced singularities of $\sF_2$.
Similar to the proof in Case (a.1.2), one shows that $T_2$ contains at most two singularities and $\sum\limits_{q\in T_2}\beta_{q}(\sF_2)\geq -1$.
Hence by \autoref{lem-7-5} one obtains
\begin{equation*}
\vol(\sF)\geq K_{\sF_2}^2+\beta_{\bar p_0}(\sF_1)+\beta_{\bar q_1}(\sF_1)+\sum_{q\in T_2}\beta_{q}(\sF_2)=4+\sum_{q\in T_2}\beta_{q}(\sF_2)\geq 3.
\end{equation*}

{\bf Case (a.2.3).} Suppose that $k=3$. Then
$$\begin{aligned}
&Z(\sF_0,F_0,p_1)=Z(\sF_0,F_0,p_2)=Z(\sF_0,F_0,p_3)=1,\\
&\cs(\sF_0,F_0,p_1)+\cs(\sF_0,F_0,p_2)+\cs(\sF_0,F_0,p_3)=F_0^2-\cs(\sF_0,F_0,p_0)=\frac12.
\end{aligned}$$
Let $T_0$ be the set of non-reduced singularities of $\sF_0$.
Similar to Case (a.1.3), one has
\begin{equation*}
\vol(\sF)\geq K_{\sF_0}^2+\beta_{p_0}(\sF_0)+\sum_{p\in T_0}\beta_{p}(\sF_0)=4+\frac12+\sum_{p\in T_0}\beta_{p}(\sF_0)\geq \frac92-\frac32=3>\frac{11}{4}.
\end{equation*}
This completes the proof of Case (a).

{\vspace{2mm}\noindent \bf Case (b).} 
The surface $Y\cong \bbp^2$.
\autoref{prop-7-1} together with the two conditions (A-1) and (A-2) in \autoref{claim-7-5} imply that $K_{\sF_0}=M_0+Z_0$,
where $M_0=L$ is a general line in $\bbp^2$, and $Z_0=D_0$ is some fixed line in $\bbp^2$.
Moreover, the strict transform of the curve $D_0$ in $S$ is contained in the support of the negative part in the Zariski decomposition of $K_{\sF}$.
According to \autoref{lem-7-9}, all the possible non-reduced singularities of $\sF_0$ are on $D_0$.

Let $p$ be a non-reduced singularity of $\sF_0$, and
let $\sigma_1:\, Y_1 \to Y$ be the blowing-up centered at $p$ with exceptional curve $E$.
Then $Y_1\cong \mathbb{P}_{\bbp^1}\big(\calo_{\bbp^1}\oplus \calo_{\bbp^1}(1)\big)$ is  the Hirzebruch surface with value $e=1$.
Moreover, the birational map $\phi:\,S \to Y$ factors through $\sigma_1$ as $\phi=\sigma_1\circ\phi_1$ by \autoref{lem-7-9}\,(ii). Let $\sF_1=\sigma_1^*(\sF_0)$ and 
$$K_{\sF_1}=\sigma_1^*(K_{\sF_0})-(\ell_{p}-1)E_1=M_1+F_0+(2-\ell_p)E_1,$$
where $M_1=\sigma_1^*(M_0)=\sigma_1^*(L)$, and $F_0$ is the strict transform of $D_0$, which is nothing but a fiber of the ruling on $Y_1 \to \bbp^1$.
In particular, $\ell_p\leq 2$ by \autoref{lem-7-9}\,(ii).
If $\ell_p=2$, then $K_{\sF_1}=M_1+F_0$.
Note that $M_1\sim C_0+F$, where $C_0=E$ is the section with negative self-intersection number equal to $-e=-1$, and $F$ is a general fiber of the ruling on $Y_1 \to \bbp^1$.
Replacing $(Y,\sF_0)$ by $(Y_1,\sF_1)$, we reach the situation as in Case (a.1).
Applying a similar argument in Case (a.1), one proves \eqref{eqn-5-8-2}.
Thus we may assume that $\ell_p=1$ for any non-reduced singularity $p$ of $\sF_0$.
In particular, for any non-reduced singularity $p$ of $\sF_0$,
the vanishing order $a_p(\sF_0)=1$ by \eqref{eqn-2-13},
and the eigenvalue $\lambda_{p}\neq 1$, cf. the resolution process above \autoref{lem-2-2}.

Let $p_1,\cdots,p_k$ be all the singularities of $\sF_0$ on $D_0$.
By \autoref{prop-2-2},
$$\sum_{i=1}^{k}Z(\sF_0,D_0,p_i)=Z(\sF_0,D_0)=K_{\sF_0}D_0+\chi(D_0)=4.$$
We can prove \eqref{eqn-5-8-2} case-by-case depending on the value of $k$ (most arguments would be similar to that in Case (a.1)).
Since $p_g(\sF)=3$, it suffices to show $\vol(\sF)\geq \frac{7}{4}$.

{\noindent \bf Case (b.1).}
Suppose $k=1$. Then by \autoref{lem-7-7}, we can resolve the singularity $p_1$ as follows.
$${\setlength{\unitlength}{6mm}
	\begin{tikzpicture}
	\draw[thick] (6,1) -- (6,3);
	\filldraw[black] (6,2) circle (1.5pt);
	\filldraw[black] (6,2) node[anchor=west]{$p_1$};
	\filldraw[black] (6,1) node[anchor=west]{$D_0$};

	\draw[->] (4.3,2) -- (5.2,2);
	\filldraw[black] (4.5,2.3) node[anchor=west]{\Large  $\sigma$};

	\draw[thick] (3,1) -- (3,3);
	\draw[thick] (-0.5,1) -- (-0.5,3);
	\draw[thick] (1.34,0.88) -- (3.5,2.5);
	\draw[thick] (2.16,0.88) -- (0,2.5);
	\draw[thick] (-1,1) -- (1,2.5);
	\filldraw[black] (0.5,2.12) circle (1.5pt);
	\filldraw[black] (1.75,1.2) circle (1.5pt);
	\filldraw[black] (2.4,1.68) circle (1.5pt);
	\filldraw[black] (-0.5,1.37) circle (1.5pt);
	\filldraw[black] (3,2.12) circle (1.5pt);
	\filldraw[black] (3,2) node[anchor=west]{$q_4$};
	\filldraw[black] (2.3,1.5) node[anchor=west]{$q_4'$};
	\filldraw[black] (1.45,1.5) node[anchor=west]{$\bar q_3$};
	\filldraw[black] (0.25,1.8) node[anchor=west]{$\bar q_2$};
	\filldraw[black] (-0.55,1.2) node[anchor=west]{$\bar q_1$};
	\filldraw[black] (-0.8,0.7) node[anchor=west]{$\ol{E}_1$};
	\filldraw[black] (3,1) node[anchor=west]{$\ol{D}_0$};
	\filldraw[black] (-0.4,2.8) node[anchor=west]{$\ol E_3$};
	\filldraw[black] (0.8,2.8) node[anchor=west]{$\ol E_2$};
	\filldraw[black] (3.3,2.7) node[anchor=west]{$E_4$};
\end{tikzpicture}}$$
Moreover,
$$K_{\sF_4}=\sigma^*(K_{\sF_0})-\sum_{i=2}^{4}\mathcal{E}_i=\sigma^*(M_0)+\ol D_0+\sum_{i=1}^{3}\ol E_i+E_4.$$
Hence the Zariski decomposition is $K_{\sF_4}=P_4+N_4$, where
$$P_4=\sigma^*(M_0)+\frac{\ol D_0}{3}+\sum_{i=1}^{3}\frac{i\ol E_i}{4}+E_4,\qquad
N_4=\frac{2\ol D_0}{3}+\sum_{i=1}^{3}\frac{(4-i)\ol E_i}{4}.$$
By \autoref{prop-2-2} with \autoref{lem-2-2},
$$\left\{\begin{aligned}
&\cs(\sF_4,E_4,q_4)^{-1}=\cs(\sF_4,\ol D_0,q_4)=\ol D_0^2=-3,\\
&\cs(\sF_4,\ol E_2,\bar q_1)^{-1}=\cs(\sF_4,\ol E_1,\bar q_1)=\ol E_1^2=-2,\\
&\cs(\sF_4,\ol E_3,\bar q_2)^{-1}=\cs(\sF_4,\ol E_2,\bar q_2)=\ol E_2^2-\cs(\sF_4,\ol E_2,\bar q_1)=-\frac32,\\
&\cs(\sF_4,E_4,\bar q_3)^{-1}=\cs(\sF_4,\ol E_3,\bar q_3)=\ol E_3^2-\cs(\sF_4,\ol E_3,\bar q_2)=-\frac43,\\
&\cs(\sF_4,E_4,\bar q_4')=E_4^2-\cs(\sF_4,E_4,\bar q_3)-\cs(\sF_4,E_4,q_4)=\frac{1}{12}.
\end{aligned}\right.$$
Thus by \autoref{lem-7-5},
$$\vol(\sF)\geq K_{\sF_4}^2+\sum_{i=1}^{3}\beta_{\bar q_i}(\sF_4)+\beta_{q_4}(\sF_4)+\beta_{q_4'}(\sF_4)=1+\frac{3}{4}+\frac13-\frac{1}{12}=2.$$

{\noindent \bf Case (b.2).}
Suppose $k=2$.
Assume without loss of generality that $Z(\sF_0,D_0,p_1)\geq Z(\sF_0,D_0,p_2)$.
Then there are two possibilities:
either $Z(\sF_0,D_0,p_1)=3$ and $Z(\sF_0,D_0,p_2)=1$,
or $Z(\sF_0,D_0,p_1)=Z(\sF_0,D_0,p_2)=2$.
If it is the first possibility, then we can resolve the singularity $p_1$ as follows by \autoref{lem-7-7}.
$${\setlength{\unitlength}{6mm}
	\begin{tikzpicture}
	\draw[thick] (6,1) -- (6,3);
	\filldraw[black] (6,2.4) circle (1.5pt);
	\filldraw[black] (6,2.4) node[anchor=west]{$p_1$};
	\filldraw[black] (6,1.5) circle (1.5pt);
	\filldraw[black] (6,1.5) node[anchor=west]{$p_2$};
	\filldraw[black] (5.7,0.7) node[anchor=west]{$D_0$};

	\draw[->] (4.3,2) -- (5.2,2);
	\filldraw[black] (4.5,2.3) node[anchor=west]{\Large  $\sigma$};

	\draw[thick] (3,1) -- (3,3);
	\draw[thick] (0.5,1) -- (0.5,3);
	\draw[thick] (1.34,0.88) -- (3.5,2.5);
	\draw[thick] (2.16,0.88) -- (0,2.5);
	\filldraw[black] (0.5,2.12) circle (1.5pt);
	\filldraw[black] (1.75,1.2) circle (1.5pt);
	\filldraw[black] (2.4,1.68) circle (1.5pt);
	\filldraw[black] (3,2.12) circle (1.5pt);
	\filldraw[black] (3,2) node[anchor=west]{$q_3$};
	\filldraw[black] (2.3,1.5) node[anchor=west]{$q_3'$};
	\filldraw[black] (1.45,1.5) node[anchor=west]{$\bar q_2$};
	\filldraw[black] (0.45,2.2) node[anchor=west]{$\bar q_1$};
	\filldraw[black] (-0.1,1) node[anchor=west]{$\ol{E}_1$};
	\filldraw[black] (2.7,0.7) node[anchor=west]{$\ol{D}_0$};
	\filldraw[black] (-0.65,2.6) node[anchor=west]{$\ol E_2$};
	\filldraw[black] (3.3,2.7) node[anchor=west]{$E_3$};
	\filldraw[black] (3,1.3) circle (1.5pt);
	\filldraw[black] (3,1.3) node[anchor=west]{$p_2$};
\end{tikzpicture}}$$
Moreover,
$$K_{\sF_3}=\sigma^*(K_{\sF_0})-\sum_{i=2}^{3}\mathcal{E}_i=\sigma^*(M_0)+\ol D_0+\sum_{i=1}^{2}\ol E_i+E_3.$$
Hence the Zariski decomposition is $K_{\sF_4}=P_4+N_4$, where
$$P_4=\sigma^*(M_0)+\ol D_0+\frac{\ol E_1}{3}+\frac{2\ol E_2}{3}+E_4,\qquad
N_4=\frac{2\ol E_1}{3}+\frac{\ol E_2}{3}.$$
By \autoref{prop-2-2} with \autoref{lem-2-2},
\begin{equation}\label{eqn-7-47}
	\left\{\begin{aligned}
	&\cs(\sF_3,\ol E_2,\bar q_1)^{-1}=\cs(\sF_3,\ol E_1,\bar q_1)=\ol E_1^2=-2,\\
	&\cs(\sF_3,E_3,\bar q_2)^{-1}=\cs(\sF_3,\ol E_2,\bar q_2)=\ol E_2^2-\cs(\sF_3,\ol E_2,\bar q_1)=-\frac32,\\
	&\cs(\sF_3,E_3,q_3)+\cs(\sF_3,E_3,q_3')= E_3^2-\cs(\sF_3,E_3,\bar q_2)=-\frac13,\\
	&\cs(\sF_3,\ol D_0,q_3)+\cs(\sF_3,D_0,p_2)=\ol D_0^2=-2.
	\end{aligned}\right.
\end{equation}
Hence the possible non-reduced singularities belong to $\{p_2,q_3,q_3'\}$.
Let $T_3$ be the set of all non-reduced singularities of $\sF_3$.
By \autoref{lem-7-5},
$$\vol(\sF)\geq K_{\sF_3}^2+\sum_{i=1}^{2}\beta_{\bar q_i}(\sF_3)+\sum_{q\in T_3}\beta_{q}(\sF_3)=\frac83+\sum_{q\in T_3}\beta_{q}(\sF_3).$$	
Note that $\cs(\sF_3,E_3,q_3)=\cs(\sF_3,\ol D_0,q_3)^{-1}$ by \autoref{lem-2-2}.
Hence $\sF_3$ admits at most two non-reduced singularities by
the last two equalities in \eqref{eqn-7-47}.
Moreover, by \autoref{cor-7-3}, one shows that the eigenvalue $\lambda_q \neq 1$ for any $q\in T_3$.
If the eigenvalue $\lambda_q\not \in \{1,2,1/2\}$ for any $q\in T_3$, then
$$\vol(\sF)\geq \frac83+\sum_{q\in T_3}\beta_{q}(\sF_3) \geq \frac83-\frac13-\frac13=2.$$
If there is an eigenvalue $\lambda_{q}=2$ or $1/2$ for some $q\in T_3$, one can
prove \eqref{eqn-5-8-2} case-by-case.
We prove for instance the case when $\cs(\sF_3,D_0,p_2)=\lambda_{p_2}=\frac12$.
Then $\cs(\sF_3,E_3,q_3)^{-1}=\cs(\sF_3,\ol D_0,q_3)=-\frac52$, and
$\cs(\sF_3,E_3,q_3')=\frac{1}{15}$.
Hence
$$\vol(\sF)\geq \frac83+\sum_{q\in T_3}\beta_{q}(\sF_3) \geq \frac83-\frac12-\frac{1}{15}=\frac{21}{10}.$$

We consider next the second possibility where $Z(\sF_0,D_0,p_1)=Z(\sF_0,D_0,p_2)=2$.
The singularities $p_1$ and $p_2$ can be resolved as follows by \autoref{lem-7-7}.

$${\setlength{\unitlength}{6mm}
	\begin{tikzpicture}
	\draw[thick] (6,1) -- (6,3);
	\filldraw[black] (6,2.4) circle (1.5pt);
	\filldraw[black] (6,2.4) node[anchor=west]{$p_1$};
	\filldraw[black] (6,1.5) circle (1.5pt);
	\filldraw[black] (6,1.5) node[anchor=west]{$p_2$};
	\filldraw[black] (5.7,0.7) node[anchor=west]{$D_0$};

	\draw[->] (4.3,2) -- (5.2,2);
	\filldraw[black] (4.5,2.3) node[anchor=west]{\Large  $\sigma$};

	\draw[thick] (1,1) -- (1,3);
	\draw[thick] (0.5,1.7) -- (3,1);
	\draw[thick] (2,1) -- (3.5,1.7);
	\draw[thick] (0.5,2.3) -- (3,3);
	\draw[thick] (2,3) -- (3.5,2.3);
	\filldraw[black] (1,1.55) circle (1.5pt);
	\filldraw[black] (1.6,1.4) circle (1.5pt);
	\filldraw[black] (2.37,1.17) circle (1.5pt);
	\filldraw[black] (0.9,1.7) node[anchor=west]{$q_4$};
	\filldraw[black] (1.5,1.58) node[anchor=west]{$q_4'$};
	\filldraw[black] (2.1,1.42) node[anchor=west]{$\bar q_3$};
	
	\filldraw[black] (1,2.45) circle (1.5pt);
	\filldraw[black] (1.6,2.6) circle (1.5pt);
	\filldraw[black] (2.37,2.83) circle (1.5pt);
	\filldraw[black] (0.9,2.28) node[anchor=west]{$q_2$};
	\filldraw[black] (1.5,2.42) node[anchor=west]{$q_2'$};
	\filldraw[black] (2.1,2.55) node[anchor=west]{$\bar q_1$};
	
	\filldraw[black] (0.7,0.7) node[anchor=west]{$\ol D_0$};
	\filldraw[black] (-0.2,1.7) node[anchor=west]{$E_4$};
	\filldraw[black] (-0.2,2.3) node[anchor=west]{$E_2$};
	\filldraw[black] (3.5,1.7) node[anchor=west]{$\ol E_3$};
	\filldraw[black] (3.5,2.3) node[anchor=west]{$\ol E_1$};
\end{tikzpicture}}$$
Moreover, let $\sF_4$ be the induced foliation. Then
$$K_{\sF_4}=\sigma^*(K_{\sF_0})-E_2-E_4=\sigma^*(M_0)+\ol D_0+\ol E_1+\ol E_3+E_2+E_4.$$
Hence the Zariski decomposition is $K_{\sF_4}=P_4+N_4$, where
$$P_4=\sigma^*(M_0)+\ol D_0+\frac{\ol E_1+\ol E_3}{2}+E_2+E_4,\qquad
N_4=\frac{\ol E_1+\ol E_3}{2}.$$
By \autoref{prop-2-2} with \autoref{lem-2-2},
\begin{equation*}
\left\{\begin{aligned}
&\cs(\sF_4,E_2,\bar q_1)^{-1}=\cs(\sF_4,E_4,\bar q_3)^{-1}=\cs(\sF_4,\ol E_1,\bar q_1)=\cs(\sF_4,\ol E_3,\bar q_3)=-2,\\
&\cs(\sF_4,E_2,q_2)+\cs(\sF_4,E_2,q_2')=E_2^2-\cs(\sF_4,E_2,\bar q_1)=-\frac12,\\
&\cs(\sF_4,E_4,q_4)+\cs(\sF_4,E_4,q_4')=E_4^2-\cs(\sF_4,E_4,\bar q_3)=-\frac12,\\
&\cs(\sF_4,E_2,q_2)^{-1}+\cs(\sF_4,E_4,q_4)^{-1}=\cs(\sF_4,\ol D_0,q_2)+\cs(\sF_4,\ol D_0,q_4)=\ol D_0^2=-3.
\end{aligned}\right.
\end{equation*}
Hence the possible non-reduced singularities belong to $\{q_2,q_2',q_4,q_4'\}$,
and $\sF_4$ admits at most two non-reduced singularities.
Let $T_4$ be the set of all non-reduced singularities of $\sF_4$.
By \autoref{lem-7-5},
$$\vol(\sF)\geq K_{\sF_4}^2+\beta_{\bar q_1}(\sF_4)+\beta_{\bar q_3}(\sF_4)+\sum_{q\in T_4}\beta_{q}(\sF_4)=3+\sum_{q\in T_4}\beta_{q}(\sF_4).$$
Similarly,  by \autoref{cor-7-3}, one shows that the eigenvalue $\lambda_q \neq 1$ for any $q\in T_4$.
If the eigenvalue $\lambda_q\not \in \{1,2,1/2\}$ for any $q\in T_4$, then
$$\vol(\sF)\geq \frac83+\sum_{q\in T_3}\beta_{q}(\sF_3) \geq 3-\frac13-\frac13=\frac73.$$
If there is an eigenvalue $\lambda_{q}=2$ or $1/2$ for some $q\in T_3$, one can
prove \eqref{eqn-5-8-2} case-by-case similar as above.

{\noindent \bf Case (b.3).}
Suppose $k=3$.
In this case, we can assume without loss of generality that
$$Z(\sF_0,D_0,p_1)=2,\qquad Z(\sF_0,D_0,p_2)=Z(\sF_0,D_0,p_3)=1.$$
By \autoref{lem-7-7}, we can resolve the singularity $p_1$ as follows by \autoref{lem-7-7}.

$${\setlength{\unitlength}{6mm}
	\begin{tikzpicture}
	\draw[thick] (6,1) -- (6,3.2);
	\filldraw[black] (6,1.4) circle (1.5pt);
	\filldraw[black] (6,1.4) node[anchor=west]{$p_3$};
	\filldraw[black] (6,2.1) circle (1.5pt);
	\filldraw[black] (6,2.1) node[anchor=west]{$p_2$};
	\filldraw[black] (6,2.8) circle (1.5pt);
	\filldraw[black] (6,2.8) node[anchor=west]{$p_1$};
	
	\filldraw[black] (5.7,0.7) node[anchor=west]{$D_0$};

	\draw[->] (4.3,2) -- (5.2,2);
	\filldraw[black] (4.5,2.3) node[anchor=west]{\Large  $\sigma$};
	
	\draw[thick] (3,1) -- (3,3.2);
	\draw[thick] (1,2.8) -- (3.5,2.8);
	\draw[thick] (1.5,1.5) -- (1.5,3.2);
\filldraw[black] (3,1.4) circle (1.5pt);
\filldraw[black] (3,1.4) node[anchor=west]{$p_3$};
\filldraw[black] (3,2.1) circle (1.5pt);
\filldraw[black] (3,2.1) node[anchor=west]{$p_2$};
\filldraw[black] (3,2.8) circle (1.5pt);
\filldraw[black] (2.93,3) node[anchor=west]{$q_2$};
\filldraw[black] (2.25,2.8) circle (1.5pt);
\filldraw[black] (2,3.05) node[anchor=west]{$q_2'$};
\filldraw[black] (1.5,2.8) circle (1.5pt);
\filldraw[black] (1.43,2.55) node[anchor=west]{$\bar q_1$};

\filldraw[black] (2.7,0.7) node[anchor=west]{$\ol D_0$};
\filldraw[black] (0.3,2.8) node[anchor=west]{$E_2$};
\filldraw[black] (1.2,1.2) node[anchor=west]{$\ol E_1$};
\end{tikzpicture}}$$
Moreover, let $\sF_2$ be the induced foliation. Then
$$K_{\sF_2}=\sigma^*(K_{\sF_0})-E_2=\sigma^*(M_0)+\ol D_0+\ol E_1+E_2.$$
Hence the Zariski decomposition is $K_{\sF_2}=P_2+N_2$, where
$$P_2=\sigma^*(M_0)+\ol D_0+\frac{\ol E_1}{2}+E_2,\qquad
N_2=\frac{\ol E_1}{2}.$$
By \autoref{prop-2-2} with \autoref{lem-2-2},
\begin{equation*}
\left\{\begin{aligned}
&\cs(\sF_2,E_2,\bar q_1)^{-1}=\cs(\sF_2,\ol E_1,\bar q_1)=\ol E_1^2=-2,\\
&\cs(\sF_2,E_2,q_2)+\cs(\sF_2,E_2,q_2')=E_2^2-\cs(\sF_2,E_2,\bar q_1)=-\frac12,\\
&\cs(\sF_2,\ol D_0,q_2)+\cs(\sF_2,\ol D_0,p_2)+\cs(\sF_2,\ol D_0,p_3)=\ol D_0^2=-1.
\end{aligned}\right.
\end{equation*}
Hence the possible non-reduced singularities belong to $\{q_2,q_2',p_2,p_3\}$,
and $\sF_2$ admits at most three non-reduced singularities.
Let $T_2$ be the set of all non-reduced singularities of $\sF_2$.
By \autoref{lem-7-5},
$$\vol(\sF)\geq K_{\sF_2}^2+\beta_{\bar q_1}(\sF_2)+\sum_{q\in T_4}\beta_{q}(\sF_2)=\frac72+\sum_{q\in T_2}\beta_{q}(\sF_2).$$
Similar as above, one can prove \eqref{eqn-5-8-2} with the help of \autoref{cor-7-3}.

{\noindent \bf Case (b.4).}
Suppose $k=4$.
In this case,
$$Z(\sF_0,D_0,p_1)=Z(\sF_0,D_0,p_2)=Z(\sF_0,D_0,p_3)=Z(\sF_0,D_0,p_4)=1.$$
By the arguments at the beginning of Case (b), we may assume that the eigenvalue $\lambda_{p_i}\neq 1$ for any $1\leq i\leq 4$.
Let $T_0$ be the set of all non-reduced singularities of $\sF_0$.
By \autoref{lem-7-5},
$$\vol(\sF)\geq K_{\sF_0}^2+\sum_{q\in T_0}\beta_{q}(\sF_0)=4+\sum_{q\in T_0}\beta_{q}(\sF_0)\geq 4-4\,\cdot\,\frac12=2.$$
Thus we complete the proof of \eqref{eqn-5-8-2}.	
	\end{proof}
	
	\subsection{The canonical map induces a fibration}\label{sec-noe-one}
	We aim to prove \autoref{prop-5-1} in this subsection.
	Thus we will assume that $\dim \Sigma =1$, where $\Sigma$ is the image of the canonical map $\varphi$.
	Similar to the arguments at the beginning of \autoref{sec-noe-two},
	we may assume that $\sF$ is relatively minimal.
	By the Stein factorization, we obtain a diagram as follows.
	$$\xymatrix{&& \wt{S} \ar[dll]_-{f} \ar[d]^-{\phi} \ar[rrr]^-{\sigma}  &&& S \ar@{-->}[d]^-{\varphi=\varphi_{|K_{\sF}|}}\\
		B \ar[rr]^-{\pi} && Y \ar[rrr]^-{\rho}_-{\text{desingularization}} &&&\Sigma}$$
	Here $\pi:\, B \to Y$ is finite, and $f:\,\wt{S} \to B$ is a family of curves with connected fibers.
	By construction, the map 
	$$\rho\circ\phi:\,\wt{S} \lra \Sigma \hookrightarrow \bbp^N,$$
	is defined by the complete linear system $|\wt{M}|$,
	where $|\wt{M}|$ is obtained by blowing-up the base points of $|M|$,
	where $M$ is the moving part of $K_{\sF}$ as in \eqref{eqn-1-2}.
	Since $|\wt{M}|$ is base-point-free and induces a fibration $f:\,\wt{S} \to B$,
	it follows that
	$$p_g(\sF)=h^0(S,M)=h^0(\wt{S}, \wt{M})=h^0\big(B,f_*\mathcal{O}_{\wt{S}}(\wt M)\big).$$
	According to the Riemann-Roch theorem,
	$$\deg(L)\geq p_g(\sF)-1,$$
	where $L=f_*\mathcal{O}_{\wt{S}}(\wt M)$ is a line bundle on $B$.
	Note that $\wt{M}=f^*(L)$.
	Hence, we have numerically,
	\begin{equation}\label{eqn-5-9}
		\wt{M} \equiv_{num} \deg(L)\,F, \qquad \text{with~}\deg(L) \geq p_g(\sF)-1,
	\end{equation}
	where $F$ is a general fiber of $f$.
	\begin{lemma}\label{lem-5-1}
		Let $M$ be the moving part of $K_{\sF}$ as in \eqref{eqn-1-2}.
		If the linear system $|M|$ has a base point, then 
		$$\vol(\sF)\geq \big(p_g(\sF)-1\big)^2.$$
	\end{lemma}
	\begin{proof}
		Let $A=\sigma_*(F)$. If $|M|$ has a base point, then $A^2 \geq 1$.
		Hence
		\[\vol(\sF) \geq P^2 \geq M^2 \geq \big((p_g(\sF)-1)A\big)^2 \geq \big(p_g(\sF)-1\big)^2. \qedhere\]
	\end{proof}
	
	To prove \autoref{prop-5-1}, we assume from now on that $|M|$ is base-point-free.
	In other words,
	$\varphi:\,S \to \Sigma$ is already a morphism.
	Hence one gets a commutative diagram as follows.
	$$\xymatrix{ && S \ar[d]^-{\phi} \ar[lld]_-{f} \ar[rrrd]^-{\varphi=\varphi_{|K_{\sF}|}} &&& \\
		B \ar[rr]^-{\pi} &&Y \ar[rrr]^-{\rho}_-{\text{desingularization}} &&&\Sigma\, \ar@{^(->}[r] & \bbp^{p_g(\sF)-1}}$$

	\begin{proof}[{Proof of \autoref{prop-5-1}}]
		Since $\sF$ is of general type, i.e., $K_{\sF}$ is big,
		it follows that $K_{\sF}\cdot F \geq 1$.
		
		\vspace{2mm}
		(i). 
		As explained at the beginning of this section, we may assume that the foliation $\sF$ is relatively minimal.
		By \autoref{lem-5-1}, we may assume that $|M|$ is base-point-free.
		The relation in \eqref{eqn-5-9} rephrases as
		\begin{equation*} 
			M \equiv_{num} \deg(L)\,F, \qquad \text{with~}\deg(L) \geq p_g(\sF)-1.
		\end{equation*}
		The moving part $M$ is clearly nef,
		from which it follows that $M\leq P$, or equivalently $N\leq Z$,
		where $P$ is the nef part and $N$ is the negative part of $K_{\sF}$ in its Zariski decomposition, and $Z$ is the fixed part of $|K_{\sF}|$ as in \eqref{eqn-1-2}.
		Hence
		\begin{equation}\label{eqn-5-10}
			\vol(\sF)=P^2 \geq P\cdot M \geq \big(p_g(\sF)-1\big)P\cdot F=\big(p_g(\sF)-1\big)\,(K_{\sF}-N)\cdot F.
		\end{equation}
		Therefore, it suffices to prove a lower bound on $P\cdot F$, or equivalently an upper bound on $N\cdot F$.
		
		The fibration $f:\,S \to B$ defines a natural foliation $\mathcal{G}$ 
		by taking the saturation of the kernel $\ker(\df:\,T_{S} \to f^*T_{B})$ in $T_{S}$.
		The given foliation $\sF$ may be equal to or different from $\mathcal{G}$.
		
		Suppose that $\sF=\mathcal{G}$. Then the genus $g(F)\geq 2$ since $\sF$ is assumed to be of general type.
		Moreover, the canonical divisor is easy to compute:
		$$K_{\sF}=K_{\mathcal{G}}=K_S \otimes f^*K_B^{-1} \otimes \mathcal{O}_{S}\Big(\sum(1-a_i)C_i\Big),$$
		where the sum is taken over all components in fibers of $f$, and $a_i$ is the multiplicity of $C_i$ in its fiber.
		Note that the support of the negative part $N$ is a sum of $\sF$-chains (cf. \autoref{thm-3-4}), and thus contained in fibers of $f$.
		Combining this with \eqref{eqn-5-10},
		$$\begin{aligned}
			\vol(\sF) &\,\geq \big(p_g(\sF)-1\big)\,(K_{\sF}-N)\cdot F\\
			&\,=\big(p_g(\sF)-1\big)\,K_{\sF}\cdot F \\
			&\,=\big(2g(F)-2)\big)(p_g(\sF)-1) \geq 2(p_g(\sF)-1).
		\end{aligned}$$
		
		In the rest part of the proof, we will always assume that $\sF$ is different from the foliation $\mathcal{G}$ defined by taking the saturation of the kernel $\ker(\df:\,T_{S} \to f^*T_{B})$ in $T_{S}$.

		Let $Z=Z_h+Z_v$,
		where $Z$ is the fixed part of $|K_{\sF}|$ as in \eqref{eqn-1-2},
		and each component in $Z_v$ is contained in fibers of $f$, while each component in $Z_h$ maps surjectively to the base $B$.
		Similarly, we can decompose the negative part as $N=N_h+N_v$.
		Let
		$$\begin{aligned}
			&Z_h=\sum a_CC,&\qquad& Z_v=\sum a_DD,\\
			&N_h=\sum b_CC,&\qquad& N_v=\sum b_DD.
		\end{aligned}$$
		Then the coefficients $\{a_C,a_D\}$'s are positive integers; while $\{b_C,b_D\}$'s belong to $[0,1)$ by \autoref{thm-3-4}, since the foliation $\sF$ is assumed to be relatively minimal.
		
		If there exists an irreducible component $C\subseteq Z_h$
		with $b_{C}=0$ (i.e., $C$ is NOT contained in the support of $N$),
		then $P\cdot F\geq a_{C} \geq 1$, from which with \eqref{eqn-5-10} it follows that
		\begin{equation}\label{eqn-7-2}
			\vol(\sF)\geq p_g(\sF)-1.
		\end{equation}
		Hence we may assume that the supports of $Z_h$ and $N_h$ are the same.
		Let $C\subseteq \text{Supp}(Z_h)=\text{Supp}(N_h)$ be an irreducible component.
		If $a_{C}\geq 2$, then
		$P\cdot F\geq a_{C}-b_{C} > 1$ since $b_{C}<1$.
		Hence one shows by \eqref{eqn-5-10} that
		\begin{equation}\label{eqn-7-3}
\vol(\sF) >p_g(\sF)-1.
\end{equation} 
		Therefore, we may assume that $a_{C}=1$.
		Let $C_1+\cdots+C_r$ be the $\sF$-chain containing $C$.
		Consider first the case when $r=1$. Then $C=C_1$ and
		$$-1=K_{\sF}\cdot C\geq \Big(C+\big(p_g(\sF)-1\big)F\Big)\cdot C.$$
		Hence $C^2\leq -p_g(\sF)$.
		As $r=1$, one sees easily that $b_{C}=\frac{1}{-C^2}$ by \autoref{lem-coefficient-N}.
		Hence by \eqref{eqn-5-10}
				\begin{equation}\label{eqn-7-4}
		\vol(\sF)\geq \big(p_g(\sF)-1\big)(1-b_{C})C\cdot F \geq p_g(\sF)-2+\frac{1}{p_g(\sF)}.
		\end{equation}
		
		Consider next the case when $r\geq 2$ and $C=C_1$.
		Then
		$$-1=K_{\sF}\cdot C_1\geq \Big(a_{C_1}C_1+a_{C_2}C_2+\big(p_g(\sF)-1\big)F\Big)\cdot C_1\geq a_{C_1}C_1^2+p_g(\sF)=C_1^2+p_g(\sF),$$
		since $a_{C_1}=a_C=1$ by assumption, and $a_{C_2}\geq 1$.
		Hence $C^2=C_1^2\leq -(p_g(\sF)+1)$.
		By \autoref{prop-3-1}, $b_{C}<\frac{1}{-C^2-1}\leq \frac{1}{p_g(\sF)}$.
		Hence by \eqref{eqn-5-10}
				\begin{equation}\label{eqn-7-5}
		\vol(\sF)\geq \big(p_g(\sF)-1\big)(1-b_{C})C\cdot F > p_g(\sF)-2+\frac{1}{p_g(\sF)}.
		\end{equation}
		
		Consider thirdly the case when $r\geq 2$ and $C=C_i$ with $1<i<r$.
		Then
		$$0=K_{\sF}\cdot C_i\geq \Big(a_iC_i+a_{C_{i-1}}C_{i-1}+a_{C_{i+1}}C_{i+1}+\big(p_g(\sF)-1\big)F\Big)\cdot C_i\geq a_iC_i^2+p_g(\sF)+1=C_i^2+p_g(\sF)+1,$$
		since $a_{C_i}=a_C=1$ by assumption, $a_{i-1}\geq 1$ and $a_{i+1}\geq 1$. 
		Hence $C^2=C_i^2\leq -(p_g(\sF)+1)$.
		By \autoref{prop-3-1}, $b_{C}<\frac{1}{-2C^2-3}\leq \frac{1}{2p_g(\sF)-1}$.
		Hence by \eqref{eqn-5-10}
				\begin{equation}\label{eqn-7-8}
				\begin{aligned}
				\vol(\sF)&\,\geq \big(p_g(\sF)-1\big)(1-b_{C})C\cdot F\\
				&\, > p_g(\sF)-1-\frac{p_g(\sF)-1}{2p_g(\sF)-1}=p_g(\sF)-\frac32+\frac{1}{2\big(2p_g(\sF)-1\big)}.
				\end{aligned}
		\end{equation}
		
		Finally, we consider the case when $r\geq 2$ and $C=C_r$.
		Then
		$$0=K_{\sF}\cdot C\geq \Big(C+a_{C_{r-1}}C_{r-1}+\big(p_g(\sF)-1\big)F\Big)\cdot C\geq C^2+p_g(\sF).$$
		Hence $C^2=C_r^2\leq -p_g(\sF)$.
		In this case, according to \autoref{lem-coefficient-N},
		$b_{C}=b_{r}=\frac1n$, where by the proof of \autoref{prop-3-1},
		$$n\geq 2\xi_{r-1}-\xi_r =-2C_r^2\xi_r-\xi_r=-(2C_r^2+1)\geq 2p_g(\sF)-1.$$
		Hence by \eqref{eqn-5-10}
				\begin{equation}\label{eqn-7-7}
		\vol(\sF)\geq \big(p_g(\sF)-1\big)(1-b_{C})C\cdot F \geq p_g(\sF)-\frac32+\frac{1}{2\big(2p_g(\sF)-1\big)}.
		\end{equation}
		
		\vspace{2mm}
		(ii).
		We use the same notations above.
		Similar as in (i), one may assume that
		$\sF$ is relatively minimal and different from the foliation $\mathcal{G}$ defined by taking the saturation of the kernel $\ker(\df:\,T_{S} \to f^*T_{B})$ in $T_{S}$.
		Furthermore, according to the proof of (i), we may also assume that
		$Z_h=C$ consists of only one component such that $C$ is contained in the support of $N_h$ and $C\cdot F=1$.
		In fact, according to the above proof, we may first assume that the coefficient $C$ is contained in the support of $N_h$ and $a_{C}=1$ for any possible component $C\subseteq Z_h$; otherwise $\vol(\sF)\geq p_g(\sF)-1$
		by \eqref{eqn-7-2} and \eqref{eqn-7-3}.
		Moreover, if $Z_h$ contains one component $C$ with $C\cdot F\geq 2$, or contains at least two components,
		then by a similar argument as above, one shows that
		$$\vol(\sF)\geq 2\Big(p_g(\sF)-2+\frac{1}{p_g(\sF)}\Big).$$
		Thus one may write
		\begin{equation}\label{eqn-7-9}
			K_{\sF}=M+C+Z_v,
		\end{equation}
		where $C$ is a section (i.e., $C\cdot F=1$) of $f$ contained in the support of $N_h$,
		$Z_v$ is contained in fibers of $f$,
		and $M=\big(p_g(\sF)-1\big)F$ with $F$ being a general fiber of $f$.
		All the statements are clear except the last one, which we simply argue as follows.
		Since $C$ is contained in the support of $N_h$, it is a rational curve,
		which implies that the base curve $B\cong \bbp^1$.		
		Note that $h^0(S,M)=h^0(B,L)=p_g(\sF)$.
		By the Riemann-Roch theorem, $\deg(L)= p_g(\sF)-1$, and hence
		$M=\big(p_g(\sF)-1\big)F$ as required.
		Let $C_1+\cdots+C_r$ be the $\sF$-chain containing $C$.
		Then by \eqref{eqn-7-8} and \eqref{eqn-7-7},
		we may assume that $C=C_1$ is the first component of the $\sF$-chain.
		In other words,
		\begin{equation}\label{eqn-7-51}
			\text{there is exactly one singularity of $\sF$ on $C$.}
		\end{equation} 
		
		According to \cite[Corollay\,4.9]{luxin-24}, we may assume that the general fiber of $f$ is a $\bbp^1$.
	    In other words, $f$ is a ruled surface over $B\cong \bbp^1$.
	    Let $\psi:\,S\to S_0$ be a contraction of the vertical exceptional curves to a $\bbp^1$-bundle such that $C_0^2=C^2$;
	    namely we contract exceptional curve disjoint with $C$,
	    where $C_0=\psi(C)$ is the image of $C$.
	    It implies that $\psi$ induces an isomorphism around neighborhoods of $C$ and $C_0$.
	    Hence there is exactly one singularity of $\sF_0=\psi_*(\sF)$ on $C_0$ by \eqref{eqn-7-51}.
	    $$\xymatrix{
	    S \ar[rr]^-{\psi} \ar[dr]_-{f} && S_0 \ar[dl]^-{f_0}\\
        &B\cong \bbp^1&}$$
	    By \eqref{eqn-7-9},
	    \begin{equation}\label{eqn-7-10}
	    	K_{\sF_0}=\psi_*K_{\sF}=M_0+C_0+Z_{v0},
	    \end{equation}
	    where $M_0=(p_g(\sF)-1)F_0$ with $F_0$ being a general fiber of $f_0$,
	    and $Z_{v0}=\psi_*Z_v$ is contained in fibers of $f_0$.
	    We claim first that $Z_{v0}\neq \emptyset$;
	    
	    Since $f_0$ is a $\bbp^1$-bundle, we may write
	    \begin{equation}\label{eqn-7-11}
	    	Z_{v0}=\sum_{i=1}^{k}t_iF_{i0},
	    \end{equation}
	    where $t_i\geq 1$ and $F_{i0}$'s are pairwise different fibers of $f_0$.
	    By construction, there is no non-reduced singularity of $\sF_0$ on $C_0$,
	    i.e., $\psi$ consists of no blowing-up center at a point on $C_0$.
	    Hence $Z_{v0}\neq 0$ according to \autoref{lem-7-9};
	    indeed, if $Z_{v0}= 0$, then $\sF_0$ has no non-reduced singularity by \autoref{lem-7-9}, and hence $\psi:\,S \cong S_0$ since $\sF$ is assumed to be relatively minimal.
	    However, this is impossible, since $S \cong S_0$ is a Hirzebruch surface, which
	    does not admit any fibration of genus $g\geq 2$; see also \autoref{lem-7-10}.
	    Conversely, 
	    \begin{claim}\label{claim-7-2}
	    	For any fiber $F_{i0}\subseteq Z_{v0}$, there exists at least one non-reduced singularity on $F_{i0}\setminus C_0$.
	    \end{claim}
	    \begin{proof}   
	    Suppose that there is no non-reduced singularity on $F_{i0}\setminus C_0$.
	    Then all the possible singularities of $\sF_0$ on $F_{i0}$ is reduced.
	    This implies that $\psi$ is an isomorphism in a neighborhood of $F_{i0}$,
	    since $\sF$ is relatively minimal.
	    Hence $\psi^{-1}(F_{i0}) \cong F_{i0}$ is a fiber of $f$, which is movable.
	    On the other hand, it is also contained in the fixed part $Z_v$ of $K_{\sF}$, because the moving part of $K_{\sF}$ is $M=\psi^*(M_0)$.
	    This gives a contradiction.
	    \end{proof}
        
        Based on the resolution of non-reduced singularities in \autoref{sec-non-reduced},
        we can resolve the possible singularities of $\sF_0$ on $Z_{v0}$ in the following two lemmas.
        \begin{lemma}
        	Let $F_{i0}$ be a fiber of $f_0$ contained in $Z_{v0}$ as in \eqref{eqn-7-11}, $p_0=C_0\cap F_{i0}$, and $e=-C_0^2$.
        	Suppose that $F_{i0}$ is $\sF_0$-invariant.
        	Then there exist a sequence blowing-ups centered over $F_{i0}$
        	$$\pi=\sigma_1\circ\cdots\circ\sigma_n:~S_n \overset{\sigma_n}{\lra} S_{n-1}
        	\overset{\sigma_{n-1}}{\lra} \cdots \overset{\sigma_2}{\lra} S_{1}\overset{\sigma_1}{\lra} S_{0},$$
        	such that the singularities of $\sF_{n}$ on the fiber $\pi^{-1}(F_{i0})$ are all non-degenerate and belong to one of the following (We denote by $T_n$ the set of singularities of $\sF_n$ over $\pi^{-1}(F_{i0})$ other than $\pi^{-1}(p_0)$, which are either non-reduced or contained in $\text{Supp}(N_n)$, where $K_{\sF_n}=P_n+N_n$ is the Zariski decomposition of $K_{\sF_n}$)).
        	\begin{enumerate}[$(i)$]
        		\item There are two non-degenerate singularities $p_1$ and $p_2$
        		(other than $p_0$) on $F_{i0}$, i.e., there is no need to blow up.
        		In this case,
        		\begin{equation}\label{eqn-7-18}
        			\cs(\sF_0,F_{i0},p_1)+\cs(\sF_0,F_{i0},p_2)=\frac{1}{e}.
        		\end{equation}
        		Moreover, $CS(\sF_0,F_{i0},p_1)\neq 1$ and $CS(\sF_0,F_{i0},p_2)\neq 1$.
        		Hence
        		\begin{equation}\label{eqn-7-20}
        			\sum_{p_j\in T_0}\beta_{p_j}(\sF_0)\geq -\frac12.
        		\end{equation}
        		
        	
        	\item There is a unique singularity $p$ of $\sF_0$ other than $p_0$ on $F_{i0}$,
        	and the inverse image $\pi^{-1}(p)$ is two rational $\sF_n$-invariant curves as in \autoref{lem-7-3}\,(iii).
        	In particular, $n=2$, and using the same notations in \autoref{lem-7-3},
        	it holds that
        	\begin{equation}\label{eqn-7-19}
        	\left\{\begin{aligned}
        	\cs(\sF_2,E_2,q_2)&\,=-\frac12,\\
        	\cs(\sF_2,E_2,q_2')&\,=-\frac{e}{2e-1},\\
        	\cs(\sF_2,E_2,q_2'')&\,=\frac{1}{2(2e-1)}.
        	\end{aligned}\right.
        	\end{equation}
        	Hence
        	\begin{equation}\label{eqn-7-21}
        	\begin{aligned}
        	K_{\sF_2}^2+\sum_{q\in T_2}\beta_{q}(\sF_2)
        	&\,=K_{\sF_2}^2+\beta_{q_2}(\sF_2)+\beta_{q_2'}(\sF_2)+\beta_{q_2''}(\sF_2)\\
        	&\,=K_{\sF_0}^2-1+\frac{e^2-e+1}{e(2e-1)}.
        	\end{aligned}
        	\end{equation}
    \end{enumerate}
In any case, it holds
\begin{equation}\label{eqn-7-29}
	K_{\sF_n}^2+\sum_{q\in T_n}\beta_{q}(\sF_n) \geq K_{\sF_0}^2-1+\frac{e^2-e+1}{e(2e-1)}.
\end{equation}
        \end{lemma}
    \begin{proof}
    	Note that $K_{\sF_0}\cdot F_{i0}=1$.
    	Hence $Z(\sF_0,F_{i0})=K_{\sF_0}F_{i0}+\chi(F_{i0})=3$.
    	Moreover, since $p_0$ is a reduced singularity by assumption.
    	It follows that either there are two other singularities $\{p_1,p_2\}$ of $\sF_0$ on $F_{i0}$
    	with
    	$$Z(\sF_0,F_{i0},p_1)=Z(\sF_0,F_{i0},p_2)=1,$$
    	or there is only one other singularity $p$ with $Z(\sF_0,F_{i0},p)=2$.
    	
    	If it is the first case, the formula \eqref{eqn-7-18} follows from \eqref{eqn-2-7} and \eqref{eqn-2-6}.
    	Indeed, $\cs(\sF_0,C_0,p_0)=-e$ by \eqref{eqn-2-7},
    	where $e=-C_0^2\geq 2$.
    	Hence $\cs(\sF_0,F_{i0},p_0)=-\frac{1}{e}$ by \eqref{eqn-2-6}.
    	Using again \eqref{eqn-2-7}, one obtains \eqref{eqn-7-18}.
    	Finally, suppose for instance that $\cs(\sF_0,F_{i0},p_1)=1$.
    	Then $\cs(\sF_0,F_{i0},p_2)=\frac{1-e}{e}<0$, which implies that $p_2$ is a reduced singularity in view of \eqref{eqn-2-6}.
    	Let $\rho':\, S' \to S_0$ be the blowing-up centered at $p_1$,
    	and $\ol F_{i0}$ be its strict transform.
    	Clearly the contraction $\psi$ factorizes through $\rho'$.
    	$$\psi:~S \overset{\rho}{\lra} S_1 \overset{\rho'}{\lra} S_0.$$
    	Moreover, there are exactly two singularities of $(\rho')^*\sF_0$ on $\ol F_{i0}$,
    	which are the inverse image of $p_0$ and $p_2$.
    	Both singularities are reduced.
    	It follows that $\rho$ is an isomorphism in a neighborhood of $\ol F_{i0}$.
    	We denote its inverse image in $S$ still by $\ol F_{i0}$ by abuse of notations.
    	According to the above arguments,
    	$\ol F_{i0}$ is smooth rational curve with $\ol F_{i0}^2=-1$ and two reduced non-degenerate singularities.
    	Hence $\ol F_{i0}$ is an $\sF$-exceptional curve, cf. \cite[\S\,5.1]{bru-04}.
    	This contradicts the relative minimality of $\sF$.
    	Note that the CS-indices of an algebraically integral foliation are just the eigenvalues of the corresponding singularities by \autoref{lem-2-2} together with \autoref{lem-7-2}.
    	Hence \eqref{eqn-7-20} follows from the definition of $\beta$-invariant \eqref{eqn-def-beta} and \eqref{eqn-7-18} with restrictions that $\cs(\sF_0,F_{i0},p_1)\neq 1$ and $\cs(\sF_0,F_{i0},p_2)\neq 1$.

    	We consider next the second case, i.e., there is only one other singularity $p$ with $Z(\sF_0,F_{i0},p)=2$.
    	The resolution of such a singularity $p$ was presented in \autoref{cor-7-1},
    	where $C=F_{i0}$ in our case.
    	Except the case stated in (ii) of our lemma,
    	in all other cases, $\ol F_{i0}^2=F_{i0}^2-1=-1$ and there are exactly two singularities of $\sF_n$ on $\ol F_{i0}$ which are reduced non-degenerate.
    	By a similar argument as above, one shows that this is impossible since $\sF$ is relatively minimal.
    	This proves that the inverse image $\pi^{-1}(p)$ is two rational $\sF_n$-invariant curves as in \autoref{lem-7-3}\,(iii) (equivalently the case when $n=2$ in \autoref{cor-7-1}(ii)).
    	In particular, $n=2$.
    	It remains to prove \eqref{eqn-7-19} and \eqref{eqn-7-21}.
    	By \eqref{eqn-2-7},
    	$$\cs(\sF_2,\ol E_1,q_2)=\ol E_1^2=-2.$$
    	Hence $\cs(\sF_2,E_2,q_2)=-\frac12$ by \autoref{lem-2-2}.
    	Again by \eqref{eqn-2-7},
    	$$\cs(\sF_2,C_0,p_0)=\cs(\sF_0,C_0,p_0)=C_0^2=-e.$$
    	Hence $\cs(\sF_2,\ol F_{i0},p_0)=-\frac{1}{e}$ by \autoref{lem-2-2}.
    	Thus
    	$$\cs(\sF_2,\ol F_{i0},q_2')=\ol F_{i0}^2-\cs(\sF_2,\ol F_{i0},p_0)=-\frac{2e-1}{e}.$$
    	So $\cs(\sF_2, E_2,q_2')=-\frac{e}{2e-1}$ by \autoref{lem-2-2}.
    	And $\cs(\sF_2,E_2,q_2'')=\frac{1}{2(2e-1)}$ again by \eqref{eqn-2-7}.
    	Clearly both $q_2$ and $q_2'$ lies on the support of the negative part, and $q_2''$ is non-reduced.
    	Hence $T_2=\{q_2,q_2',q_2''\}$.
    	Finally, \eqref{eqn-7-21} follows by a similar argument as proving \eqref{eqn-7-20} with the help of \eqref{eqn-7-17}.
    \end{proof}
    
    \begin{lemma}
    	Let $F_{i0}$ be a fiber of $f_0$ contained in $Z_{v0}$ as in \eqref{eqn-7-11} and $e=-C_0^2$.
    	Suppose that $F_{i0}$ is not $\sF_0$-invariant.
    	Then there exist a sequence blowing-ups center over $F_{i0}$
    	$$\pi=\sigma_1\circ\cdots\circ\sigma_n:~S_n \overset{\sigma_n}{\lra} S_{n-1}
    	\overset{\sigma_{n-1}}{\lra} \cdots \overset{\sigma_2}{\lra} S_{1}\overset{\sigma_1}{\lra} S_{0},$$
    	such that the singularities of $\sF_{n}$ on $\pi^{-1}(F_{i0})$ are all non-degenerate and belong to one of the following
    	($\ol F_{i0}$ is the strict transform of $F_{i0}$, $\ol E_i$ is the strict transform of the $i$-th blowing-up ($E_n=\ol E_n$ by abuse of notation), all the non-degenerate singularities of $\sF_n$ on $\pi^{-1}(F_{i0})$ are marked with dark dots in the figures of $\pi^{-1}(F_{i0})$, and $T_n$ denotes the set of singularities of $\sF_n$ over $\pi^{-1}(F_{i0})$ which are either non-reduced or contained in $\text{Supp}(N_n)$, where $K_{\sF_n}=P_n+N_n$ is the Zariski decomposition of $K_{\sF_n}$).
    	\begin{enumerate}[$(i)$]
    		\item There is exactly one singularity $q_0$ of $\sF_0$ on $F_{i0}$, which is already non-degenerate, i.e., there is no need to blow up.
    		In this case, $\beta_{q_0}(\sF_0)\geq -\frac12$.
    		
    		\item $\pi^{-1}(F_{i0})=\ol F_{i0}+\sum\limits_{j=1}^n\ol E_j$, $n\geq 2$, and
    		\begin{equation}\label{eqn-7-23}
    		\left\{\begin{aligned}
    		&K_{\sF_n}^2=K_{\sF_0}^2-(n-1),\\
    		&\sum_{q_j\in T_n}\beta_{q_j}(\sF_n) > (n-1)-t_i+\frac{t_i-1}{t_i}.
    		\end{aligned}\right.
    		\end{equation}
    		$${\setlength{\unitlength}{6mm}
    			\begin{tikzpicture}
    			\draw[very thick, dashed] (10.5,1) -- (10.5,3);
    			\draw[thick] (11.5,1) -- (8.5,3);
    			\draw[thick] (10,3) -- (7,1);
    			\filldraw[black] (9.25,2.5) circle (1.5pt);
    			
    			\filldraw[black] (8.25,1.82) circle (1.5pt);
    			\draw[thick] (9.5,1) -- (6.5,3);
    			\draw[thick] (8,3) -- (6.5,2);
    			\filldraw[black] (7.25,2.5) circle (1.5pt);

    			\filldraw[black] (9,2.2) node[anchor=west]{$q_1$};
    			\filldraw[black] (8,1.5) node[anchor=west]{$q_2$};
    			\filldraw[black] (7,2.2) node[anchor=west]{$q_3$};
    			\filldraw[black] (10.1,0.7) node[anchor=west]{$\ol F_{i0}$};
    			\filldraw[black] (11.4,0.7) node[anchor=west]{$\ol E_1$};
    			\filldraw[black] (9,0.7) node[anchor=west]{$\ol E_3$};
    			\filldraw[black] (9.5,3.3) node[anchor=west]{$\ol E_2$};
    			\filldraw[black] (7.5,3.3) node[anchor=west]{$\ol E_4$};
    			
    			\filldraw[black] (6.2,2) circle (1pt);
    			\filldraw[black] (6,2) circle (1pt);
    			\filldraw[black] (5.8,2) circle (1pt);
    			\filldraw[black] (5.6,2) circle (1pt);
    			\filldraw[black] (5.4,2) circle (1pt);
    			\filldraw[black] (5.2,2) circle (1pt);
    			
    			\filldraw[black] (4.25,1.82) circle (1.5pt);
    			\draw[thick] (5.5,1) -- (2.5,3);
    			\draw[thick] (4,3) -- (1,1);
    			\filldraw[black] (3.25,2.51) circle (1.5pt);
    			
    			\filldraw[black] (2.5,2) circle (1.5pt);
    			\filldraw[black] (1.5,1.34) circle (1.5pt);
    			
    			\filldraw[black] (5,0.7) node[anchor=west]{$\ol E_{n-1}$};
    			\filldraw[black] (3.5,3.3) node[anchor=west]{$E_n=\ol E_n$};
    			
    			\filldraw[black] (3.4,2.5) node[anchor=west]{$q_{n-1}$};
    			\filldraw[black] (3.6,1.5) node[anchor=west]{$q_{n-2}$};
    			\filldraw[black] (2,2.2) node[anchor=west]{$q_n$};
    			\filldraw[black] (0.5,1.5) node[anchor=west]{$q_{n+1}$};
    			\end{tikzpicture}}$$
    		
    	\item $\pi^{-1}(F_{i0})=\ol F_{i0}+\sum\limits_{j=1}^n\ol E_i$, $n\geq 3$, and
    	\begin{equation}\label{eqn-7-24}
    	\left\{\begin{aligned}
    	&K_{\sF_n}^2=K_{\sF_0}^2-(n-2),\\
    	&\sum_{q_j\in T_n}\beta_{q_j}(\sF_n) =\left\{\begin{aligned}
    	&\frac{1}{2}+\frac{1}{3}=\frac56, &&\text{if~}n=3,\\[2pt]
    	&\frac12+\frac{1}{(2n-1)(n-1)}+\frac{n-3}{n-2}, &&\text{if~}n\geq 4,
    	\end{aligned}\right.\\
    	&t_i\geq n-2.
    	\end{aligned}\right.
    	\end{equation}
    	$${\setlength{\unitlength}{6mm}
    		\begin{tikzpicture}
    		\draw[very thick, dashed] (10.5,1) -- (10.5,3);
    		\draw[thick] (11.5,1) -- (8.5,3);
    		\draw[thick] (10,3) -- (7,1);
    		\filldraw[black] (9.25,2.5) circle (1.5pt);
    		
    		\filldraw[black] (8.25,1.82) circle (1.5pt);
    		\draw[thick] (9.5,1) -- (6.5,3);
    		\draw[thick] (8,3) -- (6.5,2);
    		\filldraw[black] (7.25,2.5) circle (1.5pt);

    		\filldraw[black] (9,2.2) node[anchor=west]{$q_1$};
    		\filldraw[black] (8,1.5) node[anchor=west]{$q_2$};
    		\filldraw[black] (7,2.2) node[anchor=west]{$q_3$};
    		\filldraw[black] (10.1,0.7) node[anchor=west]{$\ol F_{i0}$};
    		\filldraw[black] (11.4,0.7) node[anchor=west]{$\ol E_1$};
    		\filldraw[black] (9,0.7) node[anchor=west]{$\ol E_3$};
    		\filldraw[black] (9.5,3.3) node[anchor=west]{$\ol E_2$};
    		\filldraw[black] (7.5,3.3) node[anchor=west]{$\ol E_4$};
    		
    		\filldraw[black] (6.2,2) circle (1pt);
    		\filldraw[black] (6,2) circle (1pt);
    		\filldraw[black] (5.8,2) circle (1pt);
    		\filldraw[black] (5.6,2) circle (1pt);
    		\filldraw[black] (5.4,2) circle (1pt);
    		\filldraw[black] (5.2,2) circle (1pt);
    		
    		\filldraw[black] (4.25,1.82) circle (1.5pt);
    		\draw[thick] (5.5,1) -- (2.5,3);
    		\draw[thick] (4,3) -- (1,1);
    		\filldraw[black] (3.25,2.51) circle (1.5pt);
    		
    		\draw[thick] (2,1) -- (-1,3);
    		
    		\filldraw[black] (2.5,2) circle (1.5pt);
    		\filldraw[black] (1.5,1.34) circle (1.5pt);
    		
    		\filldraw[black] (5,0.7) node[anchor=west]{$\ol E_{n-2}$};
    		\filldraw[black] (3.5,3.3) node[anchor=west]{$E_n=\ol E_n$};
    		\filldraw[black] (-1.4,3.3) node[anchor=west]{$\ol E_{n-1}$};
    		
    		\filldraw[black] (3.4,2.5) node[anchor=west]{$q_{n-2}$};
    		\filldraw[black] (3.6,1.5) node[anchor=west]{$q_{n-3}$};   		
    		\filldraw[black] (1.9,2.2) node[anchor=west]{$q_{n}$};
    		\filldraw[black] (0.4,1.3) node[anchor=west]{$q_{n-1}$};
    \end{tikzpicture}}$$	
    	\end{enumerate}
    In any case, it holds
    \begin{equation}\label{eqn-7-27}
    	K_{\sF_n}^2+\sum_{q_j\in T_n}\beta_{q_j}(\sF_n)\geq K_{\sF_0}^2-t_i+\max\left\{\frac12,\frac{t_i-1}{t_i}\right\}.
    \end{equation}
    \end{lemma}
    \begin{proof}
    Since $F_{i0}$ is not $\sF_0$-invariant,
    it follows that
    $$tang(\sF_0,F_{i0})=K_{\sF_0}F_{i0}+F_{i0}^2=1.$$
    Hence there is a unique singularity $q_0$ of $\sF_0$ on $F_{i0}$, which must be non-reduced by \autoref{claim-7-2}.
    Let $\sigma_1:\,S_1 \to S_0$ be the blowing-up centered at $q_0$.
    We claim first that the exceptional curve $E_1$ is $\sF_1=\sigma_1^*\sF_0$-invariant;
    indeed, if it is not the case, then according to the proof of \autoref{lem-7-4},
    $$tang(\sF_1,\ol F_{i0})=K_{\sF_1}\cdot\ol F_{i0}+ \ol F_{i0}^2=(\sigma_1^*K_{\sF_0}-E_1)(\sigma_1^*F_{i0}-E_1)+(\sigma_1^*F_{i0}-E_1)^2=-1,$$
    where we denote also by $\ol F_{i0}$ the strict transform of $F_{i0}$ in $S_1$ by abuse of notation.
    This gives a contradiction.
    Hence $E_1$ is $\sF_1$-invariant. Moreover, it follows from \autoref{lem-7-4} that $Z(\sF_1,E_1)=2$ and that
    the intersection $q=\ol C\cap E_1$ is not a singularity of $\sF_1$.
    
    Suppose that there are two points $q_1,q_2$ on $E_1$ with
    $$Z(\sF_1,E_1,q_1)=Z(\sF_1,E_1,q_2)=1.$$
    Then we are in case (i) by \autoref{lem-7-??}.
    It follows that the singularity $q_0$ is already a non-degenerate singularity of $\sF_0$. This is case (i) in our lemma.
    Moreover, in this case, $\lambda_{q_0}(\sF_0)\neq 1$;
    otherwise, the exceptional curve $E_1$ would be not $\sF_1$-invariant, which is impossible as argued above, where $\sigma_1:\,S_1 \to S_0$ is the blowing-up centered at $q_0$ with $E_1$ being the exceptional curve and $\sF_1=\sigma_1^*\sF_0$.
    Hence $\beta_{q_0}(\sF_0)\geq -\frac12$ by the definition of $\beta$-invariant in \eqref{eqn-def-beta-2}.
    
    
    Suppose next that there is exactly one point $p_1$ on $E_1$ with $Z(\sF_1,E_1,p_1)=2$.
    The resolution of such a singularity $p_1$ is exhibited in \autoref{cor-7-1}.
    We only remark that $E_1$ is the curve $C$ in \autoref{cor-7-1}.
    The cases (i) and (ii) in \autoref{cor-7-1} are respectively the cases (ii) and (iii) in this lemma.
    It remains to prove \eqref{eqn-7-23} and \eqref{eqn-7-24}, and to prove that the case (iii) in \autoref{cor-7-1} can not happen.
    
    \vspace{2mm}
    We prove first \eqref{eqn-7-23}.
    According to the resolution exhibited in \autoref{cor-7-1},
    $$K_{\sF_n}=\pi^*K_{\sF_0}-\sum_{j=2}^{n}\mathcal{E}_j=\pi^*K_{\sF_0}-\sum_{j=2}^{n}(j-1)\ol E_j,\qquad \pi^*(F_{i0})=\mathcal{E}_1+\ol F_{i0}=\sum_{j=1}^{n}\ol E_j+\ol F_{i0},$$
    where $\mathcal{E}_{j}=\sum\limits_{s=j}^{n}\ol E_s$ is the total inverse image of the exceptional curve of $\sigma_j$ as in \autoref{cor-7-1}.
    Remark again that the indices here different from that in \autoref{cor-7-1},
    because $E_1$ here is the curve $C$ in \autoref{cor-7-1}.
    Thus the first equation in \eqref{eqn-7-23} follows immediately.
    Moreover, $D:=\bigcup\limits_{j=1}^{n-1}\ol E_j$ is contained in the support of the negative part $N_n$ since $D$ is actually a Hirzebruch-Jung chain with $K_{\sF_n} \cdot \ol E_1=-1$ and $K_{\sF_n}\cdot \ol E_j=0$ for $2\leq j \leq n-1$.
    In particular, $q_j\in T_n$ for $1\leq j \leq n-1$.
    Note also that the contraction $\psi$ factorizes through $\pi$ as
    $$\psi:\,S \overset{\rho}{\lra} S_n \overset{\pi}{\lra} S_0.$$
    Hence
    \begin{eqnarray}
    &&\rho_*(M)+\rho_*(C)+\rho_*(Z_v)=\rho_*(K_{\sF})\nonumber\\
    &=&K_{\sF_n}=\pi^*K_{\sF_0}-\sum_{j=2}^{n}(j-1)\ol E_j\nonumber\\
    &=&\pi^*(M_0)+\pi^*(C_0)+\pi^*\Big(\sum_{i'\neq i}t_{i'}F_{i'0}\Big)+\pi^*(t_iF_{i0})-\sum_{j=2}^{n}(j-1)\ol E_j\nonumber\\
    &=&\pi^*(M_0)+\pi^*(C_0)+\pi^*\Big(\sum_{i'\neq i}t_{i'}F_{i'0}\Big)
    +t_i\ol F_{i0}+\sum_{j=1}^{n}(t_i-j+1)\ol E_j. \label{eqn-7-25}
    \end{eqnarray}
    In particular, $t_i-j+1\geq 0$ for any $1\leq j \leq n$, i.e., $t_i\geq n-1$.
    If both $q_n$ and $q_{n+1}$ are reduced, then
    $$\sum_{q_j\in T_n}\beta_{q_j}(\sF_n)=\sum_{j=1}^{n-1}\beta_{q_j}(\sF_n)=\frac{n-1}{n}> (n-1)-t_i+\frac{t_i-1}{t_i}.$$
    We explain a little about the equality $\sum\limits_{j=1}^{n-1}\beta_{q_j}(\sF_n)=\frac{n-1}{n}$ above.
    Since $\bigcup\limits_{j=1}^{n-1} \ol E_j$ is a chain of rational curves with $\ol E_j^2=-2$ for any $1\leq j \leq n-1$.
    Thus by \eqref{eqn-2-7} and \eqref{eqn-2-6},
    $\cs(\sF_n, \ol E_2, q_1)^{-1}=\cs(\sF_n, \ol E_1, q_1)=\ol E_1^2=-2$,
    and hence
    $$\cs(\sF_n, \ol E_3, q_2)^{-1}=\cs(\sF_n, \ol E_2, q_2)=\ol E_2^2-\cs(\sF_n, \ol E_2, q_1)=-\frac{3}{2}.$$
    By induction, one proves that for any $2\leq j \leq n-1$,
    $$\cs(\sF_n, \ol E_{j+1}, q_{j})^{-1}=\cs(\sF_n, \ol E_j, q_{j})=\ol E_j^2-\cs(\sF_n, \ol E_j, q_{j-1})=-\frac{j+1}{j}.$$
    Hence
    \begin{equation}\label{eqn-7-26}
    	\sum_{j=1}^{n-1}\beta_{q_j}(\sF_n)=\sum_{j=1}^{n-1}\frac{1}{j(j+1)}=\frac{n-1}{n}.
    \end{equation}
    We assume next that at least one of $\{q_n,q_{n+1}\}$ is a non-reduced.
    Then by a similar proof as \autoref{claim-7-2}, one has $t_i-j+1> 0$ for any $1\leq j \leq n$ according to \eqref{eqn-7-25}.
    In particular, taking $j=n$, we get $t_i\geq n$.
    Note also that
    $$\cs(\sF_n, E_n, q_n)+\cs(\sF_n, E_{n}, q_{n+1})=E_n^2-\cs(\sF_n, E_n, q_{n-1})=-1+\frac{n-1}{n}=-\frac{1}{n}.$$
    Hence $\beta_{q_n}(\sF_n)+\beta_{q_{n+1}}(\sF_n)>-1$.
    Together with \eqref{eqn-7-26} it follows that
    $$\sum_{q_j\in T_n}\beta_{q_j}(\sF_n)\geq-1+ \sum_{j=1}^{n-1}\beta_{q_j}(\sF_n)=-1+\frac{n-1}{n}\geq (n-1)-t_i+\frac{t_i-1}{t_i}.$$
    This proves \eqref{eqn-7-23}.
    
    \vspace{2mm}
    One can proves similarly \eqref{eqn-7-24} in case (iii),
    and in fact easier than \eqref{eqn-7-23}.
    Indeed, the first equality in \eqref{eqn-7-24} follows directly from \eqref{eqn-7-17}.
    Moreover, by \eqref{eqn-2-7} and \eqref{eqn-2-6}, if $n\geq 4$, then
    $$\left\{\begin{aligned}
    &\cs(\sF_n, E_{n}, q_{n-1})^{-1}=\cs(\sF_n,\ol E_{n-1}, q_{n-1})=\ol E_{n-1}^2=-2,\\
    &\cs(\sF_n, \ol E_{j+1}, q_{j})^{-1}=\cs(\sF_n, \ol E_j, q_{j})=\ol E_j^2-\cs(\sF_n, \ol E_j, q_{j-1})=-\frac{j+1}{j}, \quad \forall\,2\leq j \leq n-3,\\
    &\cs(\sF_n, E_{n}, q_{n-2})^{-1}=\cs(\sF_n, \ol E_{n-2}, q_{n-2})=\ol E_{n-2}^2-\cs(\sF_n, \ol E_{n-2}, q_{n-3})=-\frac{2n-1}{n-1},\\
    &\cs(\sF_n, E_{n}, q_{n})=E_n^2-\cs(\sF_n, E_{n}, q_{n-1})-\cs(\sF_n, E_{n}, q_{n-2})=-\frac{1}{2(2n-1)}.
    \end{aligned}\right.$$
    Moreover, similar as above, one shows that $D:=\bigcup\limits_{j=1}^{n-2}\ol E_j$ and $\ol E_{n-1}$ are contained in the support of the negative part $N_n$.
    In particular, $q_{j}\in T_n$ for $1\leq j \leq n-1$.
    Note that $q_n$ is reduced, and hence $q_n\not\in T_n$.
    $$\sum_{q_j\in T_n}\beta_{q_j}(\sF_n)
    =\sum_{j=1}^{n-1}\beta_{q_j}(\sF_n)=\frac12+\frac{1}{(2n-1)(n-1)}+\sum_{j=1}^{n-3}\frac{1}{j(j+1)}=\frac{1}{(2n-1)(n-1)}+\frac{3n-8}{2(n-2)}.$$
    If $n=3$, then $$\cs(\sF_3,\ol E_1,q_1)=\ol E_1^2=-3,\quad
    \cs(\sF_3,\ol E_2,q_2)=\ol E_2^2=-2,\quad \cs(\sF_3,\ol E_3,q_3)=\ol E_3^2+\frac12+\frac13=-\frac16.$$
    Hence
    $$\sum_{q_j\in T_3}\beta_{q_j}(\sF_3)
    =\sum_{j=1}^{2}\beta_{q_j}(\sF_3)=\frac12+\frac13=\frac56.$$
    The inequality $t_i\geq n-2$ can be proved using a similar argument as the proof of $t_i\geq n-1$ in case (ii) above, cf. \eqref{eqn-7-25}.
    
    \vspace{2mm}
    Finally, we prove that the case (iii) in \autoref{cor-7-1} can not happen.
    This follows from the formulas \eqref{eqn-2-7} and \eqref{eqn-2-6}.
    In fact, if the case (iii) in \autoref{cor-7-1} happens,
    then one obtains a chain of rational curves $\bigcup\limits_{j=1}^{n-1} \ol E_j$
    with $\ol E_j^2=-2$ for any $1\leq j \leq n-1$.
    $${\setlength{\unitlength}{6mm}
    	\begin{tikzpicture}
    	\draw[very thick, dashed] (10.5,1) -- (10.5,3);
    	\draw[thick] (11.5,1) -- (8.5,3);
    	\draw[thick] (10,3) -- (7,1);
    	\filldraw[black] (9.25,2.5) circle (1.5pt);
    	
    	\filldraw[black] (8.25,1.82) circle (1.5pt);
    	\draw[thick] (9.5,1) -- (6.5,3);
    	\draw[thick] (8,3) -- (6.5,2);
    	\filldraw[black] (7.25,2.5) circle (1.5pt);

    	\filldraw[black] (9,2.2) node[anchor=west]{$q_1$};
    	\filldraw[black] (8,1.5) node[anchor=west]{$q_2$};
    	\filldraw[black] (7,2.2) node[anchor=west]{$q_3$};
    	\filldraw[black] (10.1,0.7) node[anchor=west]{$\ol F_{i0}$};
    	\filldraw[black] (11.4,0.7) node[anchor=west]{$\ol E_1$};
    	\filldraw[black] (9,0.7) node[anchor=west]{$\ol E_3$};
    	\filldraw[black] (9.5,3.3) node[anchor=west]{$\ol E_2$};
    	\filldraw[black] (7.5,3.3) node[anchor=west]{$\ol E_4$};
    	
    	\filldraw[black] (6.2,2) circle (1pt);
    	\filldraw[black] (6,2) circle (1pt);
    	\filldraw[black] (5.8,2) circle (1pt);
    	\filldraw[black] (5.6,2) circle (1pt);
    	\filldraw[black] (5.4,2) circle (1pt);
    	\filldraw[black] (5.2,2) circle (1pt);
    	
    	\filldraw[black] (4.25,1.82) circle (1.5pt);
    	\draw[thick] (5.5,1) -- (2.5,3);
    	\draw[very thick, dashed] (4,3) -- (1,1);
    	
    	\filldraw[black] (5,0.7) node[anchor=west]{$\ol E_{n-1}$};
    	\filldraw[black] (3.5,3.3) node[anchor=west]{$E_n=\ol E_n$};
    	
    	\filldraw[black] (3.6,1.5) node[anchor=west]{$q_{n-2}$};
\end{tikzpicture}}$$
    By \eqref{eqn-2-7} and \eqref{eqn-2-6},
    $$\left\{\begin{aligned}
    &\cs(\sF_n,\ol E_1,q_1)=\ol E_1^2=-2,\\
    &\cs(\sF_n, \ol E_{j+1}, q_{j})^{-1}=\cs(\sF_n, \ol E_j, q_{j})=\ol E_j^2-\cs(\sF_n, \ol E_j, q_{j-1})=-\frac{j+1}{j}, \quad \forall\,2\leq j \leq n-2,\\
    &\cs(\sF_n,\ol E_{n-2},q_{n-2})^{-1}=\cs(\sF_n,\ol E_{n-1},q_{n-2})=\ol E_{n-1}^2=-2.
    \end{aligned}\right.$$
    This gives a contradiction by computing
    the invariant $\cs(\sF_n,\ol E_{n-2},q_{n-2})$.
    Indeed, taking $j=n-3$ in the second equality above, $\cs(\sF_n,\ol E_{n-2},q_{n-2})=-\frac{n-2}{n-3}$;
    while $\cs(\sF_n,\ol E_{n-2},q_{n-2})=-\frac12$ by the last equality above.
    \end{proof}
	
	We now come back to the proof of \eqref{eqn-5-11}.
	Let $\pi:\,S' \to S_0$ be a sequence of blowing-ups resolving
	the singularities (including the infinitely close ones) over $\text{Supp}(Z_{v0})$, such that all the singularities of
	$\sF'=\pi^*(\sF_0)$ on $S'$ are non-degenerate.
	Since there is exactly one singularity $p_0$ of $\sF_0$ on $C_0$,
	it follows that there is at most one fiber $F_{i0} \subseteq \text{Supp}(Z_{v0})$ which is $\sF_0$-invariant.
	Let $T$ be the set of singularities of $\sF'$,
	which are either non-reduced or on $\text{Supp}(N')$,
	where $N'$ is the negative part in the Zariski decomposition $K_{\sF'}=P'+N'$.
	Then $\pi^{-1}(p_0) \in T$ by the construction with \eqref{eqn-7-51}.
    Moreover,
	$$1=Z(\sF_0,C_0)=2+K_{\sF_0}\cdot C_0=p_g(\sF)+1+C_0^2+\sum_{i=1}^{k}t_i,\quad \Longrightarrow \quad
	e:=-C_0^2=p_g(\sF)+\sum\limits_{i=1}^{k}t_i.$$
	By \autoref{prop-2-2}, the eigenvalue $\lambda_{q_0}=-e$.
	Hence $\beta_{\pi^{-1}(p_0)}(\sF')=\beta_{p_0}(\sF_0)=\frac1e$.
	Note also that $$K_{\sF_0}^2=(M_0+C_0+\sum_{i=1}^{k}t_iF_{i0})
	=\big(p_g(\sF)-1\big)-1+\sum_{i=1}^{k}t_i=(p_g(\sF)-2)+\sum_{i=1}^{k}t_i.$$
	
	We consider first the case when none of $F_{i0} \subseteq \text{Supp}(Z_{v0})$ is $\sF_0$-invariant.
	Then by \eqref{eqn-7-12} and \eqref{eqn-7-27},
	$$\begin{aligned}
	\vol(\sF)\geq &\,\beta_{\pi^{-1}(p_0)}(\sF')+K_{\sF'}+\sum_{q\in T\setminus \{\pi^{-1}(p_0)\}}\beta_{q}(\sF')\\
	\geq&\, \frac1e+K_{\sF_0}^2+\sum_{i=1}^{k}\left(-t_i+\max\Big\{\frac12,\frac{t_i-1}{t_i}\Big\}\right)\\
	=&\,\frac{1}{e}+\big(p_g(\sF)-2\big)+\sum_{i=1}^{k}\max\Big\{\frac12,\frac{t_i-1}{t_i}\Big\}
	\geq \left\{\begin{aligned}
	&p_g(\sF)-1, &&\text{if~}k\geq 2\\
	&p_g(\sF)-\frac32+\frac{1}{p_g(\sF)+2}, &&\text{if~}k=1.
	\end{aligned}\right.
	\end{aligned}$$
	
	We consider next the case when one of $F_{i0} \subseteq \text{Supp}(Z_{v0})$, saying $F_{10}$ is $\sF_0$-invariant, and the rest $F_{i0}$'s (might be empty) are not $\sF_0$-invariant.
	Then by \eqref{eqn-7-12}, \eqref{eqn-7-29} and \eqref{eqn-7-27},
	$$\begin{aligned}
	\vol(\sF)\geq &\,\beta_{\pi^{-1}(p_0)}(\sF')+K_{\sF'}+\sum_{q\in T\setminus \{\pi^{-1}(p_0)\}}\beta_{q}(\sF')\\
	\geq&\, \frac1e+K_{\sF_0}^2-1+\frac{e^2-e+1}{e(2e-1)}+\sum_{i=2}^{k}\left(-t_i+\max\Big\{\frac12,\frac{t_i-1}{t_i}\Big\}\right)\\
	=&\,p_g(\sF)+t_1-\frac52+\frac{3}{2(2e-1)}+\sum_{i=2}^{k}\max\Big\{\frac12,\frac{t_i-1}{t_i}\Big\}\\
	\geq&\, \left\{\begin{aligned}
	&p_g(\sF)-1, &&\text{if~}k\geq 2\\
	&p_g(\sF)-\frac32+\frac{3}{2\big(2p_g(\sF)+1\big)}, &&\text{if~}k=1.
	\end{aligned}\right.\\
	\end{aligned}$$
	If the equality holds, then $k=1$ and $t_1=1$.
	This completes the proof of \eqref{eqn-5-11}.
	\end{proof}	
	
	\section{Examples}\label{sec-example} 
	In this section, we will construct several examples.
	\autoref{exam-6-1}, \autoref{exam-6-2} and \autoref{exam-6-3} show that the three Noether type inequalities in \autoref{thm-main} are all sharp;
	while \autoref{exam-6-4} illustrates that the Chern numbers can not be bounded from above by the geometric genus $p_g(\sF)$.
	
	\begin{example}\label{exam-6-1}
		For any integer $n\geq 1$, we construct a sequence of reduced foliated surfaces $(S_n,\sF_n)$ of general type,
		such that the volume $\vol(\sF_n)=n$ and the geometric genus $p_g(\sF_n)=n+2$,
		and hence the following equality holds
		\begin{equation}\label{eqn-6-1}
			\vol(\sF_n) = p_g(\sF_n)-2.
		\end{equation}
		
		\vspace{2mm}
		Let $\sF_0$ be a foliation of degree two on $\bbp^2$ with reduced singularities.
		Such a foliation exists, cf. \cite[Proposition\,3.2]{ls-20}. In fact, any foliation of degree $d$ on $\bbp^2$ can be generated by a vector of the form (\cite{go-89,bru-04})
		$$v=\big(P(x,y)+xR(x,y)\big)\frac{\partial}{\partial x}+\big(Q(x,y)+yR(x,y)\big)\frac{\partial}{\partial y},$$
		where $(x,y)$ is an affine coordinate of $\bbp^2$, $P(x,y), Q(x,y)$ are polynomials of degree $\leq d$,
		and $R(x,y)$ is a homogeneous polynomial of degree $d$ (plus some nondegeneracy conditions).
		Let
		\begin{equation}\label{eqn-6-3}
			v_{\alpha,\gamma}=x\Big(-\alpha^2+x^2+y^2\Big)\frac{\partial}{\partial x}
			+y\Big(2\alpha\gamma+(2\alpha+\gamma) y
			+x^2+y^2\Big)\frac{\partial}{\partial y},
		\end{equation}
		where $\alpha,\gamma \in \mathbb{C} \setminus \mathbb{Q}$ are general.
		The foliation $\sF_0$ defined by $v_{\alpha,\gamma}$ admits $7$ singularities:
		$$\left\{\begin{aligned}
			&(0,0),\quad (0,-2\alpha),\quad (0,-\gamma),\quad (\alpha,0),\quad (-\alpha,0),\\
			&\Big(~\frac{\alpha \sqrt{3(\alpha^2-\gamma^2)}}{2\alpha+\gamma},~\frac{-\alpha(\alpha+2\gamma)}{2\alpha+\gamma} ~\Big),\quad
			\Big(~\frac{-\alpha \sqrt{3(\alpha^2-\gamma^2)}}{2\alpha+\gamma},~ \frac{-\alpha(\alpha+2\gamma)}{2\alpha+\gamma}~\Big)
		\end{aligned}\right\}.$$
		The eigenvalues at these $7$ singularities are respectively equal to
		$$\left\{\begin{aligned}
			&\lambda_1=\frac{-\alpha}{2\gamma},\quad \lambda_2=\frac{4\alpha-2\gamma}{3\alpha},\quad \lambda_3=\frac{\gamma(\gamma-2\alpha)}{\gamma^2-\alpha^2},\quad \lambda_4=\lambda_5=\frac{\alpha+2\gamma}{2\alpha},\\
			&\lambda_6=\lambda_7=\frac{(\alpha-2\gamma)(2\alpha-\gamma)-\sqrt{(\alpha-2\gamma)^2(2\alpha-\gamma)^2-24\alpha(\alpha+2\gamma)(\gamma^2-\alpha^2)\,}~}{(\alpha-2\gamma)(2\alpha-\gamma)+\sqrt{(\alpha-2\gamma)^2(2\alpha-\gamma)^2-24\alpha(\alpha+2\gamma)(\gamma^2-\alpha^2)\,}~}
		\end{aligned}\right\}.$$
		If $\alpha,\gamma$ are sufficiently general (for instance if $\{\alpha,\gamma \}\in \mathbb{C} \setminus \mathbb{Q}$ are algebraically independent over $\mathbb{Q}$), then the eigenvalue $\lambda_i \in \mathbb{C} \setminus \mathbb{Q}$ for $1\leq i\leq 7$, and hence these are all reduced singularities.
		Moreover, the two lines $L_0:=\{x=0\}$ and $L_{\infty}:=\{y=0\}$ are both $\sF_0$-invariant in view of \eqref{eqn-6-3}.
		Let $\sigma:\,S_1 \to \bbp^2$ be the blowing-up centered at $(0,0)$,
		and $\sF_1$ be the induced foliation on $S_1$.
		Then $\sF_1$ is still reduced, and
		$$K_{\sF_1} = \sigma^*K_{\sF_0} =\sigma^*\mathcal{O}_{\bbp^2}(1) \sim C_0+F,$$
		where '$\sim$' stands for the linear equivalence, $C_0=\mathcal{E}\subseteq S_1$ is the unique section (also the exceptional curve of $\sigma$) with $C_0^2=-1$,
		and $F_1$ is a general fiber of the induced ruling $f_1:\,S_1 \to \bbp^1$.   	
		Hence 
		$$\vol(\sF_1)=\vol(\sF_0)=1,\qquad p_g(\sF_1)=p_g(\sF_0)=3.$$
		It follows that the foliations $\sF_0$ and $\sF_1$ (they are birational to each other) satisfy \eqref{eqn-6-1}.
		By construction, $S_1$ is isomorphic to the Hirzebruch surface
		$\mathbb{P}_{\mathbb{P}^1}\big(\mathcal{O}_{\mathbb{P}^1} \oplus \mathcal{O}_{\mathbb{P}^1}(1)\big)$.
		Let $F_0, F_{\infty}$ be the strict transforms of the two lines $L_0,L_{\infty}$ respectively.
		As both $L_0$ and $L_{\infty}$ are $\sF_0$-invariant,
		it follows that both $F_0$ and $F_{\infty}$ are $\sF_1$-invariant.
		Let $\pi_n:\,\bbp^1 \to \bbp^1$ be the cyclic cover of degree $n\geq 2$ branched over $f_1(F_0)$ and $f_1(F_{\infty})$, and $S_n=S_1\times _{\pi_n} \bbp^1$ be the fiber product as follows.
		$$\xymatrix{ S_n \ar[rr]^-{\Pi_n} \ar[d]_-{f_n} && S_1 \ar[d]^-{f_1} \\
			\bbp^1 \ar[rr]^-{\pi_n} && \bbp^1}$$
		
		It is clear that $S_n \cong \mathbb{P}_{\mathbb{P}^1}\big(\mathcal{O}_{\mathbb{P}^1} \oplus \mathcal{O}_{\mathbb{P}^1}(n)\big)$.
		Let $\sF_n$ be the induced foliation on $S_n$.
		Then $\sF_n$ is a reduced foliation since the eigenvalues $\lambda_i \in \mathbb{C} \setminus \mathbb{Q}$ at each of the $7$ singularities of $\sF_0$.
		Since both $F_0$ and $F_{\infty}$ are $\sF_1$-invariant,
		one obtains that (cf. \cite[\S\,2.3(4)]{bru-04}),
		$$K_{\sF_n}=\Pi_n^*(K_{\sF_1})=\Pi_n^*(C_0+F_1)=\Pi_n^{-1}(C_0)+nF_n,$$
		where $F_n$ is a general fiber of $f_n$, and $\Pi_n^{-1}(C_0)$ is the strict transform of $C_0$ in $S_n$ satisfying $\Pi_n^{-1}(C_0)^2=-n$.
		Hence
		$$\vol(\sF_n)=n, \qquad p_g(\sF_n)=n+2.$$
		Therefore, the equality \eqref{eqn-6-1} holds for the foliation $\sF_n$.
		This completes the construction.
	\end{example}
	
	\begin{example}\label{exam-6-2}
		For any integer $n\geq 2$, we construct a sequence of reduced foliated surfaces $(S_n,\sF_n)$ of general type,
		such that the volume $\vol(\sF_n)=n-2+\frac{1}{n}$ and the geometric genus $p_g(\sF_n)=n$, and hence the following equality holds
		\begin{equation}\label{eqn-6-2}
			\vol(\sF_n) = p_g(\sF_n)-2+\frac{1}{p_g(\sF_n)}.
		\end{equation}
		
		\vspace{2mm}
		Let $n\geq 2$ and $S_n=\mathbb{P}_{\mathbb{P}^1}\big(\mathcal{O}_{\mathbb{P}^1} \oplus \mathcal{O}_{\mathbb{P}^1}(n)\big)$ 
		be the Hirzebruch surface admitting a unique section $C_0$ with $C_0^2=-n<0$.
		Let $f:\,S_n \to \bbp^1$ be the geometrical ruling on $S_n$, which makes $S_n$ as a $\bbp^1$-bundle over the projective line $\bbp^1$.
		Note that any $\bbp^1$-bundle over an affine space is necessarily trivial.
		One can obtain the Hirzebruch surface $S_n$ by gluing the two trivial $\bbp^1$-bundles $\mathbb{C} \times \bbp^1$ by
		$$\mathbb{C}\times \bbp^1 \qquad\longrightarrow \qquad\mathbb{C}\times \bbp^1,$$
		$$\qquad\big(x,\,[Y_0,Y_1]\big) \quad \mapsto \quad \left(\frac{1}{x},\, [Y_0,x^nY_1]\right).$$
		Equivalently, one may obtain $S_n$ by gluing four affine spaces as follows.
		Let $(x_i,y_i)$ be the affine coordinate on $U_i\cong \mathbb{C}^2=\mathbb{C} \times \mathbb{C}$ for $1\leq i \leq 4$.
		The transition functions on their overlaps are given by
		$$\left(x_1,\,y_1\right)=\left(x_2,\,\frac{1}{y_2}\right)=\left(\frac{1}{x_3},\,x_3^ny_3\right)=\left(\frac{1}{x_4},\,\frac{x_4^n}{y_4}\right).$$
		Moreover, $C_0\cap U_1=\{y_1=0\}$.

		Let $\sF_n$ be the foliation on $S_n$ defined by $\big\{(U_i,v_i)\big\}_{i=1}^4$, where
		$$\left\{\begin{aligned}
			&v_1=h(x_1,y_1) \frac{\partial}{\partial x_1}+y_1^2g(x_1,y_1)\frac{\partial}{\partial y_1},\\
			&v_2=y_2h(x_2,1/y_2) \frac{\partial}{\partial x_2}-y_2g(x_2,1/y_2)\frac{\partial}{\partial y_2},\\
			&v_3=-x_3h(1/x_3,x_3^ny_3) \frac{\partial}{\partial x_3}+\Big(ny_3h(1/x_3,x_3^ny_3)+x_3^{n-1}y_3^2g(1/x_3,x_3^ny_3)\Big)\frac{\partial}{\partial y_3},\\
			&v_4=-x_4y_4h(1/x_4,x_4^n/y_4) \frac{\partial}{\partial x_4}-\Big(ny_4^2h(1/x_4,x_4^n/y_4)+x_4^{n-1}y_4g(1/x_4,x_4^n/y_4)\Big)\frac{\partial}{\partial y_4},
		\end{aligned}\right.$$
		where $$h(x_1,y_1)=a_0+y_1\cdot\sum_{i=1}^{n} \tilde{a}_ix_1^i;\qquad
		g(x_1,y_1)=\sum_{j=1}^{n-1}b_jx_1^j+y_1\cdot \sum_{j=1}^{n-1}\tilde b_jx_1^j.$$
		Here $a_i,\tilde a_i, b_j, \tilde b_j$ are some complex numbers satisfying certain non-degeneracy conditions to insure these $v_i$'s contain no one-dimensional zeros.
		Then one checks easily by \eqref{eqn-2-8} that
		$$K_{\sF_n}=C_0+(n-1) F,$$
		where $F$ is a general fiber of the ruling $f:\,S_n \to \bbp^1$.
		It follows that the Zariski decomposition of $K_{\sF_n}$ is the following.
		$$K_{\sF_n}=P+N,$$
		where $P=(n-1)F+\frac{n-1}{n}C_0$ is the nef part,
		and $N=\frac{1}{n}C_0$ is the negative part.
		Hence
		$$\vol(\sF_n)=n-2+\frac{1}{n},\qquad p_g(\sF_n)=n,$$
		from which the equality \eqref{eqn-6-2} follows immediately.
		To complete the construction, we should ensure that the foliation $\sF_n$ is reduced.
		A sufficient condition to ensure a singularity $p\in U_i$ of $\sF_n$ to be reduced is that
		both of the eigenvalues $\lambda_1,\lambda_2$ of $(Dv_i)(p)$ are non-zero and the quotient $\lambda_1/\lambda_2$ is not a positive rational number.
		This should be satisfied for a general choice of these complex numbers $\{a_i,\tilde a_i, b_j, \tilde b_j\}$, similar to the situation on $\bbp^2$,
		\cite[Proposition\,3.2]{ls-20}.
		For instance, one can take
		$$h(x_1,y_1)=1+y_1x_1^n,\qquad\qquad g(x_1,y_1)=\alpha x_1^{n-1}+y_1,\quad\text{where~}\alpha \in \mathbb{C} \setminus \mathbb{Q}.$$
		The number of singularities of $\sF_n$ is 
		$\# Sing(\sF)=2n+2.$
		There are $2n-1$ singularities in $U_1$:
		$$Sing(\sF) \cap U_1=\{p_i=(\xi_i,-\alpha \xi_i^{n-1}),\,~i=1,\cdots,2n-1\},\qquad \text{where~}\xi_i^{2n-1}=\frac{1}{\alpha}.$$
		By direct computation, the two eigenvalues of $(Dv_1)(p_i)$ are $$\frac{(\alpha-n) \pm \sqrt{(\alpha+n)^2+4(n-1)\alpha^2\xi_i^{2n-2}\,}\,}{2\xi_i}.$$
		Hence these $p_i$'s are reduced singularities of $\sF_n$.
		The rest three singularities of $\sF_n$ are all on the fiber
		$$F_{\infty} = \{x_3=0\} \cup \{x_4=0\}.$$
		To be explicit, the rest three singularities are
		$$\Big\{(0,0), \Big(0, \frac{-n}{n+\alpha}\Big)\Big\} \subseteq U_3, \quad \text{~and~}\quad \big\{(0,0)\big\} \subseteq U_4.$$
		The eigenvalues of $\sF_n$ at these three singularities are respectively
		equal to $-n$, $\frac{n(n+\alpha)}{\alpha}$ and $n+\alpha$.
		Hence the foliation $\sF_n$ is reduced as required.
	\end{example}
	
	\begin{remark}\label{rem-6-2}
		By \eqref{eqn-noe3}, the foliations in the above example are all transcendental.
		Moreover, it holds that $|K_{\sF_n}|=|(n-1) F|+C_0$.
		Hence the moving part of $|K_{\sF_n}|$ is $M=(n-1)F$ and the positive part in its Zariski decomposition is $P=(n-1)F+\frac{n-1}{n}C_0$. It follows that
		$$\vol(\sF)=n-2+\frac{1}{n} =P\cdot M< K_{\sF} \cdot M.$$
	\end{remark}
	
	\begin{example}\label{exam-6-3}
		For any integer $g\geq 2$, we construct an algebraically integral foliated surfaces $(S,\sF)$ of general type which is induced by a fibration of genus $g$,
		such that
		\begin{equation}\label{eqn-6-4}
		\left\{\begin{aligned}
		p_g(\sF)&\,=g,\\
		\vol(\sF) &\,=\frac{2g(g-1)}{2g+1}= p_g(\sF)-\frac32+\frac{3}{2\big(2p_g(\sF)+1\big)}.
		\end{aligned}\right.
		\end{equation}
	\end{example}
	\vspace{2mm}
	Let $g\geq 2$ and $D_0=\{x\big(1-y^{2g}\big)+y^{2g+1}=0\}\subseteq Y_0=\bbp^1\times \bbp^1$,
	where $x,y$ are respectively the affine coordinates of the first and the second factor $\bbp^1$ of $Y_0$.
	Let $pr_1:\,Y_0 \to \bbp^1$ be the projection to the first factor.
	Then the restricted map
	$pr_1\big|_{D_0}:~D_0 \to \bbp^1$
	has $2g+1$ ramified points: $p_0=(0,0)$ and $p_i=\Big(\frac{(2g+1)\xi_i}{2g},\xi_i\Big)$, where $\xi_i$'s are mutually different and $\xi_i^{2g}=2g+1$ for $1\leq i \leq 2g$.
	Moreover, the ramification indices are respectively $r_0=2g$ and $r_i=1$ for $1\leq i \leq 2g$.
	Let $\Gamma_0=\{x=0\}$ and $C_0=\{y=0\}$.
	Then $R_0=D_0+\Gamma_0+C_0$ is two-divisible.
	Let $\pi_0:\,S_0 \to Y$ be the double cover ramified exactly over $R_0$,
	and $f_0=pr_1\circ \pi_0:\,S \to \bbp^1$.
	Let $\rho:\,S \to S_0$ be the desingularization and $f=f_0\circ \rho$.
	$$\xymatrix{S \ar[rr]^-{\rho} \ar[d]^-{f}
	&& S_0 \ar[rr]^-{\pi_0} \ar[d]^-{f_0} 
&& Y_0 \ar[d]^-{pr_1}\\
\bbp^1 \ar@{=}[rr] &&\bbp^1 \ar@{=}[rr] &&
\bbp^1}$$
Since $R_0$ has a unique singularity $p_0=(0,0)$,
$S_0$ is singular exactly over $p_0$.
Such a singularity can be resolved canonically as follows, cf. \cite[\S\,V.22]{bhpv}.
$${\setlength{\unitlength}{6mm}
	\begin{tikzpicture}
	[place/.style={circle,draw,fill,inner sep=0.7mm},
	place2/.style={circle,draw,inner sep=0.7mm},
	place3/.style={circle,draw,fill,inner sep=0.3mm},]
	
	\draw[thick] (1,6) -- (1,8);
	\draw[thick] (0.4,7) -- (2.5,7);

	\filldraw[black] (0.7,5.7) node[anchor=west]{$\Gamma_0$};
	\filldraw[black] (1,7) circle (1.5pt);
	\filldraw[black] (0.5,6.8) node[anchor=west]{$p_0$};
	\filldraw[black] (2.5,6.9) node[anchor=west]{$C_0$};
	\draw[domain=1.15:2.85] plot(2*\x*\x-8*\x+9,\x+5);
	\filldraw[black] (2.4,6) node[anchor=west]{$D_0$};

	\draw[->] (4.7,7) -- (3.7,7);
	\filldraw[black] (3.8,7.3) node[anchor=west]{\Large  $\sigma_1$};
	
	\node[place] (v1) at (6.5,7) [label=above:{$\big(1,-(2g+1)\big)$}] [label=below:{$\ol \Gamma_0$}]  {};
	\node[place] (v2) at (9.5,7) [label=above:{$\big((2g+1),-1\big)$}]  {};
	\node[place] (v3) at (9.5,6) [label=right:$\ol D_0$]  {};
	\node[place] (v4) at (11.8,7) [label=above:{$\big(2g,-2\big)$}]  {};
	
	\node[place3] at (12.5,7) {};
	\node[place3] at (12.8,7) {};
	\node[place3] at (13.1,7) {};

	\node[place] (v5) at (14,7) [label=above:{$\big(2,-2\big)~$}]  {};
	\node[place] (v6) at (15.5,7) [label=above:{$~\big(1,-2\big)$}]  {};
	\node[place] (v7) at (15.5,6) [label=right:$\ol C_0$]  {};
	
	\draw (v1)--(v2); \draw (v2)--(v3); \draw (v2)--(v4);
	\draw (v5)--(v6); \draw (v6)--(v7);

	\draw[->] (9,4.7) -- (9,5.5);
	\filldraw[black] (9,5.1) node[anchor=west]{\Large  $\sigma_2$};
	
	\node[place] (v8) at (1,3.5) [label=above:{$\big(1,-(2g+2)\big)$}] [label=below:{$\ol \Gamma_0$}]  {};
	\node[place2] (v9) at (3.8,3.5) [label=above:{$\big((2g+2),-1\big)$}]  {};
	\node[place] (v10) at (6.5,3.5) [label=above:{$\big((2g+1),-4\big)$}]  {};
	\node[place2] (v11) at (9.3,3.5) [label=above:{$\big((4g+1),-1\big)$}]  {};
	\node[place2] (v12) at (6.5,2.7) [label=right:{$\big((2g+1),-1\big)$}]  {};
	\node[place] (v13) at (6.5,1.9) [label=right:{$\ol D_0$}]  {};
	
	\node[place3] at (10.5,3.5) {};
	\node[place3] at (10.8,3.5) {};
	\node[place3] at (11.1,3.5) {};
	
	\node[place] (v14) at (12,3.5) [label=above:{$\big(2,-4\big)~$}]  {};
	\node[place2] (v15) at (13.5,3.5) [label=above:{$\big(3,-1\big)~$}]  {};
	\node[place] (v16) at (15,3.5) [label=above:{$~\big(1,-4\big)$}]  {};
	\node[place2] (v17) at (15,2.7) [label=right:{$\big(1,-1\big)$}]  {};
	\node[place] (v18) at (15,1.9) [label=right:$\ol C_0$]  {};
	
	\draw (v8)--(v9); \draw (v9)--(v10); \draw (v10)--(v11);
	\draw (v10)--(v12); \draw (v12)--(v13); \draw (v14)--(v15);
	\draw (v15)--(v16); \draw (v16)--(v17); \draw (v17)--(v18);

	\draw[->] (9,0.7) -- (9,1.5);
	\filldraw[black] (9,1.1) node[anchor=west]{\Large  $\pi$};
	
	\node[place] (v28) at (1,-0.5) [label=above:{$\big(2,-(g+1)\big)$}] [label=below:{$\ol \Gamma_0$}]  {};
	\node[place2] (v29) at (3.8,-0.5) [label=above:{$\big((2g+2),-2\big)$}] 
	[label=below:{$E_{4g+3}$}] {};
	\node[place] (v30) at (6.5,-0.5) [label=above:{$\big((4g+2),-2\big)$}] [label=below:{$\quad \qquad E_{4g+2}$}] {};
	\node[place2] (v31) at (9.3,-0.5) [label=above:{$\big((4g+1),-2\big)$}] [label=below:{$E_{4g+1}$}] {};
	\node[place2] (v32) at (6.5,-1.7) [label=right:{$\big((2g+1),-2\big)$}] [label=below:{$E_{4g+4}$}]
	  {};
	
	\node[place3] at (10.5,-0.5) {};
	\node[place3] at (10.8,-0.5) {};
	\node[place3] at (11.1,-0.5) {};
	
	\node[place] (v34) at (12,-0.5) [label=above:{$\big(4,-2\big)~$}] [label=below:{$E_4$}] {};
	\node[place2] (v35) at (13.5,-0.5) [label=above:{$\big(3,-2\big)~$}] [label=below:{$E_3$}] {};
	\node[place] (v36) at (15,-0.5) [label=above:{$~\big(2,-2\big)$}] [label=right:{$E_2$}] {};
	\node[place2] (v37) at (15,-1.7) [label=right:{$\big(1,-2\big)$}] 
	[label=below:{$E_1$}] {};
	
	\draw (v28)--(v29); \draw (v29)--(v30); \draw (v30)--(v31);
	\draw (v30)--(v32); 
	\draw (v34)--(v35);
	\draw (v35)--(v36); \draw (v36)--(v37); 
\end{tikzpicture}}$$

In the above, $\sigma_1:\,Y_1 \to Y_0$ is the resolution of the ramification divisor $R_0$ into normal crossing one $R_1$, and $\sigma_2:\,Y \to Y_1$ is the further resolution to the smooth one $R_2$.
Except the ramification divisor $R_0$, we draw the rest ones using the dual graphs.
The number $(a,-b)$ over a marked point stands that the corresponding curve $C$ is of multiplicity $a$ in the fibers and $C^2=-b$.
A solid point means that the corresponding curve is ramified; while a hollow point means the corresponding curve is not ramified.
Note that the local intersection $I_{p_0}(\Gamma_0,D_0)=2g+1$, there are $2g+1$ blowing-ups contained in $\sigma_1$.
By abuse of notations, we always use $\ol \Gamma_0$ (res. $\ol C_0$ and $\ol D_0$)
to denote the strict transform of $\Gamma_0$ (resp. $C_0$ and $D_0$).
Remark that the curves $\ol C_0$ and $\ol D_0$ are not contained in the fibers.
Moreover, in the last step, we only draw the curves contained in the fiber $F_q:=f^{-1}(q)$, where $q=pr_1(p_0)\in \bbp^1$.
$$\xymatrix{S \ar[rr]^-{\rho_2} \ar[d]^-{\pi}
	&& S_1 \ar[rr]^-{\rho_1} \ar[d]
	&& S_0 \ar[d]^-{\pi_0}\\
	Y \ar[rr]^-{\sigma_2} &&Y_1 \ar[rr]^-{\sigma_1} &&
	Y_0}$$

Let $\Pi=\pi_0\circ\rho=\sigma_1\circ\sigma_2\circ\pi$.
By the formulas of double covers (cf. \cite[\S\,V.22]{bhpv}),
$$K_{S/\bbp^1}=\Pi^*\Big(K_{Y_0/\bbp^1}+\frac{R_0}{2}\Big)=\Pi^*\big(\Gamma_0+(g-1)C_0\big).$$
By construction, $f$ is relatively minimal,
and all the fibers of $f$ are normal crossing (in fact, except the fiber $F_q=f^{-1}(q)$, all the other fibers of $f$ are semi-stable).
Let $\sF$ be the induced foliation. Then $\sF$ is relatively minimal, and by \eqref{eqn-2-3},
$$K_{\sF}=K_{S/\bbp^1}-\big(F_q-F_{q,red}\big)=\Pi^*\big((g-1)C_0\big)+F_{q,red},$$
where $F_{q,red}$ is the reduced part of $F_{q}=\Pi^*(\Gamma_0)$.
Thus the Zariski decomposition of $K_{\sF}$ is $K_{\sF}=P+N$, where
$$\begin{aligned}
P&\,=\Pi^*\big((g-1)C_0\big)+\sum_{i=1}^{4g+1}\frac{iE_i}{4g+2}+E_{4g+2}+\frac{(2g-1)\ol \Gamma_0}{2g+1}+\frac{2gE_{4g+3}}{2g+1}+\frac{E_{4g+4}}{2},\\
N&\,=\sum_{i=1}^{4g+1}\frac{(4g+2-i)E_i}{4g+2}+\frac{2\,\ol \Gamma_0}{2g+1}+\frac{E_{4g+3}}{2g+1}+\frac{E_{4g+4}}{2}.
\end{aligned}$$
Indeed, one checks directly that
$\text{Supp}(N)=F_{q,red}-E_{4g+2}$, and hence the intersection matrix of $\text{Supp}(N)$ is negative definite.
Moreover,
$$P\cdot \Pi^*(C_0)=\frac{(2g-1)\ol \Gamma_0}{2g+1}\cdot \Pi^*(C_0)=\frac{2g-1}{2g+1};\qquad
P\cdot E_{4g+2}=\frac{g-1}{2g+1};\qquad
P\cdot E_i=0,\quad \forall\,i\neq 4g+2.$$
Hence $K_{\sF}=P+N$ is the Zariski decomposition of $K_{\sF}$.
Therefore,
\begin{equation}\label{eqn-8-4}
	\vol(\sF)=P^2=P\cdot \Big(\Pi^*\big((g-1)C_0\big)+ E_{4g+2}\Big)=\frac{(2g-1)(g-1)}{2g+1}+\frac{g-1}{2g+1}=\frac{2g(g-1)}{2g+1}.
\end{equation}
In particular, $\sF$ is of general type.
Note that any irreducible curve in the support of $N$ must be contained in the fixed part of $|K_{\sF}|$,
i.e., $\text{Supp}(N)=F_{q,red}-E_{4g+2}$ is contained in the fixed part of $|K_{\sF}|$.
Hence
$$\begin{aligned}
|K_{\sF}|\,&=\Big|K_{\sF}-\big(F_{q,red}-E_{4g+2}\big)\Big|+\big(F_{q,red}-E_{4g+2}\big)\\
&=\Big|\Pi^*\big((g-1)C_0\big)+E_{4g+2}\Big|+\big(F_{q,red}-E_{4g+2}\big)\\
&=\Big|\Pi^*\big((g-1)C_0\big)\Big|+E_{4g+2}+\big(F_{q,red}-E_{4g+2}\big)=\Big|\Pi^*\big((g-1)C_0\big)\Big|+F_{q,red}.
\end{aligned}$$
In particular,
$$p_g(\sF)=h^0(S,K_{\sF})=h^0\Big(S,\Pi^*\big((g-1)C_0\big)\Big)=h^0\big(Y,(g-1)C_0\big)=g.$$
Combining this with \eqref{eqn-8-4},
we obtain a foliation $\sF$ satisfying \eqref{eqn-6-4} as required.

\begin{example}\label{exam-6-4}
	In this example, we construct a sequence of foliated surfaces $(S_n,\sF_n)$ such that their geometric genera are fixed, while the three Chern numbers $\chi(\sF_n)$, $c_1^2(\sF_n)$ and $c_2(\sF_n)$ go to the infinity when $n$ goes to the infinity.
\end{example}

Let $f_1:\,S_1 \to B_1$ be a semi-stable fibration of genus $g\geq 2$ with
$$K_{f_1}=K_{S_1/B_1}^2>0,\qquad \chi_f=\deg \big( (f_1)_*\mathcal{O}_{S_1}(K_{S_1/B_1})\big)>0,\qquad e_{f_1}>0.$$
Such a fibration clearly exists, cf. \cite{ch-88,xia-87}.
Let $\sF_1$ be the induced foliation. Then by \eqref{eqn-chern-1-2},
$$c_1^2(\sF_1)=\kappa(f_1)=K_{f_1}^2,\qquad \chi(\sF_1)=\lambda(f_1)=\chi_{f_1},\qquad c_2(\sF_1)=\delta(f_1)=e_{f_1}.$$
Let $\call\in \Pic(B_1)$ be a sufficiently ample divisor such that
$$\dim H^0\big(S_1,K_{\sF_1}\otimes f_1^*(\call^{-1})\big) =\dim H^0\big(B_1,(f_1)_*\big(\mathcal{O}_{S_1}(K_{\sF_1})\big)\otimes \call^{-1}\big)=0.$$
Clearly such an ample divisor exists.
Moreover, for such a divisor $\call$ and any integer $k\geq 1$,
$$\dim H^0\big(S_1,K_{\sF_1}\otimes f_1^*(\call^{-k})\big) =\dim H^0\big(B_1,(f_1)_*\big(\mathcal{O}_{S_1}(K_{\sF_1})\big)\otimes \call^{-k}\big)=0.$$
For any number $n\geq 2$, let $D_n\in \big|\call^{\otimes n}\big|$ be a general element, such that
$D_n$ is reduced and that any fiber over $D_n$ is smooth.
Then there is a cyclic cover $\pi_n:\,B_n \to B_1$ defined by the relation $D_n \sim \call^{\otimes n}$. Let $f_n:\,S_n \to B_n$ be the fiber-product.
$$\xymatrix{ S_n \ar[rr]^-{\Pi_n} \ar[d]_-{f_n} && S_1 \ar[d]^-{f_1}\\
	B_n \ar[rr]^-{\pi_n} && B_1
}$$
By construction, $f_n$ is semi-stable and $\Pi$ is a cyclic cover defined by the relation $$\Lambda_n=f_1^{-1}(D_n) \sim f_1^*(\call)^{\otimes n}.$$
Hence $(\Pi_n)_*\mathcal{O}_{S_n}=\sum\limits_{i=1}^{n-1} \call^{-i}$.
Let $\sF_n$ be the induced foliation on $S_n$.
Then
$$c_1^2(\sF_n)=nc_1^2(\sF_1),\qquad \chi(\sF_n)=n\chi(\sF_1),\qquad c_2(\sF_n)=nc_2(\sF_1).$$
Moreover, $K_{\sF_n}=K_{S_n/B_n}=\Pi_n^*(K_{S_1/B_1})=\Pi_n^*(K_{\sF_1})$, and hence
$$\begin{aligned}
p_g(\sF_n)&\,=\dim H^0\big(S_n,K_{\sF_n}\big)=\dim H^0\big(S_n,\Pi_n^*(K_{\sF_1})\big)\\
&\,=\dim H^0\big(S_1, \mathcal{O}_{S_1}(K_{\sF_1}) \otimes (\Pi_n)_*\mathcal{O}_{S_n} \big)=\sum_{i=0}^{n-1} \dim H^0\big(S_1, \mathcal{O}_{S_1}(K_{\sF_1}) \otimes \call^{-i}\big)\\
&\,=\dim H^0\big(S_1, \mathcal{O}_{S_1}(K_{\sF_1})\big)=p_g(\sF_1).
\end{aligned}$$
This completes the construction.

	
\newcommand{\etalchar}[1]{$^{#1}$}

\end{document}